\documentclass{amsart}
\usepackage{amssymb,amsmath, amsthm,latexsym}
\usepackage{graphics}
\usepackage{psfrag}
\usepackage{amscd}
\usepackage{graphicx}
\usepackage{here}
\newcommand{\cal}[1]{\mathcal{#1}}
\theoremstyle{plain}
\newtheorem{theorem}{Theorem}
\newtheorem*{coro}{Corollary}
\newtheorem{lemma}{Lemma}[section]
\newtheorem{theo}[lemma]{Theorem}
\newtheorem{proposition}[lemma]{Proposition}
\newtheorem{corollary}[lemma]{Corollary}
\parskip=\bigskipamount

\let\egthree=\phi
\let\phi=\varphi
\let\varphi=\egthree

\begin{document}
\title{Geometry of the mapping class groups II: \\
A biautomatic structure}
\author{Ursula Hamenst\"adt}
\date{November 30, 2009}

\begin{abstract}
We show that
the mapping class group ${\cal M\cal C\cal G}(S)$ of an
oriented surface $S$ of genus $g$ with $m$
punctures and $3g-3+m\geq 2$
admits a 
biautomatic structure.
We also show that various subgroups of ${\cal M\cal C\cal G}(S)$
are undistorted.
\end{abstract}

\maketitle

\tableofcontents

\section{Introduction}

Let $\Gamma$ be a finitely generated group and let
${\cal G}$ be a finite symmetric set of generators
for $\Gamma$.
Then ${\cal G}$ defines a \emph{word norm}
$\vert \,\vert$ on $\Gamma$ by assigning to
an element $g\in \Gamma$ the smallest length
$\vert g\vert$ of a word in the generating
set ${\cal G}$ which represents $g$.
Any two such word norms $\vert\,\vert,\vert\,\vert_0$
are \emph{equivalent}, i.e. there is a constant
$L>1$ (depending on the generating sets defining
the norms) such that
$\vert g\vert/L\leq \vert g\vert_0\leq
L\vert g\vert$ for all $g\in \Gamma$.
This means that the \emph{word metrics} $d,d_0$ on 
$\Gamma$ defined by $d(g,h)=\vert g^{-1}h\vert$
and $d_0(g,h)=\vert g^{-1}h\vert_0$ are 
bilipschitz equivalent. The action of
$\Gamma$ on itself by left translation preserves
a word metric.

Let $\Gamma$ be a finitely generated group equipped with 
a word metric $d$.
An \emph{automatic
structure} for 
$\Gamma$ consists of a finite \emph{alphabet} $A$,
a (not necessarily injective) map $\pi:A\to \Gamma$
and a \emph{regular language} $L$ in 
$A$ with the following properties (Theorem 2.3.5 of \cite{E92}).
\begin{enumerate}
\item
The set $\pi(A)$ generates $\Gamma$ as a semi-group.
\item 
Via concatenation,
every word $w$ in the alphabet $A$ is mapped
to a word $\pi(w)$ in the generators $\pi(A)$ of $\Gamma$.
 The restriction of the map $\pi$ to
the set of all words from the language $L$ maps
$L$ onto $\Gamma$.
\item There is a number $\kappa >0$ with the following
property. For all $x\in A$ and each word
$w\in L$ of length $k\geq 0$, the word
$wx$ defines
a path $s_{wx}:[0,k+1]\to\Gamma$ connecting the unit
element to $\pi(wx)$. 
Since $\pi(L)=\Gamma$, there is a word 
$w^\prime\in L$ of length $\ell >0$
with $\pi(w^\prime)=\pi(wx)$.
Let $s_{w^\prime}:[0,\ell]\to \Gamma$ be the corresponding
path in $\Gamma$. Then 
$d(s_{wx}(i),s_{w^\prime}(i))\leq \kappa$ 
for 
every $i\leq \min\{k+1,\ell\}$. 
\end{enumerate}

A \emph{biautomatic structure} for the group $\Gamma$
is an automatic structure $(A,L)$ with the following
additional property.
The alphabet $A$ 
admits an inversion $\iota$ with 
$\pi(\iota a)=\pi(a)^{-1}$ for all $a$, and 
\[d(\pi(x)s_w(i),s_{w^\prime}(i))\leq \kappa\]
for all $w\in L$ all $x\in A$, for 
any $w^\prime\in L$ with 
$\pi(w^\prime)=\pi(xw)$ and all $i$
(Lemma 2.5.5 of \cite{E92}).

Examples of groups which admit a biautomatic structure
are word hyperbolic groups \cite{E92} and groups which 
admit a proper cocompact action on a ${\rm CAT}(0)$ 
cubical complex (Corollary 8.1 of \cite{S06}).
The existence of a biautomatic structure for a group $\Gamma$
has strong implications for the structure of $\Gamma$.
For example, the word problem is solvable in linear
time, and solvable subgroups are virtually abelian
\cite{E92,BH99}.

Now let $S$ be an oriented surface of finite type, i.e. $S$ is a
closed surface of genus $g\geq 0$ from which $m\geq 0$
points, so-called \emph{punctures},
have been deleted. We assume that $3g-3+m\geq 2$,
i.e. that $S$ is not a sphere with at most $4$
punctures or a torus with at most $1$ puncture.
We then call the surface $S$ \emph{non-exceptional}.
The \emph{mapping
class group} ${\cal M\cal C\cal G}(S)$ of all isotopy classes of
orientation preserving self-homeomorphisms of $S$ is finitely
generated.
We refer to the
survey of Ivanov \cite{I02} for the basic properties
of the mapping class group and for references.
Mosher \cite{M95}
showed that the mapping class group ${\cal M\cal C\cal G}(S)$
admits an automatic structure.
The main goal of this paper is to strengthen this result and to show

\begin{theorem}\label{thm1} 
The mapping class group of a non-exceptional
surface of finite type admits a biautomatic structure.
\end{theorem}

Most of the known consequences of the existence of 
a biautomatic structure are well known for mapping class
groups. Perhaps
the only improvement of 
known results is a strengthening 
of a result of Hemion \cite{He79} who showed 
that the conjugacy problem in ${\cal M\cal C\cal G}(S)$ 
is solvable. From a biautomatic structure
we obtain a uniform
exponential bound on the length of a conjugating
element \cite{BH99}.

\begin{coro}\label{cor1}
Let ${\cal G}$ be a finite symmetric set of generators
of ${\cal M\cal C\cal G}(S)$ and let ${\cal F}({\cal G})$ be
the free group generated by ${\cal G}$.
There is a constant $\mu>0$ such that
words $u,v\in {\cal F}({\cal G})$ represent
conjugate elements of ${\cal M\cal C\cal G}(S)$ if and only if there
is a word $w\in {\cal F }({\cal G})$ of length
at most $\mu^{\max\{\vert u\vert,\vert v\vert\}}$ with
$w^{-1}uw=v$ in ${\cal M\cal C\cal G}(S)$.
\end{coro}

For pseudo-Anosov elements
in ${\cal M\cal C\cal G}(S)$, the conjugacy problem can be solved in
linear time \cite{MM00}, and Mosher \cite{M86} gave an explicit algorithm
for the solution in this case. For arbitrary elements of 
the mapping class group, the statement of the corollary
was recently
also established by Jing Tao \cite{JT09}.

A finite symmetric set ${\cal G}^\prime$ of generators
of a subgroup $\Gamma^\prime$
of a finitely generated group $\Gamma$
can be extended 
to a finite symmetric set of generators of $\Gamma$. Thus
for any two word norms $\vert\,\vert$ 
on $\Gamma$ and $\vert\,\vert^\prime$ of $\Gamma^\prime$ 
there is a number $L>1$ such that
$\vert g\vert\leq L\vert g\vert^\prime$ for every $g\in \Gamma^\prime$.
However, in general the word norm in $\Gamma$ of an element
$g\in \Gamma^\prime$ can not be estimated from below
by $L^\prime \vert g\vert^\prime$ for a universal constant
$L^\prime >0$.
Define the finitely generated 
subgroup $\Gamma^\prime$ of $\Gamma$
to be \emph{undistorted} in $\Gamma$ 
if there is a constant $c>1$
such that $\vert g\vert^\prime\leq c\vert g\vert$ for all
$g\in \Gamma^\prime$. 
Thus $\Gamma^\prime<\Gamma$ is undistorted if and 
only if the inclusion $\Gamma^\prime\to \Gamma$ 
is a quasi-isometric embedding.

There are 
subgroups of ${\cal M\cal C\cal G}(S)$ 
with particularly nice geometric descriptions. To begin with, 
an \emph{essential subsurface} of $S$ is a bordered surface $S_0$ which
is embedded in $S$ as a closed subset and such that
the following two additional requirements
are satisfied.
\begin{enumerate}
\item The homomorphism $\pi_1(S_0)\to \pi_1(S)$ 
induced by the inclusion is injective.
\item Each boundary component of $S_0$ is 
an \emph{essential} simple closed curve in  
$S$, i.e. a simple closed curve which 
is neither contractible nor freely homotopic into
a puncture.
\end{enumerate}
A finite index subgroup ${\cal M\cal C\cal G}_0(S_0)$ 
of the mapping class group
${\cal M\cal C\cal G}(S_0)$ of $S_0$ can be
identified with the subgroup of ${\cal M\cal C\cal G}(S)$ of all elements
which can be represented by a homeomorphism of $S$
fixing $S-S_0$ pointwise. We give an alternative proof of the
following result of 
Masur and Minsky (the result is
implicitly but not explicitly contained in Theorem 6.12 of
\cite{MM00}).

\begin{theorem}\label{thm2}
If $S_0\subset S$ is an 
essential subsurface of 
a non-exceptional surface $S$ 
of finite type
then ${\cal M\cal C\cal G}_0(S_0)< {\cal M\cal C\cal G}(S)$ is
undistorted.
\end{theorem}

In the case that the subsurface $S_0$ of $S$ is a disjoint union of 
essential annuli, the group 
${\cal M\cal C\cal G}_0(S_0)$
equals the free abelian group generated by the
Dehn twists about the center core curves of the
annuli. In this case the above theorem was
shown by Farb, Lubotzky and Minsky \cite{FLM01}
(see also \cite{H09} for an alternative proof).

Now let $S$ be a closed surface of genus $g\geq 2$ and let 
$\Gamma<{\cal M\cal C\cal G}(S)$ be any \emph{finite}
subgroup. By the solution of the Nielsen realization
problem \cite{Ke83}, $\Gamma$ can be realized as
a subgroup of the automorphism group of a marked complex
structure $h$ on $S$. Then the quotient 
$(S,h)/\Gamma$ is a compact Riemann surface,
and the projection
$S\to S/\Gamma$ is a branched covering ramified
over a finite set $\Sigma$ of points. 
Let $S_0=S/\Gamma-\Sigma$ and 
let $N(\Gamma)$ be the normalizer of $\Gamma$ in 
${\cal M\cal C\cal G}(S)$.
Then there is an exact sequence  
\[0\to \Gamma\to N(\Gamma)\to {\cal M\cal C\cal G}_0(S_0)\to 0\]
where ${\cal M\cal C\cal C\cal G}_0(S_0)$ is the subgroup
of the mapping class group of $S_0$ of all elements
which can be represented by a homeomorphism which lifts
to a homeomorphism of $S$ \cite{BH73}.
We use this to observe (compare \cite{RS07} for a similar
statement)

\begin{theorem}\label{thm3} Let $S$ be a closed
surface of genus $g\geq 2$ and let 
$\Gamma<{\cal M\cal C\cal G}(S)$ be a finite subgroup. Then 
the normalizer of $\Gamma$ is undistorted in 
${\cal M\cal C\cal G}(S)$.
\end{theorem}

There are other relations between mapping class groups
which can be described by exact sequences of groups.
An example of such a relation is as follows.
Let $S_0$ be any non-exceptional surface
of finite type and let $S$ be the surface obtained from
$S_0$ by deleting a single point $p$. Then there is an
exact sequence \cite{B74}
\[0\to \pi_1(S_0)\to {\cal M\cal C\cal G}(S)\stackrel{\Pi}\to
{\cal M\cal C\cal G}(S_0)\to 0.\]
The image in ${\cal M\cal C\cal G}(S)$
of an element $\alpha\in \pi_1(S_0)$ is
the mapping class obtained by dragging the puncture 
$p$ of $S$ 
along a simple closed curve in the homotopy class $\alpha$.
The projection $\Pi:{\cal M\cal C\cal G}(S)\to 
{\cal M\cal C\cal G}(S_0)$ is induced 
by the map $S\to S_0$ defined by closing the puncture $p$.

Define a \emph{coarse section} for the projection
$\Pi$ to be a map 
$\Psi:{\cal M\cal C\cal G}(S_0)\to {\cal M\cal C\cal G}(S)$
with the property that there exists a number 
$\kappa >0$ such that  
\[d(\Pi\Psi(g),g)\leq \kappa\]
for all $g\in {\cal M\cal C\cal G}(S_0)$.

\begin{theorem}\label{thm4}
The projection ${\cal M\cal C\cal G}(S)\to 
{\cal M\cal C\cal G}(S_0)$ admits a coarse section which
is a quasi-isometric embedding.
\end{theorem}

Theorem \ref{thm4} contrasts a result of 
Braddeus, Farb and Putman \cite{BFP07} who showed 
that the normal subgroup $\pi_1(S_0)$ of
${\cal M\cal C\cal G}(S)$ is exponentially distorted.

The proof of Theorem \ref{thm1} builds on the
results of \cite{H09}. In that paper we constructed
a locally finite 
connected directed graph ${\cal T\cal T}$ whose
vertex set ${\cal V}({\cal T\cal T})$ 
is the set of all isotopy classes of
complete train tracks on $S$. The mapping class group
${\cal M\cal C\cal G}(S)$ acts property and 
cocompactly on ${\cal T\cal T}$ as a group of 
simplicial isometries.

A \emph{discrete path} in a metric space $(X,d)$ 
is a map $\rho:[0,k_\rho]\cap \mathbb{N}\to X$. 
For a number $L>1$, such a path is called an
\emph{$L$-quasi-geodesic} if 
\[\vert j-i\vert/L-L\leq d(\rho(i),\rho(j))+
L\vert j-i\vert +L\]
for all $i,j\in [0,k_\rho]\cap \mathbb{N}$. For convenience,
we often consider $\rho$ as an eventually constant
map $\mathbb{N}\to X$ by setting $\rho(j)=\rho(k_\rho)$
for all $j\geq k_\rho$. When referring to these 
eventually constant maps as $L$-quasi-geodesics, we
mean that their restrictions to 
$[0,k_\rho]\cap \mathbb{N}$ are
$L$-quasi-geodesics. 

A \emph{(discrete) bicombing} of a metric space $X$ 
assigns to any pair of points $x,y\in X$ a
discrete path $\rho_{x,y}:[0,k_\rho]\cap \mathbb{N}\to X$ 
so that $\rho_{x,y}(0)=x,\rho_{x,y}(k_\rho)=y$.
The bicombing is called \emph{quasi-geodesic} if there
is a number $L>1$ such that
for all $x,y\in X$ the path $\rho_{x,y}$ is a discrete
$L$-quasi-geodesic connecting $x$ to $y$. 
The bicombing is called \emph{bounded} if there is a 
number $L>1$ such that  
\[d(\rho_{x,y}(i),\rho_{z,u}(i))\leq L(d(x,z)+d(y,u))+L\]
for all $x,y,z,u\in X$ and for all $i$ (here the combing
paths $\rho_{x,y},\rho_{z,u}$ are viewed as eventually constant 
maps $\mathbb{N}\to X$).

We use directed edge-paths in ${\cal T\cal T}$ to
construct an ${\cal M\cal C\cal G}(S)$-equivariant
bounded bicombing of 
${\cal T\cal T}$. 
We proceed in three steps.

\begin{enumerate}
\item Show that directed edge-paths connect a coarsely
dense set of pairs of points in ${\cal T\cal T}$.  
\item If $x$ can be connected to $y$ by a directed 
edge-path,
single out a specific such path so that
the resulting path system is invariant under the action of the
mapping class group.
\item Show that these specific directed edge-paths  
can be equipped with paramet\-ri\-za\-tions in 
such a way that the resulting path system defines an
${\cal M\cal C\cal G}(S)$-equivariant 
bounded bicombing of ${\cal T\cal T}$.
\end{enumerate}

Section 3 of this paper is devoted to 
the first step above.
We show that while in general for two vertices
$x,y\in {\cal V}({\cal T\cal T})$ there is no directed
edge path connecting $x$ to $y$, for any pair of
points $x,y\in {\cal T\cal T}$ there is a 
pair of vertices $x^\prime,y^\prime\in 
{\cal V}({\cal T\cal T})$ 
within a uniformly bounded distance 
of $x,y$ so that $x^\prime$ can be connected
to $y^\prime$ by a directed edge path.

The second and the third step of the proof are 
much more involved, and they are carried out in 
Section 5 and Section 6. First we have a closer look
at the geometry of the graph ${\cal T\cal T}$. Using the
fact that directed edge paths are uniform quasi-geodesics,
we give in Section 5 a fairly explicit description
of the geometry of ${\cal T\cal T}$. Then we 
single out for every pair of vertices 
$(x,y)\in {\cal V}({\cal T\cal T})\times 
{\cal V}({\cal T\cal T})$ with the property that $x$ can be connected
to $y$ by a directed edge path a specific such path.  
In Section 6 we show that these paths 
equipped with suitably chosen
parametrizations
define a ${\cal M\cal C\cal G}(S)$-equivariant 
bounded bicombing of 
${\cal T\cal T}$.

In Section 7 we observe that this bicombing of 
${\cal T\cal T}$ is a regular
path system in the sense of \cite{S06}.
Theorem \ref{thm1} then follows from the results
of \cite{S06}.

Theorems 2-4 are derived 
in Section 4 from coarse density of pairs of vertices
in ${\cal T\cal T}$ which can be connected by
a directed edge path.
Namely, let
$\Gamma$ be any finitely generated group equipped with
some word metric and
let $\phi:\Gamma\to {\cal M\cal C\cal G}(S)$ be any
homomorphism with the property that 
$d(\phi g,\phi h)\leq \kappa d(g,h)$ for all
$g,h\in \Gamma$ and some $\kappa >1$. 
Assume moreover that for any $R>0$ there
is some $r>0$ such that $d(\phi(g),\phi(h))\geq R$
whenever $d(g,h)\geq r$. Then 
$\phi$ is a quasi-isometric
embedding if for some choice $\tau$ of a basepoint 
in ${\cal T\cal T}$ and for
any two points $g,h\in \Gamma$ there
is a directed edge path in ${\cal T\cal T}$
connecting $\phi(g)\tau$ to 
a point in a uniform neighborhood of
$\phi(h)\tau$ and 
which is entirely contained in a uniformly bounded neighborhood
of $\phi(\Gamma)\tau$.
In the situations described in Theorems 2-4, a
map $\phi$ which has these
properties can fairly easily be constructed.

In Section 2 we summarize the properties
of the train track complex 
${\cal T\cal T}$ of $S$ which are needed for our purpose.
The appendix contains a technical result about train tracks
which is used in Section 5 but which can be established
independently of the rest of the paper.

\bigskip

{\bf Acknowledgement:} I am very grateful to Mary
Rees for useful comments which made me aware of a gap in 
a preliminary
version of this paper.
I am also grateful to Graham Niblo for pointing out
the paper \cite{S06} to me.

\section{The complex of train tracks}

In this section we summarize some results from
\cite{H09} which will be used throughout the paper.

Let $S$ be an
oriented surface of
genus $g\geq 0$ with $m\geq 0$ punctures and where $3g-3+m\geq 2$.
A \emph{train track} on $S$ is an embedded
1-complex $\tau\subset S$ whose edges
(called \emph{branches}) are smooth arcs with
well-defined tangent vectors at the endpoints. At any vertex
(called a \emph{switch}) the incident edges are mutually tangent.
Through each switch there is a path of class $C^1$
which is embedded
in $\tau$ and contains the switch in its interior. In
particular, the half-branches which are incident
on a fixed switch are divided into two classes according to
the orientation of an inward
pointing tangent at the switch. Each closed curve component of
$\tau$ has a unique bivalent switch, and all other switches are at
least trivalent.
The complementary regions of the
train track have negative Euler characteristic, which means
that they are different from discs with $0,1$ or
$2$ cusps at the boundary and different from
annuli and once-punctured discs
with no cusps at the boundary.
A train track is called \emph{maximal} if each of
its complementary components either is a trigon,
i.e. a topological disc with three cusps at the boundary,
or a once punctured monogon, i.e. a once punctured disc
with one cusp at the boundary. 
We always identify train
tracks which are isotopic.
The book \cite{PH92} contains a comprehensive treatment
of train tracks which we refer to throughout the paper.

A train track is called \emph{generic} if all switches are
at most trivalent.
The train track $\tau$ is called \emph{transversely recurrent} if
every branch $b$ of $\tau$ is intersected by an embedded simple
closed curve $c=c(b)\subset S$ of class $C^1$ 
which intersects $\tau$
transversely and is such that $S-\tau-c$ does not contain an
embedded \emph{bigon}, i.e. a disc with two corners at the
boundary.

A \emph{trainpath} on a train track $\tau$ is a $C^1$-immersion
$\rho:[m,n]\to \tau\subset S$ which maps each interval $[k,k+1]$
$(m\leq k\leq n-1)$ onto a branch of $\tau$. The integer $n-m$ is
called the \emph{length} of $\rho$. We sometimes identify a
trainpath with its image in $\tau$. Each complementary
region of $\tau$ is bounded by a finite number of
(not necessarily embedded) trainpaths which
either are closed curves or terminate at the cusps of the
region.
A \emph{subtrack} of a train track $\tau$ is a subset $\sigma$ of
$\tau$ which itself is a train track. Thus every switch of
$\sigma$ is also a switch of $\tau$, and every branch of $\sigma$
is a trainpath on $\tau$. We write $\sigma<\tau$ if
$\sigma$ is a subtrack of $\tau$.

A \emph{transverse measure} on a train track $\tau$ is a
nonnegative weight function $\mu$ on the branches of $\tau$
satisfying the \emph{switch condition}:
for every switch $s$ of $\tau$, 
the half-branches incident on $s$ are divided into two
classes, and the sums of the weights
over all half-branches in each of the two classes 
coincide.
The train track is called
\emph{recurrent} if it admits a transverse measure which is
positive on every branch. We call such a transverse measure $\mu$
\emph{positive}, and we write $\mu>0$.
If $\mu$ is any transverse measure on a train track
$\tau$ then the subset of $\tau$ consisting of all
branches with positive $\mu$-weight is a recurrent
subtrack of $\tau$.
A train track $\tau$ is called \emph{birecurrent} if
$\tau$ is recurrent and transversely recurrent.
We call $\tau$ \emph{complete} if $\tau$ is generic,
maximal and birecurrent.

\bigskip

{\bf Remark:} As in \cite{H09}, we require every train track
to be generic. Unfortunately this leads to a slight 
inconsistency of our terminology with the terminology found
in the literature.

\bigskip

There is a 
special collection of complete train tracks on $S$ which were introduced
by Penner and Harer \cite{PH92}. Namely, a \emph{pants decomposition}
$P$ for $S$ is a collection of  $3g-3+m$ simple closed curves
which decompose $S$ into $2g-2+m$ pairs of pants. 
Here a pair of pants is a planar orientable bordered surface
of Euler characteristic $-1$ which may be non-compact.
Define a
\emph{marking} of $S$ (or complete clean marking in the terminology of 
\cite{MM00})  
to consist of a pants decomposition
$P$ for $S$ and a system of
\emph{spanning curves} for $P$. For each pants curve $\gamma\in P$
there is a unique simple closed spanning curve which is contained
in the connected component $S_0$ of $S-(P-\gamma)$ 
containing $\gamma$, which 
is not freely homotopic into the boundary or a puncture of this
component and which intersects $\gamma$ in the minimal number of points´
(one point if $S_0$ is a one-holed torus and two points
if $S_0$ if a four-holed sphere).
Note that any two choices of such a spanning curve differ by a
\emph{Dehn twist} about $\gamma$.

For each marking $F$ of $S$ we can construct a collection of
finitely many maximal transversely recurrent train tracks
as follows. Let $P$ be the pants decomposition of the marking. Choose
an open neighborhood $A$ of $P$ in $S$ whose closure in $S$ is
homeomorphic to the disjoint union of $3g-3+m$ closed annuli. Then
$S-A$ is the disjoint union of $2g-2+m$ pairs of pants.
We require that each train
track $\tau$ from our collection intersects a component of $S-A$
which does not contain a puncture of $S$ in a train
track with stops which is isotopic to
one of the four \emph{standard models} shown in Figure A
(see Figure 2.6.2 of \cite{PH92}).
\begin{figure}[ht]
\begin{center}
\psfrag{Figure C}{Figure A} 
\includegraphics 
[width=0.7\textwidth] 
{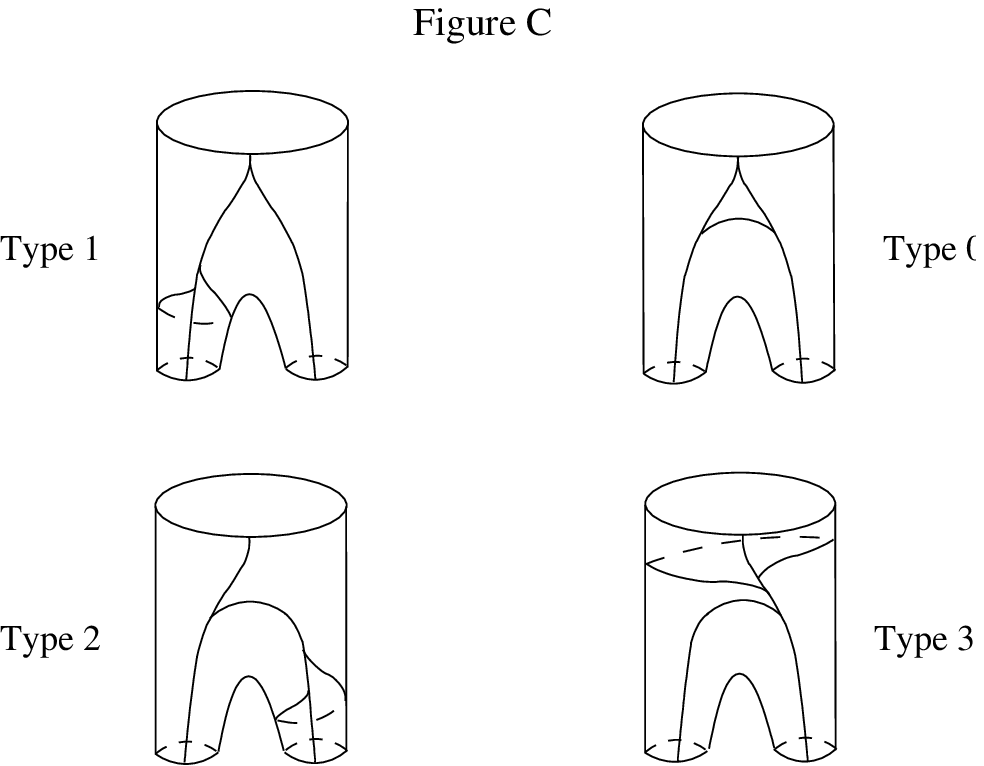}
\end{center}
\end{figure}

If $S_0$ is a component of
$S-A$ which
contains precisely one puncture, then
we require that $\tau$ intersects $S_0$ in a
train track with
stops which we obtain up to diffeomorphism of $S_0$ 
from the standard model of
type 2 or of type 3
by replacing the top boundary curve
by a puncture and by deleting the
branch which is incident on  the stop of
this boundary component. If $S_0$ is a component of $S-A$
which contains two punctures, then we require that $\tau$
intersects $S_0$ in a train track with stops which
we obtain up to diffeomorphism of $S_0$ 
from the standard model of type 1 by
replacing the two
lower boundary components by a puncture and
by deleting the two branches which are incident on the stops of
these boundary components.

The intersection of $\tau$ with a
component of the collection $A$ of $3g-3+m$
annuli is one of the
following four \emph{standard connectors} which are shown in
Figure B (see Figure 2.6.1 of \cite{PH92}).
\begin{figure}[ht]
\begin{center}
\psfrag{Figure D}{Figure B} 
\includegraphics 
[width=0.7\textwidth] 
{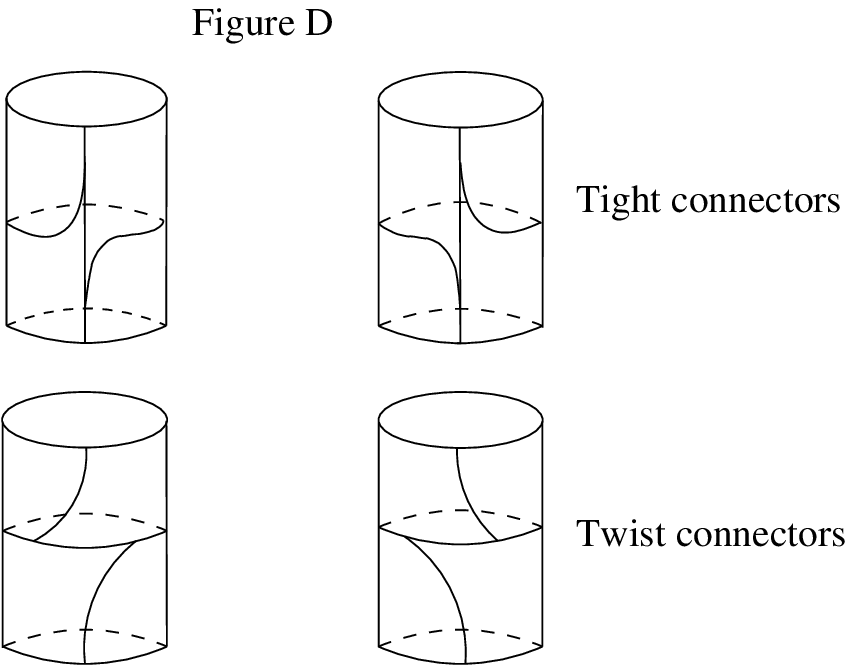}
\end{center}
\end{figure}

From the above standard pieces we can build a train track $\tau$
on $S$ by choosing for each component of $S-A$ one of the standard
models as described above and choosing for each component of $A$
one of the four standard connectors. These train tracks with stops
are then glued at their stops to a connected train track on $S$.
Any two train tracks constructed in this way from the same pants
decomposition $P$, the same choices of standard models for the
components of $S-A$ and the same choices of connectors for the
components of $A$ differ by Dehn twists about the pants curves of
$P$. The spanning curves of the marking $F$ determine a specific
choice of such a glueing \cite{PH92}. We call each of
the resulting train tracks \emph{in standard form for $F$}
provided that it is complete (see p.147 of 
\cite{PH92} for examples
of train tracks built in this way which are not
recurrent and hence not complete).

A \emph{geodesic lamination} for a complete
hyperbolic structure on $S$ of finite volume is
a \emph{compact} subset of $S$ which is foliated into simple
geodesics. A geodesic lamination $\lambda$ is \emph{minimal}
if each of its half-leaves is dense in $\lambda$.
A geodesic lamination is \emph{maximal}
if its complementary regions are all ideal triangles
or once punctured monogons (note that a minimal geodesic
lamination can also be maximal).
The space of geodesic laminations on $S$
equipped with the \emph{Hausdorff topology} is
a compact metrizable space. 

A geodesic lamination $\lambda$
is called \emph{complete} if $\lambda$ is maximal and
can be approximated in the Hausdorff topology by
simple closed geodesics. The space ${\cal C\cal L}$
of all complete geodesic laminations equipped with
the Hausdorff topology is compact. The mapping
class group ${\cal M\cal C\cal G}(S)$ naturally
acts on ${\cal C\cal L}$ as a group of 
homeomorphisms.
Every geodesic lamination $\lambda$
which is a disjoint union of finitely many minimal components
is a \emph{sublamination} of
a complete geodesic lamination, i.e. there
is a complete geodesic lamination which contains
$\lambda$ as a closed subset (Lemma 2.2 of \cite{H09}).

A train track or a geodesic lamination $\sigma$ is
\emph{carried} by a transversely recurrent train track $\tau$ if
there is a map $\phi:S\to S$ of class $C^1$ which is homotopic to the
identity and maps $\sigma$ into $\tau$ in such a way 
that the restriction of the differential of $\phi$
to the tangent space of $\sigma$ vanishes nowhere;
note that this makes sense since a train track has a tangent
line everywhere. We call the restriction of $\phi$ to
$\sigma$ a \emph{carrying map} for $\sigma$.
Write $\sigma\prec
\tau$ if the train track $\sigma$ is carried by the train track
$\tau$. Then every geodesic lamination $\lambda$ which is carried
by $\sigma$ is also carried by $\tau$. A train track
$\tau$ is complete if and only if it is generic and
transversely recurrent and if it carries
a complete geodesic lamination.
The space
of complete geodesic laminations carried by a complete
train track $\tau$ is open and closed in ${\cal C\cal L}$
(Lemma 2.3 of \cite{H09}).
In particular, the space ${\cal C\cal L}$
is totally disconnected.

For every pants decomposition $P$ of $S$ there is a finite set
of complete geodesic laminations on $S$ which contain
(the geodesic representatives of) the components of
$P$ as their minimal components. 
We call such a geodesic lamination
\emph{in standard form} for $P$. 
If $\lambda$ is a geodesic lamination in standard form for $P$
then for each component
$S_0$ of $S-P$ which does not contain a puncture of $S$ there 
are precisely three leaves of $\lambda$ contained in $S_0$
which spiral about the three different boundary components of $S_0$.
The leaves of $\lambda$ spiraling from two different sides
about a component $\gamma$ of $P$ 
define opposite orientations near $\gamma$
(as shown in Figure A of \cite{H09}).
If $S_0$ contains exactly one puncture of $S$ there there
are two leaves of $\lambda$ contained in $S_0$ which
spiral about the two boundary components of $S_0$.

For every marking $F$ of $S$
with pants decomposition $P$ and every train track 
$\tau$ in standard form for $F$ with only twist connectors
there is a unique complete geodesic lamination 
in standard form for $P$ which is
carried by $\tau$. This implies that for every marking $F$
of $S$ with pants decomposition $P$ there is a bijection 
between the complete train tracks in standard form for $F$ 
with only twist connectors and the complete geodesic laminations
in standard form for $P$.
The set of all complete geodesic laminations in standard form
for some pants decomposition
$P$ is invariant under the action of the mapping class
group, moreover there are only finitely many 
${\cal M\cal C\cal G}(S)$-orbits 
of such complete geodesic laminations.

Define the \emph{straightening} of 
a train track $\tau$ on $S$ with respect to some
complete finite volume 
hyperbolic structure $g$ on $S$ 
to be the edgewise immersed
graph in $S$ whose vertices are the switches of $\tau$ and whose
edges are the unique geodesic arcs which are homotopic with
fixed endpoints to the branches of $\tau$. For a number
$\epsilon >0$ we say that the train track $\tau$ \emph{$\epsilon$-follows}
a geodesic lamination $\lambda$ if the tangent lines of the
straightening of $\tau$ are contained in 
the $\epsilon$-neighborhood of the projectivized tangent
bundle $PT\lambda$ of $\lambda$ 
(with respect to the distance function induced by the metric $g$)
and if moreover the 
straightening of every trainpath on $\tau$ is a piecewise
geodesic whose exterior angles at the breakpoints are not
bigger than $\epsilon$. 

\begin{lemma}\label{fatening}
Let $\lambda\in {\cal C\cal L}$ be a 
complete geodesic lamination on $S$ 
in standard form for a pants decomposition $P$ of $S$
and let
$\epsilon >0$. Then there is a 
complete train track $\tau$ on $S$
in standard form for a marking $F$ of $S$ with pants decomposition
$P$ which carries $\lambda$ and 
$\epsilon$-follows $\lambda$.
\end{lemma}
\begin{proof}
The train track $\tau$ can be obtained from $\lambda$ by
collapsing a sufficiently small tubular neighborhood of $\lambda$.
We refer to Theorem 1.6.5 of \cite{PH92} 
and to Lemma 3.2 of \cite{H09} and its proof 
for more details of this construction.
\end{proof}

Note that in Lemma \ref{fatening}, 
the marking $F$ of $S$
depends on the number $\epsilon$ as well as on choices
made in the construction.

A \emph{measured geodesic lamination} is a geodesic
lamination equipped with a transverse translation
invariant measure of full support. The space 
${\cal M\cal L}$ of measured geodesic laminations
on $S$ equipped with the 
weak$^*$-topology is homeomorphic to the product of 
a sphere of dimension
$6g-7+2m$ with the real line. A measured geodesic lamination $\mu$ 
is carried by a train track $\tau$ if its support
is carried by $\tau$. Then $\mu$ defines a transverse measure
on $\tau$, and every transverse measure on $\tau$
arises in this way \cite{PH92}.

We use measured geodesic laminations to establish
another relation between train tracks
in standard form for a marking of $S$ and complete geodesic
laminations which is a variant of a result of 
Penner and Harer (Theorem 2.8.4 of \cite{PH92}).

\begin{lemma}\label{standardform}
For any marking $F$ of $S$, every complete geodesic
lamination on $S$ is carried by a unique train track
in standard form for $F$. 
\end{lemma}
\begin{proof}
A complete geodesic lamination $\lambda$ can be approximated in the
Hausdorff topology by a sequence $\{c_i\}$ of 
simple closed geodesics. For a fixed
marking $F$ of $S$, each such
geodesic is carried by a train track in standard form for $F$
(this is contained in Theorem 2.8.4 of \cite{PH92}).
Since there are only finitely many train tracks in standard form
for $F$, there is a fixed train track $\tau$ in standard form
for $F$ which carries infinitely many of the curves $c_i$.
By Lemma 3.2 of \cite{H09}, 
the set of geodesic laminations carried by a fixed train
track $\tau$ is closed in the Hausdorff topology and hence the 
geodesic lamination $\lambda$ is carried by $\tau$. 

Now assume that there is a second train track $\eta$
in standard form for $F$ which carries $\lambda$.
By Lemma 3.2 and Lemma 3.3 of \cite{H09}, there 
is a complete train track $\sigma$ which is
carried by both $\tau$ and $\eta$. The train track
$\sigma$ carries a measured geodesic lamination $\mu$ 
whose support is both minimal and maximal (see
the top of p.556 of \cite{H09} for a detailed 
discussion of this fact).
Thus $\mu$ is carried by two distinct train tracks
in standard form for $F$ which violates 
Theorem 2.8.4 of \cite{PH92}. 
\end{proof}

A half-branch $\hat b$ in a generic train track $\tau$ incident on
a switch $v$ of $\tau$ is called
\emph{large} if every trainpath containing $v$ in its interior
passes through $\hat b$. A half-branch which is not large
is called \emph{small}.
A branch
$b$ in a generic train track
$\tau$ is called
\emph{large} if each of its two half-branches is
large; in this case $b$ is necessarily incident on two distinct
switches. A branch is called
\emph{small} if each of its two half-branches is small. A branch
is called \emph{mixed} if one of its half-branches is large and
the other half-branch is small (see p.118 of \cite{PH92}).

There are two simple ways to modify a complete train track $\tau$
to another complete train track. First, we can \emph{shift}
$\tau$ along a mixed branch $b$ to a train track $\tau^\prime$ as shown
in Figure C. 
\begin{figure}[ht]
\begin{center}
\psfrag{Figure A}{Figure C} 
\includegraphics 
[width=0.7\textwidth] 
{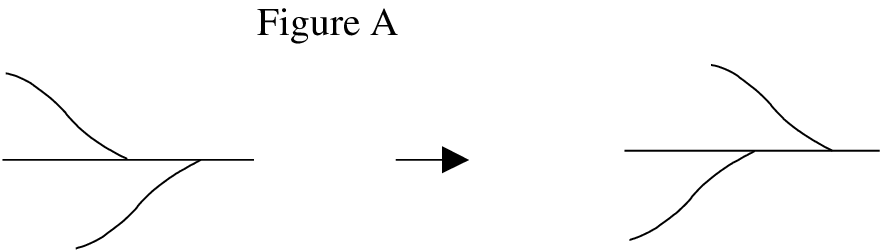}
\end{center}
\end{figure}
If $\tau$ is complete then the same is true for
$\tau^\prime$. Moreover, a train track or a geodesic 
lamination is carried by $\tau$ if and
only if it is carried by $\tau^\prime$ (see \cite{PH92} p.119).
In particular, the shift $\tau^\prime$ of $\tau$ is
carried by $\tau$. There is a natural
bijection $\phi(\tau,\tau^\prime)$ 
of the set of branches of $\tau$ onto
the set of branches of $\tau^\prime$ which is 
induced by the identity of the complement of a small
neighborhood of $b$ in $S$. The bijection $\phi(\tau,\tau^\prime)$
also induces a bijection of the set of half-branches 
of $\tau$ onto the set of half-branches of $\tau^\prime$ which we
denote again by $\phi(\tau,\tau^\prime)$.

Second, if $e$ is a large branch of $\tau$ then we can perform a
right or left \emph{split} of $\tau$ at $e$ as shown in Figure D.
\begin{figure}[ht]
\begin{center}
\psfrag{Figure B}{Figure D} 
\includegraphics 
[width=0.7\textwidth] 
{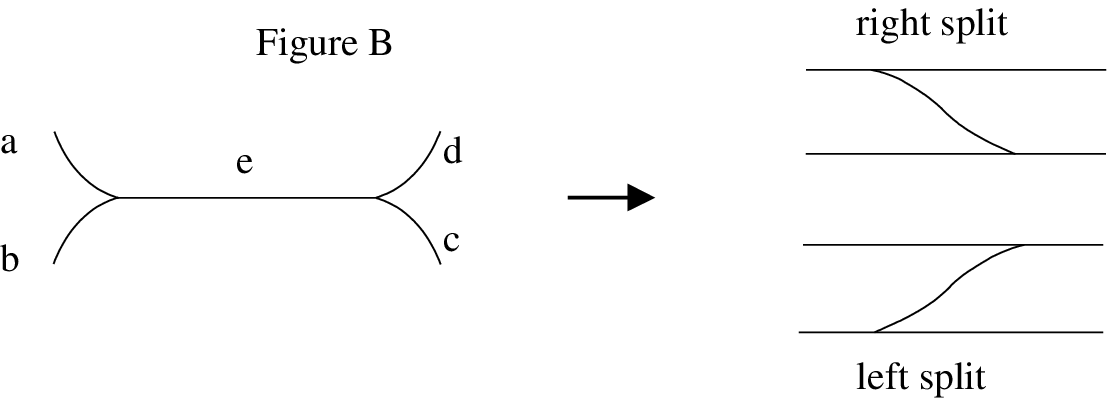}
\end{center}
\end{figure}
Note that a right split at $e$ is uniquely
determined by the orientation of $S$ and does not
depend on the orientation of $e$.
Using the labels in the figure, in the case of a right
split we call the branches $a$ and $c$ \emph{winners} of the
split, and the branches $b,d$ are \emph{losers} of the split. If
we perform a left split, then the branches $b,d$ are winners of
the split, and the branches $a,c$ are losers of the split.
The split $\tau^\prime$ of a train track $\tau$ is carried
by $\tau$, and there is a natural choice of a carrying map which
maps the switches of $\tau^\prime$ to the switches of $\tau$. The
image of a branch of $\tau^\prime$ is then a trainpath on $\tau$
whose length either equals one or two. There is
a natural bijection $\phi(\tau,\tau^\prime)$ 
of the set of branches
of $\tau$ onto the set of branches of $\tau^\prime$ which
maps the branch $e$ to a small branch $e^\prime$ which we call the
\emph{diagonal} of the split.
This bijection is induced by the identity 
on the complement of a small neighborhood
of $e$ in $S$.
The map $\phi(\tau,\tau^\prime)$ also induces a bijection of 
the set of half-branches of $\tau$ onto the set of 
half-branches of $\tau^\prime$ again denoted
by $\phi(\tau,\tau^\prime)$.

Occasionally we also have to consider the
\emph{collision} of a train track $\eta$
at a large branch $e$. This collision is
obtained from $\eta$ by a split at $e$ and 
removal of the diagonal in the split track.
Such a collision is shown in Figure 2.1.2 of \cite{PH92}.

A split of a maximal
transversely recurrent generic train track is maximal,
transversely recurrent and generic. 
If $\tau$ is a complete train track and if
$\lambda\in {\cal C\cal L}$ is carried by $\tau$, then 
for every large branch $e$ of $\tau$ there is a
unique choice of a right or left split of $\tau$ at 
a large branch $e$ of $\tau$ with the
property that the split track $\tau^\prime$ carries $\lambda$
(see p. 557 of \cite{H09} for a more complete discussion).
We call such a split a \emph{$\lambda$-split}.
The train track $\tau^\prime$ is
complete. In particular, a complete train track $\tau$ can always
be split at any large branch $e$ to a complete train track
$\tau^\prime$; however there may be a choice of a right or left
split at $e$ such that the resulting train 
track is not recurrent any
more (compare p.120 in \cite{PH92}).

For a number $L\geq 1$, an
\emph{$L$-quasi-isometric embedding} of a metric space
$(X,d)$ into a metric space $(Y,d)$ is a map
$\phi:X\to Y$ such that
\[d(x,y)/L-L\leq  d(\phi(x),\phi(y)) \leq Ld(x,y)+L\]
for all $x,y\in X$. The map $\phi$ is called
an \emph{$L$-quasi-isometry} if moreover
the $L$-neighborhood of $\phi X$ in $Y$ is all of $Y$.
An \emph{$L$-quasi-geodesic} in a metric space
$(X,d)$ is an $L$-quasi-isometric embedding of a closed connected
subset of $\mathbb{R}$ or of the intersection of 
such a closed connected subset of $\mathbb{R}$ with 
$\mathbb{Z}$.

Denote by ${\cal T\cal T}$ the directed metric graph whose set
${\cal V}({\cal T\cal T})$ of 
vertices is the set of 
isotopy classes of complete train tracks on $S$
and whose edges are determined as follows. The train track
$\tau\in {\cal V}({\cal T\cal T})$
is connected to the train track $\tau^\prime$
by a directed edge of length one 
if and only if $\tau^\prime$ can be obtained
from $\tau$ by a single split.
The graph ${\cal T\cal T}$ is connected (Corollary 2.7 of \cite{H09}).
The mapping class group ${\cal M\cal C\cal G}(S)$ of $S$ acts properly
and cocompactly on ${\cal T\cal T}$ as a group of
simplicial isometries. In particular, ${\cal T\cal T}$
is ${\cal M\cal C\cal G}(S)$-equivariantly quasi-isometric to
${\cal M\cal C\cal G}(S)$ equipped with any word metric 
(Corollary 4.4 of \cite{H09}).

Define a \emph{splitting sequence} 
in ${\cal T\cal T}$ to be a sequence 
$\{\alpha(i)\}_{0\leq i\leq m}\subset {\cal V}({\cal T\cal T})$ with
the property that for every $i\geq 0$ the train track
$\alpha(i+1)$ can be obtained from $\alpha(i)$ by
a single split.
Thus splitting sequences in ${\cal V}({\cal T\cal T})$ 
correspond precisely to directed edge-paths in ${\cal T\cal T}$.
If $\tau$ can be connected to $\eta$ by a splitting
sequence then we say that $\tau$ is \emph{splittable} to $\eta$.
If $\{\alpha(i)\}_{0\leq i\leq m}$ is a splitting sequence then
the composition 
\[\phi(\alpha(0),\alpha(m))=\phi(\alpha(m-1),\alpha(m))\circ
\dots \circ \phi(\alpha(0),\alpha(1))\]
is a bijection of the branches (or half-branches) of $\alpha(0)$
onto the 
branches (or half-branches) of $\alpha(m)$
which does not depend on the choice of the splitting sequence
connecting $\alpha(0)$ to $\alpha(m)$ (Lemma 5.1 of \cite{H09}).

For a complete train track $\tau$ and a complete geodesic
lamination $\lambda$ carried by $\tau$ 
define the \emph{cubical Euclidean cone} $E(\tau,\lambda)$
to be the full subgraph of ${\cal T\cal T}$ whose
vertices consist of the complete train tracks
which can be obtained from $\tau$ by a splitting sequence
and which carry $\lambda$. Then $E(\tau,\lambda)$ is a connected
subgraph of ${\cal T\cal T}$ and hence it can
be equipped with an intrinsic path-metric $d_E$.
We showed (Theorem 2 of \cite{H09})

\begin{theo}\label{cubicaleuclid}
There is a number $L_0>1$
such that for every
$\tau\in {\cal V}({\cal T\cal T})$ and every complete
geodesic lamination $\lambda$ carried by $\tau$ the
inclusion
$(E(\tau,\lambda),d_E)\to {\cal T\cal T}$ is an
$L_0$-quasi-isometric embedding.
\end{theo}

By Corollary 5.2 of \cite{H09},
directed edge-paths in $(E(\tau,\lambda),d_E)$
are geodesics
and hence we obtain as an immediate consequence that
directed edge-paths in ${\cal T\cal T}$
are $L_0$-quasi-geodesics.

\section{Density of splitting sequences}

The goal of this section is to show the following proposition
which is the first step in the proof of Theorem \ref{thm1}
from the introduction and which is 
the main technical tool for the proof of 
Theorem 2-4.

\begin{proposition}\label{density1}
There is a number $d_0>0$ with the
following property. For any train tracks $\tau,\sigma\in
{\cal V}({\cal T\cal T})$ 
there is a train track $\tau^\prime$ which is contained in
the $d_0$-neighborhood of $\tau$ and
which is splittable to a train
track $\sigma^\prime$ contained in the $d_0$-neighborhood of
$\sigma$.
\end{proposition}

To simplify the argument
we reduce Proposition \ref{density1} to the following

\begin{proposition}\label{density}
There is a number $d_1>0$ with the following
property. Let $F$ be any marking of $S$. Then for 
any train track
$\sigma\in {\cal V}({\cal T\cal T})$ there
is a train track $\tau\in {\cal V}({\cal T\cal T})$ 
in standard form for $F$ 
which carries a train track 
$\sigma^\prime$ contained in the $d_1$-neighborhood of 
$\sigma$. If $\sigma$ is in standard form for a marking $G$ 
with pants decomposition $Q$ then $\sigma^\prime$ can
be chosen to contain the pants decomposition $Q$ 
as an embedded subtrack.
\end{proposition}

We begin with explaining how Proposition \ref{density1} 
follows from Proposition \ref{density}. 
The mapping class group acts properly and cocompactly
on ${\cal T\cal T}$ preserving the set of train tracks
in standard form for some marking of $S$.
Thus every complete train track is
contained in a uniformly bounded neighborhood of a 
train track in standard form for some marking $F$ of $S$.
The mapping class group acts on 
the set of all markings of $S$, with finitely
many orbits. Therefore 
the diameter in ${\cal T\cal T}$ of any
set of train tracks in standard form for a fixed
marking is uniformly bounded.
This implies that 
there is a number $d_2>0$
and for every complete
train track $\tau$ on $S$ there is a marking $F$ of $S$
such that $d(\tau,\eta)\leq d_2$ for every train track
$\eta$ in standard form for $F$.

By Lemma 6.6 of \cite{H09}, there is a number
$p>0$ and for two complete train tracks
$\sigma\prec\tau$ there is a train track $\zeta$ which
can be obtained from $\tau$ by a splitting sequence
and such that $d(\sigma,\zeta)\leq p$.
As a consequence,
Proposition \ref{density1} follows from Proposition \ref{density}.

The idea of proof for Proposition \ref{density} is 
as follows.
Define a \emph{splitting and shifting sequence}
to be a sequence $\{\alpha(i)\}_{0\leq i\leq m}$ 
with the property that for every $i\geq 0$ 
the train track $\alpha(i+1)$ can be obtained
from $\alpha(i)$ by a sequence of shifts followed
by a single split. 
Theorem 2.4.1 of \cite{PH92} relates splitting and
shifting to carrying.

\begin{proposition}\label{splitshiftcarry}
If $\sigma\in {\cal V}({\cal T\cal T})$ is carried
by $\tau\in {\cal V}({\cal T\cal T})$ then 
$\tau$ can be connected to $\sigma$ by a splitting
and shifting sequence.
\end{proposition}

Now let $F,G$ be any two markings of $S$. 
We attempt to construct
a splitting and shifting sequence
connecting some train track in standard form for $F$ to
some train track in standard form for $G$.

A train track in standard form for $G$ with only twist connectors
carries a complete
geodesic lamination $\lambda$ in standard form 
for the pants decomposition $Q$ 
of the marking $G$ of $S$.
By Lemma \ref{standardform}, every complete geodesic
lamination $\lambda$ on $S$ is carried by a unique train track $\tau$
in standard form for $F$. We modify $\tau$
with a sequence of 
splits as efficiently as possible
to a train track $\xi$ which carries $\lambda$ and contains
the pants decomposition $Q$ as a subtrack.
This train track $\xi$ carries a train track
$\eta$ in standard form for
some marking of $S$ with pants decomposition $Q$ whose
distance to $\xi$ is uniformly bounded.
There is a \emph{multi-twist} $\phi$ (i.e. a concatenation
of mutually commuting Dehn twists) 
about the pants curves of $Q$ which
maps $\eta$ to a train track $\phi\eta$ in standard form for $G$.
In general, $\phi\eta$ is not carried by $\eta$, 
but we obtain enough control that we
can find a perhaps different train track in standard form 
for $F$ which can be connected with a splitting and
shifting sequence to a train track in a uniformly
bounded neighborhood of $\phi\eta$ which is in standard
form for a marking with pants decomposition $Q$.

To carry out this strategy we use the pants decomposition $Q$ 
for the construction of 
splitting sequences. However,
$Q$ may not \emph{fill} $\tau$, i.e. a carrying map
$Q\to \tau$ may not be surjective. Therefore we are lead
to investigate splitting sequences of complete train tracks
which are determined 
by modifications of subtracks. This will occupy the
major part of this section and will also be of crucial importance
in Section 5 and Section 6.

Fix a complete Riemannian metric on $S$ of finite volume.
With respect to this metric, a complementary
region $C$ of a 
train track $\sigma$ on $S$ is a hyperbolic surface whose
metric completion $\overline{C}$ is a bordered surface
with boundary $\partial C$. This boundary 
consists of a finite
number of arcs of class $C^1$, called \emph{sides}
of $C$ or of $\overline{C}$.
Each side of $C$ either is a closed curve of class $C^1$ (i.e. the
boundary component containing the 
side does not contain any cusp)
or an arc with endpoints at two not
necessarily distinct
cusps of the component. 
We call a side of $C$ which does not
contain cusps a \emph{smooth side} of $C$.
The closure of $C$ in $S$ can be obtained from
$\overline{C}$ by some identifications of subarcs
of sides (the inclusion $C\to S$ extends to an immersion
of each side of $C$, but the image arc may have
tangential self-intersections or may meet another
side tangentially). For simplicity we call the 
image in $\sigma$ of a side of $C$  
a side of $C$ as well (i.e. most of the time we view a side
of $C$ as an immersed arc of class $C^1$ in $\sigma$).
Using this abuse
of notation, a side of $C$ is just an immersed
arc or an immersed closed
curve of class $C^1$ in $\sigma$
with only tangential self-intersections.
However, we reserve the notation $\overline{C}$ for the
metric completion of $C$. 

If $T\subset \partial C$ is a smooth side
of a complementary region $C$ of $\sigma$
then we mark a point on $T$. We view this point 
as a point on the boundary of the completion
$\overline{C}$ of $C$, i.e. even in the case that
the point corresponds to a point of tangential
self-intersection of the image of $\partial C$ in $\sigma$, 
passing once through $T$ means crossing the point 
precisely once, and not passing through $T$ 
means not crossing through the point.

If $C$ is a complementary region of $\sigma$ whose
boundary contains precisely $k\geq 0$ cusps, then the \emph{Euler
characteristic} $\chi(C)$ is defined by $\chi(C)=\chi_0(C)-k/2$
where $\chi_0(C)$ is the usual Euler characteristic of
the compact topological surface with boundary
$\overline{C}$. Note that the
sum of the Euler characteristics of the complementary regions
of $\sigma$ is just the Euler characteristic of $S$
(see the discussion in Chapter 1.1 of \cite{PH92}).

A \emph{complete extension} of a train track
$\sigma$ is a complete train track $\tau$
containing $\sigma$ as a subtrack and whose switches
are distinct from the images in
$\sigma$ of the marked points on smooth boundary components
of complementary regions of $\sigma$.
Such a complete extension $\tau$ intersects
each complementary region $C$ of $\sigma$ in an
embedded graph with smooth edges. The closure of $\tau\cap C$ 
in the completion $\overline{C}$ of $C$ 
is a graph whose univalent
vertices are contained in the complement of the
cusps and marked points of the boundary $\partial C$
of $\overline{C}$.  At a univalent vertex, the graph 
is tangential to $\partial C$. 
We call two such graphs $\tau\cap C,\tau^\prime\cap C$
\emph{equivalent}
if there is a smooth isotopy of $\overline{C}$ which fixes the
cusps and the marked points in $\partial C$ and which
maps $\tau\cap C$ onto $\tau^\prime\cap C$.
The complete extensions
$\tau,\tau^\prime$ of $\sigma$ are called
\emph{$\sigma$-equivalent} if for each complementary region
$C$ of $\sigma$ the graphs $\tau\cap C$ and
$\tau^\prime\cap C$ are equivalent in this sense.
The purpose of marking a point on a smooth
boundary component $T$ of a complementary region of $\sigma$
is to control the amount of relative twisting about 
$T$ of two complete extensions $\tau,\tau^\prime$ of $\sigma$.

For two complete extensions $\tau,\tau^\prime$
of $\sigma$
define the \emph{intersection number}
$i_\sigma(\tau,\tau^\prime)$ to be the minimal number
of intersection points contained in $S-\sigma$ between
any two complete extensions $\eta,\eta^\prime$ of $\sigma$
which are $\sigma$-equivalent to $\tau,\tau^\prime$
and with the following additional properties.
\begin{enumerate}
\item[a)] A switch $v$ of $\eta$ (or $\eta^\prime$)
is also a switch of $\eta^\prime$ (or $\eta$) if and only
if $v$ is a switch of $\sigma$.
\item[b)] A switch of $\eta$ (or $\eta^\prime$)
contained in the interior
of a complementary region $C$ of $\sigma$ is not contained
in $\eta^\prime$ (or $\eta$), i.e. an intersection
point of $\eta$ with $\eta^\prime$ contained in $C$
is an interior point
of a branch of $\eta$ and of a branch of $\eta^\prime$.
\end{enumerate}
Since the number of switches of a complete train track on $S$ only
depends on the topological type of $S$, for any
complete extension $\tau$ of $\sigma$ the intersection number
$i_\sigma(\tau,\tau)$ is bounded from above by a universal constant
neither depending on $\sigma$ nor $\tau$. Moreover,
for every number
$m>0$ there is a number $q(m)>0$ not depending on $\sigma$
so that for every complete extension $\tau$ of $\sigma$
the number of $\sigma$-equivalence classes
of complete extensions $\tau^\prime$ of $\sigma$
with $i_\sigma(\tau,\tau^\prime)\leq m$ is bounded
from above by $q(m)$.

To simplify the notation we do not
distinguish between $\sigma$ as a subgraph of $\tau$ (and hence
containing switches of $\tau$ which are bivalent in $\sigma$) and
$\sigma$ viewed as a subtrack of $\tau$, i.e. the graph from which
the bivalent switches not contained in simple closed curve
components have been removed. A branch $e$ of $\sigma$
defines an embedded trainpath $\rho:[0,m]\to \tau$, unique
up to orientation, whose image is precisely $e$.
We call $\tau$
\emph{tight} at $e$ if
$e$ is a branch in $\tau$, i.e. if the length $m$ of $\rho$ 
equals one. If $e$ is a large branch of $\sigma$ then 
 $\rho$ begins and ends with a large
half-branch and hence $\rho[0,m]$ contains a large branch of 
$\tau$ (see Lemma 2.7.2 of \cite{PH92}). 

A \emph{proper subbranch} of a branch $e$ of $\sigma$ 
is a branch $b$ of $\tau$ which is a proper subset of $e$.
Then $b$ is incident on at least one switch $v$ of $\tau$ which
is not a switch of $\sigma$. There is a half-branch $c$ of
$\tau$ which is incident on $v$ and not contained in $\sigma$.
We call $c$ a \emph{neighbor} of $\sigma$ at $v$.
We distinguish three different types of 
\emph{large} proper subbranches
$b$ of a branch $e$ of $\sigma$. These types
are shown in Figure E. Note that a large branch
of any train track on $S$ is embedded in $S$.
\begin{figure}[ht]
\begin{center}
\psfrag{Figure D}{Figure E} 
\includegraphics 
[width=0.7\textwidth] 
{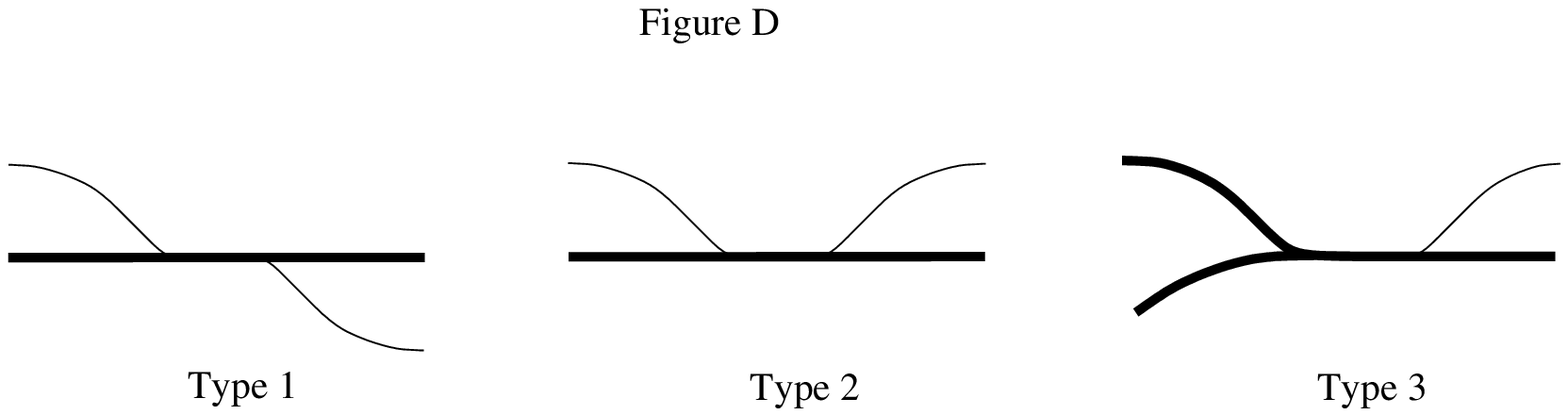}
\end{center}
\end{figure}

{\sl Type 1:} $b$ is contained in the interior of 
$e$ and the two neighbors of $\sigma$ at
the two endpoints of $b$ lie on different sides
of $e$ in a small tubular neighborhood of $b$ in $S$.

{\sl Type 2:} $b$ is contained in the interior of $e$ and
both neighbors of $\sigma$ at the two endpoints of $b$ lie
on the same side of $e$ in a small tubular neighborhood
of $b$ in $S$. 

{\sl Type 3:} One endpoint of $b$ is incident on a switch
of $\sigma$.

A split of $\tau$ at a large proper 
subbranch $b$ of $\sigma$ (i.e. of a large branch of
$\tau$ which is a proper subbranch of a branch of $\sigma$)
is called a \emph{$\sigma$-split} if the split track 
contains $\sigma$ as a subtrack.
Note that such a split always exists.
If $b$ is a large proper subbranch of $\sigma$ of type 2
then any split of $\tau$ at $b$ is a $\sigma$-split.

Let $q$ be the number of branches of a complete train track on $S$.
The number of switches of a complete train track
on $S$ then equals $2q/3<q$. For a subtrack $\sigma$ of
a complete train track $\tau$ 
let $\beta(\tau,\sigma)$ be the number of neighbors of
$\sigma$ in $\tau$,
i.e. the number of half-branches of $\tau-\sigma$ 
which are incident on a switch contained in $\sigma$.
If $\tau_1$ is obtained from $\tau$ by a split
at a large proper subbranch $b$ of $\sigma$ of type 2
then 
$\beta(\tau_1,\sigma)=\beta(\tau,\sigma)-1$.

Now let $\sigma$ be a recurrent train track on $S$. Then
there is a measured geodesic lamination $\nu$ on $S$ 
which is carried by $\sigma$ and which defines a positive
transverse measure on $\sigma$. We call such a measured
geodesic lamination \emph{filling} for $\sigma$.
For every complete extension $\tau$ of $\sigma$ 
there is a complete geodesic
lamination $\lambda$ which is carried by $\tau$ and contains
the support of $\nu$
as a sublamination. 
Namely, the positive transverse measure 
on $\sigma$ defined by $\nu$ 
can be approximated by positive transverse
measures $\mu_i$ on $\tau$ which define a 
measured geodesic lamination whose support is 
a minimal and maximal
geodesic lamination carried by $\tau$ (see
p.556 of \cite{H09} for a detailed proof of this fact).
Since the space ${\cal C\cal L}$ of all complete
geodesic laminations on $S$ is compact,
as $\mu_i\to \nu$ in the
space of transverse measures on $\tau$, up to passing to 
a subsequence the 
supports of $\mu_i$ converge in the Hausdorff topology to
a complete geodesic lamination $\lambda$ which 
contains the support of $\nu$ as a sublamination.
By Lemma 2.3 of \cite{H09}, 
$\lambda$ is carried by $\tau$.
We call $\lambda$ a \emph{complete
$\tau$-extension} of $\nu$. 

The following observation will be used throughout the paper.

\begin{lemma}\label{tightcontrol}
Let $\sigma$ be a recurrent subtrack of a complete train track $\tau$
and let $\lambda$ be a complete $\tau$-extension of a $\sigma$-filling
measured geodesic lamination. 
Then for every large branch $e$ of $\sigma$ there is a 
unique train track $\tau^\prime$ with the following properties.
\begin{enumerate}
\item $\tau^\prime$ can be obtained from $\tau$ by at 
most $q^2$ $\sigma$-splits
at large proper subbranches of $e$. 
In particular, $\tau^\prime$ contains
$\sigma$ as a subtrack.
\item $\tau^\prime$ carries $\lambda$.
\item $\tau^\prime$ is tight at $e$.
\end{enumerate}
Moreover, given any complete extension $\eta$ of $\sigma$,
if no marked point on a smooth side of a complementary region
$C$ of $\sigma$ is mapped to the branch $e$ then
\[i_\sigma(\tau^\prime,\eta)\leq 
i_\sigma(\tau,\eta)+q(\beta(\tau,\sigma)-
\beta(\tau^\prime,\sigma)).\]
Otherwise we have $i_\sigma(\tau^\prime,\eta)\leq i_\sigma(\tau,\eta)+q^3$.
\end{lemma}
\begin{proof} If $\tau$ is tight at the large branch $e$ 
of $\sigma$ then
$\tau=\tau^\prime$ satisfies the requirements
in the lemma.  

Otherwise let 
$a$ be a neighbor of $\sigma$ at a switch of $\tau$
contained in $e$.
There is a unique maximal trainpath $\rho:[-1,m]\to \tau$ with
$\rho[-1/2,0]=a$ and such that $\rho[0,m]\subset e$. Then 
$\rho(m)$ is a switch of $\sigma$ on which $e$ is incident.
Let $c(a,e)=m\leq q$ be the length of the intersection of
the trainpath $\rho$ with $e$ and let 
\[c(\tau,e)=\sum_ac(a,e)\] where the sum is taken over 
all neighbors of $e$ in $\tau$. 
Then $c(\tau,e)\leq q^2$, and 
$c(\tau,e)=0$ if and only if $\tau$ is tight at $e$. 

Let $b$ be a large proper subbranch of $e$.
Let $\lambda$ be a complete $\tau$-extension of a $\sigma$-filling
measured geodesic lamination $\nu$ and let $\tau_1$ be the 
train track obtained from $\tau$ by a $\lambda$-split at $b$.
We distinguish two cases according to the type of $b$.

If $b$ is of type 1 or of type 3 then  
there is a unique choice of a right or left split of $\tau$
at $b$ such that
the split track $\tau_1^\prime$ contains $\sigma$ as a subtrack. 
Since $\nu$ is a $\sigma$-filling measured geodesic lamination,
$\tau_1^\prime$ is also the unique train track
obtained from $\tau$ by a split at $b$  
which carries $\nu$. Now $\lambda$ is a complete
extension of $\nu$ and therefore 
$\tau_1^\prime=\tau_1$.
The natural bijection $\phi(\tau,\tau_1)$ 
of the half-branches of $\tau$ onto
the half-branches of $\tau_1$ maps any neighbor
$a$ of $\sigma$ in $\tau$ to a neighbor 
$\phi(\tau,\tau_1)(a)$ of $\sigma$
in $\tau_1$. Moreover, we have
$c(\phi(\tau,\tau_1)(a),e)\leq c(a,e)$.
If $a$ is a neighbor of $\sigma$ at an endpoint of $b$
then $c(\phi(\tau,\tau_1)(a),e)= c(a,e)-1$
(see Figure D), 
To summarize, we have $c(\tau_1,e)\leq c(\tau,e)-1$.

If $b$ is of type 2 then once again,
the train track $\tau_1$ contains $\sigma$ as a subtrack.
Moreover, we have $\beta(\tau_1,\sigma)=\beta(\tau,\sigma)-1$ and 
$c(\tau_1,e)<c(\tau,e)$. As a consequence, a splitting sequence
of length at most $q^2$ at large proper subbranches of $e$
transforms $\tau$ to a train track $\tau^\prime$ which contains
$\sigma$ as a subtrack, is tight at $e$ and carries $\lambda$.
By uniqueness of sequences of $\lambda$-splits up to order
(Lemma 5.1 of \cite{H09}), the train track
$\tau^\prime$ is uniquely determined by $\tau,\sigma,\lambda,e$.

To estimate intersection numbers between $\tau,\tau^\prime$ and
an arbitrary complete extension $\eta$ of $\sigma$, 
let again $b$ be a large proper subbranch of $e$ of type 
1 or type 3 and let $\tau_1$ be the train track obtained from
$\tau$ by a $\lambda$-split at $b$.
If $e$ does not contain the image of any marked point
on a smooth boundary component of a complementary region
of $\sigma$, then 
$\tau$ and $\tau_1$ are $\sigma$-equivalent. 

Otherwise
there are one or two (not necessarily distinct) 
complementary regions $C_1,C_2$ of $\sigma$ and smooth
sides $T_i$ of $C_i$ whose images in $\sigma$ contain $e$. 
Up to isotopy, 
a split of $\tau$ at $b$ can be realized by moving 
one of the neighbors of $\sigma$ incident on an
endpoint of $b$, say the neighbor $a$, 
across $b$ while leaving the second neighbor 
(or the branches of $\sigma$ incident on an endpoint of 
$e$ in case the branch $b$ is of type 3) fixed.
Assume that the half-branch
$a$ is contained in the complementary region 
$C_1$ of $\sigma$ and that $a$ terminates at a 
point in the smooth boundary component $T_1$ of $C_1$.
There are at most $q$ half-branches of 
$\eta$ contained in $\eta\cap C_1$ which
terminate at a point in $T_1$.
Up to isotopy of $\overline{C_1}\cup \overline{C_2}$ 
preserving the cusps and
the marked points, moving the half-branch 
$a$ of $\tau$ across the marked point in $T_1$
increases the number of
intersection points between $\tau$ and $\eta$ by 
at most $q$. Namely, up to isotopy such a move creates
at most one additional
intersection point with  
any half-branch of $\eta$ with endpoint $T_1$.
As a consequence, we have
\begin{equation}\label{type1split}
i_\sigma(\tau_1,\eta)\leq i_\sigma(\tau,\eta)+q.
\end{equation}

If $\tau_1$ is obtained from $\tau$ by a $\lambda$-split
at a large proper subbranch of $e$ of type 2 then 
the split which modifies $\tau$ to $\tau_1$
can be realized by moving 
one of the neighbors of $\sigma$ incident on 
an endpoint of $b$ across $b$
to a half-branch which is incident on a point in the 
interior of the neighbor of $\sigma$
at the second endpoint of $b$.
As before, this implies that the inequality 
(\ref{type1split}) holds true
for every complete extension $\eta$ of $\sigma$
(independent of whether or not $e$ contains the image of  
a marked point). 

To summarize,
if there is no marked point on the boundary of a complementary
region of $\sigma$ which
is mapped into $e$ and if 
$\eta$ is any complete extension of $\sigma$ then
the above discussion shows that
only splits at large proper subbranches of $e$ of type 2 
change the
intersection number between $\tau$ and 
$\eta$.
A successive application of the estimate (\ref{type1split})
yields that
\[i_\sigma(\tau^\prime,\eta)\leq i_\sigma(\tau,\eta)+
q(\beta(\tau,\sigma)-\beta(\tau^\prime,\sigma)).\]
The second estimate of intersection numbers
stated in the lemma follows in the same way from the
inequality (\ref{type1split}).
\end{proof}

For a recurrent subtrack $\sigma$
of a complete train track $\tau$, for
a large branch $e$ of
$\sigma$ and  
a complete $\tau$-extension $\lambda$ of a $\sigma$-filling
measured  
geodesic lamination,
we call the complete
train track 
$\tau^\prime$ constructed in Lemma \ref{tightcontrol}
the \emph{$(e,\lambda)$-modification}
of $\tau$.

{\bf Remark:} 1) Lemma \ref{tightcontrol} and its proof remain valid
if the large branch $e$ 
of a recurrent subtrack $\sigma$ of $\tau$
is replaced by any embedded
trainpath $\rho:[0,m]\to \tau$ which begins and ends
with a large half-branch. A large branch $e$ of a
non-recurrent subtrack of $\tau$ is an example.
In this case the complete 
geodesic lamination
$\lambda$ has to be replaced by a
measured geodesic
lamination whose support is minimal and complete and
is carried by $\tau$ and which defines a 
transverse measure on $\tau$ 
giving positive weight to the set
of arcs which are mapped homeomorphically \emph{onto} $\rho[0,m]$
by a carrying map. We call such a complete
geodesic lamination \emph{$\rho$-filling}, and we call
the train track obtained from $\tau,\rho,\lambda$ with the procedure
from the proof of Lemma \ref{tightcontrol} the
\emph{$(\rho,\lambda)$-modification} of $\tau$.
Note that a $\rho$-filling complete geodesic lamination
may not exist always.

2) Let $e_1,e_2$ be distinct large branches of a train track
$\sigma$ on $S$, let $\tau$ be a complete extension of 
$\sigma$ and let $\lambda$ be a complete $\tau$-extension of 
a $\sigma$-filling measured 
geodesic lamination
Denote by $\tau_1,\tau_2$
the complete train tracks constructed from
$\sigma$ and $\lambda$ as in Lemma \ref{tightcontrol} which
are tight at the large branch $e_1,e_2$. Then up to 
isotopy, for every neighborhood $U_1,U_2$ of $e_1,e_2$
in $S$ the intersection $\tau_i\cap (S-U_i)$ coincides
with the intersection $\tau\cap (S-U_i)$ $(i=1,2)$.
As a consequence, the train track $\tau_{1,2}$ obtained
from $\tau_1,\sigma,\lambda$ 
by the construction in Lemma \ref{tightcontrol}
which is tight at the
large branch $e_2$ coincides with the
train track $\tau_{2,1}$
obtained from $\tau_2,\sigma,\lambda$ by the
construction in Lemma \ref{tightcontrol} which is tight at $e_1$.

\bigskip

If $\sigma$ is any train track on $S$ and if
$\tau,\eta$ are two complete extension of $\sigma$,
then we defined an intersection number
$i_\sigma(\tau,\eta)$ which depends on the
choice of marked points, one on each smooth
boundary component of a complementary region of $\sigma$.
A different choice of a marked point only
changes the intersection number up to a uniformly
bounded amount (compare the proof of Lemma \ref{tightcontrol}
for a more detailed explanation and
recall that the choice of the marked
point is needed to control twisting of $\eta$ relative
to $\tau$ along the smooth boundary components of $\sigma$).
The last statement of the 
following proposition then means that there
are choices of marked points on $\sigma,\sigma_\ell$ so that
the stated inequality holds true for these choices. 

For a precise formulation, for a train track $\tau$ which
is splittable to a train track $\eta$ 
(i.e. such that 
$\tau$ can be connected to $\eta$ by a splitting sequence)
denote by $E(\tau,\eta)$ the graph 
whose vertex set consists of
all train tracks which can be obtained from $\tau$
by a splitting sequence and which are splittable to $\eta$ and
where such a vertex $\xi$ is connected to a vertex $\zeta$ 
by a directed edge of length one
if $\zeta$ can be obtained from $\xi$ by a single split.

Call a splitting sequence 
$\{\sigma_i\}$ of train tracks on $S$ \emph{recurrent} if 
each of the train tracks $\sigma_i$ is recurrent.

\begin{proposition}\label{inducing}
Given a recurrent splitting sequence $\{\sigma_i\}_{0\leq i\leq \ell}$ of 
train tracks on $S$, there is an algorithm which associates to
a complete extension $\tau$ of $\sigma_0$
and a complete $\tau$-extension 
$\lambda$ of a $\sigma_\ell$-filling measured
geodesic lamination $\nu$ a
sequence $\{\tau_i\}_{0\leq i\leq 2\ell}
\subset {\cal V}({\cal T\cal T})$ with the
following properties.
\begin{enumerate}
\item $\tau_0=\tau$, and 
for each $i\leq \ell$ the train tracks $\tau_{2i},\tau_{2i+1}$ contain 
$\sigma_i$ as a subtrack and carry $\lambda$.
\item If $\sigma_{i+1}$ is obtained from $\sigma_i$ by a 
right (or left) split
at a large branch $e_i$ then $\tau_{2i+1}$ is the 
$(e_i,\lambda)$-modification
of $\tau_{2i}$, and $\tau_{2i+2}$ is obtained from
$\tau_{2i+1}$ by a right (or left) split at $e_i$.
\item The train track $\tau_{2\ell}$ only depends
on $\tau,\sigma,\sigma_{\ell},\lambda$ but not
on the choice of a splitting sequence connecting
$\sigma$ to $\sigma_\ell$.
\item Every complete train track 
$\tau^\prime\in E(\tau,\tau_{2\ell})$ contains
a subtrack $\sigma^\prime\in E(\sigma_0,\sigma_\ell)$.
\item If $\{\eta_i\}_{0\leq i\leq 2\ell}$ is another
such sequence beginning
with a complete extension $\eta=\eta_0$ of $\sigma$
then 
\[i_{\sigma_\ell}(\tau_{2\ell},\eta_{2\ell})\leq i_\sigma(\tau,\eta)
+4q^5.\]
\end{enumerate}
\end{proposition}
\begin{proof}
Let $\sigma^\prime$ be a train track
which can be obtained from a train track $\sigma$
by a single split at a large branch $e$. Let $U$ be 
any neighborhood of $e$ in $S$. Then up to modifying
$\sigma^\prime$ with an isotopy we may assume that 
$\sigma^\prime\cap (S-U)=\sigma\cap (S-U)$ and that 
there is a map $F:S\to S$ of class $C^1$ 
which equals the identity on $S-U$ and 
which restricts to a carrying map $\sigma^\prime\to\sigma$.
In particular, there is a natural bijection $\psi$ between the
complementary regions of $\sigma$ and the complementary
regions of $\sigma^\prime$ which preserves the topological
type of the regions and which maps a complementary
region $C$ of $\sigma$ to the complementary region 
$\psi(C)$ of 
$\sigma^\prime$ containing $C-U$ (here we assume that
$U$ is sufficiently small that $C- U\not=\emptyset$
for every complementary region $C$ of $\sigma$). 

If $T$ is a smooth
side of $C$ then there is a smooth side $T^\prime$ 
of $\psi(C)$ whose image in $\sigma^\prime$ is 
mapped by the carrying map $F$ onto the image of $T$ in
$\sigma$. Let 
$\rho:[0,n]\to \sigma$ be a trainpath which
parametrizes the image of $T$ in $\sigma$. 
Then $\rho$ passes through
any branch of $\sigma$ at most twice, in opposite
direction. In particular, the length $n$ of $\rho$ 
is at most $2q$ where as before, $q$ is the
number of branches of a complete train track on $S$
(which is the maximal number of branches of any train
track on $S$). If $\rho[0,n]$ contains the branch $e$, then
the image in $\sigma^\prime$ of the side
$T^\prime$ of $\psi(C)$ does not 
pass through the diagonal branch of the split.
As a consequence,
the length of a trainpath $\rho^\prime$ on $\sigma^\prime$
parametrizing the image of $T^\prime$
is strictly smaller than the length $n$ of the trainpath on 
$\rho$ (see Figure D).

The number of distinct smooth boundary 
components of complementary regions of $\sigma$ is bounded
from above by $3g-3+m<q/2$. If 
$\{\sigma_i\}_{0\leq i\leq \ell}$ is any 
splitting sequence, then the discussion in the previous 
paragraph shows 
that there are at most
$q^2$ numbers $i\in \{1,\dots,\ell\}$ such that
$\sigma_{i+1}$ is obtained from $\sigma_i$ by a split
at a large branch which is 
contained
in the image of a smooth boundary component of a 
complementary region of $\sigma_i$. 

Now let $\tau,\eta$ be complete extensions of a recurrent
train track $\sigma=\sigma_0$.
As in the
beginning of this section,
mark a point on each smooth boundary
component of a complementary region of $\sigma$ in such a way
that no marked
point of $\sigma$ is a switch of either $\tau$ or $\eta$
(this can always be achieved with a small isotopy
of $\tau,\eta$ preserving $\sigma$ as a set).
Let $\{\sigma_i\}_{0\leq i\leq \ell}$ be a recurrent
splitting sequence issuing from $\sigma=\sigma_0$ and 
let $\lambda,\mu$
be complete $\tau,\eta$-extensions of a $\sigma_\ell$-filling
measured geodesic lamination $\nu$.
We construct sequences $\{\tau_i\}_{0\leq i\leq 2\ell},
\{\eta_i\}_{0\leq i\leq 2\ell}
\subset {\cal V}({\cal T\cal T})$ with the properties stated
in the proposition 
inductively as follows.

Let $\tau_0=\tau,\eta_0=\eta$ and 
assume that the train tracks $\tau_{2i},\eta_{2i}$ have
already been constructed for some $i\geq 0$. 
Assume that $\sigma_{i+1}$ is obtained from $\sigma_i$
by a right (or left) split at the large branch $e_i$.
Define $\tau_{2i+1},\eta_{2i+1}$ to 
be the $(e_i,\lambda)$-modification 
(or the $(e_i,\mu)$-modification, respectively)
of $\tau_{2i},\eta_{2i}$.
By construction, these train tracks
carry the geodesic laminations $\lambda,\mu$, and they
are tight at $e_i$. 

Since $\nu$ is $\sigma_\ell$-filling, 
the right (or left) split of $\sigma_i$ at $e_i$ is the unique
split so that the split track carries $\nu$. Namely, otherwise
$\nu$ is carried by the train track obtained from $\sigma_i$ by
splitting at $e_i$ and removing the diagonal of the split.
But this then means that a carrying map $\nu\to \sigma_{i+1}$ is 
not surjective which violates the assumption that
$\nu$ fills $\sigma_\ell\prec\sigma_{i+1}$. 
Define $\tau_{2i+2},\eta_{2i+2}$ to be the 
train track obtained from $\tau_{2i+1},\eta_{2i+1}$ by a right
(or left) split at the large
branch $e_i$.
Then $\tau_{2i+2},\eta_{2i+2}$ contains
$\sigma_{i+1}$ as a subtrack, and by the above reasoning,
it is the unique train track obtained from 
$\tau_{2i+1},\eta_{2i+1}$ by a split at $e_i$
which carries $\nu$. On the other hand, there is a unique
choice of a split of $\tau_{2i+1},\eta_{2i+1}$ 
at $e_i$ so that
the split track carries $\lambda,\mu$ and hence $\nu$.
But $\nu$ is a sublamination of $\lambda,\mu$ and  
therefore the train tracks
$\tau_{2i+2},\eta_{2i+2}$ carry $\lambda,\mu$. In particular,
these train tracks are complete.

As a consequence, the inductively defined sequences 
$\{\tau_i\}_{0\leq i\leq 2\ell},
\{\eta_i\}_{0\leq i\leq 2\ell}$ 
have properties 1)-2) stated in 
the proposition. 
The third property follows from the fact that
a splitting sequence connecting $\sigma$ to $\sigma_\ell$
is unique up to order (Lemma 5.1 of \cite{H09}
is also valid for splitting sequences of train tracks
which are not complete since the assumption of completeness
is nowhere used in the proof) and from the second
remark after Lemma \ref{tightcontrol}.
Namely, by this remark, the train track obtained
from a complete extension
$\tau$ of $\sigma$ and a complete $\tau$-extension
of a $\sigma$-filling measured geodesic lamination
$\lambda$ 
by two consecutive applications of Lemma \ref{tightcontrol}
at distinct large branches $e_1,e_2$ of $\sigma$ 
only depends on $\tau,\sigma,\lambda,e_1,e_2$ but not
on the order in which these two applications of 
Lemma \ref{tightcontrol} are carried out.

Property 4) follows in the same way by induction
on the length of a splitting sequence connecting
$\tau$ to $\tau_{2\ell}$. If this length vanishes
then there is nothing to show, so assume that
the claim holds true whenever the length of such a
sequence does not exceed $n-1$ for some $n\geq 1$.
Under the hypotheses used throughout this proof, 
assume that
the length of a splitting
sequence connecting $\tau$ to $\tau_{2\ell}$ equals $n$. 

Let $\tau^\prime\in E(\tau,\tau_{2\ell})$. If $\tau^\prime=\tau$
then $\tau^\prime$ contains $\sigma$ as a subtrack and there is
nothing to show. Otherwise there is a train track
$\tilde \tau\in E(\tau,\tau^\prime)\subset E(\tau,\tau_{2\ell})$
which can be obtained from
$\tau$ by a single split at a large branch $b$.
By uniqueness of splitting
sequences (Lemma 5.1 of \cite{H09}), 
we have 
$b\subset \sigma$. 

If $b$ is a large
branch of $\sigma$ (i.e. if $\tau$ is tight at $b$) then
it follows once again by uniqueness of splitting sequences that
$\tilde \tau$ contains
a subtrack $\tilde \sigma\in E(\sigma,\sigma_\ell)$ 
which can be obtained from $\sigma$
by a single split at $b$. 
Property 4) now follow s from
property 3) and the induction hypothesis, applied
to $\tilde \tau,\tau_{2\ell},\tilde \sigma,\sigma_\ell,\tau^\prime$.
Otherwise  
$b$ is a large proper subbranch of $\sigma$.
If $b$ is of type 2 then any split of $\tau$ at 
$b$ contains $\sigma$ as a subtrack. If $b$ is
of type 1 or type 3 then there is a unique
split of $\tau$ at $b$ so that the split track contains 
$\sigma$ as a subtrack, and by the previous discussion,
the split track coincides with $\tilde \tau$. Once again,
we can apply the induction hypothesis
to $\tilde \tau,\tau_{2\ell},\sigma,\sigma_\ell,\tau^\prime$  
to complete the induction step and hence the proof of property 4).

We are left with the verification of 
property 5). For this we control
the increase of intersection numbers between the
train tracks $\tau_{2i},\eta_{2i}$ and 
$\tau_{2i+2},\eta_{2i+2}$. This is done by distinguishing
two cases.

{\sl Case 1:} No marked point of a smooth
side of a complementary component of $\sigma_i$ is mapped
into the large branch $e_i$ of $\sigma_i$.
 
By two applications of 
Lemma \ref{tightcontrol}, in this case
we have 
\begin{align}
i_{\sigma_i}(\tau_{2i+1},\eta_{2i+1}) & 
\leq 
i_{\sigma_i}(\tau_{2i+1},\eta_{2i})+
q(\beta(\eta_{2i},\sigma_i)-\beta(\eta_{2i+1},\sigma_i))
\notag \\
\leq i_{\sigma_i}(\tau_{2i},\eta_{2i})&
+ q(\beta(\tau_{2i},\sigma_i)-\beta(\tau_{2i+1},\sigma_i))
+q(\beta(\eta_{2i},\sigma_i)-\beta(\eta_{2i+1},\sigma_i)).
\notag
\end{align}

Let $\psi$ be the natural bijection between the complementary
regions of $\sigma_i$ and the complementary regions of
$\sigma_{i+1}$ as introduced in the first paragraph of this proof.
Up to isotopy, for an arbitrary given 
neighborhood $U$ of $e_i$ in $S$ and for any 
complementary region $C$ of $\sigma_i$, there is
a diffeomorphism $F$ of the completion
$\overline{C}$ of $C$ onto the completion
$\overline{\psi(C)}$ of $\psi(C)$ 
respecting cusps and marked points and which
equals the identity outside of $U$.
The intersection $\tau_{2i+2}\cap \psi(C)$ is 
equivalent to $F(\tau_{2i+1}\cap C)$, and
$\eta_{2i+2}\cap \psi(C)$ is equivalent
to $F(\eta_{2i+1}\cap C)$.
This shows that
\begin{align}
i_{\sigma_{i+1}}(\tau_{2i+2},\eta_{2i+2})
& =i_{\sigma_i}(\tau_{2i+1},\eta_{2i+1})\notag \\
\leq i_{\sigma_i}(\tau_{2i},\eta_{2i})&
+ q(\beta(\tau_{2i},\sigma_i)-\beta(\tau_{2i+1},\sigma_i))
+q(\beta(\eta_{2i},\sigma_i)-\beta(\eta_{2i+1},\sigma_i)).
\notag
\end{align}

{\sl Case 2:} There is a marked point on a smooth boundary
component of a complementary region of $\sigma_i$ which 
is mapped into the branch $e_i$.

In this case, $e_i$ is contained in the image of one
or two smooth boundary
components $T_1,T_2$ of complementary regions of $\sigma_i$.
By two applications of Lemma \ref{tightcontrol}, we have
\[i_{\sigma_i}(\tau_{2i+1},\eta_{2i+1})\leq 
i_{\sigma_i}(\tau_{2i},\eta_{2i+1})+q^3
\leq i_{\sigma_i}(\tau_{2i},\eta_{2i})+2q^3.\]

The marked points on $T_1,T_2$ determine marked
points on smooth sides $T_1^\prime,T_2^\prime$ of complementary
regions of $\sigma_{i+1}$ so that we have
\[i_{\sigma_{i+1}}(\tau_{2i+2},\eta_{2i+2})=
i_{\sigma_i}(\tau_{2i+1},\eta_{2i+1})\leq 
i_{\sigma_i}(\tau_{2i},\eta_{2i})+2q^3.\]

Now by the consideration in the beginning of this proof, 
Case 2 can occur at most $q^2$ times. 
Moreover, the number of neighbors of $\sigma$ in 
$\tau,\eta$ is bounded from above by 
the upper bound
$q$ for the number of switches of $\tau,\eta$
and hence
this number can not be decreased by more than $q$ in this process. 
Together we conclude that there are at most
$q^2+2q\leq 2q^2$ among the numbers $0,\dots,\ell-1$ such that
$i_{\sigma_{i+1}}(\tau_{2i+2},\eta_{2i+2})\not=
i_{\sigma_i}(\tau_{2i},\eta_{2i})$. Since 
\[\vert i_{\sigma_{i+1}}(\tau_{2i+2},\eta_{2i+2})
-i_{\sigma_i}(\tau_{2i},\eta_{2i})\vert \leq 2q^3\]
for all $i$, this completes
the proof of the proposition.
\end{proof}

We call the sequence $\{\tau_j\}_{0\leq j\leq 2\ell}$ constructed
in Proposition \ref{inducing} from a recurrent
splitting sequence $\{\sigma_i\}_{0\leq i\leq \ell}$, a complete
extension $\tau$ of $\sigma_0$ and a complete
$\tau$-extension of a $\sigma_\ell$-filling
measured geodesic lamination a sequence
\emph{induced} by $\{\sigma_i\}$.

If $\sigma_0$ is an arbitrary (not necessarily recurrent)
subtrack of a complete train track $\tau_0$ and
if $\{\sigma_i\}_{0\leq i\leq \ell}$ 
is a splitting sequence issuing from $\sigma_0$ then
the definition of a sequence of complete train tracks
$\{\tau_i\}_{0\leq i\leq 2\ell}$ 
induced by the splitting sequence $\{\sigma_i\}_{0\leq i\leq \ell}$
always makes sense. However, if the splitting sequence
$\{\sigma_i\}$ is not recurrent then such an induced 
sequence of complete train tracks may not exist.

\begin{corollary}\label{inducing10}
For every $R>0$ there is a number $p(R)>0$ with the following
property.
Let $\{\sigma_i\}_{0\leq i\leq \ell}$ be a recurrent 
splitting sequence of train tracks on $S$.
Let $\tau,\eta$ be complete extensions of $\sigma_0$
with $d(\tau,\eta)\leq R$ and let 
$\tau^\prime,\eta^\prime$ be the endpoints of 
a sequence induced by $\{\sigma_i\}$ and
issuing from $\tau,\eta$. Then 
\[d(\tau^\prime,\eta^\prime)\leq p(R).\]
\end{corollary}

{\bf Remark:} Corollary \ref{inducing10} corresponds to
the fact that
splitting a complete train track $\tau$ along a subtrack
$\sigma$  
does not change the intersection of $\tau$ with 
the complement of a neighborhood of $\sigma$ in $S$.
Moreover, it commutes with the action of the pure
mapping class group of a bordered subsurface of $S$
which is contained in a complementary component of 
$\sigma$. We chose to introduce intersection numbers
to find an easy quantitative description 
of this fact, the main difficulty being the 
possibility of twisting about boundary components
of complementary regions.

\bigskip

For a simple geodesic
multi-curve $c$ and a train track $\tau$ which carries $c$  
we denote by $\tau(c)\subset \tau$ the
subgraph of $\tau$ of all branches which are contained 
in the image of $c$
under a carrying map.  Note that $\tau(c)$ is a recurrent
subtrack of $\tau$ and
hence either it is a disjoint union of simple closed curves which
define the multi-curve $c$, or it contains a large branch (see
\cite{PH92}).

Now let more specifically $Q$ be a pants decomposition
of $S$. If $C$ is any complementary
component of $\tau(Q)$, 
then a simple closed
curve $c$ contained 
in $C$ is disjoint from $Q$. Thus if $c$ 
is neither contractible nor freely homotopic
into a puncture of $S$ then $c$ is  freely homotopic to a
component of the pants decomposition $Q$. In particular, 
the Euler characteristic of the completion of a
complementary component of $\tau(Q)$ is at least $-1$. Thus
this completion is of one of the following
seven types, where in our terminology, a pair of pants can be a
twice punctured disc 
or a once punctured annulus or a
planar compact bordered surface of Euler 
characteristic $-1$ with three boundary
components.
\begin{enumerate}
\item A \emph{triangle}, i.e. a disc with three cusps at the
boundary. 
\item A \emph{quadrangle}, i.e. a disc with four cusps
at the boundary.
\item A punctured disc with one cusp at the boundary.
\item A punctured disc with two cusps at the boundary.
\item An annulus with one cusp at the boundary.
\item An annulus with two cusps at the boundary.
\item A pair of pants with no cusps at the boundary.
\end{enumerate}
If $C$ is a complementary component of $\tau(Q)$ which
is an annulus then the core curve
of this annulus is freely 
homotopic to a component of $Q$. 
Since $\tau(Q)$ carries $Q$, this implies that 
if the boundary of $C$ contains
two cusps then these cusps are contained in the
same boundary component, i.e. one of the boundary components of
$C$ is a smooth circle.
In other words, every complementary component of type (6)
is of the more restricted following type.
\begin{enumerate}
\item[($6^\prime$)] An annulus with one smooth boundary component
and one boundary component containing two cusps.
\end{enumerate}

For the proof of Proposition \ref{density} we need
some control on $\tau(Q)$ where $Q$ is any pants decomposition
of $S$ 
and where $\tau$ is a train track in standard form for 
a marking $F$ of $S$ which carries $Q$. 
The next lemma provides such a control.
It follows immediately from the work
of Penner and Harer \cite{PH92} and uses
an argument due to Thurston (see \cite{FLP91}). 
We present the lemma in the form needed in Section 5.

\begin{lemma}\label{markingcontrol}
Let $F$ be a marking for $S$ with
pants decomposition $P$. Let $\nu$ be a 
measured geodesic lamination; 
then there is a train track $\sigma$
with the following properties:
\begin{enumerate}
\item
$\sigma$
carries $\nu$. 
\item $\nu$ fills $\sigma$.
\item Every train track in standard
form for $F$ which carries $\nu$ 
contains $\sigma$ as a subtrack.
\end{enumerate}
\end{lemma}
\begin{proof} We begin with showing the lemma in the
case that $\nu$ is supported in a simple geodesic multi-curve.

Thus let $F$ be a marking of $S$ with pants
decomposition $P$ and let
$c$ be a simple geodesic multi-curve. Let $S_0$ be
a connected component of $S-P$ with
boundary circles $\gamma_i\in P$ (the number of these
circles is contained in $\{1,2,3\}$).
Up to homotopy, the multi-curve $c$ intersects
$S_0$ in a (perhaps empty) 
collection of disjoint simple arcs with endpoints
on the boundary of $S_0$
which are \emph{essential}, i.e.
not homotopic with fixed endpoints into the
boundary of $S_0$.

For each $i$ let $n(\gamma_i)$ be the
\emph{intersection number} between $c$ and $\gamma_i$.
Note that if $\gamma_i$ is a component of $c$ then
$n(\gamma_i)=0$.
Since any two essential simple arcs in $S_0$ 
with endpoints on the same
boundary components of $S_0$
are isotopic in $S_0$ relative
to the boundary, 
there is up to isotopy
a unique configuration of mutually disjoint simple arcs 
in $S_0$ with endpoints on the boundary of $S_0$ 
which realizes the intersection numbers $n(\gamma_i)$
(see \cite{FLP91} for details).
For this configuration there is
a unique isotopy class of a train
track (with stops) in $S_0$ which carries the
configuration with a surjective carrying map and which
can be obtained
from a standard model as shown in Figure A
by removing some (perhaps all) of the branches
(Figure 2.6.2 of \cite{PH92} shows in detail
how to remove some of the branches of a 
standard model).
These train tracks with stops can be glued to
connectors
obtained from the standard models shown in Figure B
by removing some
of the branches (see Figure 2.6.1 of \cite{PH92})
in such a way that
the resulting train track $\sigma$ has the following properties.
\begin{enumerate}
\item $\sigma$ 
carries $c$.
\item 
$c$ fills $\sigma$.
\item There is a train track in standard form for $F$
which contains $\sigma$ as a subtrack.
\end{enumerate}
Note that the direction of the winding of a component of $c$
relative to a curve from the marking $F$ which intersects
the pants curve 
$\gamma_i$ determines the connector about $\gamma_i$ in $\sigma$.
It is immediate from the construction that
a complete train track in standard form for $F$
carrying $c$ is an
extension of $\sigma$ (compare the discussion in \cite{PH92}).  

Now if $\nu$ is an arbitrary measured geodesic lamination
then the support of $\nu$ can be approximated in the
Hausdorff topology by a sequence $\{c_i\}$ 
of simple geodesic multi-curves \cite{CEG87}. 
There are only finitely many
train tracks which are subtracks of a train track in standard
form for $F$. Thus if $\{\sigma_i\}$ is a sequence
of train tracks as above for the multi-curves $c_i$ then
there is some train track $\sigma$ so that 
$\sigma=\sigma_i$ for infinitely many $i$. 
Since the set of all geodesic laminations carried by 
the fixed train track $\sigma$ is closed in the Hausdorff
topology, $\sigma$ carries $\nu$ and satisfies the
requirements in the lemma. 
\end{proof}

We use this to show

\begin{lemma}\label{specialcontrol} 
There is a number
$\kappa>0$ with the following
properties. Let $F$ be a marking for $S$ and let $Q$ by any pants
decomposition of $S$. Then there is a set ${\cal D}$ of
complete train tracks in standard form for some marking 
with
pants decomposition $Q$ (depending on the train track from ${\cal D}$)
with the following properties.
\begin{enumerate}
\item Every geodesic lamination in standard form 
for $Q$ is carried by some train track in the set ${\cal D}$.
\item The diameter of ${\cal D}$ in ${\cal T\cal T}$
is at most $\kappa$.
\item For every $\eta\in {\cal D}$ there is a train track
in standard form for $F$ which carries $\eta$.
\end{enumerate}
\end{lemma}
\begin{proof} By the discussion in the beginning of this
section, it suffices to
show the existence of a number $\chi>0$ with the following
properties. Let $F$ be a marking of $S$ and let
$Q$ by any pants decomposition of $S$.
Then for every geodesic lamination $\lambda$ in 
standard form for $Q$ there is a train track
$\tau(\lambda)$ with the following properties.
\begin{enumerate}
\item $\tau(\lambda)$ 
carries $\lambda$.
\item $\tau(\lambda)$ is in standard form for 
a marking with pants decomposition $Q$.
\item $\tau(\lambda)$ 
is carried by a train track $\tau_0(\lambda)$
in standard form for $F$.
\item If $\lambda^\prime$ is any 
other geodesic lamination in standard form
for $Q$ then we have $i_Q(\tau(\lambda),\tau(\lambda^\prime))
\leq \chi$.
\end{enumerate}

Thus let $F$ be a marking for $S$ with pants decomposition $P$, let $Q$
be a second pants decomposition and let
$\lambda,\lambda^\prime$ be two
geodesic laminations in standard form for $Q$. 
By Lemma \ref{standardform}, there are
unique train tracks $\tau,\tau^\prime$
in standard form for $F$ which carry
$\lambda,\lambda^\prime$; in particular,
$\tau,\tau^\prime$ carry $Q$. Let as before
$\tau(Q),\tau^\prime(Q)$
be the subtrack of $\tau,\tau^\prime$ of all branches of positive
$Q$-weight.
By Lemma \ref{markingcontrol}, 
the train tracks $\tau(Q),\tau^\prime(Q)$
are isotopic. This 
means that $\tau,\tau^\prime$
are complete extensions
of $\tau(Q)$. We equip the smooth boundary components of
complementary regions of $\tau(Q)$ with marked points and
use these marked points to define the intersection number
between the complete extensions $\tau,\tau^\prime$
of $\tau(Q)$. Then
$i_{\tau(Q)}(\tau,\tau^\prime)$ is bounded from
above by a universal constant.

Assume first that $\tau(Q)=Q$. 
By invariance under the
action of the mapping class group and cocompactness, 
Lemma \ref{fatening} together with Lemma 3.3 of \cite{H09} shows that
there is a number $\beta>0$ and there is a train track
$\eta$ in standard form for some marking 
of $S$ with pants decomposition $Q$ which carries
$\lambda$, which is carried by $\tau$
and such that $d(\tau,\eta)\leq \beta$.
Thus in this case the lemma is obvious (this can also
easily seen with a direct combinatorial argument).
Therefore we may assume that $\tau(Q)$ contains a
large branch.

Let $\{\sigma_i\}_{0\leq i\leq s}$ be a splitting
sequence issuing from $\sigma_0=\tau(Q)$ 
so that for each $i\leq s$ the pants decomposition 
$Q$ is carried by $\sigma_i$ and
fills $\sigma_i$. 
Then for each $i$ the 
pants decomposition $Q$ defines an integral transverse counting
measure on $\sigma_i$ by assigning to a branch $b$
the number of connected components of the preimage
of $b$ under a carrying map $Q\to \sigma_i$.
For $i<s$ 
the total $Q$-weight of $\sigma_{i+1}$, i.e. the sum of the weights 
of this counting measure over all branches of $\sigma_{i+1}$,
is bounded from above by the total $Q$-weight of $\sigma_i$
minus two. Namely, if $\sigma_{i+1}$ is obtained
from $\sigma_i$ by a split at the large branch $e$ and
if $e^\prime$ is the diagonal branch of the split in
$\sigma_{i+1}$, then the $Q$-weight of $e$ equals the sum of the
$Q$-weights of $e^\prime$ and the $Q$-weights of the two 
losing branches of the split. These weights are all
positive and integral. The weights of the branches
of $\sigma_i$ which are distinct from $e$ coincide with
the weights of their images in $\sigma_{i+1}$ under
the natural bijection $\phi(\sigma_i,\sigma_{i+1})$
of the branches of $\sigma_i$ onto the branches of
$\sigma_{i+1}$. Therefore the length
of the splitting sequence $\{\sigma_i\}$ is bounded from
above by the total $Q$-weight of $\tau(Q)$.

Assume that the sequence $\{\sigma_i\}_{0\leq i\leq s}$ is 
of maximal length. This means that for every large
branch $e$ of $\sigma_s$ the pants
decomposition $Q$ is carried by a collision of $\sigma_s$
at $e$ (i.e. a split followed by the removal of the diagonal).

By Proposition \ref{inducing}, there are 
complete extensions $\tau_1,\tau_1^\prime$ of 
$\sigma_s$ with 
$\tau_1(Q)=\tau_1^\prime(Q)=\sigma_s$ so that
$\tau_1$ carries $\lambda$, $\tau_1^\prime$ carries
$\lambda^\prime$ and that moreover
\begin{equation}\label{firstcut}
i_{\tau_1(Q)}(\tau_1,\tau_1^\prime)\leq 
i_{\tau(Q)}(\tau,\tau^\prime)+4 q^5.
\end{equation}

Let $e$ be any large branch of $\tau_1(Q)=\sigma_s$. 
The pants decomposition 
$Q$ is carried by the train track $\xi$ obtained from 
$\sigma_s$ by the collision at $e$. 
Let $\tau_2,\tau_2^\prime$ be the 
$(\tau_1(Q),\lambda)$-modification (or the
$(\tau^\prime_1(Q),\lambda^\prime)$-modification)
of $\tau_1,\tau_1^\prime$ at $e$.
Two applications of 
Lemma \ref{tightcontrol} show that
\begin{equation}\label{secondcut}
i_{\tau_2(Q)}(\tau_2,\tau_2^\prime)\leq 
i_{\tau_1(Q)}(\tau_1,\tau_1^\prime)+2q^2.\end{equation}

The train tracks $\tau_2,\tau_2^\prime$ are tight at $e$.
Let 
$\tau_3,\tau_3^\prime$ be the train tracks obtained from
$\tau_2,\tau_2^\prime$ by a split at $e$ with the property
that the split tracks carry $\lambda,\lambda^\prime$.
The number of branches of $\tau_3(Q)=\tau^\prime_3(Q)=\xi$
is strictly smaller than the number of branches of
$\tau_1(Q)$. The diagonal branch
$d=\phi(\tau_2,\tau_3)(e)$ of the split of $\tau_2$ at $e$
is a small branch of $\tau_3$ which is contained in a
complementary region $C$ of $\xi=\tau_3(Q)$ and 
which is attached at both
endpoints to a side of $C$. 
Let $d^\prime=\phi(\tau_2^\prime,\tau_3^\prime)(e)$
be the diagonal of the split of $\tau_2^\prime$ at $e$.

Let $U$ be a neighborhood of $e$ in $S$ which is
sufficiently small that $\sigma_s$ intersects
$U$ in the union of $e$ with the 
four half-branches of $\sigma_s$ which are incident on the
endpoints of $e$. Up to isotopy, the intersections of 
$\sigma_s,\xi$ with $S-U$ coincide. Thus 
if there is a complementary region $C$ of $\xi$
containing a smooth boundary component $T$ which newly
arises in the process then this boundary component
intersects $U$. Mark a point on $T$ which is mapped into  
$U$. Note that any marked point on
a smooth boundary component of a complementary region
of $\sigma_s$ induces a marked point on a smooth 
boundary component of a complementary region of 
$\xi$.

Let $\zeta,\zeta^\prime$ be complete
extensions of $\sigma_s$ which are $\sigma_s$-equivalent to 
$\tau_2,\tau_2^\prime$ and
which have the minimal number
of intersection points in 
$S-\sigma_s$, i.e. which realize the
intersection number $i_{\sigma_s}(\tau_2,\tau_2^\prime)$.
By the definition of equivalence, 
after perhaps replacing $\zeta,\zeta^\prime$ by
equivalent train tracks and after perhaps a modification 
with an isotopy we may assume 
that $\zeta,\zeta^\prime$ are tight at $e$.

Using once more the definition of equivalence,
the train tracks $\zeta,\zeta^\prime$ can be 
modified with a single split at $e$ to 
train tracks $\zeta_0,\zeta_0^\prime$ which
are complete extensions of $\xi$ and which
are equivalent to the complete extensions
$\tau_3,\tau_3^\prime$ of $\xi$. 
If $\tau_3,\tau_3^\prime$ is obtained
from $\tau_2,\tau_2^\prime$ by a right (or left) split at $e$
then $\zeta_0,\zeta_0^\prime$ is obtained from
$\zeta,\zeta^\prime$ by a right (or left) split at
$e$.

Let $d_0,d_0^\prime$ be the diagonal of the
split in $\zeta_0,\zeta_0^\prime$.
The train track $\xi$ intersects $U$ in two
disjoint embedded arcs which are joined by the two
branches $d_0,d_0^\prime$ of $\zeta_0,\zeta_0^\prime$.
We may assume that the branches $d_0,d_0^\prime$ either
are disjoint 
(if the train tracks $\tau_3,\tau_3^\prime$ are
both 
obtained from $\tau_2,\tau_2^\prime$ by the same
type of split, left or right) 
or that they intersect transversely in a
single point. 
By the definition of intersection numbers, this shows that
\begin{equation}\label{nomark}
i_{\tau_3(Q)}(\tau_3,\tau_3^\prime)=
i_\xi(\zeta_0,\zeta_0^\prime)
\leq 
i_{\tau_2(Q)}(\tau_2,\tau_2^\prime)+1.\end{equation}
Inequalities (\ref{nomark}) and (\ref{secondcut})
now yield that
\[i_{\tau_3(Q)}(\tau_3,\tau_3^\prime)\leq 
i_{\tau_2(Q)}(\tau_2,\tau_2^\prime)+1
\leq
i_{\tau_1(Q)}(\tau_1,\tau_1^\prime)+2q^2+1\]
and hence from the estimate (\ref{firstcut}) we obtain
(since $q\geq 2$) 
\[i_{\tau_3(Q)}(\tau_3,\tau_3^\prime)\leq 
i_{\tau(Q)}(\tau,\tau^\prime)+6q^5.\]

Repeat this procedure with the train track
$\xi=\tau_3(Q)$. After at most $k\leq q$ such steps where
$k$ is the number of branches of $\tau(Q)$ we
arrive at train tracks $\eta,\eta^\prime$ which contain
$Q$ as a disjoint union of simple closed curves and
carry $\lambda,\lambda^\prime$. 

To summarize, we obtain in at most $q$ steps 
two splitting sequences connecting
$\tau,\tau^\prime$ to train tracks
$\eta,\eta^\prime$ so that $\eta(Q)=\eta^\prime(Q)=Q$
and that $\eta,\eta^\prime$ carry $\lambda,\lambda^\prime$.
Each of these steps increases intersection 
numbers by at most $6q^5$. In particular,
the intersection number
$i_{Q}(\eta,\eta^\prime)$ is uniformly bounded and hence
the distance in ${\cal T\cal T}$ between $\eta,\eta^\prime$
is uniformly bounded as well.

Now $\lambda,\lambda^\prime$ is carried by 
$\eta,\eta^\prime$
and is in standard form for $Q$. Hence
by the reasoning in the third paragraph of this
proof, there are 
train tracks $\beta,\beta^\prime$ 
in standard form for some marking with 
pants decomposition $Q$ which carry $\lambda,\lambda^\prime$,
which are carried by $\eta,\eta^\prime$ and such that
$d(\eta,\beta)\leq \kappa,d(\eta^\prime,\beta^\prime)\leq 
\kappa$. As a consequence, the distance between 
$\beta,\beta^\prime$ is uniformly bounded.
Since $\lambda,\lambda^\prime$ were arbitrarily chosen
geodesic
laminations in standard form for $Q$ the lemma follows.
\end{proof}

Now we are ready to complete the proof of 
Proposition \ref{density}.  
Let $F,G$ be markings of  $S$ with 
pants decompositions $P,Q$. Let $\lambda$ be
a complete geodesic lamination in standard
form for $Q$ and let $\tau(\lambda)\in 
{\cal V}({\cal T\cal T})$ be as in Lemma \ref{specialcontrol}.
Then $\tau(\lambda)$ is in standard form
for a marking $G^\prime$ with pants decomposition
$Q$. Any two markings with pants decomposition
$Q$ differ from each other by a multi-twist about
the pants curves of $Q$. Thus if we write $k=3g-3+m$ for simplicity
of notation and if we let
$\theta_1,\dots,\theta_{k}$
be the positive Dehn twists about the components
$\gamma_1,\dots,\gamma_{k}$ of 
$Q$ then there is an integral vector 
$(n_1,\dots,n_{k})\in \mathbb{Z}^k$ such that
\[G=\theta_1^{n_1}\cdots \theta_{k}^{n_{k}}G^\prime.\]

Every pants curve $\gamma_i$ of $Q$
is the core curve of a twist connector for $\tau(\lambda)$.
Splitting a standard twist connector at the
large branch, with the small branch of the connector
as a winner, results
in deforming a train track by a (positive or negative)
Dehn twist about
the core curve of the connector.
The sign of the
twist is determined by the 
type of the twist connector which in turn is determined
by the 
spiraling direction 
of the geodesic lamination $\lambda$ in standard form
for $Q$ about the pants curve $\gamma_i$.
 
Assume after reordering that for some $p\leq k$ and for 
every $i\leq p$, either $n_i=0$ or 
the sign of $n_i$ coincides with the
sign determined by the spiraling direction of $\lambda$
about $\gamma_i$, 
and that for $i>p$ we have $n_i\not=0$ and
the sign of $n_i$ differs from this direction.
Let $\sigma$ be a train track in standard
form for the marking $G^\prime$ 
which is obtained from $\tau(\lambda)$ by reversing the
directions in the twist connectors about the curves
$\gamma_{{p+1}},\dots,\gamma_{k}$.
By Lemma 2.6.1 of \cite{PH92},  
$\sigma$ is complete, and 
$\sigma$
is splittable to the train track 
$\theta_1^{n_1}\cdots\theta_{k}^{n_{k}}\sigma$ in standard form for $G$.
The train track $\sigma$ carries a complete geodesic lamination  
$\lambda^\prime$ in standard form for $Q$.
By equivariance under the action of the mapping
class group and cocompactness, there is a universal
constant $\chi >0$ such that
\[d(\sigma,\tau(\lambda))\leq \chi.\]

By Lemma \ref{specialcontrol} there is a train track
$\tau(\lambda^\prime)$ 
which is in standard form
for a marking with pants decomposition $Q$,
which carries $\lambda^\prime$
and such that \[d(\tau(\lambda),\tau(\lambda^\prime))\leq 
\kappa.\]
Since ${\cal M\cal C\cal G}(S)$ acts isometrically
on ${\cal T\cal T}$, we have
\[d(\theta_1^{n_1}\cdots\theta_k^{n_k}\sigma,\,
\theta_1^{n_1}\cdots\theta_k^{n_k}\tau(\lambda^\prime))\leq 
\kappa +\chi.\]
But $\theta_1^{n_1}\cdots \theta_k^{n_k}\sigma$ is in
standard form for $G$ and therefore 
the train track 
$\eta=\theta_1^{n_1}\cdots \theta_{k}^{n_{k}}\tau(\lambda^\prime)$
is at distance at most $\kappa+\chi$ from a train track in
standard form for $G$. It contains the pants decomposition
$Q$ as an embedded subtrack.

On the other hand, there is a train track in standard 
form for $F$ which carries $\tau(\lambda^\prime)$, 
and $\tau(\lambda^\prime)$ carries $\eta$
by construction. Therefore there is train track in standard form 
for $F$ which carries $\eta$.
This completes the proof of Proposition \ref{density}.

\bigskip

{\bf Remark:} The results in this section are also valid
if the surface $S$ is a once puncture torus or a four punctured
sphere.

\section{Quasi-isometric embeddings}

In this section we use Proposition \ref{density1} to derive
Theorems 2-4 from the introduction.

We begin with an investigation of the mapping class group of
an essential subsurface $S_0$ of $S$. This means that $S_0$ is
a bordered subsurface of $S$ with the property that 
the inclusion $S_0\to S$ induces an injection 
$\pi_1(S_0)\to \pi_1(S)$ of
fundamental groups and that moreover 
every boundary component of $S_0$ 
is an essential simple closed curve in $S$. Let
${\cal P\cal M\cal C\cal G}(S_0)$ be the 
\emph{pure mapping class group} of $S_0$, i.e. the
subgroup of the mapping class group ${\cal M\cal C\cal G}(S_0)$ 
of $S_0$ of all mapping classes which fix each of the
boundary components and each of the punctures. Then 
${\cal P\cal M\cal C\cal G}(S_0)$ 
is a subgroup of ${\cal M\cal C\cal G}(S_0)$ of finite index. 
It can be identified
with the subgroup of the 
mapping class group of $S$ of all elements
which can be realized by a homeomorphism preserving
$S-S_0$ pointwise and fixing each of the punctures of $S$.

As in the introduction, call a finitely generated subgroup
$\Gamma$ of ${\cal M\cal C\cal G}(S)$ 
\emph{undistorted} if the inclusion
$\Gamma\to {\cal M\cal C\cal G}(S)$ is a 
quasi-isometric embedding. For example, every
subgroup of ${\cal M\cal C\cal G}(S)$ 
of finite index is undistorted.
The following result is implicitly
but not explicitly contained in Theorem 6.12 of \cite{MM00}.

\begin{proposition}\label{subsurface} For
an essential subsurface $S_0\subset S$
the subgroup ${\cal P\cal M\cal C\cal G}(S_0)$ of 
${\cal M\cal C\cal G}(S)$ is undistorted.
\end{proposition}
\begin{proof}
If $S_0=S_1\cup S_2$ for two disjoint 
essential subsurfaces $S_1,S_2$ 
of $S$ whose fundamental groups 
as subgroups of $\pi_1(S)$ have trivial intersection
then 
\[{\cal P\cal M\cal C\cal G}(S_0)={\cal P\cal M\cal C\cal G}(S_1)
\times {\cal P\cal M\cal C\cal G}(S_2).\]
Now a subgroup of a finitely generated group which is 
a direct product of two undistorted subgroups
is undistorted and hence
it suffices to show the proposition for connected
essential subsurfaces of $S$.
The case that $S_0$ is an essential
annulus is treated in detail in \cite{FLM01,H09}, 
so we assume that
the Euler characteristic of $S_0$ is negative.
If $S_0$ is a thrice punctured sphere then
${\cal P\cal M\cal C\cal G}(S_0)$ equals 
the free abelian group of Dehn twists about the 
boundary components of $S_0$.
Thus we also may assume that
$S_0$ is different from a thrice punctured sphere. 

Our goal is to show that any two elements of 
${\cal P\cal M\cal C\cal G}(S_0)$ can be connected
by a uniform quasi-geodesic in ${\cal M\cal C\cal G}(S)$
which is entirely contained in 
${\cal P\cal M\cal C\cal G}(S_0)$. 
For this 
let $\hat S_0$ be the surface which we obtain from
$S_0$ by replacing each boundary component by a puncture.
There is an exact sequence
\[0\to \mathbb{Z}^p\to {\cal P\cal M\cal C\cal G}(S_0) 
\stackrel{\Pi}\to {\cal P\cal M\cal C\cal G}(\hat S_0)\to 0\]
where $\mathbb{Z}^p$ is identified with the free abelian
group of Dehn twists about the boundary components of 
$S_0$.

Choose a pants decomposition $P$ for $S$ 
which contains the boundary of $S_0$
as a subset. Let $\tau$ be a complete train track in 
standard form for a marking $F$ with pants
decomposition $P$ and only twist connectors.
Let $\tau_1$ be the subtrack of $\tau$
which we obtain from $\tau$ by removing all
branches contained in the interior of
$S_0$. We can choose $\tau$ in such
a way that any two points in the same
connected component of 
$\tau_1$ can be connected
by a trainpath in $\tau_1$ (however, in general $\tau_1$
is neither connected not recurrent).

Let $c_1,\dots,c_p$ be the boundary circles of $S_0$.
Every complete train track $\sigma$
on $\hat S_0$ is a subtrack of a complete train track $\eta$ on $S$
which contains $\tau_1$ as a subtrack. 
Namely, up to isotopy, each boundary component $c_i$ of $S_0$ 
is contained 
in a complementary once punctured monogon region $C_i$ 
of $\sigma$. It cuts $C_i$ into an annulus $A_i\subset S_0$ and 
a once punctured disc.
Add a small branch $b_i$ to $\sigma\cup \tau_1$ 
which is contained in the closure $\overline{A_i}$ 
of the annulus $A_i$ and
connects the boundary of $C_i$ to
the boundary circle $c_i$ of $S_0$. 
Since $\sigma$ is complete and hence 
non-orientable, if for each $i\leq p$
we connect the branch $b_i$ 
to the circle $c_i$ in such a way that 
the resulting train track $\eta$ intersects 
an annulus neighborhood 
of $c_i$ in a twist connector as shown in Figure B,
then
$\eta$ is recurrent \cite{PH92}.
The train track $\eta$ is also very easily seen to 
be transversely recurrent and hence it is
complete.

The train track $\eta$ is not uniquely determined by $\tau_1$ and
$\sigma$. The choices made are the
positions of the additional switches
on the boundaries of the
complementary regions $C_i$, the 
inward pointing tangents of the
added branches $b_i$ at these switches and the homotopy
class with fixed endpoints of the branch 
$b_i\subset \overline{A_i}$. 
By invariance under the 
action of the group ${\cal P\cal M\cal C\cal G}(\hat S_0)$
and cocompactness, for any two 
such choices $\eta,\eta^\prime$ there is a
multi-twist $\phi$ about the multi-curve $c=\cup_ic_i$ 
such that 
the distance in ${\cal T\cal T}$ between $\eta,\phi\eta^\prime$
is uniformly bounded. The set ${\cal E}$ of all such 
extensions of all complete train tracks on $\hat S_0$ is 
invariant under the action of the group
${\cal P\cal M\cal C\cal G}(S_0)$, with finitely
many orbits and finite point stabilizers.

Let $F$ be a marking for $\hat S_0$.
By Theorem \ref{cubicaleuclid} and by 
Proposition \ref{density} and the following remark,
any complete train track $\eta$ on $\hat S_0$ can
be obtained from a train track $\sigma$ 
in standard form for
$F$ by a uniform quasi-geodesic 
in the train track complex ${\cal T\cal T}(\hat S_0)$ of 
$\hat S_0$ which is a concatenation
of a splitting sequence with an edge-path of 
uniformly bounded length.

For every train track $\sigma$ on $\hat S_0$ in standard form
for $F$ choose 
an extension $\Psi(\sigma)\in {\cal V }({\cal T\cal T})$ as above.
By Proposition \ref{inducing}, there is a universal 
number $p>0$ (depending on the topological type of $S$)
and for
every splitting sequence $\{\sigma_i\}_{0\leq i\leq m}$ 
of complete train tracks on $\hat S_0$ 
issuing from a train track $\sigma_0=\sigma$ in standard form
for $F$  
and for every complete $\Psi(\sigma)$-extension
of a $\sigma_m$-filling
measured geodesic lamination 
there is an induced
sequence $\{\tau_j\}_{0\leq j\leq 2m}\subset {\cal T\cal T}$ 
connecting $\tau_0=\Psi(\sigma)$ to 
a train track $\tau_{2m}$ which contains
$\sigma_m$ as well as $\tau_1$ as a subtrack.
There is a universal number $s>0$ such that for 
every $i<m$ the train track track $\tau_{2i+2}$ can
be obtained from $\tau_{2i}$ by a splitting
sequence whose length is contained $[1,s]$.

Since
splitting sequences are
uniform quasi-geodesics in both 
${\cal T\cal T}$ and the train track complex ${\cal T\cal T}(\hat S_0)$
of $\hat S_0$ \cite{H09},
this shows that there is a number
$c>0$ with the following property. For every train track 
$\xi\in {\cal V}({\cal T\cal T}(\hat S_0))$
there is a train track $\Psi(\xi)\in {\cal E}$ which contains
both $\xi$ and $\tau_1$ as a subtrack and 
is such that the distance in ${\cal T\cal T}(\hat S_0)$ 
between $\xi$ and a train track $\sigma$ in
standard form for $F$ is not bigger than 
$cd(\Psi(\sigma),\Psi(\xi))+c$.

The resulting
map $\Psi:{\cal V}({\cal T\cal T}(\hat S_0))\to {\cal E}$ 
is used to define a map
$\rho:{\cal P\cal M\cal C\cal G}(\hat S_0)\to 
{\cal P\cal M\cal C\cal G}(S_0)$
as follows. Let $\sigma$ be a fixed train track
in standard form for $F$. For 
$g\in {\cal P\cal M\cal C\cal G}(\hat S_0)$ define 
$\rho(g)\in {\cal P\cal M\cal C\cal G}(S_0)$ in such a way
that the distance between
$\Psi(g\sigma)$ and $\rho(g)(\Psi(\sigma))$ is uniformly bounded.
Since ${\cal P\cal M\cal C\cal G}(S_0)$ acts on 
${\cal E}$ with finitely many orbits and finite
point stabilizers and since ${\cal T\cal T}(\hat S_0)$ is
equivariantly quasi-isometric to 
${\cal P\cal M\cal C\cal G}(\hat S_0)$, 
the map $\rho$ is a coarse section of 
the projection 
$\Pi$. By this we mean that there is a universal constant 
$\kappa >0$ such that $d(\Pi\rho(g),g)\leq \kappa$ for all $g$.

The image
of ${\cal V}({\cal T\cal T}(\hat S_0))$ 
under the map $\Psi$ consists of train tracks which
contain each boundary component $c_i$ of $S_0$ as the core
curve of a twist connector. Splitting 
such a train track $\tau$ at the large branch in this
twist connector, with the small branch as the winner,
results in replacing $\tau$ by $\theta_c(\tau)$ where
$\theta_c$ is a Dehn twist about $c$ whose direction
(positive or negative) depends on the twist connector.
Thus if $\Gamma$ denotes the semi-group of Dehn
twists about the boundary components of $S_0$ determined
by the train track $\tau_1$ then 
for every 
$g\in \rho({\cal P\cal M\cal C\cal G}(\hat S_0))$ 
and every $\phi\in \Gamma$ there is a uniform
quasi-geodesic in ${\cal M\cal C\cal G}(S)$ 
connecting the identity to $\phi\rho(g)$ and which is
entirely contained in ${\cal P\cal M\cal C\cal G}(S_0)$.
However, the choice of the twist connector in the train track
$\tau_1$ was arbitrary and consequently
the unit element in ${\cal M\cal C\cal G}(S)$ can be
connected to any mapping class 
$g\in {\cal P\cal M\cal C\cal G}(S_0)$
by a uniform quasi-geodesic in ${\cal M\cal C\cal G}(S)$
which is entirely contained in ${\cal P\cal M\cal C\cal G}(S_0)$.
By invariance of the word metrics under left translation, 
this just means that 
${\cal P\cal M\cal C\cal G}(S_0)<{\cal M\cal C\cal G}(S)$ is 
undistorted.
\end{proof}

Now let $S_0$ be any non-exceptional surface of genus $g\geq 0$
with $m\geq 0$ punctures and let
$S$ be the surface $S_0$ punctured at one additional point $p$.
There is an exact sequence \cite{B74}
\[0\to \pi_1(S_0)\to {\cal M\cal C\cal G}(S)\stackrel{\Pi}{\to} 
{\cal M\cal C\cal G}(S_0)
\to 0.\]
The projection $\Pi$ is induced by the map $S\to S_0$ which consists in
closing the puncture $p$.
An element $\alpha$ of 
the fundamental group $\pi_1(S_0)$ of $S_0$ is mapped to the
element of ${\cal M\cal C\cal G}(S)$ obtained by dragging the
point $p$ along a loop in $S_0$ in the homotopy class 
$\alpha$. Braddeus, Farb and Putman \cite{BFP07}
showed that $\pi_1(S_0)$ is an exponentially distorted subgroup of
${\cal M\cal C\cal G}(S)$. We next observe that in contrast, the
projection $\Pi:{\cal M\cal C\cal G}(S)\to 
{\cal M\cal C\cal G}(S_0)$ has a coarse section which is
a quasi-isometric embedding. Here by a coarse section we mean
a map $\Psi:{\cal M\cal C\cal G}(S_0)\to 
{\cal M\cal C\cal G}(S)$ such that 
\[d(\Pi\Psi(g),g)\leq \kappa\]
for all $g\in {\cal M\cal C\cal G}(S)$ 
where $\kappa\geq 0$ is a universal constant.

\begin{proposition}\label{qiembeddingclosed}
Let $S_0$ be a non-exceptional
surface of genus $g\geq 0$ with $m\geq 0$ punctures
and let $S$ be the surface of genus $g$ with 
$m+1$ punctures. Then there is a 
coarse section for the projection  
$\Pi:{\cal M\cal C\cal G}(S)\to 
{\cal M\cal C\cal G}(S_0)$ which is a quasi-isometric embedding.
\end{proposition}
\begin{proof} Let ${\cal T\cal T}(S)$ and
${\cal T\cal T}(S_0)$ be the train track complex
of $S$ and of $S_0$.
We first define a map
$\Psi:{\cal V}({\cal T\cal T}(S_0))\to {\cal V}({\cal T\cal T}(S))$ 
as follows.

For a complete train track $\tau$ on $S_0$ choose any complementary
trigon $C$ of $\tau$. Mark a point $p$ in the interior of $C$ and
add two switches $v_1,v_2$ and two branches 
$b_1,b_2\subset C-\{p\}$ to $\tau$
in the following way. The switch 
$v_1$ is an interior point of a branch
of $\tau$ contained in a side of $C$, $v_2\not=p$ is a point
in the interior of $C$, $b_1$ connects $v_1$ to $v_2$ and 
$b_2$ is a small branch contained in the interior of $C$ 
whose endpoints are both 
incident on $v_2$
and which is the boundary of a subdisc of  $C$ 
containing $p$ in its interior.  
Since $\tau$ is complete, Proposition 1.3.7 of \cite{PH92}
shows that the resulting train track $\eta_0$ on 
$S=S_0-\{p\}$ is recurrent. It is also easily seen to be
transversely recurrent. The train track $\eta_0$ 
decomposes $S$ into 
trigons, once punctured monogons and one fourgon.
The fourgon can
be subdivided into two trigons by adding a single
small branch. The resulting
train track $\eta$ on $S$ is complete, and 
it contains $\tau$ as a subtrack.
This construction defines  
a map $\Psi:{\cal V}({\cal T\cal T}(S_0))\to 
{\cal V}({\cal T\cal T}(S))$. 

The map $\Psi$
depends on some choices among a finite set of possibilities:
The choice of the complementary trigon $C$, the choice of the
position of the switch $v_1$ on a side of $C$, the orientation
of the inward pointing tangent of the branch $b_1$ at 
the switch $v_1$
and the choice of the small branch subdividing the fourgon.
Any train track constructed in this way 
contains $\tau$ as
a subtrack. Moreover, for $g\in {\cal M\cal C\cal G}(S_0)$
there is some $h\in {\cal M\cal C\cal G}(S)$ such that
$h(\Psi(\tau))$ is one of the possibilities for $\Psi(g\tau)$.
Since there are only finitely many
orbits of complete train tracks on $S_0$ under
the action of the mapping class group,
by coarse equivariance of the construction we conclude that
there is a universal number $\kappa_0>0$ such that 
for any other choice $\Psi^\prime$ of such a map
we have $d(\Psi(\tau),\Psi^\prime(\tau))\leq \kappa_0$
for all $\tau\in {\cal V}({\cal T\cal T}(S_0))$.

We use the map $\Psi$ to define a map
$\Phi:{\cal M\cal C\cal G}(S_0)\to {\cal M\cal C\cal G}(S)$
as follows. The mapping class groups of $S_0,S$ act
properly and cocompactly on ${\cal T\cal T}(S_0)$,
${\cal T\cal T}(S)$. Choose $\tau\in 
{\cal V}({\cal T\cal T}(S_0))$ 
and a fundamental domain
$D$ for the action of ${\cal M\cal C\cal G}(S)$ on 
${\cal T\cal T}(S)$ 
containing $\Psi(\tau)$.
For $g\in {\cal M\cal C\cal G}(S_0)$
choose $\Phi(g)\in {\cal M\cal C\cal G}(S)$ in
such a way that $\Psi(g\tau)\in \Phi(g)D$. If
$\Phi^\prime$ is any other such map then
$d(\Phi(g),\Phi^\prime(g))\leq \kappa_1$ where
$\kappa_1>0$ is a universal constant (and $d$ is 
any distance on ${\cal M\cal C\cal G}(S)$ defined by 
a word norm of a finite symmetric generating set).

By construction, the map 
$\Phi:{\cal M\cal C\cal G}(S_0)\to {\cal M\cal C\cal G}(S)$ 
is a coarse section for the projection
${\cal M\cal C\cal G}(S)\to {\cal M\cal C\cal G}(S_0)$.
Thus we are left with showing that $\Phi$ 
is a quasi-isometric
embedding, and this holds true if this is the case
for the map $\Psi$. 
To this end, note that
$\tau$ is a subtrack of $\Psi(\tau)$.
By Proposition \ref{inducing}, a splitting sequence
$\{\tau_i\}_{0\leq i\leq \ell}\subset {\cal T\cal T}(S_0)$ 
issuing from $\tau_0=\tau$ 
induces a splitting sequence in ${\cal T\cal T}(S)$ 
issuing from $\Psi(\tau)$.
The length of this sequence is not smaller than the
length $\ell$ of the splitting sequence of $\tau$, and it
is not bigger than $q\ell$ for 
a universal constant $q>0$.
On the other hand, a point on the induced
sequence which contains $\tau_i$ as a subtrack
is a possible choice for $\Psi(\tau_i)$ and hence it 
is at uniformly bounded distance to 
$\Psi(\tau_i)$.

By Theorem \ref{cubicaleuclid} and the remark thereafter, 
splitting sequences in ${\cal T\cal T}(S)$
are uniform quasi-geodesics, and the same holds
true for splitting sequences in 
${\cal T\cal T}(S_0)$. 
As a consequence, there is a number $c>1$ such that
\[d(\tau_0,\tau_\ell)/c-c\leq 
d(\Psi(\tau_0),\Psi(\tau_\ell))\leq cd(\tau_0,\tau_\ell)+c\]
whenever $\tau_0\in {\cal V}({\cal T\cal T}(S_0))$ 
is splittable to $\tau_\ell\in {\cal V}({\cal T\cal T}(S_0))$. 
By Proposition \ref{density1}, splitting sequences connect
a coarsely dense set of pairs of points in the train 
track complex ${\cal T\cal T}(S_0)$.
This implies that the map
$\Psi:{\cal V}({\cal T\cal T}(S_0))\to {\cal V}({\cal T\cal T}(S))$
defines a quasi-isometric embedding and hence the same
holds true for 
$\Phi:{\cal M\cal C\cal G}(S_0)\to {\cal M\cal C\cal G}(S)$.
\end{proof}

Finally, for a closed surface
of genus $g\geq 2$ we investigate the normalizer 
of a finite subgroup $\Gamma$ of 
${\cal M\cal C\cal G}(S)$ (see \cite{RS07} for 
an earlier proof of this result, stated a bit differently).

\begin{proposition}\label{normalizer}
For a closed surface $S$ of genus $g\geq 2$, 
the normalizer in ${\cal M\cal C\cal G}(S)$ of a finite subgroup
is undistorted. 
\end{proposition}
\begin{proof}
Let ${\cal T}(S)$ be the Teichm\"uller space of $S$.
By the Nielsen realization problem, a finite 
subgroup 
$\Gamma$ of ${\cal M\cal C\cal G}(S)$ fixes 
a  point $x\in {\cal T}(S)$
\cite{Ke83}. This means that 
$\Gamma$ can be realized as a finite group of 
biholomorphisms of $(S,x)$. The quotient $(S,x)/\Gamma$
is a Riemann surface, and the projection
$\pi:(S,x)\to (S,x)/\Gamma$ is a branched covering ramified
over a finite number of points $p_1,\dots,p_\ell\in 
(S,x)/\Gamma$. The marked complex structure $x$ on $S$
projects to a marked complex structure on $(S,x)/\Gamma$.

Let $(S_1,x_1),(S_0,x_0)$ 
be the punctured Riemann surfaces which are obtained from
$(S,x),(S,x)/\Gamma$ by removing the branch points of the covering
$(S,x)\to (S,x)/\Gamma$. The projection $\pi$ restricts
to an unbranched covering $S_1\to S_0$.
The Teichm\"uller spaces ${\cal T}(S_0)$ of 
$S_0$, ${\cal T}(S_1)$ of $S_1$ are contractible. For every
point $y\in {\cal T}(S_0)$ which is sufficiently close
to $x_0$ there is a covering $\Psi(y)\in {\cal T}(S_1)$ of 
the Riemann surface $(S_0,y)$ 
which is of the same topological
type as the covering $(S_1,x_1)\to (S_0,x_0)$.
The marking of $\Psi(y)$ is determined in such a way that the
map $\Psi$ is continuous near $x_0$. 
This construction defines a developing map
$\Psi:{\cal T}(S_0)\to {\cal T}(S_1)$. Since
${\cal T}(S_0)$ is simply connected, the developing map
is in fact single-valued. Moreover, it is clearly
injective and hence an embedding. 
(In fact, its is not hard
to see that this construction defines an isometric
embedding of ${\cal T}(S_0)$ into ${\cal T}(S_1)$ for the
Teichm\"uller metrics). 
There is a natural projection 
$\Pi:{\cal T}(S_1)\to {\cal T}(S)$ defined by filling
in the punctures.

Let ${\cal M\cal C\cal G}_0(S_0)$ be the 
subgroup of the mapping class group 
${\cal M\cal C\cal G}(S_0)$ of $S_0$ of all
mapping classes realizable by a 
homeomorphism of $S_0$ which lifts to a homeomorphism of 
$S$. Let $N(\Gamma)$ be the normalizer of $\Gamma$
in ${\cal M\cal C\cal G}(S)$.
Then there is an exact sequence \cite{BH73}
\[0\to \Gamma\to N(\Gamma)\to {\cal M\cal C\cal G}_0(S_0)\to 0.\]
(Theorem 3 in \cite {BH73} states this only in the
case that the group $\Gamma$ is cyclic.
However, as pointed out explicitly
in \cite{BH73}, the result for all finite groups is
immediate from the argument given there and the 
Nielsen realization problem).

Since the group $\Gamma$ is finite,
the groups
${N}(\Gamma)$ and ${\cal M\cal C\cal G}_0(S_0)$ are 
quasi-isometric. Thus to show the proposition it is enough to
show that there is quasi-isometric embedding of
${\cal M\cal C\cal G}_0(S_0)$ into ${\cal M\cal C\cal G}(S)$ 
whose image is contained in a uniformly boundedd neighborhood
of $N(\Gamma)$. 
Following Proposition \ref{qiembeddingclosed}, 
it suffices in fact to show that there is
a quasi-isometric embedding of 
${\cal M\cal C\cal G}(S_0)$ 
into ${\cal M\cal C\cal G}(S_1)$ whose image 
is contained in a uniformly bounded neighborhood of 
the image of $N(\Gamma)$ under a 
a coarse section for the projection
${\cal M\cal C\cal G}(S_1)\to {\cal M\cal C\cal G}(S)$.

For this note that 
the preimage of a complete train track $\tau$ on $S_0$ under
the covering $S_1\to S_0$ 
is a $\Gamma$-invariant
graph $\xi$ in $S$ which
decomposes $S$ into polygons and once
punctured polygons.
The preimage of 
each trigon component of $\tau$ is a union of  
$n$ trigon components of $\xi$ where $n=
\vert \Gamma\vert$ is the number of sheets
of the covering. Each once punctured monogon in $\tau$ encloses one 
of the points $p_i$ and
lifts to a punctured $m_i$-gon in $S_1-\xi$  
where $2\leq m_i\leq n$ is 
the ramification index of $p_i$.

The branch points of the covering define
a set of marked points contained
in complementary regions 
of $\xi$. Each complementary region of $\xi$ contains at
most one such point. Thus $\xi$ defines a 
(non-complete) train track on the
punctured surface $S_1$, again denoted by $\xi$. Now 
a positive transverse measure
on $\tau$ lifts to a positive 
transverse measure on $\xi$ and therefore 
$\xi$ is recurrent. The same argument also 
shows that $\xi$ is transversely recurrent. Then 
$\xi$ is a subtrack of a complete train track
on $S_1$ obtained by subdividing some of 
the complementary regions 
as in the proof of Proposition \ref{qiembeddingclosed}.
As before, the resulting complete train track $\eta$ depends on 
choices among a uniformly bounded number of possibilities
(compare the proof 
of Proposition \ref{qiembeddingclosed}).

Now if the complete train track $\tau_1$ on $S_0$ 
is obtained from the complete train track
$\tau$ by a single
split at a large branch $e$ then the preimage $\xi_1$ of $\tau_1$ can
be obtained from the preimage $\xi$ of $\tau$ by a splitting 
sequence of length $n$. Namely, the preimage of 
any large branch of 
$\tau$ is the union of $n$ large branches of $\xi$.
Such a splitting sequence then induces a splitting
sequence of length at most $qn$ of the complete train track
$\eta$ on $S_1$ constructed in the previous paragraph 
where $q>0$ is a universal constant.

By Proposition \ref{density1}, splitting sequences 
in the train track complex ${\cal T\cal T}(S_0)$ 
of $S_0$ connnect
a coarsely dense set of pairs of points.
By Theorem \ref{cubicaleuclid} and the remark
thereafter, each such
splitting sequence defines a uniform quasi-geodesic in the
subgroup ${\cal M\cal C\cal G}_0(S_0)$
of the mapping class group of $S_0$. This quasi-geodesic  
lifts to a uniform quasi-geodesic
in ${\cal M\cal C\cal G}(S_1)$ contained in 
a uniformly bounded neighborhood of the image of 
$N(\Gamma)$ under the coarse section 
for the projection ${\cal M\cal C\cal G}(S_1)\to 
{\cal M\cal C\cal G}(S)$
constructed
in Proposition \ref{qiembeddingclosed}.
As a consequence, the normalizer
$N(\Gamma)$ of $\Gamma$ is undistorted.
\end{proof}

{\bf Remark:} 1) Since splitting sequences define quasi-geodesics in 
the \emph{curve graph} of a surface of finite type 
\cite{H06},
the above argument immediately
implies the following. Let $S$ be a closed surface and let 
$\Gamma$ be a finite subgroup of 
${\cal M\cal C\cal G}(S)$. Then  
there is a quasi-isometric embedding of the
curve graph of $S/\Gamma$ into the curve graph of $S$.
This was shown in \cite{RS07}.

2) In \cite{ALS08}, Aramayona, Leininger
and Souto constructed for infinitely many
$g_i>0$ injective homomorphisms of the mapping
class group of a closed surface of genus $g_i$
into the mapping class group of a closed surface
of strictly bigger genus using unbranched coverings.
The reasoning in the proof of Proposition \ref{normalizer}
can be used to show that these homomorphisms
are quasi-isometric embeddings.

\section{Distances in the train track complex}

In this section we use the results from Section 3 and
from \cite{H09} to obtain
a control on distances in the train track complex
${\cal T\cal T}$.
The main goal is the proof of Proposition \ref{shortestdistance}
which is the key ingredient for the proof of 
Theorem \ref{thm1} 
from the introduction.

Our strategy is to investigate the geometry of triangles
in the train track complex whose sides
are bounded extensions of directed edge-paths. 
There may be many such triangles with given vertices,
but we show that we can 
always find such a triangle which
is $R$-thin for a universal constant $R>0$. This 
means that a side of this triangle is contained
in the $R$-neighborhood of the union of the other
two sides. This in turn is related to
a result of Behrstock, Drutu and Sapir \cite{BDS08}
who showed that the asymptotic cone of the
mapping class  group admits the structure of a median
space. 

The first step in this direction is the 
following lemma whose 
proof uses the results from the
appendix. For its proof and for later use, we say that
a train track or a geodesic lamination $\lambda$ \emph{hits
a train track $\tau$ efficiently} if up to isotopy there
is no embedded bigon in $S$ whose frontier is composed of
two $C^1$-segments, one from the train track $\tau$ and the
second one from the train track or geodesic lamination $\lambda$.
If the train track $\sigma$ hits the train track $\tau$
efficiently then every train track or geodesic lamination
which is carried by $\sigma$ hits $\tau$ efficiently as well.

\begin{lemma}\label{reverse}
For every $R>0$ there is a number
$\beta_0=\beta_0(R)
>0$ with the following property. Let $\tau,\eta\in
{\cal V}({\cal T\cal T})$ with $d(\tau,\eta)\leq R$ and let
$\tau^\prime,\eta^\prime$ be complete train tracks which can be
obtained from $\tau,\eta$ by any splitting sequence.
If $\tau,\eta$ do not carry any common geodesic
lamination then there is a train track
$\tau^{\prime\prime}\in {\cal V}({\cal T\cal T})$ in the
$\beta_0(R)$-neighborhood of $\tau^\prime$, 
a train track $\eta^{\prime\prime}\in {\cal V}({\cal T\cal T})$
in the $\beta_0(R)$-neighborhood of $\eta^\prime$
and a splitting sequence connecting 
$\tau^{\prime\prime}$ to $\eta^{\prime\prime}$ 
which passes through the
$\beta_0(R)$-neighborhood of $\eta$.
Moreover, for any complete
geodesic lamination $\nu$ which is carried by
$\eta^\prime$, we can assume that
$\nu$ is carried by $\eta^{\prime\prime}$.
\end{lemma} 
\begin{proof} 
Fix a complete hyperbolic metric $g$ on $S$ of finite volume.
For every $R>0$ there
are only finitely many orbits under the
action of the mapping class group of
pairs $(\tau,\eta)\in {\cal V}({\cal T\cal T})\times
{\cal V}({\cal T\cal T})$ where $d(\tau,\eta)\leq R$ and such
that $\tau,\eta$ do not carry any common
geodesic lamination. Thus by invariance
under the mapping class group it is enough
to show the lemma for two fixed train tracks
$\tau,\eta\in {\cal V}({\cal T\cal T})$ which do not
carry any common geodesic lamination and with
a constant $\beta_0>0$ depending on $\tau,\eta$.

For a complete train track $\xi$ denote
by ${\cal C\cal L}(\xi)$ the set of all complete
geodesic laminations which are carried by $\xi$.
By Lemma 2.3 of \cite{H09}, the set ${\cal C\cal L}(\xi)$
is open and closed in the space
${\cal C\cal L}$ of all complete geodesic laminations
on $S$ equipped with the Hausdorff topology.
We first show that there are finitely many
complete train tracks $\tau_1,\dots,\tau_\ell$ and
$\eta_1,\dots,\eta_m$ with the following
properties.
\begin{enumerate}
\item For each
$i\leq \ell$ the train track $\tau_i$ is carried
by $\tau$ and $\cup_i{\cal C\cal L}(\tau_i)=
{\cal C\cal L}(\tau)$.
\item For each
$j\leq m$ the train track $\eta_j$ is carried
by $\eta$ and $\cup_j{\cal C\cal L}(\eta_j)=
{\cal C\cal L}(\eta)$.
\item For all $i\leq \ell,j\leq m$ the train tracks
$\tau_i,\eta_j$ hit efficiently.
\end{enumerate}

Since $\tau,\eta$ do not
carry any common geodesic lamination,
every complete geodesic lamination $\lambda\in {\cal C\cal L}(\tau)$
intersects every complete
geodesic lamination $\mu\in {\cal C\cal L}(\eta)$
transversely. Namely, a complete geodesic lamination
decomposes the surface $S$ into ideal triangles and
once punctured monogons. Thus if $\ell$ is any
simple geodesic on $S$ whose closure in $S$ is compact and if
$\ell$ does \emph{not} intersect
the complete geodesic lamination $\mu$ transversely
then $\ell$ is contained in $\mu$ and hence the
closure of $\ell$ is a sublamination
of $\mu$.
Now if $\ell$ is a leaf of the complete geodesic lamination
$\lambda$ then the closure of $\ell$
is a sublamination of $\lambda$ as well.
Since $\lambda$ is carried by $\tau$ and $\mu$ is
carried by $\eta$, this violates the assumption
that $\tau,\eta$ do not carry a common geodesic
lamination.

Recall from Section 2 the definition of 
the straightening of a train track $\tau$ with 
respect to the complete hyperbolic metric $g$. 
For a number $a>0$ call a train track
$\tau$ \emph{$a$-long} if the length of 
each edge from its straightening
ist at least $a$. Recall also from Section 2 the
definition of a train track which $\epsilon$-follows
a complete geodesic lamination $\lambda$ for some $\epsilon >0$.
For a fixed number $a>0$, a trainpath on an $a$-long train track
$\tau_\epsilon$ which $\epsilon$-follows $\lambda$ is a uniform
quasi-geodesic. As $\epsilon\to 0$, these quasi-geodesics
converge locally uniformly to leaves of $\lambda$.

For a fixed geodesic lamination
$\lambda\in {\cal C\cal L}(\tau)$, the 
compact space 
${\cal C\cal L}(\eta)$ consists of geodesic laminations
which intersect 
$\lambda\in {\cal C\cal L}(\tau)$
transversely. Thus the tangent lines of the leaves of 
$\lambda$ and the union of all tangent lines of all 
leaves of all geodesic laminations in ${\cal C\cal L}(\eta)$
are disjoint compact subsets of the projectivized 
tangent bundle of $S$. As a consequence, 
for any fixed number $a>0$
there is a number
$\epsilon >0$ depending on $\lambda$ 
such that every $a$-long train track $\xi$
which $\epsilon$-follows $\lambda$ 
hits every geodesic lamination $\nu\in {\cal C\cal L}(\eta)$
efficiently. By Lemma 2.2 and Lemma 2.3 of
\cite{H09}, this implies that
for every complete geodesic lamination
$\lambda\in {\cal C\cal L}(\tau)$ there is a
train track $\tau(\lambda)\in {\cal V}({\cal T\cal T})$
which carries $\lambda$, which
is carried by $\tau$ and which hits every geodesic
lamination $\nu\in {\cal C\cal L}(\eta)$ efficiently.
By Lemma 2.3 of \cite{H09}, 
${\cal C\cal L}(\tau(\lambda))$ is
an open subset of ${\cal C\cal L}(\tau)$. Since 
${\cal C\cal L}(\tau)$ is compact we conclude that there
are finitely many geodesic laminations
$\lambda_1,\dots,\lambda_\ell\in {\cal C\cal L}(\tau)$ such that
${\cal C\cal L}(\tau)=\cup_i{\cal C\cal L}(\tau(\lambda_i))$.
Write $\tau_i=\tau(\lambda_i)$.

Every geodesic lamination $\nu\in {\cal C\cal L}(\eta)$
hits each of the train tracks $\tau_i$ efficiently.
As a consequence, for any fixed number $a>0$ there is a 
number $\epsilon >0$ depending on $\nu$ such that every 
$a$-long train track which $\epsilon$-follows $\nu$ 
hits each of the
train tracks $\tau_i$ $(i\leq \ell)$ efficiently.
As before, this implies that we can find a finite
family $\eta_1,\dots,\eta_m\in {\cal V}({\cal T\cal T})$
of train tracks
which are carried by $\eta$, which hit each of the
train tracks $\tau_i$ $(i\leq \ell)$ efficiently and such that
$\cup_j{\cal C\cal L}(\eta_j)={\cal C\cal L}(\eta)$.
This shows the above claim.

Let \begin{equation}\label{one}
k=\max\{d(\tau,\tau_i),d(\eta,\eta_j)\mid
i\leq \ell, j\leq m\}.\end{equation} 
Let
$\tau^\prime,\eta^\prime$ be obtained
from $\tau,\eta$ by a splitting sequence and let
$\lambda\in {\cal C\cal L}(\tau^\prime)$ be
a complete geodesic lamination carried by $\tau^\prime$.
Then $\lambda\in {\cal C\cal L}(\tau)$ and hence
there is some $i\leq \ell$
such that $\lambda\in {\cal C\cal L}(\tau_i)$.
By Lemma 6.7 of \cite{H09}, there is a 
number $p_3(k)>0$ only depending on $k$ and there is a
complete train track
$\xi$ which carries $\lambda$, is
carried by both $\tau^\prime$ and $\tau_i$ and
such that 
\begin{equation}\label{two}
d(\tau^\prime,\xi)\leq p_3(k).\end{equation}
Similarly, for a complete train track $\eta^\prime$ which
can be obtained from $\eta$ by a splitting sequence
and for any $\nu\in {\cal C\cal L}(\eta^\prime)$
there is some $j\leq m$ and
a complete train track $\zeta$ which carries $\nu$, 
is carried by both $\eta^\prime$ and $\eta_j$
and such that
\begin{equation}\label{three}
d(\eta^\prime,\zeta)\leq p_3(k).\end{equation}

Since $\tau_i,\eta_j$
hit efficiently and $\tau_i$ carries $\xi$,
the train track $\xi$ hits the 
train track $\eta_j$ efficiently.
In particular, the geodesic lamination
$\nu$  (which is carried by $\eta_j$)
hits $\xi$ efficiently. Therefore
by Proposition \ref{backwards} from the
appendix, there is a number $p>0$ not depending 
on $\xi,\eta_j$ and there
is a $\nu$-collapse
$\xi^*$ of the dual bigon track $\xi_b^*$
of $\xi$ with the following
property. $\xi^*$ carries a train track $\beta$
which carries $\nu$ and which can be obtained from
$\eta_j$ by a splitting and shifting sequence of length
at most $p$. In particular, 
we have 
\begin{equation}\label{four}
d(\eta_j,\beta)\leq \kappa\end{equation} 
where $\kappa >0$ is a universal constant.

Since $\eta_j$ carries $\zeta$ and
$\zeta$ carries $\nu$, we conclude from Lemma 6.7
of \cite{H09}
that $\beta$ carries a train track $\sigma$
which carries $\nu$ and such that
$d(\sigma,\zeta)\leq p_3(\kappa)$. 
Together with the estimate (\ref{three}) above, 
this implies that
\begin{equation}\label{five} 
d(\sigma,\eta^\prime)\leq p_3(k)+p_3(\kappa).
\end{equation}

Now 
$d(\tau^\prime,\xi)\leq p_3(k)$ by inequality (\ref{two})
and therefore
the distance between $\xi^*$ and
$\tau^\prime$ is uniformly bounded
(compare Proposition \ref{backwards} and the following remark).
On the other hand, since $\xi^*$
carries $\beta$ and $\beta$ carries $\nu$,
Proposition A.6 of 
\cite{H09} shows that $\xi^*$ is splittable to a train track
$\beta^\prime$ which carries $\nu$ and is 
contained in a uniformly bounded neighborhood of
$\beta$. Since $d(\beta,\eta)\leq k+\kappa$ by inequalities
(\ref{one},\ref{four}), we conclude that 
$d(\beta^\prime,\eta)$ is uniformly bounded.
Now $\beta$ carries $\sigma$ and both $\beta^\prime$
and $\sigma$ carry $\nu$ and hence another application of 
Lemma 6.7 of \cite{H09} shows 
that $\beta^\prime$ is splittable to a train track
$\sigma^{\prime}$ which carries $\nu$ and is
contained in a uniformly bounded neighborhood of
$\sigma$. Since by inequality
(\ref{five}) the distance between $\sigma$ and
$\eta^\prime$ is uniformly bounded,
this implies the lemma with $\tau^{\prime\prime}=
\xi^*$ and $\eta^{\prime\prime}=\sigma^\prime$.
\end{proof}

Our next goal is to establish an extension
of  Lemma \ref{reverse} 
to train tracks $\tau,\eta\in {\cal V}({\cal T\cal T})$
containing a common subtrack $\beta$ 
which is a union of simple closed
curves and such that every minimal geodesic lamination carried
by both $\tau,\eta$ is a component of $\beta$.
The main idea is as follows. If $c$ is an embedded
simple closed curve of class $C^1$ in a complete
train track $\tau$ then a sequence  of $c$-splits 
of $\tau$ results in modifying $\tau$ by a sequence
of Dehn twists about $c$ up to 
a uniformly bounded error. The direction of the twist
(positive or negative) is determined by a complete
geodesic lamination which is carried by $\tau$ and
contains $c$ as a minimal component.
If $\eta$ is another complete train track containing
$c$ as an embedded simple closed curve and if $c$ 
is the only geodesic lamination which is 
carried by both $\tau,\eta$ 
then we find a version
of Lemma \ref{reverse} for the images of $\tau,\eta$ 
under a suitably chosen multi-twist about $c$. 

We first need a precise control of the distance between
a train track obtained from a train track $\tau$ by 
a sequence of $c$-splits and the image of $\tau$ under
a suitably chosen multi-twist about $c$.
For this consider for the moment an arbitrary 
train track $\tau$
and let $\rho:[0,k]\to \tau$
be any trainpath which either is embedded in $\tau$ 
or is such that $\rho[0,k-1)$ is embedded,
and $\rho(k-1)=\rho(0),\rho(k)=\rho(1)$. 
For $1\leq i\leq k-1$ call the
switch $\rho(i)$ of $\tau$ 
\emph{incoming} if the half-branch
$\rho[i-1/2,i]$ is small at $\rho(i)$, and call 
the switch $\rho(i)$ \emph{outgoing} otherwise.
Note that this depends on the orientation of $\rho$.
We call a neighbor of $\rho[0,k]$ in 
$\tau$ (i.e. a half-branch of $\tau$ which is incident
on a switch in $\rho[1,k-1]$ but which is not
contained in $\rho[0,k]$) incoming if
it is incident on an incoming switch, and we call
the neighbor outgoing otherwise.
Call the switch $\rho(i)$ a \emph{right}
(or \emph{left}) switch
if with respect to the orientation of $S$ and
the orientation of $\rho$, the neighbor of $\rho$ 
incident on $\rho(i)$ lies to the right (or left)
of $\rho$ in a small neighborhood of $\rho[0,k]$ in $S$.
A neighbor incident on a right (or left) switch is
called a \emph{right} (or \emph{left}) neighbor of 
$\rho[0,k]$. Again this depends on the orientation of $\rho$.

In the sequel we mean by an embedded simple closed
curve $c$ in a train track $\tau$ a subtrack of 
$\tau$ which is an 
embedded simple closed curve freely homotopic to $c$.
Call an embedded simple closed curve $c$ in a 
train track $\tau$ \emph{reduced} if 
for one (and hence every) trainpath
$\rho:[0,k]\to \tau$ with image $c$ and 
$\rho[0,1]=\rho[k-1,k]$, 
there are both left and right neighbors of 
$c$ in $\tau$ and 
either all left 
neighbors of $c=\rho[0,k]$ in $\tau$ 
are incoming and all right neighbors are outgoing,
or all left neighbors are outgoing and all
right neighbors are incoming.
Note that if $\tau$ is
recurrent then 
for any embedded simple closed curve $c$
in $\tau$ which is not isolated 
(i.e. not a component of $\tau$)
the following holds true. If
with respect to an orientation of $c$ and the
orientation of $S$ 
all left neighbors of $c$ in $\tau$ 
are incoming (or outgoing), then at least one of 
the right neighbors of $c$ in $\tau$ is outgoing
(or incoming).

We call a reduced simple closed 
curve $c$ in $\tau$ \emph{positive}
if the following holds true. Let $S_c$ be the bordered surface
with two boundary circles obtained by cutting 
$S$ open along $c$. 
The orientation of $S$ induces a boundary
orientation for each of the two
boundary components of $S_c$. Then $c$ is   
positive if with respect to this boundary orientation,
neighbors which lie to the left of $c$ are incoming.
Note that this holds true for both boundary components if
it holds true for one. A reduced simple closed
curve which is not positive is called 
\emph{negative}. 

A Dehn twist about an essential simple closed curve $c$ in
$S$ is defined to be \emph{positive} if the direction of
the twist coincides with the direction given by the boundary
orientation of $c$ in the surface obtained by cutting $S$
open along $c$. 
The following observation gives a quantitative
version of the idea that for an embedded reduced simple closed
curve $c$ in a 
train track $\tau$, a sequence of $c$-splits
of $\tau$ results in twisting of $\tau$ about $c$.

\begin{lemma}\label{dehnreduced}
Let $c$ be a reduced 
positive (or negative)
embedded simple closed curve
in a train track $\tau$ and let 
$\theta_c$
be the positive Dehn twist about $c$. 
\begin{enumerate}
\item 
There is a 
sequence of $c$-splits
of uniformly bounded length 
which transforms $\tau$
to $\theta_c\tau$ (or $\theta_c^{-1}\tau$). 
\item There is a number $a_1>0$ not depending
on $\tau$ and $c$ and for every
train track $\eta$ obtained from $\tau$ by a sequence of $c$-splits
there is some $i\geq 0$ such that
$d(\eta,\theta_c^{i}\tau)\leq a_1$ (or
$d(\eta,\theta_c^{-i}\tau)\leq a_1$).
\end{enumerate}
\end{lemma}
\begin{proof}
Let $c$ be an embedded reduced simple closed curve
in $\tau$. Assume for the purpose of exposition that
$c$ is positive (the case of a negative curve is completely
analogous). Let $\rho:[0,\ell]\to \tau$ be a parametrization of 
$c$ as a trainpath with $\rho(\ell)=\rho(0)$. 
Let $0\leq i_1<\dots<i_k$ be such that the switches
$s_j=\rho(i_j)$ are outgoing and that the switches
$\rho(j)$ for $j\not\in \{i_1,\dots,i_k\}$ are incoming.
Let moreover
$b_1,\dots,b_p$ $(p=\ell-k)$ 
be the incoming neighbors of $c$, ordered 
consecutively along $\rho$.
By assumption, these are precisely the 
left neighbors of $c$ in $\tau$ with respect to the
orientation defined by $\rho$.

Since 
there are both incoming and outgoing
switches along $c$, there is a large branch
contained in $c$.
If $\rho[j,j+1]$ is a large branch in $c$ then 
the neighbor $b$ of $c$ at $\rho(j)$ is incoming and
the neighbor of $c$ at $\rho(j+1)$ is outgoing. 
Up to isotopy,
the $c$-split of $\tau$  at $\rho[j,j+1]$ moves the 
neighbor $b$ across $\rho(j+1)$ preserving the orders
of the right neighbors and the left
neighbors of $c$, respectively, and preserving the 
types of the neighbors. 
This means that
in the split track $\eta$, for every $i\leq p$
the half-branch $\phi(\tau,\eta)(b_i)$ is 
an incoming left neighbor of $c$, and 
the neighbors 
$\phi(\tau,\eta)(b_i)$ are aligned along $c$ 
in consecutive order as $1\leq i\leq p$.
Only their relative positions to
the switches $s_j$ have changed.
The image $\theta_c(\tau)$ of $\tau$ under the positive
Dehn twist $\theta_c$ about $c$  
is the train track obtained from
$\tau$ by moving each of the left neighbors 
of $c$ one full turn about $c$ while leaving the right
neighbors fixed. 
Moreover, $\theta_c(\tau)$ is carried by 
$\tau$. There is a carrying map 
$\theta_c(\tau)\to \tau$ 
which is the identity away from a small annular
neighborhood of $c$ in $S$ and which 
maps every left neighbor of 
$c$ in $\theta_c(\tau)$ onto
the union of a left neighbor of $c$ in $\tau$ with 
the simple closed curve $c$,
traveled through precisely once.

Now let $\eta$ be a train track which carries
$\theta_c(\tau)$ and is obtained from
$\tau$ by a sequence of $c$-splits of maximal length
with this property.
This means that none of the neighbors $b_i$ of $\tau$ 
is moved more that one full turn about $c$, measured
with respect to the position of this neighbor
relative to the switches $s_j$ $(1\leq j\leq k)$ (which
we consider fixed). We claim that $\eta=\theta_c(\tau)$.
Namely, let $e$ be a large branch of  $\eta$ 
contained in $c$ and let
$b$ be an incoming neighbor of $c$ 
in $\tau$ such that $\phi(\tau,\eta)(b)$ is 
the incoming neighbor of $c$ 
at an endpoint of $e$.
Let moreover $s_j$ be the outgoing switch
on which $e$ is incident. If the switch of $\tau$ on 
which $b$ is incident is distinct from $\rho(i_j-1)$ 
then $b$ moved less than a full turn 
about $c$ in the splitting process.
Therefore the train track
obtained from $\eta$ by a $c$-split at $e$ carries
$\theta_c(\tau)$ which violates the assumption
hat no $c$-split of $\eta$ carries $\theta_c(\tau)$.

As a consequence, 
with respect to the parametrization $\rho^\prime$ of $c$ in $\eta$
with $\rho^\prime(i_1)=s_1$ which defines the same orientation
of $c$ as $\rho$, 
the branch $\rho^\prime[j,j+1]$ is large if and only if
this holds true for the branch $\rho[j,j+1]$. Moreover,
the neighbor of $\rho^\prime$ at $\rho^\prime(j),
\rho^\prime(j+1)$ is just $\phi(\tau,\eta)(b)$ where $b$ is the 
neighbor of $c$ at $\rho(j),\rho(j+1)$.
Since the orders of the incoming and outgoing neighbors
along $c$ are preserved in the splitting process, 
we necessarily have
$\eta=\theta_c(\tau)$ as claimed.

To show the second part of the lemma, let 
$\eta$ be a train track obtained from $\tau$ by
a sequence of $c$-splits and let $b$ be an incoming 
neighbor of $c$ in $\tau$ 
 which crossed the maximal number $\ell$  of outgoing
switches in a splitting sequence connecting
$\tau$ to $\eta$. If as before $k$ is the number
of outgoing switches along $c$ and if 
$\ell\in [ki+1,k(i+1)]$ then 
$b$ made at least $i$ full turns about $c$ and at most $i+1$ full turns.
Since the orders of the incoming neighbors of $c$ is preserved
in the process, 
each of the remaining left neighbors made at least
$i-1$ full turns about $c$ and at most $i+1$ full turns.
By the first part of this proof,
this means that $\eta$ can be transformed with a 
uniformly bounded number of splits to the train track
$\theta_c^{i+1}\tau$. Then $d(\eta,\theta_c^{i+1}\tau)$ is 
uniformly bounded. The second part of the lemma follows.
\end{proof}

For an embedded simple closed
curve $c$
in a complete train track 
$\tau$ which is not reduced we have the following

\begin{lemma}\label{nonreduced}
There is a number $a_2>0$ with the following property.
Let $c$ be an embedded simple closed curve in a 
complete train track
$\tau$. Let $\lambda$ be a complete geodesic
lamination which is carried by $\tau$ and
which contains $c$ as a minimal component. Then
there is a train track $\tau^\prime$ with the
following properties.
\begin{enumerate}
\item $\tau^\prime$ can be obtained from $\tau$ by
a sequence of $c$-splits of length at most $a_2$.
\item $\tau^\prime$ carries $\lambda$.
\item The curve $c$ in $\tau^\prime$ is reduced.
\end{enumerate}
\end{lemma}
\begin{proof}
Let $q$ be the number of branches of a complete train track
$\tau$. 
We show by induction on the number $n$ of 
switches of $\tau$ contained in the simple closed 
curve $c$ that the lemma holds true for a modification of 
$\tau$ with at most $n(q^2+1)$ $c$-splits.

If $n=2$ then
$c$ is reduced since $\tau$ is complete, so assume that
the lemma is known whenever the number of switches 
of $\tau$ contained in $c$ does not exceed 
$n-1$ for some $n\geq 3$.

Let $c$ be a simple closed curve which is a subtrack of 
a complete train track $\tau$ and such that
the number of switches of $\tau$ contained
in $c$ equals $n$. 
Assume that $c$ is 
\emph{not} reduced. Then with respect to
an orientation of $c$, there are 
two neighbors $b_1,b_2$ of $c$ in $\tau$
which lie on the same side of $c$ in an annular
neighborhood of $c$ in $S$ and such that with respect to
some orientation of $c$, $b_1$ is incoming 
and $b_2$ is outgoing. Let $\rho:[0,\ell]\to \tau$ be a trainpath which
parametrizes the oriented subarc of $c$ connecting
the switch $\rho(0)$ on which $b_1$ is incident to 
the switch $\rho(\ell)$ on which $b_2$ is incident.
Then $\rho[0,\ell]$ is embedded and 
begins and ends with a large half-branch.

Let $\lambda$ be a complete geodesic
lamination carried by $\tau$ which contains
$c$ as a minimal component.
By Lemma \ref{tightcontrol} and the following
remark, there is a sequence of $\rho$-splits
of length at most $q^2$ which modifies
$\tau$ to a complete train track 
$\tau_1$ with the following properties. $\tau_1$ 
contains $c$ as an embedded
simple closed curve, it carries $\lambda$, and
there is a large branch $e$ in $\tau_1$ contained in $c$
which is of type 2 (here the type of a proper
subbranch of $c$ is defined as in Section 3). The $\lambda$-split
of $\tau_1$ at $e$ is a complete train track 
$\tau_2$ which contains
$c$ as an embedded simple closed curve and such that the
number of switches of $\tau_2$ contained in $c$
does not exceed $n-1$. It is obtained from $\tau$ by
a sequence of $c$-splits of length at most $q^2+1$.
Thus we can now use the induction hypothesis 
to complete the proof of the lemma.
\end{proof}

Recall from Section 2 the definition of the cubical
Euclidean cone $E(\tau,\lambda)$ for a complete
train track $\tau$ and a complete geodesic lamination
$\lambda$ carried by $\tau$.
From Lemma \ref{dehnreduced} and Lemma \ref{nonreduced}
we obtain

\begin{corollary}\label{Dehn} 
There is a number $a_3>0$ with the
following property. Let $\tau\in {\cal V}({\cal T\cal T})$, let
$c<\tau$ be an embedded simple closed curve and let $\theta_c$ be
the positive Dehn twist about $c$. 
Let
$\lambda\in {\cal C\cal L}$ be a complete geodesic lamination
which is carried by $\tau$ and contains $c$ as a minimal
component and let $\{\tau_i\}_{0\leq i\leq k}\subset E(\tau,\lambda)$ 
be a splitting sequence issuing from $\tau_0=\tau$ which
consists of $c$-splits at large proper subbranches of $c$. 
Then there is some $i\in \mathbb{Z}$ such that
\[d(\tau_k,\theta_c^{i}\tau) \leq  a_3.\]
\end{corollary}

To use
Corollary \ref{Dehn} for an extension of 
Lemma \ref{reverse} we have to overcome 
a technical difficulty
arising from the fact that 
for a given complete train track $\tau$
and a large branch $e$ of $\tau$, 
there may be a choice of a right or left
split of $\tau$ at $e$ so that the 
split track is 
not recurrent. We call such a large branch 
\emph{rigid}. Note that $e$ is a rigid large branch of $\tau$
if and only if the (unique) complete train track
obtained from $\tau$ by a split at $e$ carries \emph{every}
complete geodesic lamination which is carried by $\tau$. It
follows from Lemma 2.1.3 of \cite{PH92} that a branch $e$ in
$\tau$ is rigid if and only if the train track $\zeta$ obtained
from $\tau$ by a collision at $e$, 
i.e. a split at $e$ followed by the removal
of the diagonal of the split, is \emph{not} recurrent.

\begin{lemma}\label{rigid} There is a number $a_4>0$
with the following
property. For every $\eta\in {\cal V}({\cal T\cal T})$ there is a
splitting sequence $\{\eta(i)\}_{0\leq i\leq s}\subset 
{\cal V}({\cal T\cal T})$ 
issuing from
$\eta=\eta(0)$ of length $s\leq a_4$ such that for every $i$,
$\eta(i+1)$ is obtained from $\eta(i)$ by a split at a rigid large
branch and such that $\eta(s)$ does not contain any rigid large
branches.
\end{lemma}
\begin{proof} We show the existence of a number
$a_4>0$ as in the
lemma with an argument by contradiction. Assume to the contrary
that the lemma 
does not hold true. Then there is a sequence of pairs
$(\beta_i,\xi_i)\in {\cal V}( {\cal T\cal T})
\times {\cal V}({\cal T\cal T})$ such that $\xi_i$
can be obtained from $\beta_i$ by a splitting sequence of length
at least $i$ consisting of splits at rigid large branches. Every
complete geodesic lamination which is carried by $\beta_i$ is also
carried by $\xi_i$. 

By invariance under the action of the mapping
class group and the fact that there are only
finitely many orbits for the action of ${\cal M\cal C\cal G}(S)$
on ${\cal V}( {\cal T\cal T})$,
by passing to a subsequence we may assume that there is some
$\eta\in {\cal V}({\cal T\cal T})$ such that $\beta_i=\eta$ for
all $i$. Since $\eta$ has only finitely many large branches, with a
standard diagonal procedure we can construct 
from $(\eta,\xi_i)$ an \emph{infinite}
splitting sequence $\{\eta(i)\}_{i\geq 0}\subset
{\cal V}({\cal T\cal T})$ issuing from
$\eta=\eta(0)$ such that for every $i$ the train track $\eta(i+1)$
is obtained from $\eta(i)$ by a split at a rigid large branch.
Then for every $i$, every complete geodesic lamination which is
carried by $\eta$ is also carried by $\eta(i)$. Now for every
\emph{projective measured geodesic lamination} whose support $\nu$
is carried by $\eta$ there is a complete geodesic lamination
$\lambda$ carried by $\eta$ which contains $\nu$ as a
sublamination (see the discussion in the proof of Lemma 2.3 of 
\cite{H09} and the discussion in Section 3 of this
paper). Thus for each $i$, the space ${\cal
P\cal M}(i)$ of projective measured
geodesic laminations carried by
$\eta(i)$ coincides with the space ${\cal P\cal M}(0)$ of
projective measured geodesic laminations carried by $\eta$. Since 
$\eta$ is complete.
the set ${\cal P\cal M}(0)$ contains an open subset of
the space of all projective measured geodesic laminations
(Lemma 2.1.1 of \cite{PH92}). 
Therefore $\cap_i{\cal P\cal
M}(i)={\cal P\cal M}(0)$ contains an open subset of projective
measured geodesic
lamination space which contradicts Theorem 8.5.1 of
\cite{M03}. This shows the lemma.
\end{proof}

A multi-twist $\theta$ about a simple geodesic
multi-curve $c=\cup_ic_i$ is a mapping class which 
can be represented in the form
$\theta=\theta_{c_1}^{m_1}\circ\dots\circ \theta_{c_k}^{m_k}$ for some
$m_i\in \mathbb{Z}$ and where as before, $\theta_{c_i}$ is the
positive Dehn twist about $c_i$. The next lemma is a version
of Lemma \ref{reverse} for train tracks which
carry a common multi-curve. As in Section 3,
for a train track $\tau\in
{\cal V}({\cal T\cal T})$ which is splittable to
a train track $\tau^\prime\in {\cal V}({\cal T\cal T})$
let $E(\tau,\tau^\prime)$ be the maximal subgraph
of ${\cal T\cal T}$ whose set of vertices consists of
all train tracks
which can be obtained from $\tau$ by a splitting sequence and 
which are
splittable to $\tau^\prime$.

\begin{lemma}\label{multitwist} 
For every $R>0$ there is a number
$\beta_1(R)>0$ with the following property. Let $\tau,\eta\in
{\cal V}({\cal T\cal T})$ with $d(\tau,\eta)\leq R$.
Assume that
$\tau,\eta$ contain a common subtrack which is a
multi-curve $c$ and that the components of $c$ are
precisely the minimal geodesic laminations which are carried by
both $\tau$ and $\eta$.
Let
$\tau^\prime,\eta^\prime$ be complete train tracks which can be
obtained from $\tau,\eta$ by a splitting sequence.
Then there is a multi-twist $\theta$
about $c$ with the following property.
There is a train track $\tau^{\prime\prime}$ in the
$\beta_1(R)$-neighborhood of $\tau^\prime$,
a train track $\eta^{\prime\prime}$ in the
$\beta_1(R)$-neighborhood of $\eta^\prime$
and a splitting sequence connecting 
$\tau^{\prime\prime}$ to $\eta^{\prime\prime}$ which passes
through the $\beta_1(R)$-neighborhood of $\theta (\eta)$.
For any complete geodesic lamination 
$\nu$ which is carried by $\eta^\prime$, we can assume
that $\nu$ is carried by $\eta^{\prime\prime}$.
Moreover,
$d(\theta(\tau),E(\tau,\tau^\prime))\leq a_5,
d(\theta(\eta),E(\eta,\eta^\prime))\leq a_5$
for a universal constant $a_5>0$.
\end{lemma}
\begin{proof} As in the proof of Lemma \ref{reverse}, 
by invariance under the
mapping class group it suffices to show the lemma for some fixed
train tracks $\tau,\eta\in {\cal V}({\cal T\cal T})$
which satisfy
the assumptions in the lemma with a number $\beta_1>0$ depending
on $\tau,\eta$.

Thus let $\tau,\eta$ be train tracks containing a
common subtrack $c$ which is a
multi-curve. Assume that 
every minimal geodesic lamination which is carried by both $\tau$ and
$\eta$ is a component of $c$.
By Lemma \ref{reverse}, we may assume that $c$ is not empty.

Assume that $\tau$ is splittable to a train
track $\tau^\prime$ and that $\eta$ is splittable to
a train track $\eta^\prime$.
The strategy is now as follows. By Corollary \ref{Dehn}, splitting
at large proper subbranches of $c$ results in 
twisting about the components of $c$. We first 
determine the direction of the twist and the amount
of twisting about each of the components of $c$ for
train tracks $\tilde\tau\in E(\tau,\tau^\prime)$ and
$\tilde\eta\in E(\eta,\eta^\prime)$ which are obtained from
$\tau,\eta$ by a sequence of $c$-splits of maximal length.
Then we compare these
data and determine a maximal multi-twist $\theta$ 
about the components of $c$
with the property that there are 
train tracks
$\hat\tau\in E(\tau,\tau^\prime),
\hat\eta\in E(\eta,\eta^\prime)$ which are contained 
in a uniformly bounded neighborhood of 
$\theta(\tau),\theta(\eta)$. 
Here by a maximal multi-twist we mean an element of 
maximal distance to the identity
in the free abelian group generated 
by the Dehn twists 
about the components of $c$ with respect to 
the distance function induced by the standard word metric.
Finally we show
that up to replacing
$\hat \tau,\hat \eta$ and $\tau^\prime,\eta^\prime$ by
their images under suitably chosen splitting sequences of 
uniformly bounded length, we may assume that 
$\hat\tau,\tau^\prime$ 
and $\hat\eta,\eta^\prime$ 
satisfy the assumptions in Lemma \ref{reverse}.

To carry out this strategy, note first that
by Lemma \ref{rigid}, via replacing $\tau^\prime,\eta^\prime$ by
their images under a splitting sequence of uniformly bounded length
we may assume that $\tau^\prime,\eta^\prime$ do not contain any rigid
large branch.

Let $c_1,\dots,c_k$ be the components of $c$.
To count the amount of twisting about a component of $c$
along a splitting sequence
connecting $\tau$ to $\tau^\prime$ we define
inductively a sequence $\{\tilde\tau(i)\}_{0\leq i\leq
k}\subset E(\tau,\tau^\prime)$ consisting of train
tracks $\tilde\tau(i)$ which contain $c$ as a subtrack as
follows. Put $\tilde\tau(0)=\tau$ 
and assume that $\tilde\tau(i-1)$ has been
defined for some $i\in \{1,\dots,k\}$. Let $\tilde\tau(i)\in
E(\tau,\tau^\prime)$ be the train track which
can be obtained from $\tilde\tau(i-1)$ by a
sequence of $c_i$-splits of maximal length
at large subbranches of $c_i$.
Since the components $c_i,c_j$ of $c$ are disjoint, 
splits at large subbranches of $c_i,c_j$
for $i\not=j$ commute. Therefore 
the train track $\tilde\tau(k)$ only depends on $\tau,\tau^\prime,c$
but not on the ordering of the components $c_i$ of
$c$. 

By Corollary \ref{Dehn}, 
there are numbers $b_i\in \mathbb{Z}$ such that
for $\theta_{\tau,\tau^\prime}=\theta_{c_1}^{b_1}\circ \dots\circ
\theta_{c_k}^{b_k}\in {\cal M\cal C\cal G}(S)$ we have
\begin{equation}\label{lemma5.6.1}
d(\tilde\tau(k),\theta_{\tau,\tau^\prime}(\tau))\leq ka_3.
\end{equation} 
Similarly, there are numbers
$p_i\in \mathbb{Z}$ such that for $\theta_{\eta,\eta^\prime}=
\theta_{c_1}^{p_1}\circ\dots\circ \theta_{c_k}^{p_k}\in
{\cal M\cal C\cal G}(S)$ and the train track
$\tilde\eta(k)$ obtained from $\eta$ and $\eta^\prime$ by the above
procedure we have
\begin{equation}\label{lemma5.6.2}
d(\tilde\eta(k),\theta_{\eta,\eta^\prime}(\eta))\leq ka_3.\end{equation}
By Lemma \ref{nonreduced}, we may choose $b_i=0$ (or $p_i=0$)
if there is no train track which can be obtained from
$\tau$ (or $\eta$) by a sequence of $c_i$-splits and which 
contains $c_i$ as a reduced simple closed curve.

For $i\leq k$ define $m(i)=0$ if either
the signs of $b_i,p_i$ are
distinct or if $b_i=0$ or $p_i=0$.
If the signs of $b_i,p_i$
coincide then define
$m(i)={\rm sgn}(b_i)\min\{\vert b_i\vert, \vert p_i\vert \}$. Write
$\theta=\theta_{c_1}^{m(1)}\circ\dots\circ \theta_{c_k}^{m(k)}$. 
Our goal is to show that 
the lemma holds true for this
multi-twist $\theta$. 
Note that by Corollary \ref{Dehn}, we have
\[d(\theta(\tau),E(\tau,\tau^\prime))\leq ka_3,\,
d(\theta(\eta),E(\eta,\eta^\prime))\leq ka_3.\] 

The remainder
of the proof is divided into two steps.
For convenience of notation, if $b_i=0$ or if 
$p_i=0$ then we write 
${\rm sgn}(b_i)={\rm sgn}(p_i)$.

{\sl Step 1:} Reduction to the case that 
${\rm sgn}(b_i)\not={\rm sgn}(p_i) $ for all $i$.

After reordering we may assume that
there is some $s\leq k$ such that
${\rm sgn}(b_i)\not={\rm sgn}(p_i)$ for $i\leq s$ and
that ${\rm sgn}(b_i)={\rm sgn}(p_i)$ 
for $i\geq s+1$. Choose 
train tracks $\tau_1\in E(\tau,\tau^\prime),
\eta_1\in E(\eta,\eta^\prime)$ contained in the
$ka_3$-neighborhood of
$\theta(\tau),\theta(\eta)$, which 
contain the multi-curve
$c$ as a subtrack and such that for each $i\geq s+1$ the following holds
true. If $\vert b_i\vert \leq \vert p_i\vert$ then
the cubical euclidean cone $E(\tau,\tau^\prime)$ does not contain
any train track which can be obtained from $\tau_1$
by a $c_i$-split at any large subbranch of $c_i$ in 
$\tau_1$,
and similarly for $\eta_1$ in the case that
$\vert p_i\vert \leq \vert b_i\vert$.

Every minimal geodesic lamination which is carried
by both $\tau_1,\eta_1$ is a component of $c$.
Moreover, 
\begin{equation}\label{lemma5.6.3}
d(\tau_1,\theta(\tau))\leq ka_3,\, d(\eta_1,\theta(\eta))\leq ka_3
\end{equation} and therefore
by invariance of the distance function
on ${\cal T\cal T}$ under the action of the mapping
class group  we have
\begin{equation}\label{dtauone}
d(\tau_1,\eta_1)\leq d(\tau,\eta)+2ka_3.\end{equation}

Let $i\geq s+1$ and 
assume without loss of generality
that $\vert b_i\vert \leq \vert p_i\vert$. 
Then for each
large proper subbranch $e$ of $c_i$ in $\tau_1$, 
the train track obtained from $\tau_1$ by
a $c_i$-split at $e$ is \emph{not} contained in the cubical
euclidean cone $E(\tau,\tau^\prime)$.

Assume first that $\tau_1$ contains a large proper
subbranch $e$ of $c_i$ of type 1 as defined 
in Section 3. Note that this is for example the case
if $c_i$ is reduced in $\tau_1$.
There is a unique choice of a split of $\tau_1$ 
at $e$ such that the split track $\hat\tau_1$ does 
\emph{not} carry
$c_i$. We claim that 
the train track $\hat\tau_1$ is complete. Namely,
either we have
$\hat\tau_1\in E(\tau,\tau^\prime)$, in particular
$\hat\tau_1$ is complete,
or no train track which can be obtained from
$\tau_1$ by a split at $e$ is splittable
to $\tau^\prime$. In the second case, 
$\phi(\tau_1,\tau^\prime)(e) $ is a large
branch of $\tau^\prime$ by uniqueness of splitting
sequences (Lemma 5.1 of \cite{H09}).
Since $\tau^\prime$ does not
contain any rigid large branch by assumption,
the train tracks which are obtained from $\tau^\prime$
by a single right or left 
split at $\phi(\tau_1,\tau^\prime)(e)$ are both complete. As a consequence of 
uniqueness of splitting sequences, 
the train track $\hat\tau_1$ is
splittable to a complete train track which is obtained from
$\tau^\prime$ by a single
right or left  split at $\phi(\tau_1,\tau^\prime)(e)$. 
Thus the
train track $\hat\tau_1$ is complete and
does not carry the simple
closed curve $c_i$. Moreover, $\hat \tau_1$ is splittable
to a train track obtained from $\tau^\prime$ by at most one
split.

If every large proper subbranch of $c_i$ in $\tau_1$ is 
of type 2 then no train track obtained from $\tau_1$
by a split at a large branch in $c_i$ is splittable
to $\tau^\prime$. Then $\phi(\tau_1,\tau^\prime)(c_i)$ is just
the embedded simple closed curve $c_i$ in $\tau^\prime$, 
and $\phi(\tau_1,\tau^\prime)(c_i)$ does not contain
any proper large subbranch of type 1.  
By Lemma \ref{nonreduced} and Lemma \ref{rigid},
the image $\xi$ of $\tau^\prime$ under
a (suitably chosen) 
splitting sequence of uniformly bounded length
is complete, it 
contains $c_i$ as an embedded reduced curve, and it does not
contain any rigid large branch. 
Then we can use the consideration in the previous paragraph
with the pair $\tau^\prime\prec \tau$
replaced by $\xi\prec \tau$.

To summarize, up to replacing
$\tau^\prime$ by its image under a splitting sequence of 
uniformly bounded length we may assume that 
there is a train track $\hat\tau_1\in E(\tau,\tau^\prime)$
in a uniformly bounded neighborhood of $\theta(\tau)$ which
does not carry the simple closed curve $c_i$. 
In particular, a minimal geodesic lamination which is carried
by both $\tau_1$ and $\eta$ is a component $c_j$ of $c$ for some
$j\not=i$.

Since splits at large proper subbranches of the distinct
components of $c$ commute, we construct in this way
successively in $k-s$ steps
from the train tracks $\tau_1,\eta_1$
complete train tracks $\tau_2,\eta_2$ 
with the following properties.
$\tau_2,\eta_2$ can be 
obtained from $\tau_1,\eta_1$ by a splitting sequence
of uniformly bounded length.
In particular, by the estimate (\ref{dtauone}) 
and (\ref{lemma5.6.3})
there is a universal constant
$a>0$ such that 
\begin{align}\label{lemma5.6.4}
d(\tau_2,\eta_2)\leq d(\tau,\eta)+a \text{ and }\\
d(\tau_2,\theta(\tau))\leq a,d(\eta_2,\theta(\eta))\leq a.\notag
\end{align}
The train tracks $\tau_2,\eta_2$ 
 contain the simple closed curves $c_1,\dots,c_s$ as
embedded subtracks, and they
are splittable to
train tracks $\tau_2^{\prime},\eta_2^{\prime}$
which can be obtained from $\tau^\prime,\eta^\prime$
by splitting sequences of uniformly bounded length.
A minimal geodesic
lamination which is carried by both $\tau_2,\eta_2$ coincides with
one of the curves $c_i$ for $i\leq s$. 
This shows that via replacing $\tau,\eta$ 
by $\tau_2,\eta_2$ and replacing $\tau^\prime,\eta^\prime$
by $\tau_2^\prime,\eta_2^\prime$ we may assume that
${\rm sgn}(b_i)\not={\rm sgn}(p_i)$ for all $i$.

{\sl Step 2:} The case ${\rm sgn}(b_i)\not={\rm sgn}(p_i)$ 
for all $i$.

By the choice of $b_i,p_i$
and by Lemma \ref{nonreduced}, there are train tracks
$\tau_0\in E(\tau,\tau^\prime),\eta_0\in E(\eta,\eta^\prime)$ which are
obtained from $\tau,\eta$ by a splitting sequence  
of length at most $ka_2$ 
and such that $\tau_0,\eta_0$ contain 
each of the components $c_i$ of $c$ $(i\leq k)$ 
as a reduced simple closed curve.
By Lemma \ref{dehnreduced}, for each $i$ the train track
$\tau_0$ is splittable to 
$\theta_{c_i}^{b_i}\tau_0$. Since  
${\rm sgn}(b_i)\not={\rm sgn}(p_i)$, for each $i$ the train track
$\theta_{c_i}^{b_i}\eta_0$ is splittable to $\eta_0$.

Let $\tau_1$ be the train track which can be obtained from
$\tau_0$ by a sequence of 
$\cup_{i=1}^kc_i$-splits of maximal length and which
is splittable to $\tau^\prime$. By the definition of the
multiplicities $b_i$, inquality (\ref{lemma5.6.1}) above
shows that for 
$\psi=
\theta_{c_1}^{b_1}\circ \dots\circ \theta_{c_k}^{b_k}$
we have
\[d(\psi(\tau_0),\tau_1)\leq ka_3.\]
The train track $\psi (\eta_0)$
is splittable to $\eta_0$. By invariance under the 
action of the mapping
class group, a minimal geodesic lamination carried
by both $\tau_1$ and $\psi(\eta_0)$ is a component of $c$. 

As in Step 1 above, there is a train track $\tau_2$ 
which is the image of 
$\tau_1$
under a splitting sequence of uniformly bounded length, which is
splittable to the image $\tau_2^\prime$ 
of $\tau^\prime$ under a splitting
sequence of uniformly bounded length and 
such that $\tau_2$ does not carry any
of the curves $c_i$. 
Since $\psi$ acts on ${\cal T\cal T}$ as an isometry, 
we have
\[d(\tau_2,\psi(\eta_0))\leq d(\tau,\eta)+a+2ka_2\]
where $a>0$ is a universal constant as in
the estimate (\ref{lemma5.6.4}) in Step 1 of this proof.
The train tracks 
$\psi(\eta_0)$ and $\tau_2$ 
do not carry any common geodesic
lamination. Moreover, $\psi(\eta_0)$ is splittable to 
$\eta^\prime$ with
a splitting sequence which passes through $\eta_0$.

Let $\nu$ be a complete geodesic lamination which is carried by
$\eta^\prime$.
By Lemma \ref{reverse}, applied to the train 
tracks $\tau_2,\psi(\eta_0)$
which are splittable to the train tracks 
$\tau_2^\prime, \eta^\prime$, there is a number 
$\beta>0$ only depending 
on $d(\tau_2,\psi(\eta_0))$ and hence on 
$d(\tau,\eta)$ and 
there is a splitting sequence
$\{\alpha(i)\}_{0\leq i\leq n}$ which connects a train track
$\alpha(0)$ with \[d(\alpha(0),\tau^\prime)\leq \beta\]
to a train track $\alpha(n)$ 
with 
\[d(\alpha(n),\eta^\prime)\leq\beta\] and  
which carries $\nu$. Moreover, this splitting
sequence passes through the $\beta$-neighborhood of 
$\psi(\eta_0)$ and hence it passes through a uniformly
bounded neighborhood of $\psi(\eta)$. 

Let $j\in \{0,\dots,n\}$ be such that
\[d(\alpha(j),\psi(\eta_0))\leq \beta.\]
By Lemma 6.7 of \cite{H09}, applied to the
train tracks $\alpha(j),\psi(\eta_0),\eta_0$ which 
carry the common complete geodesic lamination $\nu$,
there is a number $p_3(\beta)>0$ and 
there is a train track $\zeta$ which carries $\nu$, 
which is carried by both $\alpha(j)$ and $\eta_0$ and 
such that
\[d(\zeta,\eta_0)\leq p_3(\beta).\]
Proposition A.6 of \cite{H09} then shows
that $\alpha(j)$ is splittable to a train track
$\xi$ which is contained in a uniformly bounded
neighborhood of $\zeta$ and hence $\eta$ 
and which carries $\nu$.

An application of Lema 6.7 of \cite{H09}
to the train tracks $\xi,\eta_0,\eta^\prime$ which
carry $\nu$ produces a train track $\zeta^\prime$
in a uniformly bounded neighborhood of $\eta^\prime$
which carries $\nu$ and is carried by $\xi$.
Proposition A.6 of \cite{H09} 
then shows that $\xi$ is splittable
to a train track $\xi^\prime$ which carries $\nu$
and is contained in a uniformly bounded neighborhood of
$\eta^\prime$. This shows that 
$\alpha(0)$ can be connected to 
$\xi^\prime$ by a splitting sequence which passes
through a uniformly bounded neighborhood of $\eta$ and
completes the proof of the
lemma.
\end{proof}

Finally we investigate splitting sequences issuing from
complete train tracks $\tau,\eta$ which 
contain a common proper recurrent
subtrack $\sigma$ without closed curve
components carrying every geodesic lamination which is
carried by both $\tau,\eta$. 
Our goal is to establish an analog of Lemma \ref{multitwist}
in this case. To apply the above 
strategy we recall from Section 3 
the definition
of a sequence issuing from $\tau$ which is induced by a
splitting sequence $\{\sigma_i\}$ of $\sigma$.

\begin{lemma}\label{induceincones}
For every complete train track $\tau\in {\cal V}({\cal
T\cal T})$ which is splittable to a 
complete train track $\tau^\prime$ and for every
subtrack $\sigma$ of $\tau$ 
there is a unique train track $\xi\in
E(\tau,\tau^\prime)$ with the following properties.
\begin{enumerate}
\item There is a
splitting sequence $\{\sigma(i)\}_{0\leq i\leq p}$
issuing from $\sigma(0)=\sigma$ such that
$\xi$ can be obtained from
$\tau$ by a sequence induced by 
$\{\sigma(i)\}$.
\item If $\tilde \tau\in E(\tau,\tau^\prime)$ can be
obtained from $\tau$ by a sequence induced
by a splitting sequence of 
$\sigma$ then $\tilde\tau$
is splittable to $\xi$.
\end{enumerate}
\end{lemma}
\begin{proof} Let $\tau\in
{\cal V}({\cal T\cal T})$ be a complete train track
which is splittable to a 
train track $\tau^\prime\in {\cal V}({\cal T\cal T})$ and
let $\sigma$ be
a subtrack of $\tau$. Recall that a \emph{$\sigma$-split}
of $\tau$ is a split of $\tau$ at a large proper
subbranch of $\sigma$ with the property that the
split track contains $\sigma$ as a subtrack.
We proceed 
similarly to the procedure in the proof of Lemma \ref{tightcontrol}.

The large branches $e_1,\dots,e_k$ of 
$\sigma$ define
pairwise disjoint embedded trainpaths in $\tau$. 
For each $i$ let $\tilde \tau_i$ be the train
track obtained from $\tau$ by a sequence of 
$\sigma$-splits of maximal length at proper
subbranches of $e_i$ and with the additional
property that $\tilde \tau_i\in E(\tau,\tau^\prime)$.
By Lemma \ref{tightcontrol}, $\tilde\tau_i$ is found
in at most $q^2$ steps where $q>0$ is the number
of branches of a complete train track on $S$.

If $\tilde \tau_i$ is \emph{not} tight at $e_i$
then put a mark on the branch $e_i$ in $\sigma$.
If $\tilde\tau_i$ is tight at $e_i$ then 
$e_i$ is a large branch in $\tilde \tau_i$. In this case
we put a mark on $e_i$ 
if there is no train track contained in 
$E(\tau,\tau^\prime)$ which can be obtained from 
$\tilde \tau_i$ by a single split at $e_i$.

After reordering we may assume that
there is some $s\leq k$ such that
the branches $e_1,\dots,e_s$ are unmarked and that
the branches $e_{s+1},\dots,e_k$ are marked.
If $s=0$ then there is no train track
which can be obtained from a sequence induced
by any splitting sequence of $\sigma$ and which is
splittable to $\tau^\prime$ and we define $\xi=\tau$.
Otherwise 
let $\tau_1$ be the train track obtained from
$\tau$ by a sequence of $\sigma$-splits of maximal
length at proper large subbranches of the 
branches $e_1,\dots,e_s$ and such that $\tau_1\in E(\tau,\tau^\prime)$. 
Then $\tau_1$ is tight
at each of the branches $e_1,\dots,e_s$ of $\sigma$.
Moreover, there is a train track 
$\tau_2\in E(\tau,\tau^\prime)$ which can
be obtained from $\tau_1$ by a single split at
each of the large branches $e_i$ $(1\leq i\leq s)$.
The train track $\tau_2$ contains a subtrack
$\sigma_1$ which is obtained from $\sigma$ by a single
split at each of the branches $e_1,\dots,e_s$.

Repeat this procedure with the train track
$\tau_2$ and its subtrack $\sigma_2$. 
After finitely many steps we
obtain a train track $\xi\in E(\tau,\tau^\prime)$ which
clearly satisfies
the requirements in the lemma.
\end{proof}

For a convenient formulation of the following lemma, we say that
a train track $\eta$ is splittable to a complete geodesic
lamination $\lambda$ if $\eta$ carries $\lambda$.

\begin{lemma}\label{projection} 
Let $\tau\in {\cal V}({\cal T\cal T})$
and let ${\cal S}(\tau)\subset {\cal V}({\cal T\cal T})$ be the
set of all complete train tracks which can be obtained from
$\tau$ by a splitting sequence. Let
$E(\tau,\eta)$ be a cubical euclidean cone
where either $\eta\in {\cal S}(\tau)$
or $\eta$ is a complete geodesic
lamination carried by $\tau$. Then
there is a projection
\[\Pi_{E(\tau,\eta)}^1:{\cal S}(\tau)\to E(\tau,\eta)\] with 
the following property.
For every $\zeta\in {\cal S}(\tau)$, 
$\Pi_{E(\tau,\eta)}^1(\zeta)$ is a train track which is
splittable to both
$\zeta,\eta$, and  there is no train track $\chi\in
E(\Pi_{E(\tau,\eta)}^1(\zeta),\eta)- \Pi_{E(\tau,\eta)}^1
(\zeta)$ with this property. Moreover,
$\Pi^1_{E(\tau,\eta)}(\zeta)=\Pi^1_{E(\tau,\zeta)}(\eta)$ for all
$\eta,\zeta\in {\cal S}(\tau)$.
\end{lemma}
\begin{proof} Let $\tau\in {\cal V}({\cal T\cal T})$ 
and let $E(\tau,\eta)$ be a cubical euclidean cone
where either $\eta\in {\cal V}({\cal T\cal T})$ is such 
that $\tau$ is splittable to $\eta$ or where 
$\eta$ is a complete geodesic lamination carried by
$\tau$. For $\zeta\in {\cal S}(\tau)$ 
we construct by
induction on the length $n$ of a splitting sequence connecting
$\tau$ to $\zeta$ a projection point 
$\Pi_{E(\tau,\eta)}^1(\zeta)=\Pi_\eta^1(\zeta)\in 
E(\tau,\eta)$ with the required properties.

If $n=0$, i.e. if $\tau=\zeta$, then we define
$\Pi^1_{\eta}(\zeta)=\tau$.
By induction, assume that for some
$n\geq 1$ we determined for all $\tau\in {\cal V}({\cal T\cal T})$
which are splittable to some $\eta\in {\cal V}({\cal T\cal T})
\cup {\cal C\cal L}$ 
such a projection into $E(\tau,\eta)$ of 
the subset of ${\cal S}(\tau)$ of all complete train tracks
which can be obtained from $\tau$ by a splitting sequence
of length at most $n-1$.
Let $\{\alpha(i)\}_{0\leq i\leq n }\subset {\cal
V}({\cal T\cal T})$ be a splitting sequence of length $n$
connecting the train track $\tau=\alpha(0)$ to $\zeta=\alpha(n)$.
Let $\{e_1,\dots,e_\ell\}$ be the collection of all large
branches of
$\tau$ with the property that the splitting sequence
$\{\alpha(i)\}_{0\leq i\leq n}$ includes a split at $e_i$. Note
that $\ell\geq 1$ since $n\geq 1$. For each $i$, the choice of a
right or left split at $e_i$ is determined by the requirement that
the split track carries $\zeta$.

Assume first 
that there is a large branch $e\in
\{e_1,\dots,e_\ell\}$ such that the
train track $\tilde \alpha(1)$
obtained from $\tau$ by a split at $e$ and
which is splittable to $\zeta$ is also splittable
to $\eta$.
There is a splitting sequence
$\{\tilde \alpha(i)\}_{1\leq i\leq n}$
of length $n-1$ connecting $\tilde\alpha(1)$ to
$\tilde\alpha(n)=\alpha(n)=\zeta$
(Lemma 5.1 of \cite{H09}). The cubical euclidean 
cone $E(\tilde \alpha(1),\eta)$
is contained in the cubical euclidean cone 
$E(\tau,\eta)$, and we
have $\zeta\in {\cal S}(\tilde \alpha(1))$.
By induction hypothesis, there is
a unique train track
$\Pi^1_{E(\tilde \alpha(1),\eta)}(\zeta)\in
E(\tilde \alpha(1),\eta)\subset
E(\tau,\eta)$ with the property that
$\Pi_{E(\tilde \alpha(1),\eta)}^1(\zeta)$
is splittable to $\zeta$ but that this is not
the case for any train track contained in
$E(\Pi_{E(\tilde \alpha(1),\eta)}^1(\zeta),\eta)-
\Pi_{E(\tilde \alpha(1),\eta)}^1(\zeta)$.

Define
$\Pi^1_{\eta}(\zeta)=\Pi^1_{E(\tilde \alpha(1),\eta)}(\zeta)$.
Then $\Pi^1_{\eta}(\zeta)$ is splittable to $\zeta$ and
this is not the case for any train track in
$E(\Pi_\eta^1(\zeta),\eta)-\Pi_\eta^1(\zeta)$. On the other
hand, a splitting sequence connecting $\tau$ to $\zeta$
is unique up to order (see Lemma 5.1 of \cite{H09} for
a detailed discussion of this fact). Therefore
if $\xi\in E(\tau,\eta)$ is such that $\xi$ is
splittable to $\zeta$ and such that a splitting
sequence connecting $\tau$ to $\xi$ does not
include a split at $e$, then $\xi$ contains
$\phi(\tau,\xi)(e)$ as a large branch, 
and there is a train track
$\xi^\prime\in E(\tau,\eta)$ 
which can be obtained from $\xi$
by a split at $\phi(\tau,\xi)(e)$ and which is splittable to $\zeta$.
In particular, $\xi$ does not satisfy the requirement
in the lemma. As a consequence, the point
$\Pi_\eta^1(\zeta)$ does not depend on the above choice of the
large branch $e$.

If none of the train tracks $\xi\in E(\tau,\zeta)$
obtained from $\tau$ by a split at one of the branches
$e_1,\dots,e_\ell$ is splittable to $\eta$, then no train track
$\beta\in E(\tau,\eta)-\tau$ is splittable to $\zeta$ and
we define $\Pi^1_\eta(\zeta)=\tau$. This completes
the inductive construction of the map $\Pi_\eta^1:
{\cal S}(\tau)\to E(\tau,\eta)$.
Note that we have $\Pi_\eta^1(\zeta)=
\Pi_\zeta^1(\eta)$ for all $\zeta,\eta\in {\cal S}(\tau)$. Namely,
$\Pi_\eta^1(\zeta)$ is splittable to both $\zeta,\eta$, but
this is not the case for any train track which can
be obtained from $\Pi_\eta^1(\zeta)$ by a split.
This shows the lemma.
\end{proof}

{\bf Remark:} Lemma \ref{projection} is valid without
modification for any train track $\sigma$ which
is splittable to train tracks $\eta,\zeta$.
Namely, the proof of Lemma \ref{projection} nowhere uses
the assumption of completeness of the train tracks considered.
We used complete train tracks in the formulation
of the lemma to include projections to a cubical
euclidean cone $E(\tau,\lambda)$ where $\lambda$
is a complete geodesic lamination carried by $\tau$.

\bigskip

We use  Lemma \ref{projection} to establish a technical statement
which enables us to control distances in the train track complex.

\begin{lemma}\label{projection2}
For every $R\geq 0$ there is a number $\beta_2(R)>0$ with the
following property. Let $\tau,\eta\in {\cal V}({\cal T\cal T})$
with $d(\tau,\eta)\leq R$ and 
assume that $\tau,\eta$ have a common subtrack
$\sigma$ which carries every measured
geodesic lamination carried by
both $\tau,\eta$. Let $\tau,\eta$ be splittable to complete
train tracks $\tau^\prime,\eta^\prime$.
Then there are train  tracks  $\tau_1,\eta_1,\tau_1^\prime,
\eta_1^\prime\in {\cal V}({\cal T\cal T})$ 
with the following properties.
\begin{enumerate}
\item $\tau_1^\prime,\eta_1^\prime$ can be obtained
from $\tau^\prime,\eta^\prime$ by a splitting sequence of uniformly
bounded length (the bound does not depend on $R$).
\item $\tau_1\in E(\tau,\tau_1^\prime),\eta_1\in 
E(\eta,\eta_1^\prime)$.
\item $d(\tau_1,\eta_1)\leq \beta_2(R)$.
\item $\tau_1,\eta_1$ contain a (perhaps trivial) multi-curve $c$ as a 
common subtrack, and every minimal geodesic lamination which is
carried by both $\tau_1,\eta_1$ is a component of  $c$.
\item If $\tau=\eta$ and if $\tau^\prime,\eta^\prime\in 
E(\tau,\lambda)$ for a complete geodesic lamination
$\lambda$  carried
by $\tau$ then 
$d(\tau_1,\Pi^1_{E(\tau,\tau^\prime)}(\eta^\prime))\leq 
\beta_2(0)$.
\end{enumerate}
\end{lemma}
\begin{proof}
Let $R>0$ and let $\tau,\eta\in {\cal V}({\cal T\cal T})$ be such that
$d(\tau,\eta)\leq R$. Assume that $\tau,\eta$ are splittable to 
$\tau^\prime,\eta^\prime\in {\cal V}({\cal T\cal T})$.
By Lemma \ref{rigid}, via possibly replacing
$\tau^\prime,\eta^\prime$ by their images under a splitting
sequence of uniformly bounded length we
may assume that $\tau^\prime,\eta^\prime$ do not contain
any rigid large branch.

Assume that $\tau,\eta$ contain
a common subtrack $\sigma$ which
carries every measured
geodesic lamination carried by both $\tau,\eta$.
Since a union of recurrent subtracks of $\sigma$ is recurrent,
all measured geodesic laminations which are carried by both 
$\tau$ and $\eta$ 
are carried by the largest recurrent subtrack of 
$\sigma$ and hence we may assume that $\sigma$ is recurrent.
Let $\sigma_0$ be the complement in $\sigma$ of the closed curve
components of $\sigma$. 
We show the lemma by induction on the number $n$ of branches of
$\sigma_0$
with a constant $\beta_2(R,n)>0$ depending on $n$ and $R$.
Since $n$ does not exceed the number $q$ of branches of a complete
train track on $S$, this suffices for the purpose of the lemma. 

In the case $n=0$ we can choose $\tau_1=\tau,
\eta_1=\eta$ and there is nothing to
show. So assume that the lemma holds
true whenver the number of branches of $\sigma_0$ does not
exceed $n-1$ for some $n\geq 1$. 
Let $\tau,\eta\in {\cal V}({\cal T\cal T})$
be such that the number of branches of $\sigma_0$ 
is bounded from above by $n$.

By Lemma \ref{induceincones}, there is a 
splitting sequence $\{\sigma_i\}_{0\leq
i\leq p}$ of maximal length
issuing from $\sigma_0$ which induces a
sequence $\{\tilde \tau_i\}_{0\leq i\leq 2p}\subset 
E(\tau,\tau^\prime)$ of
maximal length issuing from $\tilde\tau_0=
\tau$. The train track
$\sigma_p$ is a subtrack of $\tilde\tau_{2p}$, and
$\sigma_p$  and
$\tilde\tau_{2p} $ only depend on $\tau, \tau^\prime, \sigma_0$
but not
on any choices made in the construction.
Similarly, there is a splitting
sequence $\{\tilde\sigma_i\}_{0\leq i\leq u}$ issuing from
$\tilde\sigma_0=\sigma_0$  which induces a
sequence $\{\tilde\eta_j\}_{0\leq j\leq 2u}\subset
E(\eta,\eta^\prime)$
of maximal length issuing from $\eta$.

The pairs
of train tracks $(\sigma_0,\sigma_p)$ and
 $(\sigma_0,\tilde \sigma_u)$ determine cubical euclidean cones $
E(\sigma_0,\sigma_p),E(\sigma_0,\tilde
\sigma_u)$ whose vertex sets consist all
train tracks which can be
obtained from $\sigma_0$ by a splitting sequence and which are
splittable to $\sigma_p,\tilde \sigma_u.$ 
By the remark following Lemma \ref{projection}, 
we can apply Lemma \ref{projection} to $E(\sigma_0,\sigma_p)$ and 
$E(\sigma_0,\tilde \sigma_u)$.
We obtain a train track 
$\xi_0=\Pi^1_{E(\sigma_0,\sigma_p)}\tilde \sigma_u=
\Pi^1_{E(\sigma_0,\tilde \sigma_u)}\sigma_p\in 
E(\sigma_0,\sigma_p)\cap
E(\sigma_0,\tilde \sigma_u)$ with the property that $\xi_0$  is
splittable to both $\sigma_p,\tilde \sigma_u $ but that this is not
the case for any train track which can be obtained from $\xi_0$
by a split. Define $\xi=\xi_0$ if $\xi_0$ is recurrent.
Otherwise let $\xi\in E(\sigma_0,\xi_0)$ be a recurrent train track
with the following property. For every large branch $e$ of $\xi$,
either a splitting sequence connecting $\xi$ to $\xi_0$
does not include a split at $e$ or the train track obtained
from $\xi$ by a split at $e$ and which is splittable
to $\xi_0$ is not recurrent.

By Lemma \ref{induceincones}, 
a splitting sequence 
connecting $\sigma_0$ to $\xi $ induces sequences
$\{\tilde\tau_i\}\subset E(\tau,\tau^\prime), 
\{\tilde\eta_j\}\subset E(\eta,\eta^\prime)$ issuing from
$\tau_0=\tau, \eta_0=\eta$ and connecting
$\tau,\eta$ to train tracks $\tau_1,\eta_1$ which contain
$\xi$ as a subtrack and which are splittable to 
$\tau^\prime,\eta^\prime$. Since $d(\tau,\eta)\leq R$ by assumption,
Corollary \ref{inducing10} shows the existence of a number
$\chi_1(R)>0$ only depending on $R$ such that 
\begin{equation}\label{dtaueta}
d(\tau_1,\eta_1)\leq \chi_1(R).\end{equation}
A measured geodesic lamination which is
carried by both $\tau_1,\eta_1$ is carried by
the union of $\xi$ with the
closed curve components of $\sigma$.
Note that by the definition of an induced splitting sequence,
a closed curve component of $\sigma$ 
is a subtrack of both $\tau_1$ and $\eta_1$.

Let ${\cal E}_{\xi}(\tau_1),
{\cal E}_{\xi}(\eta_1)$ be the
set of all large branches $e$ of $\xi$ with the property
that there is a train track
obtained from $\xi$ by a split at $e$ which is 
splittable to $\sigma_p,\tilde \sigma_u$. 
We distinguish two cases.

{\sl Case 1:} 
${\cal E}_\xi(\tau_1)\cap {\cal E}_{\xi}(\eta_1)\not=\emptyset$.

Let $e\in {\cal E}_{\xi}(\alpha_1)\cap
{\cal E}_{\xi}(\beta_1)$. Then 
there are train tracks $\tau_2\in E(\tau_1,\tau^\prime)$ and
$\eta_2\in E(\eta_1,\eta^\prime)$ obtained
from $\tau_1,\eta_1$ by a sequence of splits at large proper
subbranches of $e$ (of uniformly bounded length) which
are tight at $e$.
There is a train track $\tau_3,\eta_3$ obtained from 
$\tau_2,\eta_2$ by a single split at $e$ which is splittable to 
$\tau^\prime,\eta^\prime$. 
The estimate (\ref{dtaueta}) implies that
\[d(\tau_3,\eta_3)\leq 
\chi_2(R)\] for a number $\chi_2(R)>0$ only depending on $R$.

By the definition
of the train track $\xi$, there are now
possibilities. In the first case, 
both $\tau_3,\eta_3$ are obtained from $\tau_2,\eta_2$ by
a right (or left) split at $e$. 
The train track $\xi^\prime$ obtained from $\xi$ by
the right (or left) split at $e$ is a common subtrack of
$\tau_3,\eta_3$, but it is not recurrent.
Let $\zeta$ be the largest recurrent 
subtrack of $\xi^\prime$. The number of branches of 
$\zeta$ does not exceed $n-1$. Moreover, $\zeta$ carries
every measured geodesic lamination which is carried by
both $\tau_3,\eta_3$.

In the second case, the train track $\tau_3$ is obtained
from $\tau_2$ by a right (or left) split at $e$, and
the train track $\eta_3$ is obtained from $\eta_2$ by
a left (or right) split at $e$. 
Let $\zeta$ be the train track obtained from
$\xi$ by a collision at the large branch
$e\in {\cal E}_{\xi}(\tau_1)\cap
{\cal E}_{\xi}(\eta_1)$ of $\xi$. Then $\zeta$ carries
every measured geodesic lamination which is carried by
both $\tau_3,\eta_3$. 
The number of branches of $\zeta$ equals $n-1$.

We can now apply the induction hypothesis to 
$\tau_3,\eta_3$ which are splittable to $\tau^\prime,\eta^\prime$
and to the common subtrack of $\tau_3,\eta_3$ which is the union of
$\zeta$ with the simple closed curve components of $\sigma$.
The statement of the lemma follows.

{\sl Case 2:} ${\cal E}_{\xi}(\tau_1)\cap {\cal E}_{\xi}(\eta_1)=\emptyset.$

Let $e$ be an arbitrary large branch of $\xi$.
Up to exchanging $\tau$ and $\eta$ we may assume that
$e\not\in {\cal E}_\xi(\tau_1)$.

Let $\alpha,\beta$ be the train tracks obtained from 
$\tau_1,\eta_1$ by a sequence of $\xi$-splits of maximal length
at proper large subbranches of $e$ and which are splittable
to $\tau^\prime,\eta^\prime$. The length of a splitting
sequence connecting $\tau_1,\eta_1$ to $\alpha,\beta$ is uniformly
bounded and hence we have
\[d(\alpha,\beta)\leq \chi_3(R)\]
for a number $\chi_3(R)>0$ only depending on $R$. The train tracks
$\alpha,\beta$ contain $\xi$ as a common subtrack, and every
geodesic lamination which is 
carried by both $\alpha,\beta$ is carried by the
union of $\xi$ with the closed curve components of $\sigma$.

We distinguish three subcases.

{\sl Subcase 2.1:} Both $\alpha$ and $\beta$ are tight at $e$.

Since $e\not\in {\cal E}_\xi(\tau_1)$, 
no train track which can be obtained
from $\alpha$ by a single split at $e$ 
is splittable to $\tau^\prime$. 
By uniqueness of splitting sequences, 
$\phi(\alpha,\tau^\prime)(e)$ is large branch of 
$\tau^\prime$.
By assumption, the branch 
$\phi(\alpha,\tau^\prime)(e)$ of $\tau^\prime$ is not rigid.
Thus the train tracks obtained from $\tau^\prime$ by 
the right and the left split at $\phi(\alpha,\tau^\prime)(e)$, 
respectively,
are both recurrent and hence complete.
Then the train tracks
obtained from $\alpha$ by the right split and 
the left split at $e$, respectively, are complete as well since
these train tracks are splittable to a complete train track obtained
from $\tau^\prime$ by a
split at 
$\phi(\alpha,\tau^\prime)(e)$.

If there is a train track $\beta^\prime$ 
which is obtained
from $\beta$ by a split at $e$ and which
is splittable to $\eta^\prime$,
say if $\beta^\prime$ is obtained from $\beta$ 
by the right split, then let $\alpha^\prime$ be the
train track obtained from $\alpha$ by the 
left split at $e$ and let $\tau^{\prime\prime}$ be the
train track obtained from $\tau^\prime$ by the 
left split at 
$\phi(\alpha,\tau^\prime)(e)$. Then 
$\alpha^\prime\in E(\alpha,\tau^{\prime\prime})$, moreover  
a measured geodesic lamination
which is carried by both $\alpha^\prime$ and 
$\beta^\prime$ is carried
by the union of the simple closed curve components
of $\sigma$ with the
subtrack $\zeta$ of $\alpha^\prime$ which is obtained
from a collision of $\xi$ at $e$, i.e. a split
followed by the removal of the diagonal of the split. 
As before, the number of 
branches of $\zeta$ is strictly smaller than the number of 
branches of $\xi$. Moreover, we have
\[d(\alpha^\prime,\beta^\prime)\leq 
\chi_4(R)\]
for a number $\chi_4(R)>0$ only depending on $R$. 
As a consequence, we can now apply the induction hypothesis
to $\alpha^\prime,\beta^\prime,\tau^{\prime\prime},\eta^\prime$
and the common subtrack of $\alpha^\prime,\beta^\prime$
which is the union of $\zeta$ with the simple closed curve
components of $\sigma$ 
to deduced the statement of the lemma.

If no train track obtained from a split of $\beta$ at $e$ is 
splittable to $\eta^\prime$ then $\phi(\beta,\eta^\prime)(e)$ 
is a large branch in $\eta^\prime$. Since $\phi(\beta,\eta^\prime)(e)$
is not rigid, the train tracks obtained from $\beta$
by a right and a left split, respectively, are both 
complete (see the above discussion). 
Replace $\alpha,\beta$ by their images 
$\alpha^\prime,\beta^\prime$
under a right and left split 
at $e$, respectively,
and replace $\tau^\prime,\eta^\prime$ by their images 
$\tau^{\prime\prime},\eta^{\prime\prime}$ under
a right split and a left split at $\phi(\alpha,\tau^\prime)(e),
\phi(\beta,\eta^\prime)(e)$, respectively.
As before, a measured geodesic lamination which is
carried by both $\alpha^\prime,\beta^\prime$ 
is carried by the union of the simple closed
curve components of $\sigma$  with the 
train track $\zeta$ obtained from $\xi$ by
a collision at $e$. The statement of the
lemma now follows as above from the induction hypothesis,
applied to $\alpha^\prime,\beta^\prime,\tau^{\prime\prime},
\eta^{\prime\prime}$ and the union of $\zeta$ with the
closed curve components of $\sigma$.

{\sl Subcase 2.2:} 
Up to exchanging $\alpha$ and $\beta$, the train track 
$\alpha$ contains a large proper subbranch $b$ of $e$
of type 1 or type 3
as introduced in Section 3. 

Then 
precisely one choice of a (right or left) split
of $\alpha$ at $b$ contains $\xi$ as a subtrack.
If this is say the right split, then the 
train track $\alpha^\prime$ obtained from
$\alpha$ by the left split 
at $e$ is recurrent. Namely,
either $\alpha^\prime\in E(\tau,\tau^\prime)$
or $\alpha^\prime$ is splittable
to a train track obtained from
$\tau^\prime$ by a single split
at the large branch
$\phi(\alpha,\tau^\prime)(b)$ which is not rigid by assumption
(see the discussion under subcase 2.1). 
Moreover, $\alpha^\prime$ 
does not contain $\xi$ as a subtrack. 
Note that either $\alpha^\prime\in E(\tau,\tau^\prime)$
or $\alpha^\prime\in E(\tau,\tau^{\prime\prime})$ where
$\tau^{\prime\prime}$ can be obtained from $\tau^\prime$
by a single split at $\phi(\alpha,\tau^\prime)(b)$.
A measured geodesic lamination
which is carried by both $\alpha^\prime,\beta$ is carried
by the union of the simple
closed curve components of $\sigma$ with 
the largest subtrack $\zeta$ of 
$\xi$ which does not contain 
the branch $e$. In particular, the number of branches of 
$\zeta$ is strictly smaller than the number of branches of 
$\xi$. As before, the statement 
of the lemma now follows from the induction
hypothesis.

{\sl Subcase 2.3:} 
Up to exchanging $\alpha$ and $\beta$, the train track 
$\alpha$ contains a proper large subbranch $b$ of $e$, and
every proper large subbranch $b$ of $e$ 
in $\alpha$ is of type 2. 

Since 
in this case every split at a large proper subbranch $b$ of
$e$ is a $\xi$-split, 
no train track obtained from 
$\alpha$ by a split at a large proper
subbranch $b$ of $e$ is splittable to $\tau^\prime$.
For such a large proper subbranch $b$ of $e$, 
the branch $\phi(\alpha,\tau^\prime)(b)$ is a large branch
in $\tau^\prime$. 

Let $\tau^{\prime\prime}$ be a complete train track obtained
from $\tau^\prime$ by a split at 
$\phi(\alpha,\tau^\prime)(b)$ and let 
$\alpha^\prime$ be the train
track obtained from $\alpha$ by a split at $b$ and which
is splittable to $\tau^{\prime\prime}$. Then 
the number of half-branches of
$\alpha^\prime-\xi$ which are incident on a switch in 
$\xi$ is strictly smaller than
the number of half-branches of $\alpha-\xi$ which are
incident on a switch in $\xi$. 
Via replacing $\tau^{\prime\prime}$ by its image under a splitting
sequence of uniformly bounded length we may assume that
$\tau^{\prime\prime}$ does not contain any rigid large branch.

Carry out the above construction with the train tracks
$\alpha^\prime,\beta,\tau^{\prime\prime},\eta^\prime$.
We find a common subtrack $\xi^\prime$ 
of train tracks 
$\alpha^{\prime\prime}\in E(\alpha^\prime,\tau^{\prime\prime}),
\beta^{\prime\prime}\in E(\beta,\eta^\prime)$
such that $\xi^\prime$ 
can be obtained from $\xi$ by a splitting sequence and 
such that $\xi^\prime$ 
carries every geodesic lamination
which is carried by both $\alpha^{\prime\prime},
\beta^{\prime\prime}$. 
Moreover, the distance between $\alpha^{\prime\prime},
\beta^{\prime\prime}$
is bounded from above by a universal constant only
depending on $R$.
The number of half-branches of $\alpha^{\prime\prime}-\xi^\prime$ which
are incident on a switch in $\xi^\prime$ is strictly
smaller than the number of half-branches of $\alpha$
which are incident on a switch in $\xi$. 
Since the number of neighbors of $\xi$ in 
$\alpha$ is uniformly bounded, 
after a uniformly bounded number
of such steps, Case 1 or Subcase 2.1 or Subcase 2.2 must occur.
Use the reasoning for these cases as before to 
reduce the statement of the lemma to the induction hypothesis.

Together this completes the proof of the lemma.
\end{proof}

Let $F$ be a marking for $S$ and 
let $X\subset {\cal V}({\cal T\cal T})$ be the set of all
train tracks which can be obtained from a train track in
standard form for $F$ by a
splitting sequence. By Proposition \ref{density} and the following
discussion, there is a
number $p>0$ such that the
$p$-neighborhood of $X$ in ${\cal T\cal
T}$ is all of ${\cal T\cal T}$. Thus
if we equip $X$ with the
restriction of the metric on ${\cal T\cal T}$
then the inclusion $X\to {\cal
T\cal T}$ is a quasi-isometry.

If $\lambda$ is any complete geodesic
lamination then by Lemma \ref{standardform}, $\lambda$ is
carried by a unique train track $\tau$ in standard
form for $F$. As a consequence,
for $\eta\in X$
there is a \emph{unique} train track $\tau$ in
standard form for $F$ which is splittable to $\eta$.
Write $E(F,\eta)=E(\tau,\eta)$.

The next result is the key to
an understanding of the geometry
of the train track complex. 


\begin{proposition}\label{shortestdistance}
There is a number $\kappa>0$ with the
following property. Let $F$ be a marking for $S$ and
let $X\subset {\cal V}({\cal T\cal T})$ be the
set of all complete train tracks which
can be obtained from a train track
in standard form for $F$ by a splitting sequence.
Then for every
$\eta\in X$ there is a map
$\Pi_{E(F,\eta)}:X\to E(F,\eta)$ such that for every $\zeta\in X$
the following is satisfied.
\begin{enumerate}
\item There is a splitting sequence connecting a train
track $\tau^\prime$ in standard form for $F$
to $\zeta$ which passes through the $\kappa$-neighborhood
of $\Pi_{E(F,\eta)}(\zeta)$.
\item There is a splitting sequence
connecting a point in the $\kappa$-neighborhood of
$\zeta$ to a point in the $\kappa$-neighborhood of
$\eta$ which passes through
the $\kappa$-neighborhood of $\Pi_{E(F,\eta)}(\zeta)$.
\item $d(\Pi_{E(F,\eta)}(\zeta),\Pi_{E(F,\zeta)}(\eta))
\leq \kappa$ for all $\eta,\zeta\in X$.
\item If $\eta,\zeta\in E(\tau,\lambda)$ for a train track
$\tau$ in standard form for $F$ which carries the complete
geodesic lamination
$\lambda\in {\cal C\cal L}$ then
$\Pi_{E(F,\eta)}(\zeta)=\Pi_{E(\tau,\eta)}^1(\zeta)=
\Pi_{E(\tau,\zeta)}^1(\eta)$.
\end{enumerate}
\end{proposition}
\begin{proof} Let $F$ be any marking of $S$,
let $X$ be the set of all complete train tracks
which can be obtained from a train track in standard
form for $F$ by a splitting sequence and let
$\tau^\prime,\eta^\prime\in X$ be arbitrary. Then there are unique
train tracks $\tau,\eta$ in standard form for $F$ so that
$\tau$ is splittable to $\tau^\prime$ and 
$\eta$ is splittable to
$\eta^\prime$. Let ${\cal V}(\tau),{\cal V}(\eta)$
be the set of all measured
geodesic laminations carried by $\tau,\eta$.
Since every geodesic lamination $\lambda$ contains a minimal
component and hence supports a transverse measure
(whose support may be a proper sublamination of $\lambda$),
if ${\cal V}(\tau)\cap {\cal V}(\eta)=\{0\}$ then 
by Lemma \ref{reverse} we can define
$\Pi_{E(F,\tau^\prime)}(\eta^\prime)=\tau$ and 
$\Pi_{E(F,\eta^\prime)}(\tau^\prime)=
\eta$. 

On the other hand, if ${\cal V}(\tau)\cap
{\cal V}(\eta)\not= \{0\}$ 
and if the measured geodesic lamination $\lambda$ 
with minimal support is carried
by both $\tau$ and $\eta$ then the set 
$\xi,\zeta$ of all branches of $\tau,\eta$
whose $\lambda$-weight is positive is a recurrent subtrack
of $\tau,\eta$. By Lemma \ref{markingcontrol}, 
the train tracks $\xi,\zeta$ are isotopic.
Now the union of two subtracks of $\tau$ is again a subtrack
and therefore $\tau,\eta$ contain a
common recurrent subtrack $\sigma$ which carries
every measured geodesic 
lamination in ${\cal V}(\tau)\cap
{\cal V}(\eta)$.

Since the diameter in ${\cal T\cal T}$ of the set of all
train tracks in standard form for $F$ is uniformly bounded,
Lemma \ref{projection2}, applied to the train tracks 
$\tau,\eta$ which are splittable to
$\tau^\prime,\eta^\prime$ and their common subtrack $\sigma$, yields
the existence of a universal constant $\beta_2>0$ and of 
train tracks $\tau_1^\prime,\eta_1^\prime$
and train tracks $\tau_1\in E(\tau,\tau_1^\prime),
\eta_1\in E(\eta,\eta_1^\prime)$ 
with the following 
properties.

\begin{enumerate}
\item $d(\tau_1,\eta_1)\leq\beta_2$.
\item
The train track 
$\tau_1^{\prime},\eta_1^\prime$ can
be obtained from $\tau^\prime,\eta^\prime$ by a splitting sequence
of length at most $\beta_2$.
\item $\tau_1,\eta_1$ contain a common embedded
multi-curve $c$ such that every minimal geodesic lamination
carried by both $\tau_1,\eta_1$ is a component of $c$.
\item If $\tau^\prime,\eta^\prime\in E(\tau,\lambda)$ for some
complete geodesic lamination $\lambda$ then \\
$d(\tau_1,\Pi^1_{E(\tau,\tau^\prime)}\eta^\prime)\leq \beta_2$.
\end{enumerate}

Apply Lemma \ref{multitwist} to the train tracks
$\tau_1,\eta_1$ which are splittable to $\tau_1^\prime,\eta_1^\prime$.
We find a multi-twist $\theta$ about the multi-curve $c$, 
train tracks $\tau^{\prime\prime},
\eta^{\prime\prime}$ in a uniformly bounded neighborhood
of $\tau_1^\prime,\eta_1^\prime$ and hence of $\tau^\prime,\eta^\prime$
and a splitting
sequence which connects $\tau^{\prime\prime}$ to
$\eta^{\prime\prime}$ and which passes through a 
uniformly bounded neighborhood of $\theta(\tau_1),\theta(\eta_1)$.
Moreover, $d(\theta(\tau_1),E(\tau,\tau^\prime)),
d(\theta(\eta_1),E(\eta,\eta^\prime))$ is uniformly bounded.
Note that $d(\theta(\tau_1),\theta(\eta_1))\leq \beta_2$.

Define $\Pi_{E(\tau,\tau^\prime)}\eta^\prime$
to be a point in $E(\tau,\tau^\prime)$ of smallest distance
to $\theta(\tau_1)$, and define
$\Pi_{E(\eta,\eta^\prime)}\tau^\prime$ to be a point
in $E(\eta,\eta^\prime)$ of  
smallest distance to $\theta(\eta_1)$.
By construction, the train tracks
$\Pi_{E(\tau,\tau^\prime)}\eta^\prime,
\Pi_{E(\eta,\eta^\prime)}\tau^\prime$ satisfy 
properties 1),2),3)
stated in the proposition. Moreover, we may
assume that property 4) holds true as well.
%
%
\end{proof}

\section{A bounded bicombing of the train track complex}

A \emph{discrete bicombing} of a connected
metric space $(X,d)$
assigns to every ordered pair of points
$(x,y)\in X\times X$ a discrete path
$\rho_{x,y}:[0,k_\rho]\cap \mathbb{N}\to X$ 
connecting $x=\rho_{x,y}(0)$ to
$y=\rho_{x,y}(k_\rho)$. The path $\rho_{x,y}$ is called the
\emph{combing line} connecting $x$ to $y$.
We view each such path as an eventually constant
map defined on $\mathbb{N}$. The bicombing is
\emph{reflexive} if $\rho_{x,x}(i)=x$ for all $x\in X$ and
all $i$, and
\emph{$L$-quasi-geodesic} for some $L\geq 1$
if for all $x,y\in X$ the path
$i\to \rho_{x,y}(i)$ $(i\in [0,k_\rho]\cap 
\mathbb{N})$ is an
$L$-quasi-geodesic. Call the
bicombing \emph{$L$-bounded} for some
$L>0$ if for all $x,y,x^\prime,y^\prime\in X$
and all $i\geq 0$ we have 
\[d(\rho_{x,y}(i),\rho_{x^\prime,y^\prime}(i))
\leq L(d(x,x^\prime)+d(y,y^\prime))+L.\]
As an example, if $X$ is a ${\rm Cat}(0)$-space then any two
points in $X$ can be connected by a unique geodesic
parametrized by arc length, and
these geodesics define a
reflexive $1$-quasi-geodesic $1$-bounded bicombing of $X$.

The purpose of this section is to construct
for some $L\geq 1$ 
a reflexive $L$-quasi-geodesic $L$-bounded
${\cal M\cal C\cal G}(S)$-equivariant 
bicombing
of the train track complex ${\cal T\cal T}$.
We begin with defining for a train track 
$\tau\in {\cal V}({\cal T\cal T})$ which is splittable
to a train track $\eta\in {\cal V}({\cal T\cal T})$ a
combing line connecting $\tau$ to
$\eta$. This combing line is entirely contained in 
the cubical euclidean cone
$E(\tau,\eta)$, and it is 
constructed from a particular
splitting sequence connecting $\tau$ to $\eta$. 
Here as before, $E(\tau,\eta)$ is the full subgraph
of ${\cal T\cal T}$ whose vertex set
consists of all train tracks which can
be obtained from $\tau$ by a splitting sequence
and which are splittable to $\eta$.

Note first that any bicombing constructed directly
from splitting sequences is not  
bounded. The difficulty arises
already in the case of a simple closed
curve $c$ embedded in a train track $\tau$ as the
waist curve of a twist connector as shown in 
Figure B. The curve consists of a large branch and a small
branch, connected at two common switches. A single
split at the large branch, with the small branch as the
winner, results in the train track $\theta_c(\tau)$ where
$\theta_c$ 
is a (positive or negative) Dehn twist about
$c$. 

If the twist connector is 
attached to a standard train track with stops 
of type 0 as
shown in Figure A, then $\tau$ can be modified with a
single shift to a train track $\tau^\prime$ with the 
property that $c$ is an embedded simple closed curve
in $\tau^\prime$ which consists of three branches,
one large branch, one small branch and one mixed branch.
Now a splitting sequence of length 2 results in 
$\theta_c\tau^\prime$. By invariance,
the distance
in ${\cal T\cal T}$ between 
$\theta_c^k\tau$ and $\theta_c^k\tau^\prime$ does not depend on 
$k\in \mathbb{Z}$. On the other hand, there are 
$2k$ splits needed to transform $\tau^\prime$ to 
$\theta_c^k\tau^\prime$, but $k$ splits suffice to
transform $\tau$ to $\theta_c^k\tau$. Therefore splitting
sequences do not give rise to a bounded
bicombing of ${\cal T\cal T}$: Their
parametrizations are not synchronized.

The main observation for the construction of
a bounded bicombing of ${\cal T\cal T}$ is  
that we can ``accelerate''
splitting sequences in a suitable way to yield 
indeed a quasi-geodesic bounded 
bicombing of ${\cal T\cal T}$ as required.

We begin with singling out subsets of a recurrent
train track which are used for ``acceleration'' of
splitting sequences. 
Namely, let $\rho:[0,m]\to \tau$ be an embedded trainpath
which begins and ends with a large half-branch
$\rho[0,1/2],\rho[m-1/2,m]$. We call such a 
trainpath \emph{two-sided large}. By the results of \cite{PH92}
(see Lemma 2.3.2 of \cite{PH92} and the discussion
thereafter),
such a trainpath contains at least one large branch.

As in Section 5, with respect to the
orientation of $S$ and the orientation of $\rho$,
we can distinguish right and 
left neighbors of $\rho$, namely  
half-branches of $\tau$ incident on switches
in the interior of $\rho$ which lie to
the right or to the left of $\rho$, respectively,
in a small neighborhood of $\rho[0,m]$ in $S$. A switch
on which a right (or left) neighbor is incident will
be called a right (or left) switch.
Call a switch 
$\rho(i)$ for some $i\in \{1,\dots,m-1\}$ 
incoming if the half-branch $\rho[i-1/2,i]$ is small, 
and call the switch outgoing otherwise.
A neighbor of $\rho$ incident on an incoming (or outgoing) switch
is called incoming (or outgoing). 

Define a two-sided large embedded trainpath 
$\rho$ in $\tau$ to be \emph{reduced} if 
either all left
neighbors of $\rho$ are incoming and
all right neighbors are outgoing, or if all left
neighbors of $\rho$ are 
outgoing and all right neighbors are
incoming.
If $\rho$ is reduced and if all left neighbors
of $\rho$ are incoming then the same holds true for 
the trainpath $\rho$ with the orientation reversed.
In this case we call $\rho$ a \emph{positive reduced
trainpath}, and otherwise $\rho$ is called a
\emph{negative reduced trainpath}.
Note that a subpath $\rho^\prime$
of a reduced trainpath $\rho$ is reduced 
if and only if it is two-sided large.
An embedded trainpath which consists of a single
large branch will also be called reduced.

If $m\geq 2$ and if $\rho:[0,m]\to \tau$ 
is an embedded reduced trainpath then
for every large branch $e=\rho[i,i+1]$ of $\tau$ contained in $\rho[0,m]$
there is a unique split of $\tau$ at $e$ so that 
the branches $\rho[i-1,i],\rho[i+1,i+2]$ 
are winners of the split.
We called such a split a \emph{$\rho$-split}.
For convenience of terminology, if $\rho$ is
a reduced trainpath of length one then we call every split 
of $\tau$ at this branch a $\rho$-split.

For every train track $\tau$ 
which is splittable to a train track $\sigma$ there is a
natural bijection $\phi(\tau,\sigma)$ from the branches
of $\tau$ to the branches of $\sigma$
(compare Lemma 5.1 of \cite{H09} and its proof). This
bijection also induces a natural bijection of 
the half-branches of $\tau$ 
onto the half-branches of $\sigma$. 
We denote this bijection again by $\phi(\tau,\sigma)$.

For every embedded reduced trainpath $\rho:[0,m]\to \tau$
and every large branch $e$ in $\rho[0,m]$, the train track 
$\tau_1$ obtained 
from $\tau$ by a $\rho$-split at $e$ contains 
$\phi(\tau,\tau_1)(\rho[0,m])$ as an embedded trainpath.
Namely, the split just consists in 
moving one of the neighbors of $\rho$ incident on an
endpoint of $e$ along $\rho$ 
across the second endpoint of $e$ (see Figure D).
If $e=\rho[0,1]$ or $e=\rho[m-1,m]$ then this
trainpath is not two-sided large any more. However,
if we denote this trainpath on $\tau_1$ again
by $\rho$ then a $\rho$-split of $\tau_1$ at a large branch
in $\rho$ is defined.
Thus we can talk about a sequence of splits at large
branches contained in $\rho[0,m]$.
A successive application of this observation shows
that whenever $\tau^\prime$ can be obtained from $\tau$
by a sequence of $\rho$-splits then
$\phi(\tau,\tau^\prime)(\rho)$ is an embedded
trainpath in $\tau^\prime$. Moreover,
every two-sided large subpath of this trainpath is
reduced.

The following
lemma is a fairly immediate consequence of 
Lemma \ref{tightcontrol}
and the remark thereafter. For its
formulation, denote as before by $q$
the number of branches of a complete train track
on $S$.

\begin{lemma}\label{splitadmissible} 
Let $\tau$ be a
train track and
let $\rho:[0,m]\to \tau$ be any embedded reduced
trainpath. Then there is a unique train track $\tau^\prime$
with the following properties.
\begin{enumerate}
\item $\tau^\prime$ is obtained from $\tau$ by a sequence of 
at most $q^2$ $\rho$-splits at large branches contained in $\rho[0,m]$.
\item $\phi(\tau,\tau^\prime)(\rho[0,m])$ is an embedded
trainpath 
in $\tau^\prime$ which does not
contain a large branch.
\end{enumerate}
The splitting sequence connecting $\tau$ to $\tau^\prime$
includes a split at each branch in $\rho[0,m]$.
\end{lemma}
\begin{proof} Our goal is to show that after a transformation
of $\tau$ with at most $m(m+1)/2$ 
$\rho$-splits which includes at least one split at each
branch in $\rho[0,m]$ 
we arrive at a train track 
$\tau^\prime$ with the property that
$\phi(\tau,\tau^\prime)(\rho[0,m])$ does not
contain any large branch.

To see that this is the case
we proceed by induction 
on the length $m$ of the trainpath $\rho$. If this length
equals one then 
the train track $\tau^\prime$ obtained from $\tau$
by a single split at $\rho[0,1]$ contains
$\phi(\tau,\tau^\prime)(\rho[0,1])$ as a small branch
and there is nothing to show.
Thus assume that the claim holds true whenever the 
length of $\rho$ does not exceed $m-1$ for some $m\geq 2$.

Let
$\rho:[0,m]\to \tau$ be an embedded reduced trainpath 
of length $m$ and let $i\geq 0$ be the smallest 
number such that the branch $\rho[i,i+1]$ is large. 
Since the half-branch $\rho[0,1/2]$ is large, 
the trainpath $\rho[0,i+1]$ is 
\emph{one-way} in the terminology used on 
p.127 of \cite{PH92}.
In particular, the
branches $\rho[j,j+1]$ are all mixed for $j<i$
(Lemma 2.3.2 of \cite{PH92}).
Let $\tau_1$ be the train track obtained
from $\tau$ by a $\rho$-split at $\rho[i,i+1]$.

Assume first that $i=0$. Then 
the branch $\phi(\tau,\tau_1)(\rho[0,1])$ is small. 
The trainpath
$\phi(\tau,\tau_1)(\rho[1,m])$ is 
two-sided large and hence reduced, and its length
equals $m-1$. By induction hypothesis,
$\tau_1$ can be transformed with at most $(m-1)m/2$ 
$\phi(\tau,\tau_1)(\rho[1,m])$-splits 
including a split at each branch in 
$\phi(\tau,\tau_1)(\rho[1,m])$ 
to a train track
$\tau^\prime$ so that $\phi(\tau,\tau^\prime)(\rho[1,m])$
does not contain any large branch. Now the half-branch
$\phi(\tau,\tau^\prime)(\rho[0,1/2])$ is small 
and hence
$\phi(\tau,\tau^\prime)(\rho)$ does not contain any large branch
as claimed. The number of splits needed to transform $\tau$
to $\tau^\prime$ is at most $(m-1)m/2+1$.

By possibly reversing the orientation of $\rho$
we are left with the
case that both branches $\rho[0,1]$ and $\rho[m-1,m]$ 
are mixed. Then we have $1\leq i\leq m-2$, and 
$\phi(\tau,\tau_1)(\rho[0,m])$ is a reduced trainpath
in $\tau_1$ 
which contains $\phi(\tau,\tau_1)(\rho[i-1,i])$ as a large branch.
The branches $\phi(\tau,\tau_1)(\rho[j,j+1])$
are mixed for $j\leq i-2$.
Successively we can modify $\tau$ in this way with 
$i\leq m-2$ $\rho$-splits at the branches in 
the subarc $\rho[1,i]$ of $\rho$ to a train track 
$\tilde\tau$ which contains 
$\tilde\rho=\phi(\tau,\tilde\tau)(\rho[0,m])$ as 
a reduced
trainpath and such that the branch 
$\phi(\tau,\tilde\tau)(\rho[0,1])$ is large.
Then we can apply the consideration in the previous paragraph to 
$\tilde\tau$ and the
reduced trainpath $\tilde\rho$ to deduce the lemma.
%
\end{proof}

We say that the train track $\tau^\prime$ constructed
in Lemma \ref{splitadmissible} from
$\tau$ and an embedded reduced trainpath 
$\rho:[0,m]\to \tau$ is obtained from $\tau$ by the
\emph{full $\rho$-multi-split}.

Lemma \ref{splitadmissible} and its proof also show the following.
If $\rho:[0,m]\to \tau$ is an embedded reduced 
trainpath
and if $0\leq i< j\leq m$ are such that the
embedded trainpath
$\rho[i,j]$ is two-sided large then
the train track obtained from $\tau$ by 
the full $\rho[i,j]$-multi-split is splittable
to the train track obtained from $\tau$ by 
the full $\rho$-multi-split.

Recall from Section 5 the definition of a reduced circle
(i.e. a reduced embedded simple closed curve of class $C^1$) in 
a train track $\tau$.
Using the notations from Section 5, if $c$ is a reduced
circle in $\tau$ and if $b$ is any small branch contained
in $c$ then $c-b$ is an embedded arc in $\tau$ which
can be parametrized as a reduced trainpath
in $\tau$.

As before, let $\theta_c$ be the positive Dehn twist
about $c$.
By Lemma \ref{dehnreduced},
there
is a splitting sequence consisting of $c$-splits
which transforms $\tau$ to $\theta_c\tau$ (in case $c$ is positive)
or to $\theta_c^{-1}\tau$ (in case $c$ is negative).
To simplify the notation,
we write $\theta_c^{\pm}\tau$ to denote this 
train track. We have

\begin{lemma}\label{circleadmis}
Let $c$ be a reduced circle in a
train track $\tau$ which consists of at least three branches. 
Let $b$ be a small branch of 
$\tau$ contained in $c$ and
let $\rho$ be a trainpath parametrizing $c-b$. Then 
the train track obtained from $\tau$ 
by the full $\rho$-multi-split is splittable
to $\theta_c^{\pm}\tau$.
\end{lemma}
\begin{proof}
Let $c$ be a reduced circle in $\tau$ and let 
$b$ be a small branch in $c$. Then 
$c-b$ can be parametrized as a reduced
trainpath $\rho:[0,m]\to c\subset\tau$. We assume that
the length $m$ of $\rho$ is at least 2.
For simplicity of notation, assume also 
that the neighbor $b$ of $c$
at $\rho(m)$ is right outgoing with
respect to the orientation of $c$ induced
by the orientation of $\rho$ (note that 
this neighbor is outgoing since the branch $b$ is small).
Then the train track obtained from $\tau$ by the full
$\rho$-multi-split can be described as the 
train track obtained by sliding all left (and hence
incoming) neighbors
of $\rho$ successively along $\rho$ 
past the endpoint of the half-branch $b$ 
(compare the discussion in the
proof of Lemma \ref{dehnreduced} and of
Lemma \ref{splitadmissible}). On the
other hand, $\theta_c^{\pm}\tau$ is obtained from
$\tau$ by sliding all left neighbors of $c$ 
successively along $\rho$ precisely once 
across each endpoint of each
right neighbor of $c$.
This implies the lemma.
\end{proof}

Let $\tau$ be any (not necessarily
recurrent or maximal) 
train track which is
splittable to a train track $\eta$.
Define a \emph{splittable $\eta$-path} in $\tau$ 
to be an embedded reduced trainpath
$\rho:[0,m]\to \tau$
with the following property.
There is a train track $\tau^\prime$ which is 
splittable
to $\eta$ and which is obtained from $\tau$ by
a sequence of $\rho$-splits including at least one split
at each of the branches in $\rho[0,m]$.
This means that for every $i\leq m-1$ there is
a train track $\tau_i\in E(\tau,\eta)$ which can be obtained from
$\tau$ by a sequence of $\rho$-splits and 
with the following additional properties. The branch 
$\phi(\tau,\tau_i)(\rho[i,i+1])$ is large, and
the train track obtained from $\tau_i$ by the
$\rho$-split at this large branch is 
splittable to $\eta$. We call a sequence of $\rho$-splits
of maximal length with the property that the split
track is splittable to $\eta$ the \emph{$\rho-\eta$-multi-split}
of $\tau$.

The splittable $\eta$-path 
$\rho$ is called \emph{maximal} if for every 
splittable $\eta$-path $\rho^\prime$ 
which intersects
$\rho$ in at least one branch we have $\rho^\prime\subset \rho$.

A single large branch $e$ in $\tau$ such that no 
train track obtained from $\tau$ by a split at
$e$ is splittable to $\eta$ is called a \emph{maximal 
non-splittable $\eta$-path.}
A maximal splittable or non-splittable $\eta$-path
is simply called a \emph{maximal $\eta$-path}.

A \emph{splittable $\eta$-circle} is a reduced 
simple closed curve $c$
embedded in $\tau$ such that 
there is a train track $\sigma$ which is splittable
to $\eta$ and which is
obtained from $\tau$ by
a sequence of $c$-splits including a split at every
branch of $c$.
Note that if $c$ is a reduced circle
in $\tau$ of length two, i.e. if 
$c$ consists of a single large branch and a single small branch,
then $c$ is a splittable $\eta$-circle only
if $\theta_c^{\pm}\tau$ is splittable to $\eta$.

\begin{lemma}\label{unique}
Let $\tau$ be a train track
which is splittable to a train track $\eta$.
Then every large branch $e$ of $\tau$ is
contained either in a unique
maximal $\eta$-path or in a unique splittable $\eta$-circle.
\end{lemma}
\begin{proof}
By the definition of a maximal $\eta$-path in a 
train track $\tau$
and by uniqueness of splitting sequences,
a large branch $e$ of $\tau$ such that
no train track which can
be obtained from $\tau$ by a split at $e$ is splittable to $\eta$
is a maximal non-splittable
$\eta$-path, and it is the unique
maximal $\eta$-path containing $e$.

If $e$ is any large branch of $\tau$
such that a splitting sequence
connecting $\tau$ to $\eta$ includes a split
at $e$ then $e$ is contained in
a splittable $\eta$-path or in a splittable
$\eta$-circle.

Let for the moment
$\rho:[0,m]\to \tau$ be any two-sided large trainpath 
with the property that with respect to the 
given orientation of  $\rho$ and the orientation of $S$,
either 
all left switches are incoming and all right switches
are outgoing or all left switches are outgoing and
all right switches are incoming (this is 
a local property). We claim that
either $\rho$ 
is embedded in $\tau$ or $\rho$ is a circle of
class $C^1$. 

We argue by contradiction and we assume that 
$\rho$ is not embedded and not a circle of class
$C^1$. Since $\rho$ is two-sided large, the self-intersection
of $\rho$ can not consist of a single
switch. Thus up to reversing the
orientation of $\rho$ there is some $i\in \{1,\dots, m-1\}$
such that 
the half-branch $\rho[i,i+1/2]$ is large
and that both half-branches which are incident and
small at $\rho(i)$ are contained in the image of $\rho$.
Since $\rho$ is two-sided large, 
there is some $j\not=i$ such that
$\rho[j-1,j+1]$ is a subarc of $\rho$ with 
$\rho(j)=\rho(i)$ 
and such that $\rho[i-1,i+1]\cup\rho[j-1,j+1]$ 
contains all three half-branches which are incident on $\rho(i)$.
Now if $\rho(j+1)=\rho(i+1)$ then
the half-branch $\rho[i-1/2,i]$ is an incoming neighbor of 
$\rho[j-1,j+1]$, and $\rho[j-1/2,j]$ is 
an incoming neighbor of $\rho[i-1,i+1]$.
Moreover, one of these neighbors is a left
neighbor along $\rho$, and the other neighbor is
a right neighbor. This violates the assumption on
$\rho$. Similarly, if $\rho(j-1)=\rho(i+1)$ 
and if $\rho(i)$ is a right (or left) incoming switch
for $\rho[i-1,i+1]$ 
then $\rho(i)$ is a right (or left) outgoing switch for 
$\rho[j-1,j+1]$.
Again this violates
the assumption on $\rho$ and shows the claim.

As a consequence, the union
of two intersecting
two-sided large embedded positive (or negative) reduced 
trainpaths is necessarily either embedded
in $\tau$ or an embedded circle of class $C^1$ 
and hence it is a two-sided large positive (or negative)
reduced trainpath or a reduced circle. Moreover,
two such paths either are disjoint or they 
intersect in at least one branch.

Now let $\rho:[0,m]\to \tau$ be any 
splittable $\eta$-path or splittable $\eta$-circle
of length $m\geq 2$. Then there is some 
$i<m$ such that the branch $\rho[i,i+1]$ is large.
Since $\rho$ is an $\eta$-path, the $\eta$-split
of $\tau$ at the branch $\rho[i,i+1]$ 
is just the unique $\rho$-split of $\tau$ 
at $\rho[i,i+1]$. This implies that 
if $\rho^\prime:[0,n]\to \tau$ is another splittable
$\eta$-path whose image in $\tau$ contains $\rho[i,i+1]$ 
then either $\rho^\prime$ consists of the single large
branch $\rho[i,i+1]$ and hence 
$\rho^\prime\subset \rho$, or the type of  
$\rho^\prime$ (positive or negative) coincides with the
type of $\rho$ (positive or negative). Thus
by the discussion in the previous paragraph,
the union of two splittable $\eta$-paths 
which intersect in at least one branch 
is a splittable $\eta$-path or a splittable
$\eta$-circle.

Let $e$ be any large branch of 
$\tau$ so that a splitting sequence connecting
$\tau$ to $\eta$ includes a split at $e$.
Let $\rho$ be the union of all
splittable $\eta$-paths in 
$\tau$ which pass through $e$. By the discussion in the previous
paragraph, $\rho$ is a splittable $\eta$-path or
a splittable $\eta$-circle, 
and it is the unique maximal $\eta$-path
or splittable $\eta$-circle containing $e$.
Moreover, any two such maximal $\eta$-paths 
or splittable $\eta$-circles either
coincide or 
are disjoint. This shows the lemma.
\end{proof}

Let again $\tau$ be a (not necessarily complete) train
track on $S$ which is splittable
to a train track $\eta$.
A \emph{splittable $\eta$-configuration} 
in $\tau$ is defined 
to be a maximal 
splittable $\eta$-path or a splittable
$\eta$-circle. If $\rho$ is a maximal splittable
$\eta$-path then there is a unique train track
$\tilde\tau$ which can be obtained from $\tau$ by
a sequence of $\rho$-splits of maximal length and which
is splittable  
to $\eta$. In other words,
we require that no train track which can be
obtained from $\tilde \tau$ by
a $\phi(\tau,\tilde\tau)(\rho)$-split at a large
branch in $\phi(\tau,\tilde \tau)(\rho)$ is
splittable to $\eta$. If $c$ is a splittable
$\eta$-circle then there is a unique train track
$\tilde\tau$ which 
can be obtained from $\tau$ by a splitting 
sequence of maximal length and which is splittable to 
both $\theta_c^{\pm}(\tau)$ and $\eta$. In both cases we say that
$\tilde\tau$ is obtained from 
$\tau$ by the \emph{$\rho-\eta$-multi-split}. 
A \emph{non-splittable
$\eta$-configuration} is a single large branch $e$ in
$\tau$ such that no train track which can be obtained from $\tau$
by a split at $e$ is splittable to $\eta$.

We define the
train track obtained from $\tau$ by an \emph{$\eta$-move}
to be the unique train track $\tau^\prime$ with the
following property. Let $\rho_1,\dots,\rho_k$ be the
splittable $\eta$-configurations of $\tau$. By Lemma
\ref{unique}, these are uniquely determined pairwise disjoint
embedded reduced trainpaths or reduced circles 
in $\tau$. In particular, 
splits at branches contained in $\rho_i,\rho_j$ for 
$i\not=j$ commute, and we
define $\tau^\prime$ to be the train track
obtained from $\tau$ by 
a successive modification with a $\rho_i-\eta$-multi-split
where $i=1,\dots,k$. The train track $\tau^\prime$ is 
uniquely determined by $\tau$ and $\eta$.
Moreover, the length of a splitting sequence 
connecting $\tau$ to $\tau^\prime$ is bounded from
above by a number $p>0$ only depending on the 
topological type of $S$.

Extending slightly the notations used earlier on,
for an arbitrary train track $\tau$ on $S$ which is splittable
to a train track $\eta$ let $E(\tau,\eta)$ be the 
connected directed graph whose set of vertices is the set 
of all train tracks which can be obtained from 
$\tau$ by a splitting sequence and which are
splittable to $\eta$. Connect $\sigma\in E(\tau,\eta)$
to $\sigma^\prime\in E(\tau,\eta)$ by a 
directed edge if $\sigma^\prime$ can be obtained from
$\sigma$ by a single split.

For a recurrent train track $\tau$ which is
splittable to a recurrent train track $\eta$
define now inductively a sequence 
$\{\gamma(\tau,\eta)(i)\}_{0\leq i\leq k}\subset
E(\tau,\eta)$ beginning at $\tau=\gamma(\tau,\eta)(0)$ 
and ending at $\eta$
by requiring that for each $i<k$ the train track
$\gamma(\tau,\eta)(i+1)$ is obtained from 
$\gamma(\tau,\eta)(i)$ by an
$\eta$-move. We call the sequence the \emph{balanced
splitting path} connecting $\tau$ to $\eta$, and
we denote it by $\gamma(\tau,\eta)$.
By construction, if
$\tau$ is splittable to $\eta$ and if 
$g\in {\cal M\cal C\cal G}(S)$ is arbitrary
then $g\gamma(\tau,\eta)$ is the balanced
splitting path connecting $g\tau$ to $g\eta$.

For convenience of notation, we extend a discrete
path in a metric space $(X,d)$ 
defined on a subset $[0,n]\cap\mathbb{N}$ 
of the natural numbers to a path defined 
on $\mathbb{N}$ which 
is constant on $[n,\infty)\cap \mathbb{N}$.
Call two eventually constant 
paths $c_1:\mathbb{N}\to X,
c_2:\mathbb{N}\to X$ 
\emph{weight-$L$ fellow travellers} if
\[d(c_1(i),c_2(i))\leq L(d(c_1(0),c_2(0))+d(c_1(\infty),c_2(\infty)))
\text{ for all }i\]
where $c_i(\infty)$ is defined by requiring that  
$c_i(j)=c_i(\infty)$ for all sufficiently large $j$.
If $c_1,c_2$ are weight-$L$ fellow travellers then
the Hausdorff distance in $X$ 
between the images 
$c_1(\mathbb{N})$ and $c_2(\mathbb{N})$ is bounded from above by
$L(d(c_1(0),c_2(0))+d(c_1(\infty),c_2(\infty)))$.

Now we are ready to formulate the 
main result of this section.

\begin{theo}\label{flatconfig3}
There is a number $L>0$ with the following property.
Let 
$\tau,\tau^\prime\in {\cal V}({\cal T\cal T})$ 
be splittable to
$\eta,\eta^\prime\in {\cal V}({\cal T\cal T})$.
Then the balanced splitting paths
\[\gamma(\tau,\eta),
\gamma(\tau^\prime,\eta^\prime)\] 
are weight-$L$ fellow travellers in ${\cal T\cal T}$.
\end{theo}

The remainder of this section is devoted to the proof of
Theorem \ref{flatconfig3}. 
We begin the proof
with collecting some first
basic properties
of balanced splitting paths.

\begin{lemma}\label{basicbalanced}
Let $\tau$ be a train track which is
splittable to a train track 
$\eta$ and 
let $\sigma\in E(\tau,\eta)$. Then the following
holds true.
\begin{enumerate}
\item $\gamma(\tau,\sigma)(1)\in E(\tau,\gamma(\tau,\eta)(1))$. 
\item If $\xi\in E(\tau,\gamma(\tau,\eta)(1))$ then
$\gamma(\tau,\xi)(1)=\xi$ and 
$\gamma(\tau,\eta)(1)\in E(\tau,\gamma(\xi,\eta)(1))$.
\end{enumerate}
\end{lemma}
\begin{proof}
The first part of the lemma is fairly immediate
from the definitions. Namely, assume that
the train track 
$\tau$ is splittable to the train track $\eta$ and let
$\sigma\in E(\tau,\eta)$. 
Let $\rho$ be a $\sigma$-configuration in $\tau$.
If $\rho$ is a splittable $\sigma$-circle then
by the definitions and uniqueness of splitting sequences
(Lemma 5.1 of
\cite{H09} which is also valid for 
train tracks which are not complete),
$\rho$ is an $\eta$-configuration
in $\tau$ and the $\rho-\sigma$-multi-split of $\tau$
is splittable to the $\rho-\eta$-multi-split of $\tau$.

On the other hand, if $\rho$ is a splittable $\sigma$-path then 
by Lemma \ref{unique} and the definitions, 
$\rho$ is a subpath of a maximal splittable
$\eta$-path $\rho^\prime$ or of a 
splittable $\eta$-circle. 
Once again, the $\rho-\sigma$-multi-split of $\tau$ 
is splittable to the $\rho^\prime-\eta$-multi-split of $\tau$ 
by construction, by uniqueness of splitting sequences and
by Lemma \ref{circleadmis}. Note that
the $\eta$-configuration
$\rho^\prime$ may contain several 
distinct $\sigma$-configurations. The first part of the
lemma follows.

The same argument also shows the second part of the
lemma. Namely, let $\rho:[0,n]\to \tau$ be any embedded
reduced trainpath and assume that
$\nu$ is obtained from $\tau$ by a sequence
of $\rho$-splits. Then
$\rho^\prime=\phi(\tau,\nu)(\rho)$ is an embedded trainpath
in $\nu$, and any two-sided large trainpath in $\nu$ which is
contained in $\rho^\prime$ is reduced. Now if
$i\geq 0$ is the smallest number
such that the half-branch $\rho^\prime[i,i+1/2]$ is large
then no sequence of $\rho^\prime$-splits 
can include a split at any of the branches
$\rho^\prime[u,u+1]$ for $0\leq u\leq i-1$. 
Thus by symmetry, if $j\geq i$ is the largest
number such that $\rho^\prime[j-1/2,j]$ is a large half-branch
then $\rho^\prime[i,j]$ is embedded and reduced,
and the train track obtained from $\nu$ by the 
full $\rho^\prime[i,j]$-multi-split 
equals the train track obtained from
$\tau$ by a full $\rho$-multi-split.
This implies that $\gamma(\nu,\zeta)(1)=\zeta=\gamma(\tau,\zeta)(1)$ for 
any train track $\zeta$ which can be obtained from
$\tau$ by a sequence of $\rho$-splits and such that
$\nu\in E(\tau,\zeta)$. In particular, we have
$\gamma(\nu,\gamma(\tau,\eta)(1))(1)=\gamma(\tau,\eta)(1)$.
The same argument is also
valid in the case that $\rho$ is a reduced circle $c$ in $\tau$
and that $\nu,\zeta$ are obtained from $\tau$ by a sequence of
$\rho$-splits and are splittable to $\theta_c^{\pm}\tau$.

Now let $\tau$ be splittable to $\eta$ and let
$\xi\in E(\tau,\gamma(\tau,\eta)(1))$. Let
$\rho_1,\dots,\rho_k$ be the $\eta$-configurations in $\tau$.
Then $\xi$ can be obtained from $\tau$ by successively
splitting $\tau$ with a sequence of $\rho_i$-splits
for $i=1,\dots,k$. Since any $\rho_i$-split commutes
with any $\rho_j$-split for $i\not=j$, the discussion
in the previous paragraph shows that
$\xi=\gamma(\tau,\xi)(1)$ and that 
$\gamma(\tau,\eta)(1)$ is splittable to $\gamma(\xi,\eta)(1)$.
This shows the second part of the lemma.
\end{proof}

As a first step towards a proof of  
Theorem \ref{flatconfig3} we investigate
balanced splitting paths whose images are contained
in a fixed cubical euclidean cone 
$E(\tau,\lambda)$. As in Section 3, denote by $d_E$ the intrinsic
path metric on the connected graph $E(\tau,\lambda)$.

\begin{lemma}\label{flatconfig}
There is a number $L_1>0$ with the following
property.
Let $\lambda$ be a complete geodesic
lamination carried by a complete train track $\tau$
and let $\sigma,\eta\in E(\tau,\lambda)$. Then the
balanced splitting paths
$\gamma(\tau,\sigma)$ and $\gamma(\tau,\eta)$
are weight-$L_1$ fellow travellers in $(E(\tau,\lambda),d_E)$.
\end{lemma}
\begin{proof}
By Lemma 5.4 of \cite{H09}, for 
$\sigma,\eta\in E(\tau,\lambda)$ there is a
train track $\Theta_-(\sigma,\eta)\in E(\tau,\lambda)$
so that $\Theta_-(\sigma,\eta)$ is splittable to
both $\sigma,\eta$ and that there is a geodesic
in $(E(\tau,\lambda),d_E)$ connecting 
$\sigma$ to $\eta$ which passes through $\Theta_-(\sigma,\eta)$.
In particular, we have
\begin{equation}\label{star}d_E(\sigma,\eta)=
d_E(\sigma,\Theta_-(\sigma,\eta))+d_E(\Theta_-(\sigma,\eta),\eta).
\end{equation}
This implies that
the balanced splitting paths  
$\gamma(\tau,\sigma),\gamma(\tau,\eta)$ are weight-$L$ fellow
travellers in $(E(\tau,\lambda),d_E)$ 
if the balanced splitting paths 
$\gamma(\tau,\sigma),\gamma(\tau,\Theta_-(\sigma,\eta))$ and
$\gamma(\tau,\eta)$, $\gamma(\tau,\Theta_-(\sigma,\eta))$
are weight-$L$ fellow travellers in $(E(\tau,\lambda),d_E)$.
As a consequence, it suffices to show the lemma in the 
particular case that $\sigma$ is splittable to $\eta$.

By Corollary 5.2 of \cite{H09},
splitting paths in 
$(E(\tau,\lambda),d_E)$ are geodesics.
Thus  it is enough to show the lemma for 
balanced splitting paths
$\gamma(\tau,\sigma),\gamma(\tau,\eta)$ 
connecting $\tau$ to train tracks
$\sigma,\eta\in E(\tau,\lambda)$ with the additional property that
$\eta$ can be obtained from $\sigma$ by a
single split at a large branch $e$. 

Assume that this is the case. 
Let $\{\sigma(i)\}_{0\leq i\leq k}$ 
be the balanced splitting path
connecting $\tau=\sigma(0)$ to 
$\sigma=\sigma(k)$ and let 
$\{\eta(i)\}_{0\leq i\leq
\ell}$ be the balanced splitting path connecting 
$\tau=\eta(0)$ 
to $\eta=\eta(\ell)$. 
Then there is a largest number $i\leq k$ such that
$\eta(i)=\sigma(i)$. If $i=k$ then we have
$\sigma(k)=\sigma=\eta(k)$. 
Since $\eta$ can be obtained from $\sigma$ by a
single split, 
by definition of an $\eta$-move
we obtain $\ell=k+1$, and the
balanced splitting paths $\gamma(\tau,\sigma),\gamma(\tau,\eta)$ 
connecting $\tau$ to $\sigma,\eta$ 
are weight-$1$ fellow travellers for the distance
$d_E$.

If $i<k$ then also $i<\ell$ and 
$\sigma(i+1)\not=\eta(i+1)$. 
By uniqueness of splitting sequences (Lemma 5.1 of \cite{H09}),
there is a train track $\xi$ which can
be obtained from $\eta(i)=\sigma(i)$ by a 
splitting sequence,
which is splittable to $\sigma$ and which
is splittable 
with a single
split at a large branch $e^\prime$ to $\eta(i+1)$. Moreover,
we have
$\phi(\xi,\sigma)(e^\prime)=e$.

By the first part of Lemma \ref{basicbalanced}, 
applied to $\sigma(i)=\eta(i)$ and to 
$\sigma\in E(\sigma(i),\eta)$, the train track 
$\sigma(i+1)$ is splittable to $\eta(i+1)$
and therefore $\sigma(i+1)$ is splittable to $\xi$.
The second part of Lemma \ref{basicbalanced},
applied to $\xi\in E(\eta(i),\eta(i+1))$, then
implies that $\sigma(i+1)=\xi$. 
As a consequence, the train track $\eta(i+1)$ can
be obtained from $\sigma(i+1)$ by a single split at
$e^{\prime}$.
Inductively we deduce that for every 
$j\in \{i+1,\dots,k\}$ the
train track $\eta(j)$ can be obtained from $\sigma(j)$ by a single
split at $\phi(\sigma(i+1),\sigma(j))
(e^{\prime})$. In other words, 
we have $k=\ell$ and 
the balanced splitting paths 
$\gamma(\tau,\sigma),\gamma(\tau,\eta)$ 
are weight-$1$ fellow travellers for the distance
$d_E$.
 This completes the proof of the lemma.
\end{proof}

If a train track $\tau$ is splittable to a train track $\eta$ then 
the intrinsic path metric $d_E$ on the cone $E(\tau,\eta)$ 
is defined.

\begin{lemma}\label{flatconfig2}
There is a number $L_2>0$ with the following
property. Let $\tau$ be a train track
which is splittable to a train track $\eta$
and let $\sigma\in E(\tau,\eta)$. Then the 
balanced splitting paths $\gamma(\tau,\eta),
\gamma(\sigma,\eta)$ are weight-$L_2$ fellow
travellers in $(E(\tau,\eta),d_E)$.
\end{lemma}
\begin{proof} Let the train track
$\tau$ be splittable to the train track $\eta$.
By definition of a balanced splitting path, if $e$ is a large
branch in $\tau$ such that a splitting sequence
connecting $\tau$ to $\eta$ includes a split at $e$ then
a splitting sequence connecting $\tau$ to
$\gamma(\tau,\eta)(1)$ includes a split at $e$ as well. 
Since directed edge-paths in $(E(\tau,\eta),d_E)$ are
geodesics (Corollary 5.2 of \cite{H09} is also valid
if the train tracks $\tau,\eta$ are not complete),
it therefore suffices to show the 
lemma in the particular case that $\sigma\in 
E(\tau,\gamma(\tau,\eta)(1))$.

For this we claim
that if $\tau$ is splittable to $\eta$ and if 
$\sigma\in E(\tau,\gamma(\tau,\eta)(1))$
then for each $u$ the train track 
$\gamma(\tau,\eta)(u)$ is splittable
to $\gamma(\sigma,\eta)(u)$, and
$\gamma(\sigma,\eta)(u)$ is splittable to
$\gamma(\tau,\eta)(u+1)$.

We proceed by induction on the length of a splitting
sequence connecting $\tau$ to $\eta$.
If this length vanishes 
then $\tau=\eta=\sigma$ and there is nothing to show, so
assume that the claim is known whenever 
the length of a splitting sequence connecting
$\tau$ to $\eta$ is at most $n-1$ for some $n\geq 1$.

Let $\tau$ be splittable to $\eta$ and assume that
the length of a splitting sequence connecting
$\tau$ to $\eta$ equals $n$.
If $\sigma=\tau$ then once again there is nothing to show, 
so assume that $\sigma$ can be obtained from 
$\tau$ by a non-trivial splitting sequence and is splittable 
to $\gamma(\tau,\eta)(1)$. 
By the second part of Lemma \ref{basicbalanced}, 
$\gamma(\tau,\eta)(1)$ is
splittable to $\gamma(\sigma,\eta)(1)$. 
Now the length of a splitting sequence 
connecting $\sigma$ to $\eta$ is at most $n-1$
and hence we can apply the 
induction hypothesis to $\sigma,\gamma(\tau,\eta)(1),\eta$
and to the balanced splitting paths 
$\gamma(\sigma,\eta)$ and 
$\gamma(\gamma(\tau,\eta)(1),\eta)$. 
This yields the above claim.

As a consequence,
we have
\[d_E(\gamma(\tau,\eta)(u),\gamma(\sigma,\eta)(u))\leq p\]
for all $u$ where $p>0$ is an upper bound for the
length of any splitting sequence connecting a 
train track to its image under a move.
Thus the balanced
splitting paths $\gamma(\tau,\eta),\gamma(\sigma,\eta)$
are weight-$p$ fellow travellers in $E(\tau,\eta),d_E)$. 
This shows the lemma.
\end{proof}

As an immediate corollary we conclude that
balanced splitting paths connecting points
in the same cubical euclidean cone are uniform fellow
travellers.

\begin{corollary}\label{flatstripcontrol}
There is a number $L_3>0$ with the following
property.
Let $E(\tau,\lambda)\subset {\cal T\cal T}$ be any cubical euclidean cone
and let $\sigma_1,\sigma_2\in E(\tau,\lambda)$ be
train tracks which are splittable to train tracks
$\eta_1,\eta_2\in E(\tau,\lambda)$. Then
the balanced splitting paths
$\gamma(\sigma_1,\eta_1),
\gamma(\sigma_2,\eta_2)$ are weight-$L_3$ fellow
travellers.
\end{corollary}
\begin{proof}
Let $\sigma_i,\eta_i$ be as in the corollary. Let
$d_E$ be the intrinsic path metric on $E(\tau,\lambda)$.
By Lemma 5.4 of \cite{H09}, there is a unique
train track $\Theta_-(\sigma_1,\sigma_2)\in E(\tau,\lambda)$
such that $\sigma_1,\sigma_2\in E(\Theta_-(\sigma_1,\sigma_2),\lambda)$ 
and that 
there is a geodesic in $(E(\tau,\lambda),d_E)$
connecting $\sigma_1$ to $\sigma_2$ which passes through
$\Theta_-(\sigma_1,\sigma_2)$. In particular, 
we have $d_E(\sigma_1,\sigma_2)=
d_E(\sigma_1,\Theta_-(\sigma_1,\sigma_2))+
d_E(\Theta_-(\sigma_1,\sigma_2),\sigma_2)$.

By Lemma \ref{flatconfig}, there is a number
$L_1>0$ such that the balanced splitting paths 
$\gamma(\Theta_-(\sigma_1,\sigma_2),\eta_1)$ and
$\gamma(\Theta_-(\sigma_1,\sigma_2),\eta_2)$ 
with the same starting point 
$\Theta_-(\sigma_1,\sigma_2)$ 
are weight-$L_1$ fellow
travellers in 
$(E(\Theta(\sigma_1,\sigma_2),\lambda),d_E)\subset
(E(\tau,\lambda),d_E)$. 
Moreover, by Lemma \ref{flatconfig2} and its proof, there is a number
$L_2>0$ such that the
balanced splitting paths 
$\gamma(\Theta_-(\sigma_1,\sigma_2),\eta_i)$ and
$\gamma(\sigma_i,\eta_i)$
$(i=1,2)$ 
with the same endpoints 
are weight-$L_2$ fellow travellers
in $(E(\Theta_-(\sigma_1,\sigma_2),\lambda),d_E)
\subset (E(\tau,\lambda),d_E)$. 
Since $d_E(\sigma_1,\sigma_2)=
d_E(\sigma_1,\Theta_-(\sigma_1,\sigma_2))+
d_E(\Theta_-(\sigma_1,\sigma_2),\sigma_2)$,
the balanced splitting paths
$\gamma(\sigma_1,\eta_1),\gamma(\sigma_2,\eta_2)$ are
weight-$(L_1+L_2)$-fellow travellers in
$(E(\tau,\lambda),d_E)$. Together with Theorem \ref{cubicaleuclid}
this yields the corollary.
\end{proof}

Next we extend Corollary \ref{flatstripcontrol} to
train tracks related by carrying rather than
splitting. We begin with the most basic case.

\begin{lemma}\label{shiftingbasic}
There is a number $L_4>0$ with the following property.
Let $\tau,\sigma$ be complete train tracks which 
are splittable 
to the same complete train track $\eta$. If 
$\sigma$ can be obtained from $\tau$ by a single
shift then 
\[d(\gamma(\tau,\eta)(i),\gamma(\sigma,\eta)(i))\leq L_4
\text{ for all }i.\]
\end{lemma}
\begin{proof} Let $\tau,\sigma$ be splittable to $\eta$ and assume
that $\sigma$ is obtained from $\tau$ by a 
single shift at a mixed branch $b$.
Let $v,w$ be the two switches of $\tau$
on which $b$ is incident.
Note that since $\tau$ is complete
and hence recurrent, the switches $v,w$ are distinct:
otherwise $b$ defines an embedded circle of class $C^1$ in
$\tau$ which contains a single switch.
There is a trainpath
$\zeta:[0,3]\to \tau$ where 
$b=\zeta[1,2]$ and such that
the half-branches $\zeta[1,3/2],\zeta[2,5/2]$ 
are large (see Figure C). Up to isotopy,
the shift then moves the neighbor of
$\zeta$ at $\zeta(1)=v$ across the switch 
$\zeta(2)=w$.
This means that shifting
exchanges the two switches $v,w$ along $\zeta$.
The natural bijection
$\phi(\tau,\sigma)$  
of the branches of $\tau$ onto the branches
of $\sigma$ preserves the type of the 
branches (i.e. large, mixed, small). In particular,
the branch $\phi(\tau,\sigma)(b)$ in $\sigma$ is mixed,
and $\tau$ can be obtained from $\sigma$
by a single shift at $\phi(\tau,\sigma)(b)$.

Throughout we use the following immediate consequence
of uniqueness of splitting sequences
(Lemma 5.1 of \cite{H09}). 
If $\sigma\in {\cal V}({\cal T\cal T})$ 
is splittable to $\eta\in {\cal V}({\cal T\cal T})$
and if $\xi$ is obtained from $\sigma$ by a splitting
sequence and carries $\eta$ then $\xi$ is splittable
to $\eta$. 
We subdivide the argument into seven steps.

{\sl Step 1:}

Let $e$ be a large branch in $\tau$ and assume that
a splitting sequence connecting $\tau$ to $\eta$
includes a split at $e$. Then the minimal
cardinality of the preimage of a point $x$ 
contained in the interior of $e$ under \emph{any}
carrying map $F:\eta\to\tau$ is at least two
(see the proof of Lemma 5.1 of \cite{H09}). 
Up to isotopy, for any neighborhood $U$ of 
the branch $b$ in $S$ there is
a map $F:S\to S$ of class $C^1$ which is homotopic to the
identity, whose restriction to $\sigma$ is a carrying
map $\sigma\to \tau$ and which is the identity on $S-U$. 
Therefore  
the minimal cardinality of the preimage of any
point $y$ in the interior of $\phi(\tau,\sigma)(e)$ 
under any carrying map $\eta\to \sigma$ 
is at least
two as well. This implies that a splitting
sequence connecting $\sigma$ to $\eta$ includes a split at
$\phi(\tau,\sigma)(e)$ (see the proof of Lemma 5.1 of
\cite{H09}). In other words, a large branch $e$ of $\tau$
is a non-splittable $\eta$-configuration in $\tau$ if and
only if $\phi(\tau,\sigma)(e)$ is a non-splittable
$\eta$-configuration in $\sigma$.

{\sl Step 2:} 

Let $e$ be a large branch in $\tau$
which is not incident on an endpoint of $b$.
Then $e$ is contained in the complement of a small
neighborhood $U$ of $b$ in $S$.
Up to isotopy,  
the intersection of 
$\tau$ with $S-U$ coincides with the
intersection of $\sigma$ with $S-U$. 
The intersection with
$S-U$ of the train track $\tau_1$ 
obtained from $\tau$ by a right (or left) split
at $e$ is isotopic to the
intersection with $S-U$ of the train track $\sigma_1$
obtained from $\sigma$ by a right (or left) split 
at $e$. Thus $\sigma_1$ carries $\tau_1$. Therefore if 
$\tau_1$ is splittable to
$\eta$ then $\sigma_1$ carries $\eta$ and hence 
$\sigma_1$ is splittable
to $\eta$.

{\sl Step 3:} 

Let $e$ be a large branch in $\tau$ such that one
of the endpoints of $e$ is a switch $v$ on which 
the branch $b$ is incident. This means that
there is a trainpath $\rho:[0,2]\to \tau$ of length two
such that $\rho[0,1]=b$ and $\rho[1,2]=e$. 
Assume that the train track $\tau_1$ obtained from
$\tau$ by a right (or left) split at $e$ is 
splittable to $\eta$ and that the same holds true
for the train track $\sigma_1$ obtained from $\sigma$
by a right (or left) split at $\phi(\tau,\sigma)(e)$.
Figure F (which is just Figure 2.3.3 of \cite{PH92},
reproduced here for convenience) shows that 
if $b$ is a winner of the split connecting $\tau$ to $\tau_1$ then
$\phi(\tau,\sigma)(b)$ is a loser of the split
connecting $\sigma$ to $\sigma_1$. In Figure F,
this is for example the case if $\tau$ is as in the
left hand side of the figure and if $\tau_1$ is obtained
from $\tau$ by a right split. 

Assume without loss of generality that $b$ is a winner
of the right (or left) split connecting $\tau$ to $\tau_1$.
Then there is a train track $\tau_2$ which is
obtained from $\tau_1$ by a single right (or left) 
split at $\phi(\tau,\tau_1)(b)$ and such that
$\tau_2$ can be obtained from $\sigma_1$ by a single
shift. Since $\sigma_1$ is splittable to $\eta$,
the train track $\tau_2$ carries $\eta$ and hence $\tau_2$  
is splittable to $\eta$. Therefore, in this case
$\rho[0,2]$ is contained in an $\eta$-configuration of $\tau$.
There is a train track $\xi_1\in E(\tau,\gamma(\tau,\eta)(1))$ which can
be obtained from a train track $\xi_2\in E(\sigma,\gamma(\sigma,\eta)(1))$
by a single shift.
\begin{figure}[ht]
\begin{center}
\psfrag{Figure C}{Figure F}
\includegraphics
[width=0.7\textwidth]
{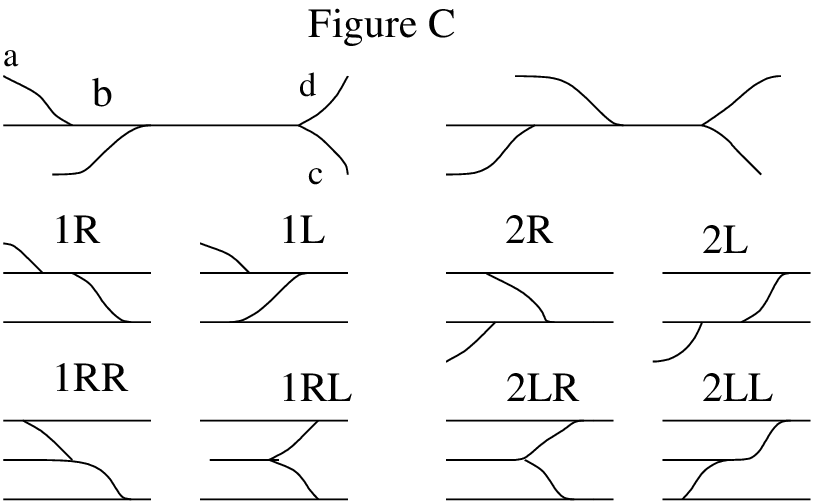}
\end{center}
\end{figure}

{\sl Step 4:} 

Let $e$ be a large branch in $\tau$ such that one
of the endpoints of $e$ is a switch $v$ on which 
the branch $b$ is incident, i.e. that
there is a trainpath $\rho:[0,2]\to \tau$ of length two
with $\rho[0,1]=b$ and $\rho[1,2]=e$. 
Assume that the train track $\tau_1$ obtained from
$\tau$ by a right (or left) split at $e$ is 
splittable to $\eta$ and that the same holds true
for the train track $\sigma_1$ obtained from $\sigma$
by a left (or right) split at $\phi(\tau,\sigma)(e)$.
We claim that in this case both $\tau$ and $\sigma$
can be modified with a splitting sequence of length 2 
to a train track $\xi\in E(\tau,\eta)\cap E(\sigma,\eta)$.

By the discussion in Step 3 above 
and uniqueness of splitting sequences,
the branch $b$ is a winner of the split connecting
$\tau$ to $\tau_1$, and $\phi(\tau,\sigma)(b)$ is a winner
of the split connecting $\sigma$ to $\sigma_1$.
The branches $b_1=\phi(\tau,\tau_1)(b)$ and 
$c_1=\phi(\sigma,\sigma_1)(\phi(\tau,\sigma)(b))$ are large.
In Figure F, this corresponds for example to the case that
$\tau$ is as in the left had side of the figure, that
$\tau_1$ is obtained from $\tau$ by a 
right split and that $\sigma_1$ is obtained from $\sigma$ by
a left split.

If the train track $\tau_1^\prime$ 
obtained from $\tau_1$ by a right
(or left) 
split at $b_1$ is splittable to $\eta$ then $\tau_1^\prime$
can be modified with a single shift to 
the train track $\sigma_1^\prime$ 
obtained from $\sigma$ by a 
right (or left) split at $\phi(\tau,\sigma)(e)$. Then
$\sigma_1^\prime$  
carries $\eta$ and hence is splittable to $\eta$. 
This violates uniqueness of splitting 
sequences connecting $\sigma$ to $\eta$ (Lemma 5.1 of \cite{H09}). 
Thus if a train track obtained from $\tau_1$ by a split
at $b_1$ is splittable to $\eta$, then this is the train track
$\tau_2$
obtained from $\tau_1$ by a left (or right) split at $b_1$.
Figure F shows that $\tau_2$ is
isotopic to a train track obtained from $\sigma_1$ by a 
right (or left) split at $c_1$ and the claim holds true.

If no train track obtained from $\tau_1$ by a 
split at $b_1$ is splittable to $\eta$ then
$\phi(\tau_1,\eta)(b_1)$ is a large branch in $\eta$.
The train track $\tau^\prime$ obtained from $\tau_1$ by
a right (or left) split at $b_1$ is splittable
to the train track $\eta^\prime$
obtained from $\eta$ by a right
(or left) split at $\phi(\tau_1,\eta)(b_1)$. 
Since $\sigma_1$ is splittable to $\eta$, it is also
splittable to $\eta^\prime$.
On the other hand, the train track $\sigma^\prime$ 
is obtained from 
$\sigma$ by a right (or left) split at $\phi(\tau,\sigma)(e)$
is a modification of $\tau^\prime$ with a single shift and
hence it carries $\eta^\prime$.
This contradicts
uniqueness of splitting sequences connecting $\sigma$
to $\eta^\prime$ (which also holds
true in the case that $\eta^\prime$ is not recurrent
and hence not complete).
The above claim is proven.

{\sl Step 5:}

Let $\rho:[0,n]\to \tau$ be
an $\eta$-configuration of $\tau$ 
which contains the mixed branch $b$. 

Figure F shows
that $\rho^\prime=\phi(\tau,\sigma)(\rho-b)$ is a reduced
trainpath in $\sigma$ or a reduced circle.
Namely, via perhaps reversing the orientation of $\rho$
we may assume that $b=\rho[i,i+1]$ for some $i\in \{0,\dots,n-2\}$ 
and that 
the neighbor of $\rho$ at $\rho(i)$ is incoming.
(This means that in Figure C, 
$\rho$ passes from the left to the right.)
The neighbor of $\rho$ at $\rho(i+1)$ is
also incoming. Since $\rho$
is reduced, the neighbor of $\rho$ 
at $\rho(i+1)$ is contained in the same
side of $\rho$ as the neighbor at $\rho(i)$ in a small
neighborhood of $\rho[0,n]$ in $S$.
(Thus in Figure C, 
if a neighborhood of $b$ in $\tau$ 
is as in the left part of the picture
then $\rho[i-1,i]$ is the
branch beginning at the top left corner of the picture,
and if a neighborhood of 
$b$ in $\tau$ is as in the right part of the picture then
$\rho[i-1,i]$ is the branch beginning at the bottom
left corner.)

The trainpath $\rho[i,n]$ is two-sided large and reduced.
Since $\rho[i+1,i+3/2]$ is a large half-branch, 
by definition of an $\eta$-configuration
there is a train track $\tau_1$ which
can be obtained from $\tau$ by a sequence of 
$\rho$-splits at branches contained in 
$\rho[i+2,n]$ such that
$\phi(\tau,\tau_1)(\rho[i+1,i+2])$ is a large branch.
There also is a train track $\sigma_1$ which can
be obtained from $\sigma$ by a sequence of 
$\rho^\prime$-splits and which can be obtained
from $\tau_1$ by a single shift at $\phi(\tau,\tau_1)(b)$.
By the definition of an $\eta$-configuration, 
the train track $\tau_2$ obtained from
$\tau_1$ by a $\rho$-split 
at the branch
$\phi(\tau,\tau_1)(\rho[i+1,i+2])$ (with the mixed branch
$\phi(\tau,\tau_1)(b)$ as a winner) followed by
a $\rho$-split at the branch 
$\phi(\tau,\tau_1)(\rho[i,i+1])$ is splittable to $\eta$.
By Step 3 above, the train track $\tau_2$  
can be modified with a single shift at the 
mixed branch $\phi(\tau,\tau_2)(b)$
to a train track 
$\sigma_2$ obtained
from $\sigma$ by a sequence of 
$\rho^\prime$-splits. 

Reapply this consideration to the maximal
two-sided large trainpath contained in 
$\phi(\tau,\tau_2)(\rho)$. Since the winner of any
split in a splitting sequence which modifies
$\tau$ to its image under the $\rho-\eta$-multi-split is 
contained in $\rho$, together with Steps 1-3 above we conclude
that there is a train track $\tilde \sigma$ which can
be obtained from $\gamma(\tau,\eta)(1)$ by a single shift
and which is splittable to $\gamma(\sigma,\eta)(1)$.

{\sl Step 6:} 

Let $\rho:[0,n]\to \tau$ be an $\eta$-configuration
such that the mixed branch $b$ is incident on a switch
contained in $\rho[0,n]$ but is not contained in $\rho[0,n]$.
There are now two possibilities. 

In the first case,
$b$ is incident on a switch $\rho(i)$ for some
$i\in \{1,\dots,n-1\}$. Assume that $b$ is an incoming neighbor of 
$\rho$, i.e. that the half-branch $\rho[i,i+1/2]$ is large.
By the discussion in Step 5 above,
there is a reduced trainpath $\rho^\prime$ in $\sigma$
which contains $\phi(\tau,\sigma)(\rho\cup b)$. Moreover,
if $\tau_1$ is obtained from $\tau$ by a sequence
of $\rho$-splits at branches in $\rho[i+1,n]$
and if $\phi(\tau,\tau_1)(\rho[i,i+1]$ is a large branch,
then $\phi(\tau,\tau_1)(b)$ is a loser of the $\rho$-split
of $\tau_1$ at $\phi(\tau,\tau_1)(\rho[i,i+1])$.
As a consequence, this case corresponds to an exchange of the
roles of $\tau$ and $\sigma$ in Step 5 above.

Together with Step 1 and Step 5 above, we conclude the following.
If either the mixed branch $b$ is contained in 
an $\eta$-configuration of $\tau$ or if $b$ is incident on 
a switch contained in the interior of an $\eta$-configuration of $\tau$
then $\gamma(\sigma,\eta)(1)$ can be obtained from
$\gamma(\tau,\eta)(1)$ by a single shift.

If $b$ is incident on the endpoint $\rho(0)$ of the $\eta$-configuration
$\rho$ of $\tau$ then $\phi(\tau,\sigma)(\rho)$ is a
reduced trainpath in $\sigma$. After modifying 
$\tau,\sigma$ with a sequence of 
$\rho$-splits (or $\phi(\tau,\sigma)(\rho)$-splits)
and perhaps reversing the orientation of $\rho$ we
may assume that $\rho[0,1]$ is a large branch.
By Step 3 above, since $b$ is not contained in the 
$\eta$-configuration $\rho$, the branch $b$  
can not be a winner of the
$\eta$-split at $\rho[0,1]$. This means that
either $\phi(\tau,\sigma)(b)\cup \phi(\tau,\sigma)(\rho)$
is contained in an $\eta$-configuration of $\sigma$
and $\gamma(\sigma,\eta)(1)$ can be obtained from
$\gamma(\tau,\eta)(1)$ by a single shift,
or $\phi(\tau,\sigma)(b)$ is a loser of an
$\eta$-split of $\sigma$ at $\phi(\tau,\sigma)(\rho[0,1])$
and by Step 4 above,
both $\tau,\sigma$ are splittable to the same
train track $\xi\in E(\tau,\eta)\cap E(\sigma,\eta)$ with 
a splitting sequence of uniformly bounded length.

{\sl Step 7:}

Steps 1-6 above 
and Corollary \ref{flatstripcontrol} now imply
the following. Either $\gamma(\sigma,\eta)(1)$ can
be obtained from $\gamma(\tau,\eta)(1)$ by a single
shift, or 
\[d(\gamma(\tau,\eta)(i),\gamma(\sigma,\eta)(i))\leq 2L_3D
\text{ for all }i\]
where $L_3>0$ is as in Corollary \ref{flatstripcontrol} and where
$D>0$ is a universal constant.

If $\gamma(\sigma,\eta)(1)$ can be obtained from
$\gamma(\tau,\eta)(1)$ by a single shift then 
use the above consideration for 
to $\gamma(\tau,\eta)(1)$ and
$\gamma(\sigma,\eta)(1)$. Since 
by invariance under the action of the mapping class
group and cocompactness, the distance in ${\cal T\cal T}$ between
any two shift equivalent train tracks is uniformly bounded,
we conclude that 
$\gamma(\tau,\eta)$ and $\gamma(\sigma,\eta)$ are
weight-$L_4$ fellow travellers for a universal constant 
$L_4>1$ as claimed. This completes the proof of the lemma.
\end{proof}

As a preparation for a general control of balanced splitting paths
for train tracks related by
shifting we need

\begin{lemma}\label{splittingcontrol}
For every $k>0$ there is a number $n(k)>0$ with the
following property. Let 
$\tau,\eta$ be complete train tracks with
$d(\tau,\eta)\leq k$ which carry a common complete
geodesic lamination $\lambda$. Then $\tau,\eta$ are
splittable to the same complete train track $\sigma$ with 
a splitting sequence of length at most $n(k)$ each.
\end{lemma}
\begin{proof}
Let $\tau,\eta$ be any two train tracks which carry
a common complete geodesic lamination $\lambda$.
Then the set of all complete geodesic laminations carried
by both $\tau,\eta$ is open and closed in ${\cal C\cal L}$
(Lemma 2.3 of \cite{H09}), and
it is non-empty. Therefore there is a minimal complete
geodesic lamination $\mu$ carried by both $\tau,\eta$
(this is explained in detail in the proof of 
Lemma 3.3 and Lemma 3.4 of \cite{H09}).
By Corollary 2.4.3 of \cite{PH92}, the train tracks
$\tau,\eta$ are splittable to a common train track $\zeta$ 
which carries $\mu$. Since $\mu$ is complete,
the train track $\zeta$ is complete as well.

Now there are only finitely many orbits under the action of
the mapping class group of pairs of complete train tracks
which carry a common complete geodesic lamination
and whose distance is at most $k$. 
By invariance under the action of the mapping class group,
the lemma follows.
\end{proof}

The next lemma is the main technical step towards the proof of
Theorem \ref{flatconfig3}.

\begin{lemma}\label{shift}
There is a number
$L_5>0$ with the following
property. Let $\tau,\sigma\in {\cal V}({\cal T\cal T})$ 
be splittable 
to $\eta\in {\cal V}({\cal T\cal T})$ 
and assume that $\sigma$ is carried 
by $\tau$.
Then the balanced splitting paths
$\gamma(\tau,\eta),\gamma(\sigma,\eta)$ are
weight-$L_5$ fellow travellers.
\end{lemma}
\begin{proof} 
For convenience of terminology, call two complete  
train tracks $\zeta_1,\zeta_2$ \emph{shift equivalent}
if $\zeta_1$ can be obtained from $\zeta_2$ by
a sequence of shifts. This clearly defines an
equivalence relation on the set of all
complete train tracks on $S$.

We divide the proof of the lemma into two steps.

{\sl Step 1:}
 
By Theorem 2.4.1 of \cite{PH92},
if $\sigma\in {\cal V}({\cal T\cal T})$ 
is carried by $\tau\in {\cal V}({\cal T\cal T})$ then $\sigma$ can
be obtained from $\tau$ by a splitting and shifting
sequence. In other words, $\tau$ can be connected
to $\sigma$ by a sequence $\{\eta(i)\}_{0\leq i\leq k}$ 
of minimal length such that for all $i$,
$\eta(i+1)$ can be obtained from $\eta(i)$ either 
by a single split or a single shift. We call such a
sequence a \emph{stretched-out splitting and shifting
sequence}.

We claim 
that there is a number
$L>1$ with the following property. Assume that
$\sigma$ can be obtained from
$\tau$ by a stretched-out splitting and shifting
sequence of length at most $n$. If $\tau,\sigma$
are splittable to a common complete train track $\eta$ 
then the balanced splitting paths
$\gamma(\tau,\eta)$ 
and $\gamma(\sigma,\eta)$ 
satisfy 
\begin{equation}\label{progressfellow}
d(\gamma(\tau,\eta)(i),\gamma(\sigma,\eta)(i))
\leq nL
\text{ for all }i.
\end{equation}

To determine a
number $L>1$ with this property we need the 
following preparation. 
Let $\alpha,\beta\in {\cal V}({\cal T\cal T})$.
Assume that $\beta$ is obtained from $\alpha$ by a 
single shift and that $\alpha$ 
is splittable
to a train track $\zeta$. Then $\zeta$ is carried by $\beta$.
Let $\lambda$ be a complete
geodesic lamination carried by $\zeta$. By Proposition A.6
of \cite{H09}, there is a number $k>0$ such that 
$\beta$ is splittable to a train track $\xi$ 
which carries $\lambda$ and is such that $d(\zeta,\xi)\leq k$.
By Lemma \ref{splittingcontrol},
there is a train track $\zeta^\prime$ which can be obtained from
both $\zeta,\xi$ by a splitting sequence of length at most
$n(k)$. Thus $\alpha,\beta,\zeta,\xi$ are 
all splittable to $\zeta^\prime$.

By Corollary \ref{flatstripcontrol}, the balanced
splitting paths $\gamma(\alpha,\zeta), 
\gamma(\alpha,\zeta^\prime)$
are weight-$L_3$ fellow travellers. Thus
Lemma \ref{shiftingbasic}, applied to $\alpha,\beta,\zeta^\prime$,
implies that 
there is a number $L>2L_3n(k)$ so that the balanced
splitting paths $\gamma(\alpha,\zeta)$
and $\gamma(\beta,\zeta^\prime)$
satisfy
\begin{equation}\label{step5est}
d(\gamma(\alpha,\zeta)(i),
\gamma(\beta,\zeta^\prime)(i))\leq L/2
\text{ for all }i.
\end{equation}

We show by induction on $n\geq 0$ that the inequality
(\ref{progressfellow}) 
holds true for this number $L$. 
Note first that if the minimal length of 
a stretched out splitting and shifting sequence connecting
$\tau$ to $\sigma$ vanishes then $\tau=\sigma$ and  
there is nothing to show.
So assume that the claim holds true whenever
there is a stretched out splitting and shifting sequence 
connecting $\tau$ to $\sigma$ whose 
length does not exceed
$n-1$ for some $n\geq 1$.

Let $\sigma$ be obtained from
$\tau$ by a stretched-out splitting and shifting sequence
of length $n$. If this sequence can be
arranged to begin with a split at a large
branch $e$ of $\tau$ then this split is
an $\eta$-split since $\sigma$ is splittable
to $\eta$ by assumption. Let $\tilde \tau$
be the split track. Then $\tilde \tau$ can
be connected to $\sigma$ by a stretched-out
splitting and
shifting sequence of length $n-1$. 
Therefore by the induction hypothesis, the
balanced splitting paths 
$\gamma(\tilde\tau,\eta),\gamma(\sigma,\eta)$
satisfy the estimate (\ref{progressfellow}) for
$n-1$.
By Corollary \ref{flatstripcontrol}, the balanced
splitting paths $\gamma(\tau,\eta),
\gamma(\tilde \tau,\eta)$ are weight-$L_3$ fellow
travellers and hence they satisfy the estimate
(\ref{progressfellow}) with $n=1$.
Together we conclude that
the estimate (\ref{progressfellow}) holds true for the
balanced splitting paths $\gamma(\tau,\eta),
\gamma(\sigma,\eta)$.

If there is no splitting and shifting
sequence of length $n$ connecting $\tau$ to $\sigma$
which begins
with a single split then let $\tilde\sigma$ be
a train track obtained from $\tau$ by a single
shift which can be connected to $\sigma$ by a 
stretched-out splitting
and shifting sequence of length $n-1$.
Let $\zeta$ be a train track with the property
that both $\eta,\tilde \sigma$ are splittable
to $\zeta$ and that the length of a splitting
sequence connecting $\eta$ to $\zeta$ is at most $n(k)$.
Such a train track exists by the above discussion. 
Apply the induction hypothesis
to the balanced splitting paths $\gamma(\tilde \sigma,\zeta),
\gamma(\sigma,\zeta)$. We conclude that
\begin{equation}\label{step61}
d(\gamma(\tilde\sigma,\zeta)(i),\gamma(\sigma,\zeta)(i))\leq
(n-1)L \text{ for all }i.
\end{equation}
On the other hand, the estimate
(\ref{step5est}) yields that 
\begin{equation}\label{step62}
d(\gamma(\tau,\eta)(i),\gamma(\tilde \sigma,\zeta)(i))\leq L/2
\text{ for all }i.
\end{equation}
Moreover, since $L>2L_3n(k)$ by assumption,
Corollary \ref{flatstripcontrol} applied to
$\sigma,\eta,\zeta$ implies that
\begin{equation}\label{step63}
d(\gamma(\sigma,\zeta)(i),\gamma(\sigma,\eta)(i))\leq L/2
\text{ for all }i.
\end{equation}
The estimates (\ref{step62},\ref{step61},\ref{step63})
together yield the inequality 
(\ref{progressfellow}) which
completes the induction step. The estimate
(\ref{progressfellow}) is proven.

{\sl Step 2:}

Using Step 1 above, we are now able to complete the
proof of the lemma. Namely, assume that
$\sigma\prec\tau$ are both splittable to a complete
train track $\eta$. Let $\lambda$ be a complete geodesic
lamination carried by $\eta$. By Proposition A.6 of 
\cite{H09}, $\tau$ is splittable to a complete
train track $\tau^\prime$ which carries
$\lambda$ and is such that 
\[d(\tau^\prime,\sigma)\leq \kappa_0\]
where $\kappa_0 >0$ is a universal constant.

By Lemma 6.7 of \cite{H09},
there is a number $\kappa_1>0$ only depending on 
$\kappa_0$ and there is a train track $\sigma^\prime$
which carries $\lambda$ and which can be obtained 
from both $\tau^\prime,\sigma$ by a 
stretched-out splitting and shifting sequence
of length at most $\kappa_1$.
Another application of Lemma 6.7 of \cite{H09}
yields a universal number $\kappa_2>0$ and a 
train track $\eta^\prime$ with 
\[d(\eta,\eta^\prime)\leq \kappa_2\]
which is carried by both $\sigma^\prime$ and $\eta$
and hence by $\tau^\prime,\sigma$.
Proposition A.6 of \cite{H09} and 
Lemma \ref{splittingcontrol} then show the existence
of a train track $\eta^{\prime\prime}$ with 
\[d(\eta,\eta^{\prime\prime})\leq \kappa_3\]
for a universal constant $\kappa_3>0$ such that
both $\sigma^\prime$ and $\eta$ are splittable
to $\eta^{\prime\prime}$.

By Corollary \ref{flatstripcontrol}, 
applied to $\tau,\eta,\tau^\prime,\eta^{\prime\prime}$, the
balanced splitting paths
\[\gamma(\tau,\eta),\gamma(\tau^\prime,\eta^{\prime\prime})\]
are weight-$L_3$-fellow travellers.
Since $\tau^\prime$ can be connected to $\sigma^\prime$
by a stretched-out splitting and shifting sequence
of length at most $\kappa_1$ and since
$\sigma^\prime,\tau^\prime$ are both splittable to 
$\eta^{\prime\prime}$, Step 1 above shows that
the balanced splitting paths
\[\gamma(\tau^\prime,\eta^{\prime\prime}),
\gamma(\sigma^\prime,\eta^{\prime\prime})\]
are uniform fellow travellers. Another application
of Step 1 yields that the balanced splitting
paths \[\gamma(\sigma^\prime,\eta^{\prime\prime}),
\gamma(\sigma,\eta^{\prime\prime})\]
are uniform fellow travellers as well.
Finally Corollary \ref{flatstripcontrol} shows that
the balanced splitting paths
\[\gamma(\sigma,\eta^{\prime\prime}),\gamma(\sigma,\eta)\]
are uniform fellow travellers. 
Together the lemma follows.
\end{proof}

We use Lemma \ref{shift} to show

\begin{corollary}\label{near}
For every $R>0$ there is a number $L_6=L_6(R)>0$ 
with the following property. 
Let $\tau,\eta\in {\cal V}({\cal T\cal T})$ be
splittable to $\tau^\prime,\eta^\prime$.
Assume that $\tau^\prime,\eta^\prime$ carry a common
complete geodesic lamination $\lambda$
and that $d(\tau,\eta)\leq R,d(\tau^\prime,\eta^\prime)\leq R$.
Then 
\[d(\gamma(\tau,\tau^\prime)(i),
\gamma(\eta,\eta^\prime)(i))\leq L_6\text{ for all }i.\]
\end{corollary}
\begin{proof}
Let $\tau,\tau^\prime,\eta,\eta^\prime$ be as in the lemma.
By Lemma \ref{splittingcontrol}, $\tau^\prime,\eta^\prime$
are splittable to the same complete train track $\zeta$ with a
splitting sequence of length at most $n(R)$.
Let $\mu$ be a complete geodesic lamination carried
by $\zeta$.

By Lemma 6.7 of \cite{H09}, there is a complete
train track $\beta_0$ which is carried by
both $\tau,\eta$, which carries $\mu$ and such that
\[d(\tau,\beta_0)\leq p_0(R),d(\eta,\beta_0)\leq p_0(R)\]
where $p_0(R)>0$ only depends on $R$.
Another application
of Lemma 6.7 of \cite{H09} and of Lemma \ref{splittingcontrol}
shows that $\beta_0$ is splittable to a 
train track $\zeta^\prime$
which can be obtained from $\zeta$ by a splitting sequence
of length bounded from above by a number $p_1(R)>0$ 
only depending on $R$ (compare the discussion in Step 2 of the
proof of Lemma \ref{shift}).
Then Lemma \ref{shift} yields that
\[d(\gamma(\tau,\zeta^\prime)(i),
\gamma(\beta_0,\zeta^\prime)(i))\leq p_2(R),
d(\gamma(\eta,\zeta^\prime)(i),\gamma(\beta_0,
\zeta^\prime)(i))\leq p_2(R)\]
for all $i$ where $p_2(R)>0$ only depends on $R$.

On the other hand, since the length of a splitting
sequence connecting $\tau^\prime,\eta^\prime$ to 
$\zeta^\prime$ is bounded form above by a number
only depending on $R$,
Lemma \ref{flatconfig} shows that there is a number
$p_3(R)>0$ such that 
\[d(\gamma(\tau,\tau^\prime)(i),\gamma(\tau,\zeta^\prime)(i))\leq p_3(R)\]
and that 
\[d(\gamma(\eta,\eta^\prime)(i),\gamma(\eta,\zeta^\prime)(i))\leq p_3(R)\]
for all $i$.
Together this shows the corollary.
\end{proof}

For any recurrent train track $\sigma$
which is splittable to a recurrent train track $\sigma^\prime$
there is a unique balanced splitting path
$\gamma_0(\sigma,\sigma^\prime)$ connecting
$\sigma$ to $\sigma^\prime$. Let $\tau$
be a complete extension of $\sigma$ and assume that
the complete train track $\tau^\prime$ is the endpoint
of a sequence issuing from 
$\tau$ which is induced by a splitting
sequence connecting $\sigma$ to $\sigma^\prime$ as in 
Proposition \ref{inducing}. For each $i$ there is 
a train track $\tau(i)\in E(\tau,\tau^\prime)$ which
contains the train track $\gamma_0(\sigma,\sigma^\prime)(i)$
as a subtrack. An example is the endpoint of a sequence
issuing from $\tau$ which is induced by 
a splitting sequence connecting $\sigma$ to 
$\gamma_0(\sigma,\sigma^\prime)(i)$.
Such a train track is not unique, but
by Lemma \ref{tightcontrol}, the diameter in ${\cal T\cal T}$
of the set of all train tracks with this property is
uniformly bounded (by $q^3$). 

Our next goal is to estimate the
distance between the train tracks $\tau(i)$ and 
$\gamma(\tau,\tau^\prime)(i)$.

\begin{lemma}\label{doubleinduced}
There is a number $L_7>0$ with the following property.
Let $\sigma$ be a recurrent 
train track which is splittable to the recurrent train
track $\sigma^\prime$. Let
$\tau$ be a complete extension of $\sigma$ and let
$\tau^\prime$ be the endpoint of a sequence issuing from 
$\tau$ which is induced
by a splitting sequence connecting $\sigma$ to $\sigma^\prime$.
For each $i$ let $\tau(i)\in E(\tau,\tau^\prime)$ be a train 
track which contains $\gamma_0(\sigma,\sigma^\prime)(i)$ as a
subtrack. Then 
\[d(\gamma(\tau,\tau^\prime)(i),\tau(i))\leq L_7 \text{ for all }i.\]
\end{lemma}
\begin{proof}
We divide the proof of the lemma into four steps.

{\sl Step 1:} 

Let $\sigma$ be a recurrent train track which is
splittable to a recurrent train track $\sigma^\prime$.
Then $\sigma$ is connected to $\sigma^\prime$ by a balanced
splitting path $\gamma_0(\sigma,\sigma^\prime)$.
Let $\tau$ be a complete extension of $\sigma$ and assume that
there is a complete train track $\tau^\prime$ which can
be obtained from $\tau$ by a sequence induced by 
a splitting sequence connecting $\sigma$ to $\sigma^\prime$.
Let $\gamma(\tau,\tau^\prime)$ be the balanced splitting
path connecting $\tau$ to $\tau^\prime$. By the forth property
in Proposition \ref{inducing}, 
for each $i$ the train track $\gamma(\tau,\tau^\prime)(i)$
contains a subtrack $\sigma_i\in E(\sigma,\sigma^\prime)$.
We claim that
$\sigma_i$ is splittable to $\gamma_0(\sigma,\sigma^\prime)(i)$.

For this we proceed by induction on the length of 
the balanced splitting path $\gamma(\tau,\tau^\prime)$.
If this length vanishes then there is nothing to show, so
assume the claim holds true whenever the length of this 
path is at most $k-1$ for some $k\geq 1$.

Let $\sigma,\sigma^\prime,\tau,\tau^\prime$ be such that
the length of the balanced splitting
path $\gamma(\tau,\tau^\prime)$ equals $k$.
The splittable $\tau^\prime$-configurations
in $\tau$ are 
pairwise disjoint trainpaths or circles
with images in $\sigma$.
Let $\rho:[0,m]\to \tau$ be such a 
splittable $\tau^\prime$-configuration
and let $\xi$ be the train track 
obtained from $\tau$ by a
$\rho-\tau^\prime$-multi-split.
By construction, $\xi$ contains
a subtrack $\zeta$ which can be obtained
from $\sigma$ by a splitting sequence at branches
contained in $\rho[0,m]$. Since splits at large branches
of $\sigma$ contained in the distinct
$\tau^\prime$-configurations of $\tau$ commute, 
it now suffices to show that
$\zeta$ is splittable
to $\gamma_0(\sigma,\sigma^\prime)(1)$.

Consider first the case that  
$\rho$ is a reduced trainpath in $\tau$. Then  
every two-sided large trainpath
$\rho_0:[0,n]\to \sigma$ with
$\rho_0[0,n]\subset \rho[0,m]$ is reduced. There is
a maximal two-sided large trainpath $\rho_0$ in $\sigma$
which contains every two-sided large trainpath 
in $\sigma$ with image in $\rho[0,m]$.
If $\{\xi_j\}$ is a sequence of
$\rho$-splits transforming $\tau$ to 
$\xi$, then for each $j$ there is a subtrack
$\zeta_j\in E(\sigma,\sigma^\prime)$ 
of $\xi_j$ so that either 
$\zeta_{j}=\zeta_{j-1}$ or that $\zeta_j$ can 
be obtained from $\zeta_{j-1}$ by 
a $\rho_0$-split. As a consequence,
$\zeta$
is indeed splittable to $\gamma_0(\sigma,\sigma^\prime)(1)$.

Now let $\rho[0,m]=c$ be a splittable
reduced circle in $\tau$. Since
$c$ is splittable, by the definition of an induced sequence
the image of $c$ in $\sigma$ contains
a large branch in $\sigma$ and hence $c$ is a reduced circle in 
$\sigma$. Let as before $\theta_c$ be the positive Dehn twist
about $c$.
Then $\theta_c^{\pm}\tau$
contains $\theta^{\pm}_c\sigma$ as a subtrack.
By definition of a balanced sequence, any train track
which can be obtained from $\sigma$
by a sequence of $c$-splits
and which is splittable to both $\sigma^\prime$ and 
$\theta_c^{\pm}(\sigma)$ 
is also splittable to $\gamma_0(\sigma,\sigma^\prime)(1)$.
On the other hand, by the definition of a 
$\rho-\tau^\prime$-multi-split, the train track $\xi$
is splittable to  
$\theta_c^{\pm}(\tau)$.
Together this just means once again that $\zeta$ 
is splittable
to $\gamma_0(\sigma,\sigma^\prime)(1)$. 

As a consequence, the subtrack $\sigma_1$ of 
$\gamma(\tau,\tau^\prime)(1)$ is splittable to 
$\gamma_0(\sigma,\sigma^\prime)(1)$.
An inductive application
of Lemma \ref{basicbalanced} shows that for each $i\geq 1$ the train track
$\gamma_0(\sigma_1,\sigma^\prime)(i-1)$ is splittable 
to $\gamma_0(\sigma,\sigma^\prime)(i)$. The induction
hypothesis, applied to $\gamma(\tau,\tau^\prime)(1),\tau^\prime,
\sigma_1,\sigma^\prime$,  
then implies that for each $i$ the
train track $\sigma_i$ is splittable to $\gamma_0(\sigma,\sigma^\prime)(i)$.
This completes
the proof of the above claim.

{\sl Step 2:} 

Let $\rho_0:[0,n]\to \sigma$ be an embedded 
reduced trainpath.
Then $\rho_0$ defines an embedded two-sided large
trainpath
$\rho:[0,m]\to \tau$ with $\rho[0,m]=\rho_0[0,n]$ as sets.  
The trainpath $\rho$ may not be reduced. 
Similarly, a reduced circle in $\sigma$
defines an embedded circle in $\tau$ which 
may not be reduced.

Assume however for the moment  
that each $\sigma^\prime$-configuration
in $\sigma$ determines in this way a reduced trainpath or a 
reduced circle in $\tau$. We claim that then the 
train track $\gamma(\tau,\tau^\prime)(1)$ 
contains the
train track $\gamma_0(\sigma,\sigma^\prime)(1)$
as a subtrack.

To see that this is indeed the case let 
$\{\sigma_i\}_{0\leq i\leq j}$ be
a splitting sequence connecting 
$\sigma_0=\sigma$ to $\sigma_j=\gamma_0(\sigma,\sigma^\prime)(1)$.
Let $\tau_1$ be the endpoint of 
a sequence issuing from $\tau$ which is
induced by the
sequence $\{\sigma_i\}_{0\leq i\leq j}$ as in 
Proposition \ref{inducing}. 
If $\rho_i$ $(i=1,\dots,k)$ are the trainpaths or circles
in $\tau$ whose images are the 
$\sigma^\prime$-configurations of $\sigma$ then
it follows from the definition of an induced sequence
that $\tau_1$ can be obtained from
$\tau$ by a successive transformation
with $\rho_i$-splits. The splitting sequence connecting
$\tau$ to $\tau_1$ includes a split at each branch 
of $\tau$ contained in $\rho_i$.
Since the trainpaths or circles $\rho_i$ are 
reduced by assumption, we conclude that 
for each $i$, $\rho_i$ is contained in a 
$\tau^\prime$-configuration of $\tau$ and 
$\tau_1$ is splittable to $\gamma(\tau,\tau^\prime)(1)$.

By the fourth part of Proposition \ref{inducing},
the train track $\gamma(\tau,\tau^\prime)(1)$ contains
a subtrack $\hat\sigma$ which can be obtained from 
$\sigma$ by a splitting sequence and which 
is splittable to $\sigma^\prime$. By Step 1 above,
$\hat \sigma$ is splittable to 
$\gamma_0(\sigma,\sigma^\prime)(1)$.
On the other hand, the discussion in the previous
paragraph shows that $\gamma_0(\sigma,\sigma^\prime)(1)$
is splittable to $\hat\sigma$. This shows that
$\gamma(\tau,\tau^\prime)(1)$ contains
$\gamma_0(\sigma,\sigma^\prime)(1)$ as a subtrack as claimed.

{\sl Step 3:} 

The discussion in Step 2 above shows the following.
If for every $\sigma^\prime$-configuration
$\rho_0$ in $\sigma$ the trainpath (or circle) defined
by $\rho_0$ in $\tau$ is reduced, then
$\gamma(\tau,\tau^\prime)(1)$ contains $\gamma_0(\sigma,\sigma^\prime)(1)$
as a subtrack. As a consequence, by induction
the lemma holds true for 
the balanced splitting path $\gamma(\tau,\tau^\prime)$
(with a universal constant $L=L_7>0$) 
if for every $i\geq 0$ the following
is satisfied. Let $\sigma_i$ be the subtrack of  
the train track $\gamma(\tau,\tau^\prime)(i)$ which can
be obtained from $\sigma$ by a splitting sequence
and which is splittable to $\sigma^\prime$. 
Then every $\sigma^\prime$-configuration in 
$\sigma_i$ defines a reduced  
trainpath or a reduced 
circle in $\gamma(\tau,\tau^\prime)(i)$.

The goal of this step is to determine 
a set of complete extensions
of $\sigma$ with this property.
For this
define a branch $a$ of $\sigma$ to be 
\emph{$\sigma^\prime$-split} if a splitting
sequence connecting $\sigma$ to $\sigma^\prime$
includes a split at $a$ (under the natural
identification of the branches of train tracks in a
splitting sequence, see Lemma 5.1 of \cite{H09}). Call a 
complete extension $\tau$ of $\sigma$
\emph{simple for $\sigma^\prime$} if the following holds true.
\begin{enumerate}
\item Let $b$ be a half-branch of 
$\tau$ which is contained in 
a complementary region $C$ of $\sigma$ and which is incident on 
a switch $v$ contained in 
a non-smooth side $\xi$ of $C$.
There is a unique trainpath $\zeta:[0,\ell]\to \tau$
such that $\zeta[1/2,1]=b$, $\zeta[1,\ell]\subset
\xi$ and that $\zeta(\ell)$ is a cusp of $C$.  
If the branch $a$ of $\sigma$ containing the switch $v$ in its
interior is $\sigma^\prime$-split then
we require that $a$ is incident on the switch $\zeta(\ell)$ of $\sigma$,
i.e. $\zeta[1,\ell]$ is contained in 
a single branch of $\sigma$. Moreover we require that 
$\zeta[1,\ell]$ consists of mixed branches. The neighbors
of $\zeta[1,\ell]$ are all contained in $C$.
\item Let $\xi:[0,n]\to \sigma$ be a smooth side
of a complementary region $C$ of $\sigma$. We require that
there is no half-branch of $\tau$ contained in $C$ which is
incident on a switch contained in the interior of 
any $\sigma^\prime$-split
branch in $\xi[0,n]$.
\end{enumerate}

In the sequel we call a complete extension
of $\sigma$ which is simple for $\sigma^\prime$
a simple complete extension of $\sigma$ 
whenever no confusion is possible.
We claim that a simple complete extension $\tau$ 
of $\sigma$ has the property
described in the first paragraph of this step.

For this we first show that every $\sigma^\prime$-configuration 
$\rho_0:[0,n]\to \sigma$ of 
$\sigma$ defines a reduced trainpath or a reduced circle in 
a simple complete extension $\tau$ of $\sigma$.
Namely, by definition of a $\sigma^\prime$-configuration,
every branch of $\sigma$ contained in 
$\rho_0[0,n]$ is $\sigma^\prime$-split. 
Let $b$ be a half-branch of $\tau$ contained in 
a complementary region $C$ of $\sigma$ which 
is incident on a switch $v\in \rho_0[0,n]$. By the second
part in the definition of a simple complete extension of
$\sigma$, the switch $v$ of $\tau$ is contained 
in a non-smooth side $\xi$ of $C$. 
Let $\zeta:[0,\ell]\to \tau$ be the trainpath with 
$\zeta[1/2,1]=b$ and $\zeta[1,\ell]\subset \xi$ which terminates
at a cusp $\zeta(\ell)=v_0$ of $C$. 
By the first part in the definition of a simple
complete extension of $\sigma$, $\zeta[1,\ell]$ is contained in 
a single branch $a$ of $\sigma$.
The branch $a$ is small at $v_0$.
Since $\rho_0$ is two-sided large and contains $a$,
the switch $v_0$ is contained in the interior of 
$\rho_0[0,n]$. Now by assumption, the subarc of $a$ 
connecting $v$ to $v_0$ consists of mixed branches of $\tau$, and
the neighbor of $\rho_0$ at 
$v_0$ is a half-branch of $\sigma$ contained in the boundary of 
$C$ which is of the same type (incoming or outgoing and
left or right) as $b$ along $\rho_0$.

As a consequence, if $b\subset \tau-\sigma$ is any neighbor of
the embedded arc $\rho_0[0,n]$ in $\tau$ then 
there is a neighbor of $\rho_0[0,n]$ in $\sigma$ of the same
type as $b$. Since $\rho_0$ is
reduced by assumption, this implies that the trainpath
$\rho$ in $\tau$ which coincides with $\rho_0$ as a set 
is reduced as well.
In other words, if $\tau$ is a simple complete extension
of $\sigma$ then 
every $\sigma^\prime$-configuration in $\sigma$
defines a reduced trainpath or a reduced circle in $\tau$.

Step 2 implies that 
$\tau_1=\gamma(\tau,\tau^\prime)(1)$ contains the train track
$\sigma_1=\gamma_0(\sigma,\sigma^\prime)(1)$ as an embedded subtrack.
We claim that $\tau_1$ is a simple complete extension of 
$\sigma_1$.

To see that this is the case,
let again $\rho_0$ be a $\sigma^\prime$-configuration
in $\sigma$ and let $e$ be a large branch of $\sigma$
contained in 
$\rho_0$. Then $e$ is $\sigma^\prime$-split
and hence since $\tau$ is a simple complete extension of $\sigma$, 
the train track $\tau$ is tight at $e$.
Let $\tilde \sigma$ be the train track
obtained from $\sigma$ by a $\sigma^\prime$-split at
$e$ and let $\tilde\tau$ be the complete extension
of $\tilde \sigma$ obtained from 
$\tau$ by a single split at $e$. Then 
up to isotopy, for any small neighborhood $U$ in $S$ 
of the branch $e$
the intersection of $\sigma,\tau$ with $S-U$ coincides with the
intersection of $\tilde \sigma,\tilde \tau$ with $S-U$. 
Therefore if there is a neighbor $\tilde b$ 
of $\tilde \sigma$ in $\tilde \tau$
which does not meet the
requirements in the definition of an extension
of $\tilde\sigma$ which is simple for $\sigma^\prime$ 
then $\tilde b=\phi(\tau,\tilde \tau)(b)$
where $b$ is incident on a switch contained in the interior of
a $\sigma^\prime$-split branch of $\sigma$ which is
a winner of the split connecting $\sigma$ to $\tilde \sigma$.

Now let the branch $a$ of $\sigma$ 
be a winner of the split connecting 
$\sigma$ to $\tilde \sigma$. Assume that 
$a$ is $\sigma^\prime$-split. Let $C$ be the
complementary region of $\sigma$ which 
has a cusp at an endpoint $v_0$ of $e$ and which
contains $a$ in one of its sides, say in the side $\xi$.
The complementary region
$\tilde C$ of the train track $\tilde \sigma$
which contains up to isotopy 
the complement in $C$ of a
small neighborhood of the branch $e$
has a side containing $\phi(\sigma,\tilde \sigma)(e\cup \xi)$
(see Figure G).
\begin{figure}[ht]
\begin{center}
\psfrag{Figure}{Figure G} 
\includegraphics 
[width=0.7\textwidth] 
{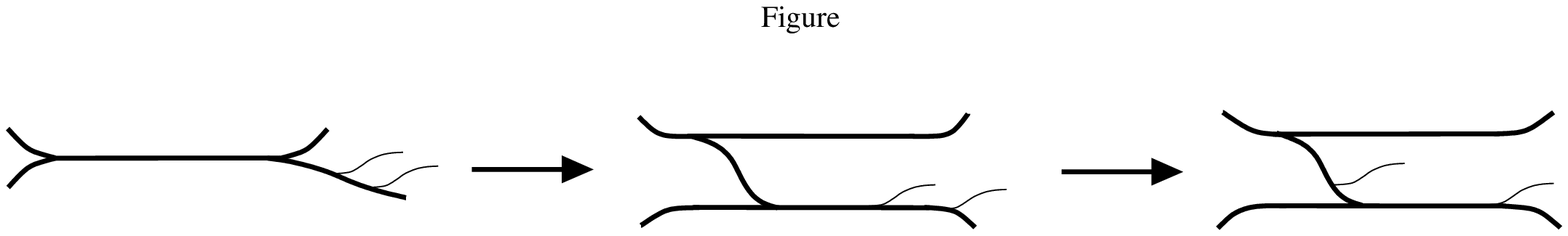}
\end{center}
\end{figure}

Let $\zeta:[0,\ell]\to a$ be the parametrization of the branch
$a$ as a trainpath on $\tau$ with $\zeta(0)=v_0$.
For $0<i<\ell$ 
denote by $b_i$ the neighbor of $a$ in $\tau$ incident on the
switch $\zeta(i)$.
Since $a$ is a winner of the split transforming $\sigma$ to 
$\tilde\sigma$, the branch $\phi(\tau,\tilde \tau)(\zeta[0,1])\subset 
\tilde \tau$ is large (and of type 3).
Since $a$ is $\sigma^\prime$-split by assumption, 
the train track $\alpha_1$ obtained from $\tilde \tau$ by the
(unique) $\tilde\sigma$-split at the large branch
$\phi(\tau,\tilde \tau)(\zeta[0,1])$ is splittable
to $\tau^\prime$. The train track $\alpha_1$ contains
$\tilde \sigma$ as a subtrack, and 
$\phi(\tau,\alpha_1)(b_1)$ 
is a neighbor of $\tilde \sigma$
in $\alpha_1$ which satisfes the first requirement
in the definition of a simple complete 
extension of $\tilde \sigma$ (see Figure G for an
illustration of this fact). Repeat this construction with the large
branch $\phi(\tau,\alpha_1)(\zeta[1,2])$ and the neighbor
$\phi(\tau,\alpha_1)(b_2)$ of $\tilde \sigma$.
After a uniformly
bounded number of such steps we obtain a train track $\alpha_2$
which is splittable to $\tau^\prime$ and such that for 
every neighbor $b$ of $a$ in $\tau$, the half-branch
$\phi(\tau,\alpha_2)(b)$ is incident on the 
branch $\phi(\sigma,\tilde \sigma)(e)$.

A second application of this
procedure to the second winner of the split connecting
$\sigma$ to $\tilde \sigma$ 
shows that there is a train track 
$\alpha_3\in E(\tau,\gamma(\tau,\tau^\prime)(1))$ which is 
a simple complete extension of $\tilde \sigma$.
By induction on the length of a splitting sequence connecting
$\sigma$ to $\gamma_0(\sigma,\sigma^\prime)(1)$ we conclude that
there is a 
train track $\alpha_4\in E(\tau,\gamma(\tau,\tau^\prime)(1))$ 
which is a simple
complete extension of $\gamma_0(\sigma,\sigma^\prime)(1)$. 
This means in particular that every large branch of $\alpha_4$ contained
in a $\sigma^\prime$-split branch of $\gamma_0(\sigma,\sigma^\prime)(1)$
is a large branch of $\alpha_4$. Then necessarily
$\alpha_4=\gamma(\tau,\tau^\prime)(1)$ (see Step 1 above). 

A successive application
of this argument shows the following.
If $\tau$ is a simple complete extension of 
$\sigma$ then for every $i$ the train track $\gamma(\tau,\tau^\prime)(i)$
contains $\gamma_0(\sigma,\sigma^\prime)(i)$ as a subtrack.
In particular, the balanced splitting path
$\gamma(\tau,\tau^\prime)$ satisfies
\begin{equation}\label{simpletry}
d(\gamma(\tau,\tau^\prime)(i),\tau(i))\leq L\end{equation}
with a universal constant $L>0$.

{\sl Step 4:} 

In this final step we use steps 1)-3) to complete the
proof of the lemma.

Let $\tau$ be any complete extension of a recurrent train track
$\sigma$.
Let $\{\sigma_i\}_{0\leq i\leq u}$
be a recurrent splitting sequence issuing from 
$\sigma=\sigma_0$ and let $\tau^\prime$ be a train track
obtained from $\tau$ by a sequence induced by the splitting
sequence $\{\sigma_i\}_{0\leq i\leq u}$. Then $\tau^\prime$
contains $\sigma_u$ as a subtrack. Let 
$\lambda$ be the complete $\tau^\prime$-extension of a 
$\sigma_u$-filling measured geodesic lamination.
The goal of this step is to show that
$\tau$ can be modified 
to a complete train track $\eta$ 
with the following properties.
\begin{enumerate}
\item $\eta$ carries $\lambda$.
\item $\eta$ contains a train track $\tilde\sigma\in
E(\sigma,\sigma^\prime)$ as a subtrack which can
be obtained from $\sigma$ by a splitting sequence of 
uniformly bounded length.
\item $\eta$ can be obtained from $\tau$ by a splitting
and shifting sequence of uniformly bounded length.
\item $\eta$ is a complete extension of $\tilde \sigma$ which 
is simple for $\sigma^\prime$.
\end{enumerate}

Before we show how to construct such a train track $\eta$, 
we assume that such a train track $\eta$ exists and use it 
to complete the proof of the lemma.
Namely, assume a train track 
$\eta$ with properties 1)-4) above is given
which contains $\tilde \sigma\in E(\sigma,\sigma^\prime)$ 
as a subtrack.
Let $\tilde \tau\in E(\tau,\tau^\prime)$ be obtained from
$\tau$ by a sequence induced by a splitting sequence 
connecting $\sigma$ to $\tilde\sigma$. Since the length
of a splitting sequence connecting $\sigma$ to $\tilde \sigma$
is uniformly bounded, the distance between
$\tau$ and $\tilde \tau$ is uniformly bounded and hence
by the third property above, 
the distance between $\tilde \tau$ and $\eta$ is uniformly bounded
as well. 

Since $\eta$ carries the
complete extension $\lambda$ of a 
$\sigma^\prime$-filling measured geodesic lamination,
Proposition \ref{inducing} shows that 
there is a sequence issuing from $\eta$ which is induced
by a splitting sequence connecting $\tilde \sigma$ to
$\sigma^\prime$ and which consists of train tracks 
carrying $\lambda$. Let $\eta^\prime$ be the endpoint of this
sequence. Corollary \ref{inducing10} shows that
the distance between $\tau^\prime$ and $\eta^\prime$ is uniformly
bounded. If for $i\geq 0$ 
we denote by $\eta(i)\in E(\eta,\eta^\prime)$ a train track
which contains $\gamma_0(\tilde \sigma,\sigma^\prime)(i)$ as a
subtrack then by the fourth property above, 
inequality (\ref{simpletry}) can be applied and shows that
\[d(\gamma(\eta,\eta^\prime)(i),\eta(i))\leq L.\]

Apply Lemma \ref{flatconfig2} to the balanced
splitting paths $\gamma_0(\sigma,\sigma^\prime),
\gamma_0(\tilde \sigma,\sigma^\prime)$ connecting
$\sigma,\tilde \sigma$ to $\sigma^\prime$. We conclude that 
for each $i$, the train track 
$\gamma_0(\sigma,\sigma^\prime)(i)$ is splittable to 
$\gamma_0(\tilde \sigma,\sigma^\prime)(i)$ with a splitting
sequence of uniformly bounded length.
Corollary \ref{inducing10} then shows that the distance
between $\eta(i)$ and a train track $\tau(i)\in E(\tau,\tau^\prime)$ 
which contains $\gamma_0(\sigma,\sigma^\prime)(i)$ as a subtrack
is bounded from
above by a universal constant.
Since by Corollary \ref{near}
the balanced splitting paths
$\gamma(\tau,\tau^\prime),\gamma(\eta,\eta^\prime)$ are uniform
fellow travellers, the distance between
$\gamma(\tau,\tau^\prime)(i)$ and $\tau(i)$ is uniformly bounded
as claimed in the lemma.

To construct a modification of $\tau$ to a complete
train track with the properties described in the second
paragraph of this step,
observe first that by 
the discussion in the proof of Proposition
\ref{inducing}, there is a train track
$\tilde \sigma\in E(\sigma,\sigma^\prime)$ which can
be obtained from $\sigma$ by a splitting sequence of
uniformly bounded length and such that a splitting
sequence connecting $\tilde \sigma$ to $\sigma^\prime$
does not include a split at any large branch 
which is contained
in a smooth boundary component of a complementary region of 
$\tilde \sigma$. By definition of an induced
sequence, there is a train track
$\tilde\tau\in E(\tau,\tau^\prime)$ which can
be obtained from $\tau$ by a splitting sequence
of uniformly bounded length and which
contains $\tilde \sigma$ as a subtrack.

Let $\xi$ be a non-smooth side of a complementary region 
$C$ of $\tilde \sigma$ and let $b$ be a half-branch 
of $\tilde \tau$ which is 
contained in $C$ and is incident on a switch $v\in \xi$.
Let $\zeta:[0,\ell]\to \xi$ be the unique trainpath
with $\zeta[1/2,1]=b$, 
$\zeta[1,\ell]\subset \xi$ and such that
$\zeta(\ell)$ is a cusp of $C$. Note that 
$\zeta[1,3/2]$ is  large half-branch. In particular, 
$\zeta[1,2]$ is the initial branch of a one-way
trainpath as defined on p.127 of \cite{PH92}. Thus 
if $\zeta[1,\ell]$ does not
contain any large branch then all 
switches of $\tau$ along
$\zeta$ are incoming, and all branches $\zeta[i,i+1]$
$(1\leq i\leq \ell-1)$ are mixed.
Assume that $b$ lies to the left (or right) of the branch
$a$ of $\xi\subset \sigma$ containing $v$ in its interior with respect
to the orientation of $a$ induced by the oriented trainpath $\zeta$.
Then all left (or right) neighbors of 
$\zeta[1,\ell]$ can be shifted forward along $\xi$ past
every right (or left) neighbor. In the resulting 
train track $\tau_1$, these shifted neighbors satisfy the
requirements for a simple extension of $\tilde \sigma$.
The train track $\tau_1$ is shift
equivalent to $\tilde\tau$ and hence it carries $\lambda$.

If $\zeta[0,\ell]$ contains a large branch then 
there is a sequence of shifts and $\sigma$-splits 
of uniformly bounded length which transforms
$\tau$ to a train track $\tau_2$ which carries $\lambda$, 
which contains $\sigma$ as a subtrack and which is
of the required form (compare the argument in the
proof of Lemma \ref{tightcontrol}).
Together this completes the 
proof of the lemma.
\end{proof}

As an immediate corollary we obtain

\begin{corollary}\label{induced}
For every $R>0$ 
there is a number $L_8=L_8(R)>0$ with the following
property. Let $\sigma_0$ be a recurrent train track 
and let 
$\{\sigma_i\}$ be a recurrent splitting
sequence issuing from $\sigma_0$. 
Let $\tau,\eta$ be complete extensions of 
$\sigma_0$ with $d(\tau,\eta)\leq R$ 
and let $\tau^\prime,\eta^\prime$ be 
the endpoints of a sequence induced
by the splitting sequence $\{\sigma_i\}$ and issuing from
$\tau,\eta$. Then the balanced splitting paths
$\gamma(\tau,\tau^\prime),\gamma(\eta,\eta^\prime)$ 
are weight-$L_8$ fellow
travellers.
\end{corollary}
\begin{proof}
Let $\sigma$ be a recurrent train track which 
is splittable to a recurrent train track $\sigma^\prime$
and let $\gamma_0(\sigma,\sigma^\prime)$ be the balanced
splitting path connecting 
$\sigma$ to $\sigma^\prime$.

Let $R>0$ and  let $\tau,\eta$ 
be complete extensions of $\sigma$ with $d(\tau,\eta)\leq R$. Let
$\tau^\prime,\eta^\prime$ be a train track 
obtained from $\tau,\eta$ by 
a sequence induced by a splitting sequence
connecting $\sigma$ to $\sigma^\prime$ as in 
Proposition \ref{inducing}. 
Then for every $i\geq 0$ there are
train tracks $\beta(\tau,\tau^\prime)(i)\in E(\tau,\tau^\prime),
\beta(\eta,\eta^\prime)(i)\in E(\eta,\eta^\prime)$ 
which are obtained from
$\tau,\eta$ by a sequence
induced by a splitting sequence
connecting $\sigma$ to 
$\gamma_0(\sigma,\sigma^\prime)(i)$. 
The train tracks $\beta(\tau,\tau^\prime)(i),
\beta(\eta,\eta^\prime)(i)$ contain
$\gamma_0(\sigma,\sigma^\prime)(i)$ as a subtrack.
By Corollary \ref{inducing10}, 
there is a number $p=p(R)>0$ such that
\begin{equation}
d(\beta(\tau,\tau^\prime)(i),\beta(\eta,\eta^\prime)(i))
\leq p(R)\text{ for all }i.\notag
\end{equation}

On the other hand, by Lemma \ref{doubleinduced} 
there is a number
$L_7>0$ such that
\begin{equation}\label{closetobeta}
d(\beta(\tau,\tau^\prime)(i),
\gamma(\tau,\tau^\prime)(i))\leq L_7,
d(\beta(\eta,\eta^\prime)(i),\gamma(\tau,\tau^\prime)(i))\leq L_7
\text{ for all }i.\notag\end{equation}
Together this shows the corollary.
\end{proof}

By the discussion in Section 5, for a complete
train track $\tau$ which contains an embedded simple closed
curve $c$, splitting $\tau$ with a sequence of $c$-splits results
in twisting $\tau$ along $c$. On the other hand, 
$c$-splits of $\tau$ commute with splits at large branches
of $\tau$ which are disjoint from $c$. Thus if 
$\tau$ is splittable to a 
complete train track $\tau^\prime$ which contains
$c$ as an embedded simple closed curve, then 
there is a splitting sequence $\{\tau_i\}$ connecting
$\tau$ to $\tau^\prime$ and
there is a number $\ell \geq 0$ with the following property. 
The train track $\tau_\ell$ can
be obtained from $\tau$ by a splitting sequence at large
branches disjoint from $c$, and $\tau^\prime$ can be 
obtained from $\tau_\ell$ by a sequence of $c$-splits.
We call $\tau_\ell$ the \emph{$(S-c)$-contraction} of the 
pair $(\tau,\tau^\prime)$. Let $\theta_c$ be the positive
Dehn twist about $c$. By 
Corollary \ref{Dehn}, there is a number $p\in \mathbb{Z}$ such that
\[d(\tau^\prime,\theta_c^p(\tau_\ell))\leq a_3.\]
We call $p$ a \emph{$c$-twist parameter} of the pair
$(\tau,\tau^\prime)$.

The following observation is
immediate from the definitions, from 
Lemma \ref{dehnreduced} and 
Corollary \ref{Dehn} and from Corollary \ref{flatstripcontrol}.

\begin{lemma}\label{sccontraction}
There is a number $L_9>0$ with the following property.
Let $\tau\in {\cal V}({\cal T\cal T})$ 
be splittable to $\tau^\prime\in {\cal V}({\cal T\cal T})$ 
and let $c$ be a
simple closed curve which is embedded in both $\tau,\tau^\prime$.
Let $\eta\in {\cal V}({\cal T\cal T})$ 
be the $(S-c)$-contraction of the pair
$(\tau,\tau^\prime)$ and let 
$p\in \mathbb{Z}$ be the $c$-twist parameter of 
$(\tau,\tau^\prime)$. Then for every $i\leq \vert p\vert$ we have
\[d(\gamma(\tau,\tau^\prime)(i),\theta_c^{({\rm sgn}\,i)i}
\gamma(\tau,\eta)(i))\leq L_9,\]
and $d(\gamma(\tau,\tau^\prime)(i),
\theta_c^p\gamma(\tau,\eta)(i))\leq L_9$ for $i\geq \vert p\vert$.
\end{lemma}
\begin{proof} Let $\tau,\tau^\prime,c$ be as in the lemma.
Choose a complete $\tau^\prime$-extension 
$\lambda$ of 
a $c$-filling 
measured geodesic lamination.
By Lemma \ref{nonreduced}, there is a complete
train track 
$\tau_1\in E(\tau,\lambda)$ 
which can be obtained from $\tau$ 
by a sequence of $c$-splits of
length at most $a_2$
and which contains $c$ as a reduced circle. 
Let $\eta$ be the $(S-c)$-contraction of the pair
$(\tau,\tau^\prime)$. Then $\eta$ is splittable 
with a sequence of $c$-splits of length at most $a_2$ to
a train track 
$\tau_2\in E(\tau_1,\lambda)$. 
By Lemma \ref{dehnreduced} and 
Corollary \ref{Dehn} and by the choice
of $\lambda$, we may assume that
$\theta_c^{p}\tau_2\in E(\tau_1,\lambda)$ and 
that
$\tau_2$ is splittable to $\theta_c^{p}\tau_2$.
By the definition of a move, for each $i\leq \vert p\vert$ 
we have 
\[\gamma(\tau_1,\theta_c^p\tau_2)(i)=
\theta_c^{({\rm sgn}\,i)i}\gamma(\tau_1,\tau_2)(i).\]
The lemma is now immediate from this observation and from
Corollary \ref{flatstripcontrol}.
\end{proof}

We use  Corollary \ref{flatstripcontrol} and 
Lemma \ref{sccontraction} to show

\begin{lemma}\label{onesidedcomb}
There is a number $L_{10}>0$ with the following properties.
Let $F$ be a marking for 
$S$ and let $\tau,\eta$ be complete train tracks
in standard form for $F$. Assume that
$\tau,\eta$ are splittable to 
$\tau^\prime,\eta^\prime\in {\cal V}({\cal T\cal T})$. 
Then the balanced splitting paths
$\gamma(\tau,\tau^\prime),\gamma(\eta,\eta^\prime)$ are
weight-$L_{10}$ fellow travellers.
\end{lemma}
\begin{proof} 
Let $\tau,\eta$ be train tracks in standard form for a marking
$F$ which are splittable to train tracks 
$\tau^\prime,\eta^\prime$. By Proposition 
\ref{shortestdistance} and by Theorem \ref{cubicaleuclid} and the
following remark, 
there is a number $a>0$ such that
\[d(\tau^\prime,\eta^\prime) \geq 
a(d(\tau^\prime,\Pi_{E(\tau,\tau^\prime)}(\eta^\prime))+
d(\Pi_{E(\eta,\eta^\prime)}(\tau^\prime),\eta^\prime))-1/a.\]
Thus by two applications of Corollary \ref{flatstripcontrol}, 
applied to the train track $\tau$ which is splittable to the
train tracks 
$\tau^\prime,\Pi_{E(\tau,\tau^\prime)}(\eta^\prime)$ and
to the train track $\eta$ which is splittable to 
the train tracks 
$\eta^\prime,\Pi_{E(\eta,\eta^\prime)}(\tau^\prime)$, 
it suffices to show that
the balanced splitting paths
\[\gamma(\tau,\Pi_{E(\tau,\tau^\prime)}(\eta^\prime)),
\gamma(\eta,\Pi_{E(\eta,\eta^\prime)}(\tau^\prime))\] are
uniform fellow travellers.

By the construction of the projections
$\Pi_{E(\tau,\tau^\prime)}(\eta^\prime),
\Pi_{E(\eta,\eta^\prime)}(\tau^\prime)$ 
(see in particular the proof of Lemma \ref{projection2})
and by Lemma \ref{sccontraction}, 
for this it suffices to show that for every 
$R>0,k>0$ there is a constant 
$L(R,k)>0$ with the following property.
Let $\sigma$ be a recurrent train track on $S$ without
closed curve components.
Assume that the number of branches of 
$\sigma$ is at most $k$.
Let $\tau,\eta$ be complete extensions
of $\sigma$ with $d(\tau,\eta)\leq R$. Assume that
every minimal geodesic lamination which is carried
by both $\tau,\eta$ and which is not
a simple closed curve is carried by $\sigma$. 
Let $\tau^\prime,\eta^\prime$ be complete train tracks
which can be obtained from $\tau,\eta$ by a splitting
sequence and such that $d(\tau^\prime,\eta^\prime)\leq R$.
Then the balanced splitting paths $\gamma(\tau,\tau^\prime),
\gamma(\eta,\eta^\prime)$ satisfy
\begin{equation}\label{ellk}
d(\gamma(\tau,\tau^\prime)(i),\gamma(\eta,\eta^\prime)(i))\leq 
L(R,k)\text{ for all }i.\end{equation}
Since $k$ is bounded from above by the number of branches
of complete train tracks on $S$, this is sufficient for the 
proof of the lemma.

We show this claim by induction on $k$. 
Note first that by the results from Section 5 and by
Corollary \ref{flatstripcontrol} we may assume that
$\tau^\prime,\eta^\prime$ are obtained from $\tau,\eta$
be a splitting sequence consisting of splits at
branches contained in a common proper subtrack of 
$\tau,\eta$. Thus 
the case $k=0$ is
immediate from Lemma \ref{sccontraction}. 

Now assume that
the lemma holds true for some $k-1\geq 0$. Let $R>0$ and let
$\tau,\eta\in {\cal V}({\cal T\cal T})$ be such that
$d(\tau,\eta)\leq R$ and that $\tau,\eta$ contain a common
recurrent subtrack $\sigma$ without closed curve components
which carries every minimal geodesic lamination which is
carried by both $\tau,\eta$ and which is not a simple closed
curve. Let $\tau,\eta$ be splittable to 
$\tau^\prime,\eta^\prime$ and assume that
$d(\tau^\prime,\eta^\prime)\leq R$.
By the above remark, for the purpose of the proof
of inequality (\ref{ellk}) 
we may assume that a splitting sequence
connecting $\tau,\eta$ to $\tau^\prime,\eta^\prime$ does not include
a split at a large branch disjoint from $\sigma$.

As in the proof of Lemma \ref{projection2}, we use 
Lemma \ref{induceincones} and Lemma \ref{projection} to find a 
train track $\sigma^\prime$ 
with the following properties.
\begin{enumerate}
\item $\sigma^\prime$ can be obtained from
$\sigma$ by a splitting sequence.
\item There are train tracks $\hat \tau\in E(\tau,\tau^\prime),
\hat \eta\in E(\eta,\eta^\prime)$ which are obtained from
$\tau,\eta$ by a sequence induced from a splitting sequence
connecting $\sigma$ to $\sigma^\prime$.
\item If $\tilde \sigma$ can be obtained from $\sigma$
by a splitting sequence and if there are train tracks
$\tilde \tau\in E(\tau,\tau^\prime),
\tilde \eta\in E(\eta,\eta^\prime)$ obtained from $\tau,\eta$ 
by a sequence induced
by a splitting sequence connecting $\sigma$ to $\tilde \sigma$
then $\tilde \sigma\in E(\sigma,\sigma^\prime)$.
\end{enumerate}

Let $\ell,n>0$ be the length of the balanced splitting
path $\gamma(\tau,\tau^\prime),\gamma(\eta,\eta^\prime)$.
Let $i_0(\tau)\geq 0$ be the smallest number
with the property that
$\gamma(\tau,\tau^\prime)(i_0+1)$
does \emph{not} contain
a \emph{recurrent} subtrack contained in 
$E(\sigma,\sigma^\prime)$.
If no such number exists then define 
$i_0(\tau)=\ell$. Define similarly $i_0(\eta)\geq 0$.

Consider first the case that $i_0(\tau)=\ell$.
Then the train track $\sigma^\prime$ is recurrent.
Let as above $\hat\tau\in E(\tau,\tau^\prime)$ 
be the endpoint of a sequence
issuing from $\tau$ which is induced by a splitting
sequence connecting $\sigma$ to $\sigma^\prime$.
Since $\tau^\prime$ contains $\sigma^\prime$ as a subtrack,
the distance between $\tau^\prime$ and
$\hat\tau$ is uniformly bounded.
Let $\hat\eta\in E(\eta,\eta^\prime)$ be the train track
obtained from $\eta$ by a sequence induced by 
a splitting sequence connecting $\sigma$ to 
$\sigma^\prime$. By Corollary \ref{inducing10}, 
the distance between
$\hat\eta$ and $\hat \tau$ is bounded from above by 
a constant only depending on $R$. As a consequence,
there is a number $\chi_1(R)>0$ only depending on 
$R$ such that
\[d(\eta^\prime,\hat \eta)\leq \chi_1(R).\]
Thus in this case the estimate (\ref{ellk}) follows from
Corollary \ref{flatstripcontrol}, applied to 
$\tau$ which is splittable to $\hat\tau,\tau^\prime$ 
and to $\eta$ which is splittable to $\hat \eta,\eta^\prime$, and 
Corollary \ref{induced}.

In the case that $i_0(\tau)<\ell,i_0(\eta)<n$
let $i_0=\min\{i_0(\tau),i_0(\eta)\}$ and assume that
$i_0=i_0(\tau)$.
We claim that
the distance between $\gamma(\tau,\tau^\prime)(i_0)$ and 
$\gamma(\eta,\eta^\prime)(i_0)$ is bounded from above
by a constant only depending on $R$.
 
Namely,
by the definition of $i_0(\tau)$, 
the train track $\gamma(\tau,\tau^\prime)(i_0)$ contains a 
recurrent subtrack $\zeta_1\in E(\sigma,\sigma^\prime)$.
Since by assumption a splitting sequence connecting
$\tau$ to $\tau^\prime$ does not include a split at a 
large branch which is disjoint from $\sigma$, 
the distance between $\gamma(\tau,\tau^\prime)(i_0)$ and 
the train track $\hat \tau$ obtained from $\tau$ by a sequence
induced by a splitting sequence connecting $\sigma$ to 
$\zeta_1$ is uniformly bounded. Note that
$\hat\tau\in E(\tau,\gamma(\tau,\tau^\prime)(i_0))$ by the
definition of an induced sequence.
By Corollary \ref{flatstripcontrol}, 
applied to the
train tracks $\tau,\hat\tau,\gamma(\tau,\tau^\prime)(i_0)\in
E(\tau,\tau^\prime)$, 
if $i_1\geq 0$ denotes
the length of the balanced splitting path
$\gamma(\tau,\hat\tau)$ then $\vert i_1-i_0\vert$ is 
uniformly bounded. Moreover, there is a constant $\chi_2>0$ 
such that \begin{equation}\label{chi2}
d(\gamma(\tau,\tau^\prime)(i),\gamma(\tau,\hat\tau)(i))\leq \chi_2
\text{ for all }i\leq \min\{i_0,i_1\}.\end{equation}

Let $i_2\geq 0$ be the length
of the balanced splitting path connecting
$\sigma$ to $\zeta_1$. The first step in the proof of 
Lemma \ref{doubleinduced}, 
applied to $\tau,\hat\tau,\gamma(\tau,\tau^\prime)(i_0)$ and their 
recurrent subtracks $\sigma,\zeta_1$, shows
that $i_2\leq \min\{i_0,i_1\}$. Moreover, $i_1-i_2$ is uniformly bounded.
The estimate (\ref{chi2}) then implies that 
there is a universal
constant $\chi_3>0$ such that
\begin{equation}\label{chi3}
d(\gamma(\tau,\tau^\prime)(i_0),\gamma(\tau,\hat\tau)(i_2))\leq \chi_3,
\,d(\gamma(\tau,\hat\tau)(i_2),\hat\tau)\leq \chi_3.
\end{equation}

Let $\hat \eta$ be the endpoint of a sequence issuing from $\eta$
which is induced by a splitting sequence connecting $\sigma$
to $\zeta_1$. By Corollary \ref{inducing10}, the distance
between $\hat\tau,\hat \eta$ is bounded from above by
a constant $p(R)>0$ only depending on $R$.
Corollary \ref{induced} shows that there is a constant
$\chi_4(R)>0$ only depending on $R$ such that
\begin{equation}\label{chi4}
d(\gamma(\tau,\hat\tau)(i_2),\gamma(\eta,\hat \eta)(i_2))\leq \chi_4(R).
\end{equation}
Since $d(\gamma(\tau,\hat\tau)(i_2),\hat \tau)\leq \chi_3$
by the estimate (\ref{chi3}),
this implies that 
\begin{equation}\label{shortinduced}
d(\gamma(\eta,\hat\eta)(i_2),\hat \eta)\leq \chi_4(R)+\chi_3+p(R).
\end{equation}
In particular, if $\nu_2\in E(\sigma,\zeta_1)$ is a recurrent subtrack
of $\gamma(\eta,\hat\eta)(i_3)$ then 
the length of a splitting sequence connecting
$\nu_2$ to $\zeta_1$ is bounded by a number only depending on $R$.

On the other hand, since $i_2\leq i_0$ and since
$\hat \eta$ is splittable to $\eta^\prime$,
Lemma \ref{basicbalanced} shows that 
$\gamma(\eta,\hat\eta)(i_2)$ is splittable
to $\gamma(\eta,\eta^\prime)(i_0)$. Therefore the recurrent subtrack
$\nu_2\in E(\sigma,\zeta_1)$ of $\gamma(\eta,\hat\eta)(i_2)$ is
splittable to a recurrent subtrack $\zeta_2\in E(\sigma,\sigma^\prime)$
of $\gamma(\eta,\eta^\prime)(i_0)$.
Reversing the roles of $\tau,\tau^\prime$ and $\eta,\eta^\prime$
(which is possible since so far we only used the
fact that both $\gamma(\tau,\tau^\prime)(i_0)$ and
$\gamma(\eta,\eta^\prime)(u_0)$ contain a recurrent subtrack
in $E(\sigma,\sigma^\prime)$)
yields a train track $\nu_1\in E(\sigma,\sigma^\prime)$ 
which is splittable to $\zeta_2$ with
a splitting sequence of uniformly bounded length and which is
splittable to $\zeta_1$. 

Lemma 5.4 of \cite{H09} (which is valid for train tracks
which are not necessarily complete) 
shows that there is a unique 
train track $\nu=\Theta_-(\zeta_1,\zeta_2)\in 
E(\nu_2,\sigma^\prime)\cap E(\nu_1,\sigma^\prime)
\subset E(\sigma,\sigma^\prime)$ 
which is splittable to both 
$\zeta_1,\zeta_2$
and such that there is a geodesic in 
$E(\sigma,\sigma^\prime)$
connecting $\zeta_1$ to $\zeta_2$ which passes through 
$\nu$. Since $\nu_1,\nu_2$ are both
splittable to $\zeta_1,\zeta_2$, the length of a splitting sequence
connecting $\nu$ to $\zeta_1,\zeta_2$ does not exceed the length
of a splitting sequence connecting $\nu_2$ to 
$\zeta_1$, $\nu_1$ to $\zeta_2$. 
In particular, this length is bounded from above by a number
only depending on $R$. Moreover, both $\nu_1,\nu_2$ are
splittable to $\nu$ with a splitting sequence of uniformly bounded
length.

As a consequence,
there is a number $\chi_5(R)>0$ only
depending on $R$ such that 
\begin{equation}\label{chi5}
d(\hat\eta,\gamma(\eta,\eta^\prime)(i_0))\leq \chi_5(R).
\end{equation}
The estimates 
(\ref{chi5},\ref{shortinduced},\ref{chi4},\ref{chi3}) then show that
\[d(\gamma(\tau,\tau^\prime)(i_0),
\gamma(\eta,\eta^\prime)(i_0))\leq \chi_5(R)+2\chi_4(R)+2\chi_3+p(R)\]
and hence 
the distance between $\gamma(\tau,\tau^\prime)(i_0+1)$ and
$\gamma(\eta,\eta^\prime)(i_0+1)$ is bounded from above by a constant
only depending on $R$.
Now by the choice of $i_0+1$, the
train track $\gamma(\tau,\tau^\prime)(i_0+1)$ does not
contain a recurrent subtrack which can
be obtained from $\sigma$ by a splitting sequence.
Therefore a minimal geodesic lamination which
is carried by both $\gamma(\tau,\tau^\prime)(i_0+1)$
and $\gamma(\eta,\eta^\prime)(i_0+1)$ and which is not a simple
closed curve 
is carried by a subtrack of $\gamma(\tau,\tau^\prime)(i_0+1)$
with fewer than $k$ branches (compare the discussion in the
proof of Lemma \ref{projection2}).
The lemma now follows from the induction hypothesis, applied
to $\gamma(\tau,\tau^\prime)(i_0),\tau^\prime$ and
$\gamma(\eta,\eta^\prime)(i_0),\eta^\prime$.
\end{proof}

Finally we are ready to complete the proof of 
Theorem \ref{flatconfig3}. We have to show the
existence of a number $L>1$ such that for any
complete train tracks $\tau,\eta\in 
{\cal V}({\cal T\cal T})$ which are splittable
to complete train tracks $\tau^\prime,\eta^\prime$ the balanced
splitting paths $\gamma(\tau,\tau^\prime),
\gamma(\eta,\eta^\prime)$ are weight-$L$ fellow
travellers. 

For this observe first that 
it suffices to consider the particular
case that $\tau,\eta$ are in standard form for 
some marking of $S$. Namely, let $\lambda$ be a complete
geodesic lamination carried by $\tau^\prime$. By invariance
under the action of the mapping class group and
cocompactness, there is a number $\chi_0>0$ 
and there is a train track $\tau_1$ in standard form for
a marking $F$ of $S$ which carries $\lambda$ and such that
\[d(\tau,\tau_1)\leq \chi_0.\]
By Lemma 6.7 and Proposition A.6 of \cite{H09},
$\tau_1$ is splittable to a train track $\tau_1^\prime$
in a uniformly bounded neighborhood of $\tau^\prime$
which carries $\lambda$.
Corollary \ref{near} then shows that
the balanced splitting paths 
$\gamma(\tau,\tau^\prime),\gamma(\tau_1,\tau_1^\prime)$
are uniform fellow travellers.
Hence indeed we may
assume that $\tau,\eta$ are in standard
form for a marking of $S$.

Thus let 
$\tau,\eta$ be complete train tracks in standard form
for markings $F,G$ of $S$. 
Assume that $\tau$ is splittable 
to a complete train track $\tau^\prime$. By Proposition \ref{density}
and by Proposition A.6 of \cite{H09}, there is a universal constant
$\chi_1>0$ and 
there is a train track $\eta_0$ with 
\[d(\eta,\eta_0)\leq \chi_1\] which can be obtained
from a train track in standard form for $F$ by a splitting 
sequence. 
Using the notations from Section 5, 
let $\beta=\Pi_{E(\tau,\tau^\prime)}(\eta_0)
\in E(\tau,\tau^\prime)$ 
be as in Proposition \ref{shortestdistance}.
Let 
$\lambda$ be any complete geodesic lamination
carried by $\tau^\prime$.
By Proposition \ref{shortestdistance}, 
there is a universal number $\kappa>0$ and there are
train tracks 
$\eta_1^\prime\in {\cal V}({\cal T\cal T})$ with 
$d(\tau^\prime,\eta_1^\prime)\leq \kappa$ 
and  $\eta_1\in {\cal V}({\cal T\cal T})$ with 
$d(\eta_0,\eta_1)\leq \kappa$ and hence
$d(\eta,\eta_1)\leq \chi_1+\kappa$ 
with the following properties.
\begin{enumerate}
\item $\eta_1^\prime$ carries $\lambda$.
\item  $\eta_1$ can be connected
to $\eta_1^\prime$ by a splitting sequence.
\item There is a train track $\beta^\prime\in 
E(\eta_1,\eta_1^\prime)$ with 
$d(\beta^\prime,\beta)\leq \kappa$.
\item There is a splitting sequence connecting
a point in the $\kappa$-neighborhood of 
$\tau$ to a point in the $\kappa$-neighborhood of
$\eta_1$ which passes through the $\kappa$-neighborhood of 
$\beta$.
\end{enumerate}

We begin with showing 
that there is a universal number $L_{11}>0$ 
not depending on $\tau,\eta_1,\tau^\prime,\eta_1^\prime$ 
such that the balanced splitting paths 
$\gamma(\tau,\tau^\prime),
\gamma(\eta_1,\eta_1^\prime)$ are weight-$L_{11}$ fellow
travellers.

Since by Theorem \ref{cubicaleuclid} and the remark thereafter 
splitting sequences are uniform quasi-geodesics
in ${\cal T\cal T}$, 
by the third and the forth property above there is a 
universal constant $a>0$ such that
\[d(\tau,\eta)\geq a(d(\tau,\beta)+
d(\beta^{\prime},\eta_1))-1/a.\]
Together with Corollary \ref{flatstripcontrol}, 
applied to the train tracks $\tau,\beta$ which
are splittable to $\tau^\prime$ and to the train tracks
$\eta_1,\beta^\prime$ which are splittable to 
$\eta_1^\prime$,
it now suffices to show that the
balanced splitting paths 
$\gamma(\beta,\tau^\prime),\gamma(\beta^\prime,\eta_1^\prime)$ are 
uniform fellow travellers.
However, since $\tau^\prime,\eta_1^\prime$ both carry
$\lambda$ and since $d(\beta,\beta^\prime)\leq \kappa,
d(\tau^\prime,\eta_1^\prime)\leq \kappa$ 
this follows from Corollary \ref{near}.

We are left with showing that the balanced splitting
paths $\gamma(\eta_1,\eta_1^\prime)$ and 
$\gamma(\eta,\eta^\prime)$ are weight-$L$-fellow
travellers for a universal constant $L>0$.
For this recall that 
$d(\eta,\eta_1)\leq \kappa+\chi_1$.
Choose complete geodesic laminations $\nu,\lambda$
which are carried by $\eta^\prime,\eta_1^\prime$.
By invariance under the action of the mapping class
group and cocompactness, 
there is a marking $F$ of $S$ such that the
distance between $\eta,\eta_1$ 
and any train track in standard form for $F$ is 
uniformly bounded. Let 
$\xi_1,\xi_2$ be train tracks in
standard form for $F$ which carry the complete geodesic
laminations $\nu,\lambda$.
By Lemma 6.7 and Proposition A.6 of \cite{H09},
$\xi_1,\xi_2$ are splittable to 
complete train tracks 
$\xi_1^{\prime}$ and $\xi_2^{\prime}$ which
carry $\nu,\lambda$ and which are contained in a uniformly
bounded neighborhood of $\eta^\prime,\eta_1^\prime$.
By Corollary \ref{near}, 
the balanced splitting paths
$\gamma(\eta,\eta^\prime),\gamma(\xi_1,\xi_1^\prime)$ and
$\gamma(\eta_1,\eta_1^\prime),\gamma(\xi_2,\xi_2^\prime)$
are uniform fellow travellers. Thus it suffices
to assume that $\eta,\eta_1$ are in standard form
for the same marking of $S$.
However, Lemma \ref{onesidedcomb} shows that in this case
the balanced splitting paths 
$\gamma(\eta,\eta^{\prime})$ and
$\gamma(\eta_1,\eta_1^{\prime})$ are
weight-$L_{10}$ fellow travellers.

Together  we conclude that the
balanced splitting paths 
\[\gamma(\tau,\tau^\prime) \text{ and }
\gamma(\eta,\eta^\prime)\] 
are weight-$L$-fellow 
travellers for a universal constant $L>0$. 
This completes the proof of 
Theorem \ref{flatconfig3}.

\section{A biautomatic structure}

In this section we show Theorem \ref{thm1} from the introduction.
Our strategy is to use balanced splitting paths
to construct a regular path system on ${\cal T\cal T}$ 
(or, rather, on an ${\cal M\cal C\cal G}(S)$-graph
obtained from ${\cal T\cal T}$ by adding some edges)
in the sense of \cite{S06}.

For this we first add three families 
of edges to ${\cal T\cal T}$.
The purpose of the edges in the different families is
distinct, and to keep track of this purpose
we color the edge in each of the families
with a fixed color (red, yellow, green).
We then can talk about a colored graph.

For every marking $F$ of $S$, 
connect any two train tracks in standard form for $F$ by
a yellow edge. The resulting graph $G_1$ is of finite
valence and contains ${\cal T\cal T}$ as a subgraph. 
The mapping class group acts properly and cocompactly on $G_1$.
Note that the stabilizer 
in ${\cal M\cal C\cal G}(S)$ of a fixed
train track in standard form for $F$ is 
a subgroup of the stabilizer of the marking $F$ which in turn
is a finite subgroup of  
${\cal M\cal C\cal G}(S)$.

If $\eta$ can be obtained from $\tau$ by a move as
described in Section 6 then add a green edge to the graph $G_1$
which connects $\tau$ to $\eta$.
As before, since $\eta$ can
be obtained from $\tau$ by a splitting sequence of uniformly
bounded length, the resulting extension
$G_2$ of $G_1$ is of finite valence,
and the mapping class group acts cocompactly on this graph.

By Proposition \ref{density} and Proposition A.6 of \cite{H09},
there is a number $\chi>0$ and for every 
marking $F$ of $S$ and every train track $\eta\in 
{\cal V}({\cal T\cal T})$ there is a train track 
$\eta^\prime\in {\cal V}({\cal T\cal T})$ 
which can be obtained from a train track
in standard form for $F$ by a splitting sequence and
such that $d(\eta,\eta^\prime)\leq \chi$.
For any two vertices $\eta,\eta^\prime\in {\cal T\cal T}$ 
with $d(\eta,\eta^\prime)\leq \chi$ 
add a red edge to $G_2$ which connects $\eta$ to $\eta^\prime$.
The resulting colored graph $X$ 
is locally finite
and admits a properly discontinuous
cocompact simplicial action by
${\cal M\cal C\cal G}(S)$. In particular, the inclusion
${\cal T\cal T}\to X$ defines a ${\cal M\cal C\cal G}(S)$-equivariant
quasi-isometry between ${\cal T\cal T}$ and $X$.

Following \cite{S06}, 
a \emph{path} in $X$ is a finite sequence
$e_1,\dots,e_n$ of oriented edges in $X$ such that
for each $i$ the terminal point of $e_i$ equals
the initial point of $e_{i+1}$. 
The mapping class group naturally acts on 
paths in $X$. A ${\cal M\cal C\cal G}(S)$-invariant
set of paths in $X$ is called a 
${\cal M\cal C\cal G}(S)$-invariant 
\emph{path system} ${\cal P}$ 
on $X$. 

Fix a marking $F$ of $S$ and let $\tau$ be a train track
in standard form for $S$. Define a 
${\cal M\cal C\cal G}(S)$-invariant path system ${\cal P}$
on $X$ to consist of all paths $\gamma=(e_1,\dots,e_n)$
with the following properties.
\begin{enumerate}
\item $e_1$ is a yellow edge connecting a train track
$\tau^\prime=g\tau$ for some $g\in {\cal M\cal C\cal G}(S)$
to a (not necessarily different) train track 
$\tau^{\prime\prime}$ in standard form for $gF$.
\item Each of the edges $e_2,\dots,e_{n-1}$ is green, and
the concatenation of these edges defines a 
balanced splitting path connecting the terminal 
point of $e_1$ (a train track in standard form for $F$)
to the terminal point of $e_{n-1}$.
\item The edge $e_n$ is red and connects the terminal
point of $e_{n-1}$ to a train track in the orbit
${\cal M\cal C\cal G}(S)\tau$ of $\tau$.
\end{enumerate}

By construction, each path in ${\cal P}$
starts and terminates at vertices from the orbit
${\cal M\cal C\cal G}(S)\tau$, and for any
two points in ${\cal M\cal C\cal G}(S)\tau$ there is
a path in ${\cal P}$ connecting these two points.
Following \cite{S06}, this means that
${\cal P}$ is \emph{transitive} on 
${\cal M\cal C\cal G}(S)\tau$.
Moreover, ${\cal P}$ is ${\cal M\cal C\cal G}(S)$-invariant.
By Theorem \ref{flatconfig3}, the system ${\cal P}$
satisfies moreover the
\emph{2-sided fellow traveller property}:
There is a constant $c>0$ such that
\[d(\gamma_1(j),\gamma_2(j))\leq 
c(d(\gamma_1(0),\gamma_2(0))+
d(\gamma_1(\infty),\gamma_2(\infty)))+c\text{ for all }j\]
where as before, we view a path as an eventually constant map
defined on $\mathbb{N}$.

Theorem \ref{thm1} from the 
introduction is now an immediate consequence of 
Theorem 1.1 of \cite{S06} and the following

\begin{proposition}\label{regular}
The path system ${\cal P}$ is regular.
\end{proposition}
\begin{proof}
Following p.24 of \cite{S06}, an 
\emph{automaton} is a triple $(M,S_0,S_\Omega)$ 
where $M$ is a directed graph and where
$S_0,S_\Omega$ are subsets of the vertex set $V_M$ of 
$M$ called the sets of \emph{start} and 
\emph{accept} vertices, respectively.
An \emph{automaton over a graph $X$ }
is a quadruple ${\cal M}=(M,m,S_0,S_\Omega)$ 
in which $(M,S_0,S_\Omega)$ is an automaton and
$m:M\to X$ is a simplicial map.
If $G$ is a group of automorphisms of $X$ then
an automaton over $(X,G)$ is an automaton
${\cal M}=(M,m,S_0,S_\Omega)$ over $X$ equipped with
an action of $G$ on $M$ satisfying the following properties.
\begin{enumerate}
\item The sets $S_0$ and $S_\Omega$ are $G$-invariant.
\item The map $m:M\to X$ is $G$-equivariant.
\end{enumerate}

By \cite{S06}, a path system in a graph $X$ is called
regular if there is a finite-to-one automaton
${\cal M}$ over $(X,G)$ such that ${\cal P}={\cal P}({\cal M})$.

We construct an automaton ${\cal M}$ over the graph $X$ 
defined in the beginning of this section as follows.
Define a \emph{train track with traffic control}
to be a complete train track equipped with a labeling
of each large branch with one of four labels
green, red, right, left. 
The set $V_{M}$ of vertices of the graph $M$  
is a union $V_{M}=V_T\cup V_0\cup V_\Omega$ 
where $V_T$ is the set of all train tracks with traffic
control. The set $V_0$ is the set of start vertices,
and the set $V_\Omega$ is the set of accept vertices
which are given as follows.
Let $\Gamma<{\cal M\cal C\cal G}(S)$ 
be the stabilizer of the fixed train track $\tau$ in 
standard form for the marking $F$.
For each $g\in {\cal M\cal C\cal G}(S)/\Gamma$ there is 
an additional vertex $v_0(g\Gamma)\in V_0$ and 
a vertex $v_\Omega(g\Gamma)\in V_\Omega$. 

The directed edges of the graph $M$ are determined
as follows. Every vertex $v_0(g\Gamma)\in V_0$ 
determines a train track $g\tau\in {\cal V}({\cal T\cal T})$ 
in the ${\cal M\cal C\cal G}(S)$-orbit of $\tau$.
The train track $g\tau$ is in standard form for the
marking $gF$. For every train track $\eta$ in standard form
for $gF$ 
connect $v_0(g\Gamma)$ to the train track with traffic control
which is just $\eta$ with all large branches labeled green.
Note that the number of directed edges issuing from 
$v_0(g\Gamma)$ is uniformly bounded.

A train track $[\xi]$ with traffic control is connected
to a train track $[\eta]$ with traffic control by a directed 
edge if 
the train track $\eta$ underlying $[\eta]$ 
can be obtained from the train track $\xi$ 
underlying $[\xi]$ by a move
and if the traffic controls of $[\xi],[\eta]$ satisfy the 
following compability conditions.
\begin{enumerate}
\item 
If a large branch $e$ of 
$\xi$ is labeled red in $[\xi]$ then we require that
a splitting sequence
connecting $\xi$ to $\eta$ does not include 
a split at $e$. Then $\phi(\xi,\eta)(e)$ is a large
branch of $\eta$. We require that the label of this
large branch in the train track with traffic control
$[\eta]$ is red.
\item If $e$ is a large branch in $\xi$ labeled
green, right or left in $[\xi]$ and if 
a splitting sequence connecting $\xi$ to $\eta$
does not include a split at $e$ then 
$\phi(\xi,\eta)(e)$ is a large branch of $\eta$. We
require that the label of this large branch 
$\phi(\xi,\eta)(e)$ in $[\eta]$ is red.
\item Let $e$ be a large branch in $\xi$ whose 
label in $[\xi]$ is right
(or left).   
If a splitting sequence connecting
$\xi$ to $\eta$ includes a split at $e$ then
we require that
this split is a right (or left) split. 
\item 
Assume that $\eta$ can be obtained from $\xi$ by a 
move and let $\rho:[0,n]\to \xi$ be 
a splittable $\eta$-configuration which 
is a (left or right) reduced trainpath.
Let $e_1,\dots,e_s$ $(s\geq 0)$ be the large branches
of $\eta$ contained in $\phi(\xi,\eta)(\rho)$.
If for $i\leq s$ the $\phi(\xi,\eta)(\rho)$-split 
at $e_i$ is a right 
(or left) split then we require that the label
in $[\eta]$ of the large 
branch $e_i$ is left (or right).
\item If $\rho$ is a splittable $\eta$-circle $c$ and
if the train track $\xi^\prime$ obtained from 
$\xi$ by a $\rho-\eta$-multi split is distinct from
$\theta^{\pm}_c(\xi)$ then there are large branches
$e_1,\dots,e_s$ in $\xi^\prime\cap c$ such that
for each $i$ the train track obtained from
$\xi^\prime$ by a right (or left) split at $e_i$ is splittable
to $\theta_c^{\pm}(\xi)$. The branch $\phi(\xi^\prime,\eta)(e_i)$ 
in $\eta$ is
large in $[\eta]$, and we require that the label of this branch
in $[\eta]$ is left (or right). 
\item 
If $\eta$ is obtained from $\tau$ by a move and if
the label of a large branch $e$ of $\eta$ is not determined from 
$\tau$ by one of the five rules above 
then we require that the
label of $e$ is green.
\end{enumerate}

This list of requirements determines for every train track $[\xi]$ 
with traffic control a uniformly bounded number of directed
edges issuing from $[\xi]$. 
Finally, if $[\zeta]$ is any train track with traffic control
and if $g\in {\cal M\cal C\cal G}(S)$ is such that
$d(\zeta,g\tau)\leq \chi$ then we require that there
is a directed edge connecting $[\zeta]$ to $v_\Omega(g\Gamma)$.

The action of the mapping class group on 
${\cal V}({\cal T\cal T})$ induces an action on the graph
$M$. Moreover, there is a finite-to-one 
${\cal M\cal C\cal G}(S)$-equivariant
simplicial map $M\to X$ which maps the set of  
directed path in $M$ connecting 
a vertex in $V_0$ to a vertex in $V_\Omega$ 
onto the path system ${\cal P}$. This shows that the 
path system ${\cal P}$ is indeed regular and complete the
proof of the proposition.
\end{proof}

{\bf Remark:} The automaton which defines a
biautomatic structure on ${\cal M\cal C\cal G}(S)$ is
completely explicit. Since the number of branches of
a complete train track on a surface $S$ of genus
$g\geq 0$ with $m$ punctures does not exceed 
$18g-18+6m$ \cite{PH92} and since for every $p>0$ the number
of trivalent graphs with $p$ vertices is uniformly
exponential in $p$, the complexity of the automaton
(i.e. the number of its states and the cardinality
of its alphabet) is uniformly exponential in the
complexity of $S$.

\appendix
\section{Train tracks hitting efficiently}

In this appendix we construct complete train tracks with some
specific properties which are used to obtain a geometric control
on the train track complex ${\cal T\cal T}$. First,
define a \emph{bigon track} on $S$ to be an embedded 1-complex on
$S$ which satisfies all the requirements of a train track except
that we allow the existence of complementary bigons. Such a bigon
track is called \emph{maximal} if all complementary components are
either bigons or trigons or once punctured monogons. Recurrence,
transverse recurrence, birecurrence and carrying for bigon tracks
are defined in the same way as they are defined for train tracks.
Any complete train track is a maximal birecurrent bigon track in
this sense. 

A \emph{tangential measure} $\nu$ for a maximal bigon track
$\zeta$ assigns to each branch $b$ of $\zeta$ a nonnegative weight
$\nu(b)\in [0,\infty)$ with the following properties. Each side of
a complementary component of $\zeta$ can be parametrized as a
trainpath $\rho$ on $\zeta$ (here as in the case of
train tracks, a trainpath is a $C^1$-immersion 
$\rho:[0,m]\to \zeta$ which maps each segment $[i,i+1]$ 
onto a branch of $\zeta$).
Denote by $\nu(\rho)$ the sum of the
weights of the branches contained in $\rho[0,m]$ counted with
multiplicities. If $\rho_1,\rho_2$ are the two distinct sides of a
complementary bigon then we require that
$\nu(\rho_1)=\nu(\rho_2)$, and if $\rho_1,\rho_2,\rho_3$ are the
three distinct sides of a complementary trigon then we require
that $\nu(\rho_i)\leq \nu(\rho_{i+1})+\nu(\rho_{i+2})$ (where
indices are taken modulo 3). A bigon track is transversely
recurrent if and only if it admits a tangential measure which is
positive on every branch \cite{PH92}.

A bigon track is called \emph{generic} if all switches are
at most trivalent. A bigon track $\tau$ which is not generic can
be \emph{combed} to a generic bigon track
by successively modifying $\tau$ as shown in Figure H.
By Proposition 1.4.1 of \cite{PH92} (whose proof is also valid
for bigon tracks),
the combing of a recurrent bigon track is recurrent.
However, the combing of a transversely recurrent bigon track
need not be transversely recurrent (see the discussion on
p.41 of \cite{PH92}).

\begin{figure}[ht]
\begin{center}
\psfrag{a1}{$a_1$}
\psfrag{bm}{$b_m$}
\psfrag{c1}{$c_1$}
\psfrag{c2}{$c_2$}
\psfrag{Figure C}{Figure H} 
\includegraphics 
[width=0.7\textwidth] 
{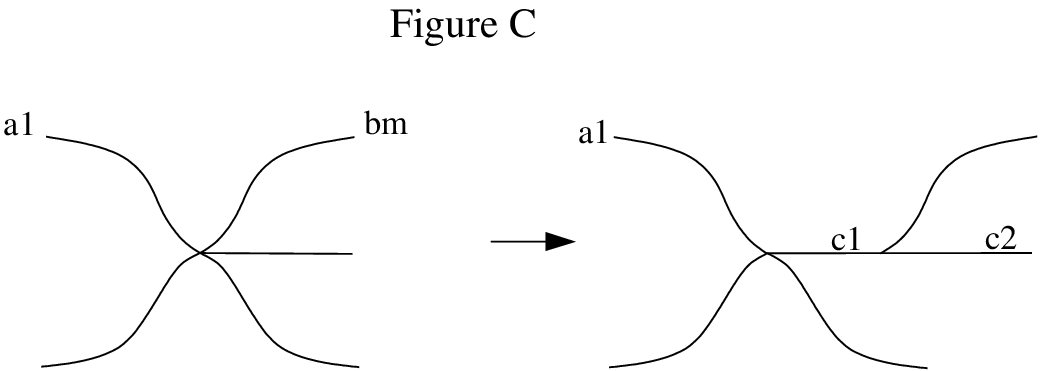}
\end{center}
\end{figure}

The next Lemma gives a criterion for a non-generic
maximal transversely recurrent bigon track to be
combable to a generic maximal transversely recurrent bigon track.
For its formulation, we say that a positive
tangential measure
$\nu$ on a maximal bigon track
$\sigma$ satisfies the \emph{strict triangle inequality
for complementary trigons} if
for every
complementary trigon of $\sigma$ with sides
$e_1,e_2,e_3$ we have $ \nu(e_i)<\nu(e_{i+1})+
\nu(e_{i+2})$. By
Theorem 1.4.3 of \cite{PH92}, a \emph{generic} maximal
train track
is transversely recurrent if and only if it admits a
positive tangential measure satisfying the strict triangle
inequality for complementary trigons.

\begin{lemma}\label{combing} Let $\zeta$ be a maximal bigon track
which admits a positive
tangential measure satisfying the strict triangle
inequality for complementary trigons.
Then $\zeta$ can be combed to a
generic transversely recurrent bigon track.
\end{lemma}
\begin{proof} Let $\sigma$ be an arbitrary maximal
bigon track.
Then $\sigma$
does not have any bivalent switches.
For a switch $s$ of $\sigma$ denote the valence of $s$ by
$V(s)$ and define the excessive
total valence ${\cal V}(\sigma)$ of
$\sigma$ to be
$\sum_s (V(s)-3)$ where
the sum is taken over all switches
$s$ of $\sigma$; then ${\cal V}(\sigma)=0$
if and only if $\sigma$ is generic. By induction it is
enough to show that a maximal non-generic
bigon track $\sigma$ which admits
a positive tangential measure $\nu$ satisfying the
strict triangle inequality for complementary trigons can be
combed to a bigon track $\sigma^\prime$ which admits
a positive tangential measure $\nu^\prime$ satisfying the
strict triangle inequality for complementary trigons and such that
${\cal V}(\sigma^\prime)<{\cal V}(\sigma)$.

For this let $\sigma$ be such a non-generic maximal bigon track
with tangential measure $\nu$ satisfying the strict triangle
inequality for complementary trigons and let $s$ be a switch of
$\sigma$ of valence at least $4$. For a fixed orientation
of the tangent line of $\sigma$ at $s$
assume that there are $\ell$
incoming and $m$ outgoing branches where $1\leq \ell\leq m$ and
$\ell +m\geq 4$. We number the incoming branches in
counter-clockwise order $a_1,\dots,a_\ell$ (for the given
orientation of $S$) and do the same for the outgoing branches
$b_1,\dots,b_m$. Then the branches $b_m$ and $a_1$ are contained
in the same side of a complementary component of $\sigma$, and the
branches $b_{m-1}, b_m$ are contained in adjacent (not necessarily
distinct) sides $e_1,e_2$ of a complementary component $T$ of
$\sigma$. Assume first that $T$ is a complementary trigon. Denote
by $e_3$ the third side of $T$; by assumption, the total weight
$\nu(e_3)$ is strictly smaller than $\nu(e_1)+\nu(e_2)$ and
therefore there is a number $\epsilon\in (0,\min\{\nu(b_i)\mid 1\leq
i\leq m\})$ such that $\nu(e_3)<\nu(e_1)+\nu(e_2)-2\epsilon$. 
Move the endpoint of the branch $b_m$ to a point in the interior of
$b_{m-1}$ as shown in Figure H; we obtain a bigon track
$\sigma^\prime$ with ${\cal V}(\sigma^\prime)< {\cal V}(\sigma)$.

The branch $b_{m-1}$ decomposes in $\sigma^\prime$ into the union
of two branches $c_1,c_2$ where $c_1$ is incident on $s$ and on
an endpoint of the image $b_m^\prime$ of $b_m$ under our move.
Assign the weight $\epsilon$ 
to the branch $c_1$, the weight $\nu(b_m)-\epsilon$
to the branch $b_m^\prime$ and the weight 
$\nu(b_{m-1})-\epsilon$ to the
branch $c_2$. The remaining branches of $\sigma^\prime$ inherit
their weight from the tangential measure $\nu$ on $\sigma$. This
defines a positive weight function on the branches of
$\sigma^\prime$. By the choice of $\epsilon$, this weight function 
defines a 
tangential measure on $\sigma^\prime$ which
satisfies the strict triangle inequality for
complementary trigons.

Similarly, if the complementary component $T$ containing
$b_m$ and $b_{m-1}$ in its boundary is
a bigon or a once punctured monogon, then we can shift $b_m$ along
$b_{m-1}$ as before and modify our tangential measure to a
positive tangential measure on the combed track with the desired
properties. \end{proof}

Our next goal is to transform a maximal recurrent bigon
track $\eta$ which admits a positive tangential measure satisfying
the strict triangle inequality for complementary trigons
to a complete train track. This will be done with 
splits, combings and \emph{collapses of embedded bigons},
where by an embedded bigon we mean a bigon $B$ in $\eta$ 
whose boundary does not have self-intersections. 
A collapse of an embedded bigon $B$ 
consists in identifying the two boundary arcs
of $B$ with a map $F:S\to S$ which is homotopic to the identity
and equals the identity
in the complement of an arbitrarily
small neighborhood of $B$. 
This procedure is described more explicitly in the following lemma.

\begin{lemma}\label{collapse}
Let $\eta$ be a maximal bigon track which admits
a positive tangential measure $\nu$ satisfying the strict triangle
inequality for complementary trigons. Let $B\subset\eta$ be
an embedded bigon. Then there is a maximal bigon track
$\tilde \eta$ which 
carries $\eta$ with a carrying
map which equals the identity
outside a small neighborhood of $B$.
The tangential measure $\nu$ on $\eta$ induces 
a positive tangential measure $\tilde \nu$ on $\tilde \eta$ 
satisfying the strict triangle inequality for complementary trigons.
The number of bigons of $\tilde \eta$ is strictly smaller
than the number of bigons of $\eta$.
\end{lemma}
\begin{proof}
Let $\eta$ be a maximal bigon track which admits
a positive tangential measure $\nu$ satisfying the
strict triangle inequality for complementary trigons.
Assume that $\eta$ contains an embedded bigon $B$.

We construct
from $\eta$ and $\nu$ a maximal birecurrent bigon track $\tilde \eta$
which carries $\eta$. Namely,
assume that the side $E$ of the bigon $B$ 
is the image of an embedded trainpath $\rho_E:[0,\ell]\to \eta$
and that 
the second side $F$ of $B$ 
is the image the embedded trainpath $\rho_2:[0,m]\to \eta$.
Assume also that $\rho_1(0)=\rho_2(0)$. 
We collapse the bigon $B$ to a single arc in $S$
with a homeomorphism $\Psi:E\to F$ which is defined as follows. 

If for
some $p\geq 1$, $q\geq 1$  we have 
\[\sum_{j=1}^{q-1}
\nu(\rho_2[j-1,j])< \sum_{i=1}^p 
\nu(\rho_1[i-1,i])<\sum_{j=1}^q
\nu(\rho_2[j-i,j])\] 
then $\Psi$ maps the subarc $\rho_1[0,p]$ of
$E$ homeomorphically onto a subarc of $F$ which contains
$\rho_2[0,q-1]$ and has its endpoint in the interior of the
branch $\rho_2[q-1,q]$. If 
\[\sum_{i=1}^p\nu(\rho_1[i-1,i])=\sum_{j=1}^q
\nu(\rho_2[j-1,j])\] 
then $\Psi$ maps $\rho_1[0,p]$ onto
$\rho_2[0,q]$, i.e. an endpoint of 
$\rho_1[p-1,p]$ is mapped to an
endpoint of $\rho_2[q-1,q]$. 
The resulting bigon track $\tilde \eta$ carries
$\eta$, and it is maximal. 

There is a natural carrying map $\Phi:\eta \to
\tilde \eta$ which maps each complementary trigon of $\eta$ to a
complementary trigon of $\tilde \eta$. The
positive tangential measure $\nu$ on $\eta$ induces a positive
weight function $\tilde \nu$ on the branches of $\tilde \eta$.
Note that the total weight of $\tilde \nu$ is
strictly smaller than the total weight of $\nu$ and that the
$\nu$-weight of a side $\rho$ of a complementary component $T\not=
B$ in $\eta$ coincides with the $\tilde \nu$-weight of the side
$\Phi(\rho)$ of the complementary component
$\Phi(T)$ in $\tilde \eta$.
In particular, the weight function $\tilde \nu$
is a tangential measure on $\tilde \eta$
which satisfies the strict triangle inequality for
complementary trigons.
The number of complementary bigons in $\tilde \eta$ is strictly
smaller than the number of complementary bigons in $\eta$.
More precisely, there is a bijection 
between the complementary bigons of $\tilde \eta$
and the complementary bigons of $\eta$ distinct from $B$.
The image of the bigon $B$ under the map $\Phi$ is an
embedded arc in $\tilde \eta$. The number of branches of
$\tilde \eta$ does not exceed the number of branches of $\eta$.
\end{proof}

We use Lemma \ref{collapse} to show

\begin{lemma}\label{finitecombinatorics}
For every $m>0$ there is a number $h(m)>0$ with the 
following property.
Let $\eta$ be a maximal bigon track with at most $m$ bigon which 
admits a positive tangential measure satisfying the
strict triangle inequality for complementary trigons.
Let $\lambda$ be a complete geodesic lamination carried by $\eta$.
Then there is a complete train track $\eta^\prime$ 
which carries $\lambda$ and which is obtained from $\eta$ by at
most $h(m)$ splits, combings and collapses.
\end{lemma}
\begin{proof}
We modify a bigon track $\eta$ as in the lemma 
in a uniformly
bounded number of steps to a complete
train track $\eta^\prime$ with
a (non-deterministic) algorithm as follows.

The set of input data
for the algorithm is the set ${\cal B}$ of quadruples
$(\eta,\lambda,\nu,B)$
which consist of a
maximal birecurrent bigon track
$\eta$, a complete geodesic lamination $\lambda$ carried by
$\eta$, a positive tangential measure $\nu$ on $\eta$
which satisfies the strict triangle inequality
for complementary trigons and
a complementary bigon $B$ of $\eta$. If $\eta$ does not
have any complementary bigons, i.e. if
$\eta$ is a train track, then we put $B=\emptyset$.
The algorithm
modifies the quadruple $(\eta,\lambda,\nu,B)\in {\cal B}$ to a
quadruple $(\eta^\prime,\lambda,\nu^\prime,B^\prime)\in {\cal B}$
with $B^\prime=\emptyset$
as follows.

{\sl Step 1:}

Let $(\eta,\lambda,\nu,B)\in {\cal B}$ be an input
quadruple with
$B\not=\emptyset$. 
Check whether the boundary $\partial B$ of $B$ is embedded in $\eta$.
If this is not the case then go to Step 2.

Otherwise collapse $B$ with the procedure described in 
Lemma \ref{collapse}. This yields a bigon track 
$\tilde \eta$ which carries $\eta$ and hence $\lambda$
and which admits 
a positive transverse measure $\tilde \nu$
satisfying the strict triangle inequality for complementary
trigons. The number of bigons of $\tilde \eta$ equals the number
of bigons of $\eta$ minus one.

Choose an input quadruple of the form
$(\tilde \eta,\lambda,\tilde\nu,\tilde B)\in {\cal B}$
for a complementary bigon $\tilde B$ of $\tilde \eta$
(or $\tilde B=\emptyset$ if $\tilde \eta$ is a train track)
and repeat Step 1 with this input quadruple.

{\sl Step 2:}

Let $(\eta,\lambda,\nu,B)\in {\cal B}$ be an input quadruple
such that $B\not=\emptyset$ and that the
boundary $\partial B$ of $B$ is
\emph{not}
embedded. Then the sides $E,F$ of $B$
are immersed arcs of class $C^1$ in $S$ which
intersect tangentially or have tangential self-intersections.
Check whether the two cusp points of $B$ meet at the same 
switch of $\eta$. If this
is not the case, continue with Step 3.

Otherwise the two cusps of $B$ meet at a common
switch $s$ of $\eta$ which is necessarily at least
4-valent. By Lemma \ref{combing} and its proof, we can modify $\eta$ with a
sequence of combings near
$s$ to a maximal birecurrent bigon track $\tilde\eta$
in such a way that the two cusps of the
complementary bigon $\tilde B$ in $\tilde \eta$ corresponding
to $B$ under the combing are distinct and such that
the tangential measure $\nu$ on $\eta$
induces a tangential measure $\tilde \nu$
on $\tilde\eta$ which satisfies the strict
triangle inequality for complementary trigons.
The bigon track $\tilde\eta$ carries the complete
geodesic lamination $\lambda$, moreover the number
of its bigons coincides with the number of bigons of $\eta$.
Continue with Step 1 above with the input
quadruple $(\tilde \eta,\lambda,\tilde \nu,\tilde B)\in {\cal B}$.

{\sl Step 3:}

Let $(\eta,\lambda,\nu,B)\in {\cal B}$ be an input quadruple
where $B$ is a bigon in $\eta$ whose boundary
$\partial B$ is not embedded and whose sides $E,F$ 
have distinct endpoints.
Check whether the boundary $\partial B$ of $B$ contains
any isolated self-intersection points.
Such a point is a switch $s$ contained in
the interior of at least two distinct embedded subarcs
$\rho_1,\rho_2$ of
$\partial B$ of class $C^1$ with the additional property
that $(\rho_1-\{s\})\cap (\rho_2-\{s\})=\emptyset$.
If $\partial B$ does not contain such an isolated
self-intersection point
then go to Step 4.
 
Otherwise any such isolated self-intersection
point $s$ is a switch of $\eta$ which is at
least 4-valent. Thus using once more
Lemma \ref{combing}, we can
modify $\eta$ with a sequence of combings
to a complete birecurrent bigon track
$\tilde \eta$ with the property that
all self-intersection points of the boundary
of the bigon $\tilde B$ in $\tilde \eta$
corresponding to $B$
are non-isolated, i.e. they
are contained in a self-intersection branch, and that
the tangential measure $\nu$ on $\eta$
induces a tangential
measure $\tilde \nu$ on $\tilde \eta$ satisfying
the strict triangle inequality for
complementary trigons.
Continue with Step 4 with the input quadruple
$(\tilde \eta,\lambda,\tilde \nu,\tilde B)\in {\cal B}$.

{\sl Step 4:}

Let $(\eta,\lambda,\nu,B)\in {\cal B}$ be an input quadruple
where $B$ is a complementary bigon for $\eta$
whose boundary $\partial B$ does not contain any
isolated self-intersection points, whose
sides have distinct endpoints and such that
$\partial B$ has tangential
self-intersection branches.

Check whether there is a
branch $e$ of $\eta$ contained in the self-intersection
locus of $\partial B$
which is not incident on any one of the two
cusps $s_1,s_2$ of $B$. If there is no such branch continue
with Step 5. 

Otherwise note that 
since the
interior of the bigon $B$ is an embedded topological disc in $S$,
such a self-intersection 
branch $e$ is necessarily a large branch. Now $\eta$
carries $\lambda$ and therefore
there is a bigon track $\tilde \eta$ which is the
image of $\eta$ under a split at $e$ and which carries $\lambda$.
To each complementary region of
$\tilde \eta$ naturally corresponds a complementary
region of $\eta$ of the same
topological type.
In particular, the number of complementary bigons in
$\tilde\eta$ and $\eta$ coincide, and the bigon
$B$ in $\eta$ corresponds to a bigon $\tilde B$ in
$\tilde \eta$.
The bigon track $\tilde \eta$
is recurrent, and it admits
a positive tangential measure $\tilde \nu$ induced
from the measure $\nu$ on $\eta$
which satisfies
the strict triangle inequality for complementary trigons.
The number of branches contained in the self-intersection locus
of the boundary $\partial \tilde B$ of
the bigon $\tilde B$
is strictly smaller than the number of branches in
the self-intersection locus of $\partial B$.

After a number of splits of this kind which is bounded from
above by the number of branches of the bigon track $\eta$
and hence by the number of bigons of $\eta$ 
we obtain from $\eta$ a
bigon track $\eta_1$ which is maximal and birecurrent.
There is a natural bijection from the
set of complementary bigons of $\eta$
onto the set of complementary bigons of $\eta_1$.
If $B_1$ is the bigon of $\eta_1$ corresponding to $B$
then
the self-intersection locus of the boundary
$\partial B_1$ of $B_1$ is a union of branches
which are incident on one of the two
cusps $s_1\not=s_2$ of $B_1$. As before, $\eta_1$ admits
a positive tangential measure $\nu_1$ which satisfies
the strict triangle inequality for complementary trigons
and is induced from $\nu$. Moreover, $\eta_1$ carries 
the complete geodesic lamination $\lambda$.

If the boundary $\partial B_1$ of $B_1$ is embedded
then we proceed
with Step 1 above for the input
quadruple $(\eta_1,\lambda,\nu_1,B_1)$.
Otherwise proceed with Step 5 using again the
input quadruple $(\eta_1,\lambda,\nu_1,B_1)$.

{\sl Step 5:}

Let $(\eta,\lambda,\nu,B)$ be an input quadruple such that
the cusps of $B$ are distinct and that 
there is no branch of self-intersection 
of the boundary $\partial B$ of $B$ 
which is entirely contained in 
the interior of a side of $B$.
Check whether there is a self-intersection
branch $b$ of $\partial B$ which is incident on
a cusp $s$ of $B$. Note that the branch
$b$ can \emph{not} be large, so it is either small
or mixed.

For a small branch $b$, there are two possibilities
which are shown in Figure I. 
\begin{figure}[ht]
\begin{center}
\psfrag{Figure F}{Figure I}
\includegraphics{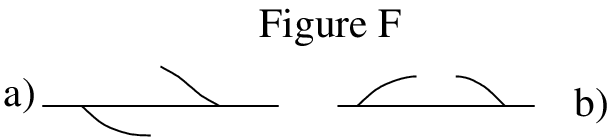}
\end{center}
\end{figure}

Assume first that $b$ is as in the left hand side of 
Figure I. 
The boundary of a 
bigon in a bigon track is the boundary of an embedded disc in $S$
and hence it 
admits a natural orientation
induced from the orientation of $S$. This easily implies that
the small branch $b$ is contained in the intersection
of the two \emph{distinct} sides of $B$.
Since there is no branch in the interior of 
a side of $\partial B$ which is a branch
of self-intersection of $\partial B$, and
since the tangential measure $\nu$ on $\eta$ is
positive by assumption, the construction in
Lemma \ref{collapse} can be used to collapse the bigon
$B$ in $\eta$ to a single simple closed curve.
We obtain a maximal birecurrent bigon
track $\tilde \eta$ which carries $\lambda$ and
admits a positive tangential measure $\tilde \nu$ satisfying
the strict triangle inequality for complementary trigons.
The number of bigons of $\tilde \eta$ is strictly
smaller than the number of bigons of $\eta$.
Choose an arbitrary complementary bigon $\tilde B$
in $\tilde \eta$ or put $\tilde B=\emptyset$ if there
is no such bigon and
continue with Step 1 above and the input
quadruple $(\tilde \eta,\lambda,\tilde \nu,\tilde B)\in {\cal B}$.

A small branch $b$ of self-intersection of $\partial B$ 
as shown on the right hand side of Figure I
can not be collapsed. In this case the
branch $b$ coincides with a side $E$
of the bigon $B$, and the second side $F$ of
$B$ contains $b$ as a proper subarc.
However, this configuration violates the assumption that 
the tangential measure
$\nu$ on $\eta$ is \emph{positive}.

If the branch $b$ is mixed then
$b$ and the cusp $s$ of $B$ are
contained in the interior
of a side $E$ of $B$. The bigon track
$\eta$ can be modified with a single shift to
a maximal birecurrent bigon track $\eta_1$ such
that the switch $s$ is not contained any more
in the interior of a side of the bigon $B_1$
corresponding to $B$ in $\eta$.
The tangential measure $\nu$ on $\eta$
naturally induces a positive tangential
measure $\nu_1$ on $\eta_1$ which satisfies the
strict triangle inequality for complementary trigons.
We now proceed with Step 1 for the input
quadruple $(\eta_1,\lambda,\nu_1,B_1)$.
The algorithm stops if there is no bigon left.
This completes the description of the algorithm.

The algorithm produces a 
not necessarily generic maximal recurrent
train track which admits a positive tangential measure
satisfying the strict triangle inequality for complementary
trigons and hence which can be combed to a complete train track
$\eta^\prime$ carrying $\lambda$. The train track $\eta^\prime$
depends on choices made in the above construction. However,
the number of possibilities in each step is uniformly bounded and
hence $\eta^\prime$ depends on choices among a uniformly
bounded number of possibilities.
The lemma follows.
\end{proof}

A train track $\tau$ on the
surface $S$ \emph{hits efficiently} a train track or a geodesic
lamination $\sigma$ if $\tau$ can be isotoped to a train track
$\tau^\prime$ which intersects $\sigma$ transversely in such a way
that $S-\tau^\prime-\sigma$ does not contain any embedded bigon.
A \emph{splitting and shifting
sequence} is a sequence
$\{\tau_i\}\subset {\cal V}({\cal T\cal T})$
such that for every $i$ the train track $\tau_{i+1}$ can
be obtained from $\tau_i$ by a sequence of shifts and
a single split. Denote by $d$ the distance on ${\cal T\cal T}$.
The following technical result  
relates a train track $\tau$ to train
tracks which hit $\tau$ efficiently. Its proof 
relies on a construction
in Section 3.4 of \cite{PH92} and uses Lemma \ref{finitecombinatorics}.

\begin{proposition}\label{backwards}
There is a number $p>0$ and for every
$\tau\in {\cal V}({\cal T\cal T})$ and every complete geodesic lamination
$\lambda$ which hits $\tau$ efficiently there is a complete train
track $\tau^{*}\in {\cal V}({\cal T\cal T})$
with the following properties.
\begin{enumerate}
\item $d(\tau,\tau^*)\leq p$.
\item $\tau^{*}$ carries $\lambda$.
\item
Let $\sigma\in {\cal V}({\cal T\cal T})$
be a train track which hits $\tau$ efficiently and
carries $\lambda$; then $\tau^{*}$
carries a train track $\sigma^\prime$ which carries
$\lambda$ and can be obtained from $\sigma$ by
a splitting and shifting sequence
of length at most $p$.
\end{enumerate}
\end{proposition}

\begin{proof} By Lemma 3.4.4 and
Proposition 3.4.5 of \cite{PH92}, for every
complete train track $\tau$ there is a maximal birecurrent
\emph{dual bigon track} $\tau_b^*$ with the following property. A
geodesic lamination or a train track $\sigma$ hits $\tau$
efficiently if and only if $\sigma$ is carried by $\tau_b^*$. We
construct the train track $\tau^{*}$ with the properties stated
in the proposition from
this dual bigon track and a complete geodesic
lamination $\lambda\in {\cal
C\cal L}$ which hits $\tau$ efficiently and hence is carried by
$\tau_b^*$.

For this we recall from p.194 of \cite{PH92} the precise
construction of the dual bigon track $\tau_b^*$ of a complete
train track $\tau$. Namely, for each branch $b$ of $\tau$ choose a
compact embedded arc $b^*$ of class $C^1$ 
meeting $\tau$ transversely in a single point in
the interior of $b$ and such that all these arcs are pairwise
disjoint. Let $T\subset S-\tau$ be a complementary trigon of
$\tau$ and let $E$ be a side of $T$ which is composed of the
branches $b_1,\dots,b_\ell$. Choose a point $x\in T$ and extend
all the arcs $b_1^*,\dots,b_\ell^*$ within $T$ in such a way that
they end at $x$, with the same inward pointing tangents at $x$. In
the case $\ell\geq 2$ we then add an arc of class $C^1$ 
which connects $x$ within
$T$ to a point $x^\prime\in T$ and whose 
inward pointing tangent at $x$
equals the outward pointing tangent at $x$ of the arcs
$b_1^*,\dots,b_\ell^*$. We do this in such a way that the
different configurations from the different sides of $T$ are
disjoint. If $y^\prime\in T$ is the point in $T$ arising in this
way from a second side, then we connect $x^\prime$ (or $x$ if
$\ell=1$) and $y^\prime$ by an embedded arc of class $C^1$ 
whose outward pointing
tangent at $x^\prime,y^\prime$ 
coincides with the inward pointing
tangents of the arcs constructed before which end at
$x^\prime,y^\prime$. In a similar way we construct the
intersection of $\tau_b^*$ with a complementary once punctured
monogon of $\tau$. Note that the resulting graph $\tau_b^*$ is in
general not generic, but its only vertices which are not trivalent
arise from the sides of the complementary components of $\tau$.
Figure J shows the intersection of the dual bigon track $\tau_b^*$
with a neighborhood in $S$ of a complementary trigon of $\tau$ and
with a neighborhood in $S$ of a complementary once punctured
monogon.

\begin{figure}[ht]
\begin{center}
\psfrag{a1}{$a_1$}
\psfrag{bm}{$b_m$}
\psfrag{c1}{$c_1$}
\psfrag{c2}{$c_2$}
\psfrag{Figure D}{Figure J} 
\includegraphics 
[width=0.7\textwidth] 
{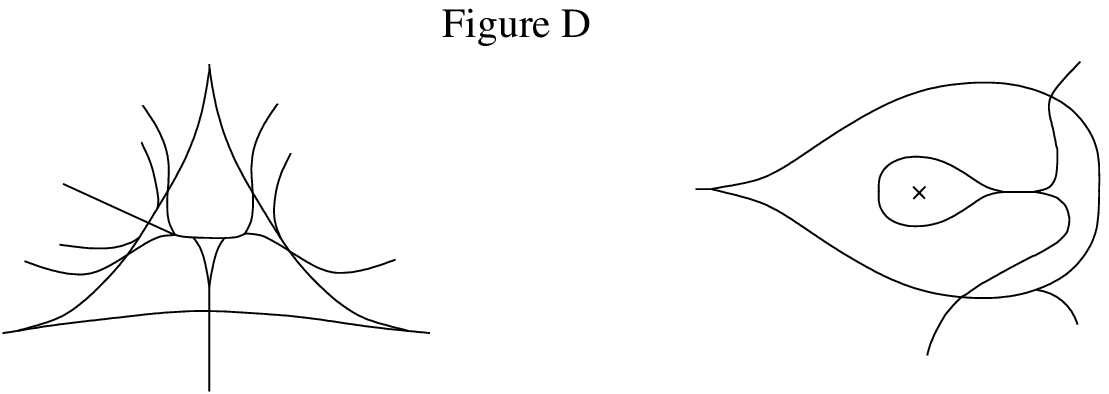}
\end{center}
\end{figure}

Following \cite{PH92}, $\tau_b^*$ is a maximal birecurrent bigon
track, and the number of its branches is bounded from above by a
constant only depending on the topological type of $S$.
Each complementary trigon of
$\tau$ contains exactly one complementary trigon of $\tau_b^*$ in
its interior, and  these are the only complementary trigons. Each
complementary once punctured monogon of $\tau$ contains exactly
one complementary once punctured monogon of $\tau_b^*$ in its
interior. All other complementary components of $\tau_b^*$ are
bigons. The number of these bigons is uniformly bounded.

Now let $\mu$ be a positive integral transverse measure on $\tau$
with the additional property that the $\mu$-weight of every branch
of $\tau$ is at least $4$.
This weight then defines a
\emph{simple multi-curve} $c$ carried by $\tau$
in such a way that $\mu$
is just the counting measure for $c$ (see \cite{PH92}).
Here a simple multi-curve consists of a disjoint
union of essential simple closed curves which
can be realized disjointly; we allow that some of the 
component curves
are freely homotopic. For every
side $\rho$ of a complementary component of $\tau$ there are at
least $4$ connected subarcs of $c$ which are mapped by the natural
carrying map $c\to \tau$ \emph{onto} $\rho$. Namely, the number of
such arcs is just the minimum of the $\mu$-weights of a branch
contained in $\rho$.

Assign to a branch $b^*$ of $\tau_b^*$ which is dual to the branch
$b$ of $\tau$ the weight $\nu(b^*)=\mu(b)$, and to a branch of
$\tau_b^*$ which is contained in the interior of a complementary
region of $\tau$ assign the weight $0$. The resulting weight
function $\nu$ is a tangential measure for $\tau_b^*$, but it is
not positive (this relation between transverse measures
on $\tau$ and tangential measures on $\tau_b^*$ is
discussed in detail in Section 3.4 of \cite{PH92}).
However by construction, every branch of vanishing
$\nu$-mass is contained in the interior of a complementary trigon
or once punctured monogon of $\tau$, and positive mass can be
pushed onto these branches by ``sneaking up'' as described on p.39
and p.200 of \cite{PH92}. Namely, the closed multi-curve $c$ defined by
the positive integral transverse measure $\mu$ on $\tau$ hits the
bigon track
$\tau_b^*$ efficiently. For every branch $b$ of $\tau$
the weight $\nu(b^*)=\mu(b)$ equals the number of
intersections between $b^*$ and $c$.
For each side of a complementary component
$T$ of $\tau$ there are at least $4$ arcs from $c$ which are
mapped by the carrying map onto this side. If the side consists of
more than one branch then we pull two of these arcs into $T$ as
shown in Figure K. If the side consists of a single branch then we
pull a single arc into $T$ in the same way.

\begin{figure}[ht]
\begin{center}
\psfrag{Figure E}{Figure K} 
\includegraphics 
[width=0.7\textwidth] 
{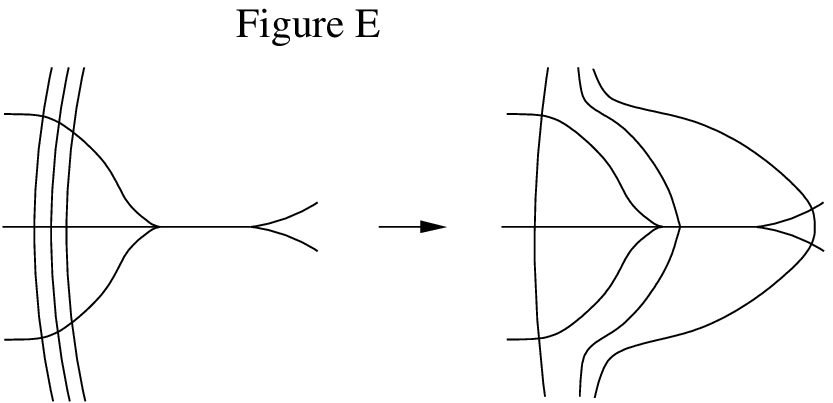}
\end{center}
\end{figure}

For a branch
$e$ of $\tau_b^*$ define $\mu^*(e)$ to be the
number of intersections between $e$ and the deformed
multi-curve. The resulting weight function
$\mu^*$ is
a positive integral tangential measure for $\tau_b^*$.
Note that the weight of each
side of a complementary trigon in $\tau_b^*$ is exactly 2 by
construction, and the weight of a side of a once punctured
monogon is 2 as well. In particular, the tangential measure
$\mu^*$ satisfies the strict triangle inequality for complementary
trigons: If $T$ is any complementary trigon with sides
$e_1,e_2,e_3$ then $\mu^*(e_i)<\mu^*(e_{i+1})+\mu^*(e_{i+2})$
(compare the proof of Lemma \ref{combing}).

Apply the algorithm from the proof 
of Lemma \ref{finitecombinatorics} to the
bigon track $\tau_b^*$, the tangential measure
$\mu^*$ for $\tau_b^*$ constructed from an integral transverse
measure $\mu$ on $\tau$ with $\mu(b)\geq 4$ for every
branch $b$ 
and a complete geodesic lamination $\lambda$ which
hits $\tau$ efficiently and hence is carried
by $\tau_b^*$. Since the number of branches of
$\tau_b^*$ is bounded from above by a constant
only depending on the topological type of
$S$, in at most $p$ steps for a universal number 
$p>0$ the algorithm 
constructs from these data a complete
train track $\tau^*$ which carries $\lambda$.
The train track $\tau^*$ is not unique, and it depends
on $\lambda$ and $\mu$. However,
since the number of steps in the algorithm is 
uniformly bounded and since each step involves
only a uniformly bounded number of choices,
the number of such train tracks which can be obtained
from $\tau_b^*$ by this procedure is uniformly bounded.
Moroever, the algorithm is equivariant with respect to
the action of the mapping class group 
and therefore by invariance under
the action of ${\cal M\cal C\cal G}(S)$, the distance between
$\tau$ and $\tau^*$ is uniformly bounded.
In other words, $\tau^*$ has properties 1) and 2) stated in the
proposition.

To show property 3),
let $\sigma$ be a complete train track on $S$ which hits $\tau$
efficiently and carries a complete geodesic lamination $\lambda\in
{\cal C\cal L}$. Then $\sigma$ is carried by $\tau_b^*$
and hence it is carried by every bigon track which
can be obtained from $\tau_b^*$ by a sequence
of combings, shifts and collapses.
On the other hand, if $\eta_i$ $(0\leq i\leq k)$
are the successive bigon tracks obtained from
an application of the algorithm to
$\eta_0=\tau_b^*$ and if $\eta_i$ is obtained from
$\eta_{i-1}$ by a split at a large branch $e$,
then this split is a $\lambda$-split.
By Lemma A.3 and Lemma 6.7 
of \cite{H09}
(which are local statements and hence they are also
valid for bigon tracks)
there is a universal number $m>0$ such that
if $\sigma$ is carried by $\eta_{i-1}$ and carries
$\lambda$ then there is
a train track $\tilde \sigma$ which
carries $\lambda$, which is carried by $\eta_i$ and which
can be obtained from $\sigma$ by a splitting and shifting
sequence of length at most $m$.
Since the number of splits which occur during the
modification of $\tau_b^*$ to $\tau^*$
is uniformly bounded, this means that
$\tau^*$ also satisfies the third requirement
in the proposition.
This completes the proof of the proposition.
\end{proof}

{\bf Remark:} We call the train
track $\tau^*$ constructed in the
proof of Proposition \ref{backwards} from a complete
train track $\tau$ and a complete geodesic lamination $\lambda$
which hits $\tau$ efficiently
a \emph{$\lambda$-collapse} of $\tau_b^*$. Note
that a $\lambda$-collapse is not unique, 
but the number of different train
tracks which can be obtained from the construction is
bounded by a constant only depending on the topological
type of $S$. In general, a $\lambda$-collaps of $\tau_b^*$ 
is neither carried by the dual bigon
track of $\tau$ nor carries this bigon track.
In particular, in general a $\lambda$-collapse of
$\tau$ does not hit $\tau$ efficiently.

\bigskip

\noindent
MATHEMATISCHES INSTITUT DER UNIVERSIT\"AT BONN\\
ENDENICHER ALLEE 60, D-53115 BONN, GERMANY

\end{document}